\documentclass[10pt,reqno]{amsart}
\usepackage{flafter,amsmath,amssymb,latexsym,psfrag,graphicx,color,indentfirst,bm,amsthm}
\usepackage{geometry}
\usepackage[T1]{fontenc}
\usepackage[utf8]{inputenc}
\usepackage[all]{xy}
\usepackage[colorlinks,
linktocpage=true,
linkcolor=blue,
citecolor=blue,
urlcolor=blue,
anchorcolor=black
]{hyperref}
\usepackage{float}
\usepackage[makeroom]{cancel}
\usepackage{tikz}
\usepackage{circuitikz}
\usepackage{mathtools}
\usepackage[numbers,sort&compress]{natbib}
\usepackage{appendix}
\usepackage{multirow}
\usepackage{caption}
\usepackage{subfigure}
\usepackage{graphicx}
\usepackage{dutchcal}
\usepackage{comment}
\allowdisplaybreaks
\usepackage{tikz}
\usepackage{tikz-cd}
\tikzset{mynode/.style={inner sep=2pt,fill,outer sep=0,circle}}
\usetikzlibrary{fadings}
\usepackage{amsmath,amssymb,braket,tikz}
\usetikzlibrary{tikzmark,calc}

\newtheorem{theorem}{Theorem}[section]
\newtheorem{lemma}{Lemma}[section] 
\newtheorem{corollary}{Corollary}[section] 
\newtheorem{proposition}{Proposition}[section] 
\newtheorem{remark}{Remark}[section]  

\newtheorem{definition}[theorem]{Definition}

\numberwithin{equation}{section}

\newcommand{\R}{\mathbb R}
\newcommand{\N}{\mathbb N}
\newcommand{\T}{\mathbb T}
\newcommand{\Z}{\mathbb Z}

\makeatletter
\newcommand{\Extend}[5]{\ext@arrow0099{\arrowfill@#1#2#3}{#4}{#5}}
\makeatother

\def\({\left(}
\def\){\right)}
\def\<{\left\langle}
\def\>{\right\rangle}

\def\be{{\beta}}

\def\R{{\mathbb{R}}}

\def\T{{\mathbb{T}}}
\def\N{{\mathbb N}}
\def\Z{{\mathbb{Z}}}

\def\J{{\bf J}}

\def\u{{\bf u}}

\def\beaa{\begin{eqnarray*}}
\def\eeaa{\end{eqnarray*}}
\def\ba{\begin{array}}
\def\ea{\end{array}}

\def\be#1{\begin{equation} \label{#1}}
\def \eeq{\end{equation}}

\def\upppi{U_\Delta^3(G_{\alpha}^i,L_{x_1}^2L_{x_2}^2)}

\usepackage{mathrsfs}

\numberwithin{equation}{section}
\numberwithin{figure}{section}

\newcommand\obullet[1]{\ThisStyle{\ensurestackMath{%
			\stackon[1pt]{\SavedStyle#1}{\SavedStyle\kern.6\LMpt\bullet}}}}
\newcommand\ocirc[1]{\ThisStyle{\ensurestackMath{%
			\stackon[1pt]{\SavedStyle#1}{\SavedStyle\kern.6\LMpt\circ}}}}

\begin{document}

\title[On Scattering for 2D quintic NLS under partial harmonic confinement]{On Scattering for  two-dimensional quintic  Schr\"odinger equation under partial harmonic confinement}

\author{Zuyu Ma}
\address{Zuyu Ma
\newline \indent The Graduate School of China Academy of Engineering Physics,
	Beijing 100088,\ P. R. China}
\email{mazuyu23@gscaep.ac.cn}
\author{Yilin Song}
\address{Yilin Song
\newline \indent The Graduate School of China Academy of Engineering Physics,
	Beijing 100088,\ P. R. China}
\email{songyilin21@gscaep.ac.cn}
\author{Ruixiao Zhang}
\address{Ruixiao Zhang
\newline \indent The Graduate School of China Academy of Engineering Physics,
	Beijing 100088,\ P. R. China}
\email{zhangruixiao21@gscaep.ac.cn}

\author{Zehua Zhao}
\address{Zehua Zhao
\newline \indent Department of Mathematics and Statistics, Beijing Institute of Technology, Beijing, China.
\newline \indent Key Laboratory of Algebraic Lie Theory and Analysis of Ministry of Education, Beijing, China.}
\email{zzh@bit.edu.cn}

\author{Jiqiang Zheng}
\address{Jiqiang Zheng
	\newline \indent Institute of Applied Physics and Computational Mathematics,
	Beijing, 100088, China.
	\newline\indent
	National Key Laboratory of Computational Physics, Beijing 100088, China}
\email{zheng\_jiqiang@iapcm.ac.cn, zhengjiqiang@gmail.com}


\maketitle
	
\begin{abstract}
In this article, we study the scattering theory for the following two dimensional defocusing quintic nonlinear Schr\"odinger equation (NLS) with partial harmonic oscillator which is given by
\begin{align}\label{NLS-abstract}
\begin{cases}\tag{PHNLS}
i\partial_tu+(\partial_{x_1}^2+\partial_{x_2}^2)u-x_2^2u=|u|^4u,&(t,x_1,x_2)\in\R\times\R\times\R,\\
u(0,x_1,x_2)=u_0(x_1,x_2).
\end{cases}
\end{align}	

First, we establish the linear profile decomposition for the Schr\"odinger operator $e^{it(\partial_{x_1}^2+\partial_{x_2}^2-x_2^2)}$ by utilizing the classical linear profile decomposition associated with the Schr\"odinger equation in $L^2(\R)$. Then, applying the normal form technique, we approximate the nonlinear profiles using solutions of the new-type quintic dispersive continuous resonant (DCR) system. This allows us to employ the concentration-compactness/rigidity argument introduced by Kenig and Merle in our setting and prove scattering for equation \eqref{NLS-abstract} in the weighted Sobolev space.
  
The second part of this paper is dedicated to proving the scattering theory for this mass-critical (DCR) system. Inspired by Dodson's seminal work [B. Dodson, Amer. J. Math. 138 (2016), 531-569], we develop long-time Strichartz estimates associated with the spectral projection operator $\Pi_n$, along with low-frequency localized Morawetz estimates, to address the challenges posed by the Galilean transformation and spatial translation.
\end{abstract}

\tableofcontents

	
\section{Introduction}
We consider the defocusing quintic nonlinear Schr\"odinger equation with partially harmonic oscillator of the form
\begin{align}\label{NLS}
\begin{cases}
i\partial_tu+\Delta_{\R^2}u-(\alpha x_1^2+x_2^2)u=\mu|u|^4u,&(t,x)=(t,x_1,x_2)\in\R\times\R\times\R,\\
u(0,x)=u_0(x),
\end{cases}
\end{align}
where $u:\R\times\R^2\to\Bbb{C}$ is a complex-valued function and $\mu\in\{-1,1\}$.  The operator $-\Delta_{\R^2}+\alpha x_1^2+x_2^2$ is corresponding to the harmonic oscillator when $\alpha=1$ and the partially harmonic oscillator when $\alpha=0$.  When $\mu=1$, the equation  refers to the defocusing case while $\mu=-1$  refers to the focusing case. The solution to \eqref{NLS} enjoys the following the mass and energy conservation laws:
\begin{align*}
\mbox{Mass: }	M(u(t))&=\int_{\R\times\R}|u(t,x_1,x_2)|^2dx_1dx_2,\\
\mbox{Energy: }E(u(t))&=\int_{\R\times\R}\left(\frac{1}{2}|\partial_{(x_1,x_2)}u(t,x_1,x_2)|^2+\frac{1}{2}(|x_2|^2+\alpha|x_1|^2)|u(t,x_1,x_2)|^2+\frac{\mu}{6}|u(t,x_1,x_2)|^6\right)dx_1dx_2.
\end{align*}

The equation $\eqref{NLS}$ arises from the study of Bose-Einstein condensation(BEC) and can also be derived in NLS with a constant magnetic potential.  In the practical  experiment, the BEC is observed in presence of a confined potential trap and its macroscopic behavior strongly depends on the shape of this trapping potential. We refer the readers to \cite{FO,JP,Pitaeskii} for more physical backgrounds for the equation \eqref{NLS}.

\subsection{Recent progress on the global well-posedness and scattering for \eqref{NLS} and related models}
Before presenting the main results, we briefly review the global well-posedness and scattering result of the following equation
\begin{align}\label{NLS-2}
\begin{cases}
i\partial_tu+\Delta_{x,y}u-(\alpha |x|^2+\beta |y|^2)u=\mu|u|^{p-1}u,&(t,z)=(t,x,y)\in\R\times\R^{d-m}\times\R^{m},\\
u(0,x,y)=u_0(x,y),
\end{cases}
\end{align}
where $\alpha,\beta\in\{0,1\}$, $\mu\in\{-1,1\}$ and $1<p\leqslant 1+\frac{4}{d-2}$.
	
For the case of $\alpha=\beta=0$, then \eqref{NLS-2}  degenerates to the following classical nonlinear Schr\"odinger equation,
\begin{align}\label{NLS-classical}
\begin{cases}
i\partial_tu+\Delta u=\mu|u|^{p-1}u,&(t,x)\in\R\times\R^d,\\
u(0,x)=u_0(x), 1<p\leqslant 1+\frac{4}{d-2}.
\end{cases}
\end{align}
We note that the equation $\eqref{NLS-classical}$ is invariant under the following scaling transformation
\begin{align*}
u(t,x)\rightarrow  u_\lambda(t,x)=\lambda^{\frac{2}{p-1}} u(\lambda^2t,\lambda x),\quad\forall\lambda>0.
\end{align*}  
Moreover, we have the following 
\begin{align*}
\|u_\lambda\|_{\dot{H}^{s_c}(\R^d)}=\|u_0\|_{\dot{H}^{s_c}(\R^d)},\quad s_c=\frac{d}{2}-\frac{2}{p-1}.
\end{align*}
If $s_c=0$,  $\eqref{NLS-classical}$ is called mass-critical and if $s_c=1$,  $\eqref{NLS-classical}$ is called energy-critical.  If $s_c\in(0,1)$, $\eqref{NLS-classical}$ is called energy-subcritical.
	
For the defocusing energy-critical case, i.e. $p=1+\frac{4}{d-2}$ and $\mu=1$,  Bourgain \cite{Bourgain-JAMS} gave the first result on the global well-posedness and scattering for radial initial data with $d=3$. He developed the induction on energy argument and the Lin-Strauss Morawetz estimates to exclude the energy concentration near the origin. Later, Colliander-Keel-Staffilani-Takaoka-Tao \cite{CKSTT-Ann} exploited the interaction Morawetz inequality to prove the scattering for general data. Ryckman-Visan and Visan \cite{Ryckman-Visan,Visan} proved the scattering for $\eqref{NLS-classical}$ with dimensions $d\geq4$. For $\mu=-1$, i.e. the focusing case, Kenig-Merle \cite{Kenig-Merle} established the scattering/blow-up dichotomy of radial solutions in dimension $d\in\{3,4,5\}$. Later, Killip and Visan  \cite{Killip-Visan-AJM} used the double-Duhamel trick to remove the radial assumption of scattering theory in higher dimensions $d\geqslant5$.  Inspired by the work on solving the mass-critical problem \cite{Dodson-d=3,dodson-d=2,Dodson-d=1}, Dodson \cite{D5}  overcame the logarithmic loss in double-Duhamel argument and prove the scattering result for non-radial $\dot{H}^1$ initial data  in  dimension four.   
	
For  the defocusing mass-critical case, i.e. $p=1+\frac{4}{d}$ and $\mu=1$, Tao-Visan-Zhang \cite{Tao-Visan-Zhang2} first proved the existence of the minimal mass blow-up solution which is almost-periodic and reduced the proof of  scattering  to the exclusion of such almost-periodic solution. In  \cite{Tao-Visan-Zhang}, they applied the frequency-localized Morawetz estimate to show the scattering for the mass-critical NLS in higher-dimensions $d\geq3$ with radial initial data.  Afterwards, Killip-Tao-Visan \cite{KTV2009}  proved the scattering when $d=2$. They also treat the case of $\mu=-1$ in the same paper. For the focusing mass-critical NLS when $d\geq3$, Killip-Visan-Zhang \cite{KVZ2008} proved the global well-posedness and scattering for  radial initial data. But  their argument is not available at $d=1$ and non-radial data.  Dodson \cite{Dodson-d=1,dodson-d=2,Dodson-d=3} developed the long-time Strichartz estimate and the frequency-localized Morawetz estimates to prove the scattering for the general initial data when $d\geq1$. For the focusing case, i.e. $\mu=-1$, Dodson \cite{dodson-focusing} utilized the new frequency-localized Morawetz estimates to prove the scattering for $d\geq1$.  
	
For the case of $\alpha=\beta=1$, \eqref{NLS-2} turns to the nonlinear Schr\"odinger equation with harmonic oscillator, which is given by
\begin{align}\label{NLS-harmonic}
\begin{cases}
i\partial_tu+\Delta u-|x|^2u=\mu|u|^{p-1}u,&(t,x)\in\R\times\R^d,\\
u(0,x)=u_0(x)\in\Sigma(\R^d),
\end{cases}
\end{align}
where $\Sigma(\R^d)$ denotes the classical weighted Sobolev space
\begin{align*}
\Sigma(\R^d)=\left\{f\in H^1(\R^d):\|f\|_{L^2}+\|\nabla f\|_{L^2}+\|xf\|_{L^2}<\infty\right\}.
\end{align*}
Since the spectrum of $-\Delta+|x|^2$ are purely discrete, the effect of this operator to the long-time behaviors of the nonlinear problem is similar to the case of NLS posed on the pure torus. In the outstanding work of Colliander-Keel-Staffilani-Takaoka-Kao \cite{CKSTT-Invent}, they constructed the solution to cubic NLS on $\Bbb{T}^2$ may not convergence to the free Schr\"odinger group when $t\to\pm\infty$. Thus the scattering of \eqref{NLS-harmonic}  will not occur for the cubic NLS on $\Bbb{T}^2$. Hani and Thomann \cite{Hani-Thomann} proved the modified scattering of \eqref{NLS-harmonic} with cubic nonlinearity when dimensions $1\leq d\leq5$.  More recently, Chapert \cite{Chapert} constructed the smooth solution to \eqref{NLS-harmonic} whose Sobolev norms blow-up logarithmically as $t\to\infty$.    The global well-posedness of $\eqref{NLS-harmonic}$ in the energy sub-critical case was proved by Carles \cite{Carles1} with $\mu=1$ and Zhang \cite{Z3} with $\mu=-1$, where the mass-energy of the initial data is below that of the ground state.   Later, Killip-Visan-Zhang \cite{Killip-Visan-Zhang} studied the  global well-posedness for the energy-critical case with radial initial data in $d\geqslant3$ and Jao \cite{Jao-CPDE} extended \cite{Killip-Visan-Zhang} to the general data. Jao   \cite{Jao-DCDS} generalized the global well-posedness in \cite{Jao-CPDE} to the case of general potential with quadratic growth.    We also refer to \cite{Jao-Killip-Visan-1D,Jao-APDE} for the partial results concerning the $L^2$ profile decomposition associated to the mass-critical problem. 
	
For the case of $\alpha=0$, $\beta=1$, it is corresponding to the model of NLS with partial harmonic potential. It is worth to emphasize that the impact of the partial harmonic oscillator can be compared to the influence of  waveguide manifold $\R^{d-m}\times\Bbb{T}^m$. The nonlinear Schr\"odinger equation posed on the wave-guide manifold can be written as
\begin{equation}\label{NLS-product}
\begin{cases}
i\partial_t u + \Delta_{x,y}u = \mu |u|^{p-1}u,&(t,x,y)\in\R\times\R^{d-m}\times\Bbb{T}^m,\\
u(0,x,y) = u_0(x,y) \in H^s{(\R^{d-m}\times \mathbb{T}^m)}.
\end{cases}
\end{equation}
For the case of $m=1$, $d\geqslant2$, Tzvetkov-Visciglia \cite{TV} proved the scattering for $1+\frac{4}{d-1}<p<1+\frac{4}{d-2}$. For the right-endpoint, i.e. the energy-critical case, the global well-posedness and scattering was proved in Zhao \cite{Zhao-JDE} with $d=3,4$. For the left endpoint case, i.e. the mass-critical case, Cheng-Guo-Zhao \cite{Cheng-Guo-Zhao} and Cheng-Guo-Yang-Zhao \cite{Cheng-Guo-Yang-Zhao} proved the scattering when $d=2,3$ respectively.   We also refer to Luo \cite{Luo-normalized,Luo-JMPA,Luo-MA} for recent developments regarding the focusing case. For NLS on more general waveguide manifold $\Bbb{R}^{d-m}\times\Bbb{T}^m$,  we  refer to \cite{Cheng-Zhao-Zheng,Hani-Pausader,Ionescu-Pausader,Zhao-d=4,Zhao-Zheng} and references therein.

Inspired by the work of scattering theory on waveguide manifold, one can expect scattering occur for NLS with partial harmonic oscillators. Let  the potential be added in $y$-directions, the strong dispersion may lead to scatter. Antonelli-Carles-Silva \cite{Antonelli} proved the scattering for $\eqref{NLS-2}$ in the weighted Sobolev spaces, where $\mu=1$, $m=1$, $d\in\{2,3,4\}$ and $1+\frac{4}{d-1}<p<1+\frac{4}{d-2}$. The focusing case was treated by Ardila-Carles \cite{AC} when the initial mass-energy is below that of the ground state. The main contributions of   \cite{Antonelli} is to establish the  global-in-time Strichartz estimates for the Schr\"odinger operator $e^{it(-\Delta_{x,y}+|y|^2)}$. If the potential is added in all directions, the effect of the potential implies that the dispersive estimate can only  hold locally-in-time instead of global-in-time. By Mehler's formula, the associated heat kernel can be written as
\begin{align*}
e^{t(\Delta-|x|^2)}(x,y)=e^{\alpha(t)(x^2+y^2)}e^{\frac{\sinh(t)\Delta}{2}}(x,y),\quad(x,y)\in\R^d\times\R^d,
\end{align*}
where $\alpha(t)=\frac{1-\cosh(t)}{2\sinh(t)}$.  Moreover, using the analytic continuation, we can write the solution to the  linear Schr\"odinger equation with harmonic potential explicitly,
\begin{align*}
e^{it(\Delta-|x|^2)}f(x)=\frac{1}{(2\pi i\sin(t))^\frac{d}{2}}\int_{\R^d}e^{i\sin^{-1}t\big(\frac{x^2+y^2}{2}\cos(t)-xy\big)}f(y)dy.
\end{align*}
Then one can verify that
\begin{align*}
\big\|e^{it(-\Delta_{y}+|y|^2)}f(x,y)\big\|_{L_{y}^\infty(\R^m)}\lesssim|\sin(t)|^{-\frac{m}{2}}\|f(x,y)\|_{L_{y}^1(\R^m)},\quad\forall t\notin\frac{\pi}{2}\Bbb{Z}.
\end{align*}
However, by the strong dispersion of the $x$-direction, we have the global-in-time dispersive estimate:
\begin{align*}
\big\|e^{it(-\Delta_{x,y}+|y|^2)}f(x,y)\big\|_{L_{x}^\infty L_{y}^2(\R^{d-m}\times\R^m)}\lesssim|t|^{-\frac{d-m}{2}}\|f(x,y)\|_{L_{x}^1L_{y}^2(\R^{d-m}\times\R^m)}.
\end{align*}
Therefore, we can expect the scattering for the whole range of subcritical and critical regime, i.e. $1+\frac{4}{d-m}\leqslant p\leqslant 1+\frac{4}{d-2}$. However, it is a challenge problem to show the scattering at the two end-points, where the left endpoint is corresponding to the $(d-m)$-dimensional mass critical and the right endpoint is corresponding to the $d$-dimensional energy-critical. For the right endpoint case, the technical difficulty is the delicate global-in-time Strichartz estimates. We refer to   \cite{Hani-Pausader,Zhao-JDE} for more details of the waveguide manifold, which is similar to the partial confinement. For the left end-point,  the global well-posedness is direct and clear. However, obtaining the scattering behavior is hard. Cheng-Guo-Guo-Liao-Shen \cite{Cheng-Guo-Guo-Liao-Shen} proved the scattering for \eqref{NLS-2} where $d=3$, $m=1$ and $p=3$ by utilizing the concentration-compactness argument and Dodson's long-time Strichartz estimate of nonlinear solutions. For the case of $d=2$, $p=3$, Deng-Su-Zheng \cite{Deng-Su-Zheng} established the growth of higher-order Sobolev norms. 
	
In this article, our goal is to establish the scattering theory for left end-point case when $d=2$, which is   mass-critical.

\subsection{Main results}
In this article, we will focus on the following defocusing quintic NLS with partial harmonic oscillator on $\R^2$,
\begin{align}\label{NLS-final}\begin{cases}
i\partial_tu+(\Delta_{\R^2}-x_2^2)u=|u|^4u,&(t,x_1,x_2)\in\R\times\R\times\R,\\
u|_{t=0}=u_0.
\end{cases}\tag{PHNLS}
\end{align}

The solution to \eqref{NLS-final} satisfies the following two conservation laws:
\begin{align*}
\mbox{Mass: }	M(u(t))&=\int_{\R\times\R}|u(t,x_1,x_2)|^2dx_1dx_2,\\
\mbox{Energy: }E(u(t))&=\int_{\R\times\R}\left(\frac{1}{2}|\partial_{x_1,x_2}u(t,x_1,x_2)|^2+\frac{1}{2}|x_2|^2|u(t,x_1,x_2)|^2+\frac{1}{6}|u(t,x_1,x_2)|^6\right)dx_1dx_2.
\end{align*}
By the conservation of energy, it is natural to work on the following weighted Sobolev spaces 
\begin{align*}
\Sigma_{x_1,x_2}(\R\times\R):=\bigg\{f\in L^2(\R^2):\|\nabla f\|_{L^2(\R^2)}+\|x_2f\|_{L^2(\R^2)}+\|f\|_{L^2(\R\times\R)}<\infty\bigg\}.
\end{align*}

Our main result is the global well-posedness and scattering to the equation \eqref{NLS-final}.
\begin{theorem}\label{Thm1}For any $u_0\in\Sigma_{x_1 , x_2}(\R\times\R)$, there exists a unique global solution $u\in C(\R,\Sigma_{x_1 , x_2}(\R\times\R))$. Moreover, $u$ scatters in $\Sigma_{x_1,x_2}(\R\times\R)$ for both time directions, i.e. there exists $u_\pm\in\Sigma_{x_1 , x_2}(\R\times\R)$ such that 
\begin{align*}
\lim_{t\to\pm\infty}\|u(t)-e^{it(\Delta_{\R^2}-x_2^2)}u_{\pm}\|_{\Sigma_{x_1 , x_2}(\R\times\R)}=0.
\end{align*}
\end{theorem}
	
The proof of Theorem \ref{Thm1} is based on the concentration-compactness/rigidity argument introduced by Kenig-Merle \cite{Kenig-Merle}. The main ingredient in this argument is the linear and nonlinear profile decomposition. We also note that the proof of Theorem \ref{Thm1} is based on Theorem \ref{Thm2}, which is given below.
	
In establishing the linear profile decomposition for bounded sequences in $\Sigma_{x_1 , x_2}(\R^2)$, the remainder terms can be only small  in $L^6_{t,x_1} \mathcal{H}_{x_2}^{1-\varepsilon_0}(\R\times\R\times\R)$ with $\varepsilon_0\in(0,\frac{1}{2})$. More precisely, we describe the defect of compactness of the following global-in-time Strichartz estimates
\begin{equation*}
e^{it (\partial_{x_1}^2 + \partial_{x_2}^2 - x_2^2)} : \Sigma_{x_1 , x_2}(\R^2) \to L^6_{t,x_1} \mathcal{H}_{x_2}^{1-\varepsilon_0}(\R\times\R\times\R).
\end{equation*}
We remark that we cannot obtain the remainder term appearing in the linear profile decomposition is small in $L_{t,x_1}^6\mathcal{H}_{x_2}^{1}$. Thus we replace $\mathcal{H}_{x_2}^1$ by the weaker Sobolev space $\mathcal{H}_{x_2}^{1-\varepsilon_0}$ for some $\varepsilon_0\ll1$. Fortunately, we observe that the weaker space-bound $L_{t,x_1}^6\mathcal{H}_{x_2}^{1-\varepsilon_0}$ can lead to scatter.
	
To accomplish such linear profile decomposition,  we prove the following type linear profile decomposition which can be understood as   linear profile decomposition with the spectral projector,
\begin{equation*}
e^{it \partial_{x_1}^2} : L^2_{x_1} \mathcal{H}_{x_2}^{1}(\R\times\R) \to L^6_{t,x_1} \mathcal{H}_{x_2}^{1-\varepsilon_0}(\R\times\R\times\R).
\end{equation*}
	
Next, we focus on the nonlinear profile decomposition. The nonlinear profiles are solutions to equation $\eqref{NLS-final}$ with initial data being the bubble function in linear profile decomposition. The main difficulty is to find a suitable approximating solution at the large scale. To do this, we consider the sequence of solutions with the highly-concentrated initial data as the nonlinear profiles,
\begin{align}\label{nonlinaer-profile}
\begin{cases}
i\partial_tu_\lambda+(\partial_{x_1}^2+\partial_{x_2}^2-x_2^2)u_\lambda=|u_\lambda|^4u_\lambda,\\
u_\lambda(0,x_1,x_2)=\frac{1}{\lambda^\frac{1}{2}}\phi(\frac{x_1}{\lambda},x_2), \quad \lambda > 0.
\end{cases}
\end{align} 
Here we only rescale the initial data $\phi$ in one direction instead of the whole directions. Acting the operator $e^{it (\partial_{x_2}^2 - x_2^2)}$ to both the nonlinear profile $u_{\lambda}$ and denote the new solution by the following 
\begin{align*}
w_\lambda(t,x_1,x_2)=e^{-it(\partial_{x_2}^2-x_2^2)}u_\lambda(t,x_1,x_2).
\end{align*}
By direct calculation,  we have
\begin{align*}
\begin{cases}
i\partial_tw_\lambda+\partial_{x_1}^2w_\lambda=e^{-it(\partial_{x_2}^2-x_2^2)}\left(|e^{it(\partial_{x_2}^2-x_2^2)}w_\lambda|^4e^{it(\partial_{x_2}^2-x_2^2)}w_\lambda\right),\\
w_\lambda|_{t=0}=\lambda^{-\frac{1}{2}}\phi(\frac{x_1}{\lambda},x_2).
\end{cases}
\end{align*}
By using the scaling transformation, we denote
 $$\tilde{v}(t,x_1,x_2)=\lambda^\frac{1}{2}w_\lambda(\lambda^2t,\lambda x_1,x_2).$$
It is clearly that $\tilde{v}$ satisfies
\begin{align*}
\begin{cases}
i\partial_t\tilde{v}+\partial_{x_1}^2\tilde{v}=e^{-i\lambda^2t(\partial_{x_2}^2-x_2^2)}\left(\left|e^{i\lambda^2 t(\partial_{x_2}^2-x_2^2)}\tilde{v}\right|^4e^{i\lambda^2 t(\partial_{x_2}^2-x_2^2)}\tilde{v}\right),\\
\tilde{v}|_{t=0}=\phi(x_1,x_2).
\end{cases}
\end{align*}
Let $\Pi_n$ be the spectral projector associated to the harmonic oscillator $-\partial_{x_2}^2+x_2^2$.
Acting the spectral projector $\Pi_n$ on both side of the above equation and decomposing $\tilde{v}=\sum\limits_{n\in\N}\Pi_{n}\tilde{v}$, then we get
\begin{align*}
\begin{cases}
i\partial_t\tilde{v}_n+\partial_{x_1}^2\tilde{v}_n=e^{i\lambda^2t(2n+1)}\Pi_n\Big(\sum\limits_{\substack{n_1,n_2,n_3,n_4,n_5\in\N}}e^{-i\lambda^2t(1+2n_1-2n_2+2n_3-2n_4+2n_5)}\tilde{v}_{n_1}\overline{\tilde{v}_{n_2}}\tilde{v}_{n_3}\overline{\tilde{v}_{n_4}}\tilde{v}_{n_5}\Big),\\
\tilde{v}_n|_{t=0}=\phi_n\stackrel{\triangle}{=}\Pi_n\phi.
\end{cases}
\end{align*}
	
Taking $\lambda\to\infty$ and to make sure that the right-handside of the above equation is well-defined, we have $n_1-n_2+n_3-n_4+n_5=n$ in the summation. Thus we derive the limiting equation
\begin{align}\label{limit-equation}
\begin{cases}
i\partial_t v_n+\partial_{x_1}^2v_n=\sum\limits_{\substack{n_1,n_2,n_3,n_4,n_5\in\Bbb{N}\\n_1-n_2+n_3-n_4+n_5=n}}\Pi_n(v_{n_1}\overline{v_{n_2}}v_{n_3}\overline{v_{n_4}}v_{n_5}),\\
v_n|_{t=0}=\phi_n(x_1,x_2).
\end{cases}
\end{align}
Using the relation between $w_\lambda$ and $u_\lambda$ and sum over $n\in\N$, we derive the nonlinear profiles $u_\lambda$ as
\begin{align*}
u_{\lambda}=e^{it(\partial_{x_2}^2-x_2^2)}\sum_{n\in\N}\lambda^{-\frac{1}{2}}v_n\left(\frac{t}{\lambda^2},\frac{x_1}{\lambda},x_2\right).
\end{align*}
Therefore, it  remains to prove the scattering result for the following equation, which is called the dispersive continuous resonant system(DCR),
\begin{align}\label{DCR}
\begin{cases}\tag{DCR}
i\partial_tv+\partial_{x_1}^2v=F(v),\\
v|_{t=0}=\phi(x_1,x_2),
\end{cases}
\end{align}
where 
\begin{align*}
F(v)=\sum_{\substack{n_1,n_2,n_3,n_4,n_5,n\in\N\\n_1-n_2+n_3-n_4+n_5=n}}\Pi_n(v_{n_1}\overline{v_{n_2}}v_{n_3}\overline{v_{n_4}}v_{n_5}).
\end{align*}
The study of (DCR) system arises from the relativistic quantum mechanics and weak turbulence theory. 	In \cite {Masmoudi-Nakanishi}, the authors showed that the nonlinear Klein-Gordon equation can be approximated by two coupled Schr\"odinger system as the speed of light tends to infinity in energy-space. Recently, Hani-Pausader-Tzvetkov-Visciglia \cite{Hani-Pausader-Tzvetkov-Visciglia} studied the modified scattering of the cubic nonlinear Schr\"odinger equation on waveguide $\R\times\Bbb{T}^d$ with $1\leq d\leq4$ by reducing the problem to the continuous resonant Schr\"odinger system. The (DCR) system can be understood as the dispersive-generalized version of their continuous resonant systems. The result of \cite{Hani-Pausader-Tzvetkov-Visciglia} is tightly linked to the weak turbulence theory. 
	
In our second result, we prove the global well-posedness and scattering for  \eqref{DCR}.
\begin{theorem}\label{Thm2}
For any $\phi\in L_{x_1}^2\mathcal{H}_{x_2}^1(\R\times\R)$, there exists a unique global solution $v\in C(\R,L_{x_1}^2\mathcal{H}_{x_2}^1(\R\times\R))\cap L_{t,x_1}^6\mathcal{H}_{x_2}^{1}(\R\times\R\times\R)$. Moreover, $v$ scatters in the both time directions, i.e. there exist $v_\pm\in L_{x_1}^2\mathcal{H}_{x_2}^1(\R\times\R)$ such that
\begin{align}\label{scattering-def}
\lim_{t\to\pm\infty}\big\|v(t)-e^{it\partial_{x_1}^2}v_\pm\big\|_{L_{x_1}^2\mathcal{H}_{x_2}^1(\R\times\R)}=0,
\end{align}
where $\mathcal{H}_{x_2}^1$ denotes the Sobolev spaces associated to the harmonic oscillator equipped with the norm
\begin{align*}
\|f\|_{\mathcal{H}^{1}(\R)}^2=\sum_{n\in\N}(2n+1)\|\Pi_nf\|_{L^2(\R)}^2.
\end{align*}
\end{theorem}

Since the remainder term appearing  linear profile decomposition can be  small in $L_{t,x_1}^6\mathcal{H}_{x_2}^{1-\varepsilon_0}$ for $\varepsilon_0\in(0,1]$, thus we need to prove that the weaker regularity in $x_2$ variable can also imply the scattering. By using the nonlinear estimate of \eqref{DCR}, we have the following scattering criterion.
\begin{proposition}[Scattering Criterion]
Let $u$ be the solution to \eqref{DCR}. If $u$ obeys
\begin{align*}
\|u\|_{L_{t,x_1}^6L_{x_2}^2(\R\times\R\times\R)}<\infty,
\end{align*}
then we have
\begin{align*}
\|u\|_{L_{t,x_1}^6\mathcal{H}_{x_2}^1(\R\times\R\times\R)}<\infty.
\end{align*}
Furthermore, $u$ scatters in the sense of \eqref{scattering-def}.
\end{proposition}
	
\begin{remark}
 We  emphasize that due to the better decay of the bilinear estimate and the higher-order nonlinear terms, we can use it to derive the long-time Strichartz estimate instead of  complicated bilinear estimates (cf. Lemma 7.12 in \cite{Cheng-Guo-Guo-Liao-Shen}).   
\end{remark}

\begin{remark}
We note that the $\eqref{DCR}$-system is a mass-critical Schr\"odinger system, which can be compared with 1D mass-critical NLS. Thus, naturally the proof of Theorem \ref{Thm2} is inspired by the seminal work of Dodson \cite{Dodson-d=1} in $1D$. See also \cite{dodson-d=2,Dodson-d=3} for higher dimensions.
\end{remark}

\begin{remark}
In \cite{Cheng-Guo-Zhao}, the authors studied the resonant system associated to the quintic NLS on $\R\times\Bbb{T}$. One main difference between two cases is that we prove better nonlinear estimates for our case, which allows us to reduce the regularity to $L^2$ level with respect to $x_2$ variable, which is \textbf{different} from \cite{Cheng-Guo-Zhao}. See Section 5 for more details. 
\end{remark}

\begin{remark}
In this paper, we study the long time behavior of the defocusing energy-subcritical/mass-critical nonlinear Schr\"odinger equations with partial harmonic oscillators. It is of great interests to study the scattering of the focusing problem. The main difficulty is to classify the threshold of the scattering/blowing-up dichotomy. We refer to \cite{AC,BBJV,Stefanov} for more results.
\end{remark}

 \begin{remark}
It is also interesting to study the long time behavior of the defocusing energy-critical nonlinear Schr\"odinger equations with partial harmonic oscillators such as 4D cubic NLS with 1D (or 2D) partial harmonic oscillation and 3D quintic NLS with 2D partial harmonic oscillation. More ingredients are expected to be necessary. We refer to \cite{Hani-Pausader} and \cite{Ionescu-Pausader} for waveguide-type models which share similar spirits. We leave them for interested readers.
\end{remark}

\subsection{Outline of the proof}
The nonlinear model considered in this article can be compared to the NLS on the waveguide $\R\times\Bbb{T}$, which was studied by Cheng-Guo-Zhao \cite{Cheng-Guo-Zhao}. The key \textbf{difference} between the two cases is that in the partial harmonic oscillator model, the eigenfunction of harmonic oscillator do not enjoy the algebraic property, which makes this problem more difficult. Differences between the two types of models (partial harmonic potential v.s. waveguide)  will be discussed in detail in the following sections.
	
The proof of the main results can be separated into two parts. In the first part, we prove the scattering for equation \eqref{NLS-final} under the assumption of Theorem \ref{Thm2} via the concentration-compactness/rigidity argument. First, we give the linear profile decomposition of the Schr\"odinger operator with partial harmonic oscillator $e^{itH_p}$, where
\begin{align*}
H_p=-\partial_{x_1}^2-\partial_{x_2}^2+x_2^2,\quad x=(x_1,x_2)\in\R\times\R.
\end{align*}
This profile decomposition can be proved via the classical  $L^2$ profile decomposition which was developed by Merle-Vega \cite{Merle-Vega} and Begout-Vargas \cite{Begout-Vargas}.
	
By constructing the approximated solution, we prove the existence of the critical element when assuming that Theorem \ref{Thm1} fails. We emphasize that Theorem \ref{Thm2} is crucial in proving the  existence of such special solutions.  Since the spectral projector cannot commute with the nonlinear terms, we need to use the limiting equations, which is supposed to  be global and scatter in both time directions. This method was borrowed from  Hani-Pausader \cite{Hani-Pausader} and Ionescu-Pausader \cite{Ionescu-Pausader}. Then, we use the normal form technique to obtain the additional regularity of the limiting equations, which can help us  approximate the nonlinear profiles. The presence of the spectral projection brings some difficulties in proving that the error terms is small in $L_{t,x_1}^6\mathcal{H}_{x_2}^{1}\cap L_t^\infty L_{x_1}^2\mathcal{H}_{x_2}^1$.   Inspired by the work of \cite{Cheng-Guo-Guo-Liao-Shen}, we use the normal form method to obtain the additional regularity. Finally, by the interaction Morawetz estimates and the compactness of the critical element, we can preclude the possibility of such solution, thus can  complete the proof of Theorem \ref{Thm1}. We refer to Section \ref{sec:profile} for more details.
	
The second part is devoted to prove Theorem \ref{Thm2}. The $\eqref{DCR}$ system is a  mass-critical nonlinear Schr\"odinger equation with respect to the $x_1$ variable. The spectral resonance on the second variable makes the problem more complicated.  Therefore, we will follow the idea of Dodson \cite{Dodson-d=1} which give the scattering for the one-dimensional mass-critical nonlinear Schr\"odinger equations. Note that the $\eqref{DCR}$ system is scale-invariant, thus the reduction to the minimal mass blow-up solution to the equation \eqref{NLS-classical} with $p=1+\frac{4}{d}$ can be directly applied to our case. We also observe that the linear profile decomposition established in the first part can be applied to this part as well. In other words, the exclusion of the almost-periodic solution is more complicated. Using the standard argument as in Tao-Visan-Zhang \cite{Tao-Visan-Zhang}, we can show that if  Theorem \ref{Thm2} fails, then there exist the minimal blow-up solution, which is concentrated in space and frequency. 

\begin{proposition}[Reduction to the almost-periodic solution]
Assume that Theorem \ref{Thm2} fails, then there exists a minimal mass $m_0$ and a non-trivial minimal blow-up solution $v_c\in C(I,L_{x_1}^2\mathcal{H}_{x_2}^1(\R\times\R))$ with $M_S(v_c)=m_0$ and
\begin{align*}
\|v_c\|_{L_{t,x_1}^6L_{x_2}^2((I\cap(t_0,\infty))\times\R\times\R)}=\|v_c\|_{L_{t,x_1}^6L_{x_2}^2((I\cap(-\infty,t_0^\prime))\times\R\times\R)}=\infty, \quad\forall t_0, t_0'\in I,
\end{align*}
where the maximal lifespan $I\supset [0,\infty)$ and $M_S(u)$ is defined in \eqref{equ:Msudef} below.  Moreover, there exists $C(\eta)>0$ and the parameters $N(t):I\to(0,1]$, $\xi(t): I\to\R$ and $x_1(t):I\to\R$ such that
\begin{align}\label{compact-intro}
\int_{|x_1-x_1(t)|\geqslant\frac{C(\eta)}{N(t)}}\|v_c(t,x_1,x_2)\|_{\mathcal{H}_{x_2}^1(\R)}^2dx_1+\int_{|\xi-\xi(t)|\geqslant C(\eta)N(t)}\|\mathcal{F}_{x_1}v_c(t,\xi,x_2)\|_{\mathcal{H}_{x_2}^1(\R)}^2d\xi<\eta.
\end{align}
Furthermore, by the invariance of Galilean transform, we can take $N(0)=1$, $\xi(0)=0$ and 
\begin{align*}
|N^\prime(t)|+|\xi^\prime(t)|\lesssim N^3(t).
\end{align*} 
\end{proposition}
	
\begin{remark}
By the standard compactness argument, $\eqref{compact-intro}$ can be expressed as the following
\begin{align*}
\left\{\frac{1}{N(t)^\frac12}e^{-ix_1\cdot\xi(t)}v_c\left(t,\frac{x_1-x_1(t)}{N(t)},x_2\right):t\in I\right\} \mbox{ is pre-compact in }L_{x_1}^2\mathcal{H}_{x_2}^1(\R\times\R).
\end{align*}
	
\end{remark}

Since the (DCR) system is $L^2$-critical, to preclude the possibility of the almost-periodic solution in the sense of \eqref{compact-intro}, utilizing the interaction Morawetz estimate is usually necessary. Unfortunately,  $v_c$ is only in $L_{x_1}^2\mathcal{H}_{x_2}^1(\R\times\R)$, thus we cannot prove that the right-hand side of the interaction Morawetz estimate is bounded. But inspired by Dodson's seminal work \cite{Dodson-d=1}, we can further reduce the minimal blow-up solutions to the following two scenarios:

\begin{proposition}[Two special scenarios for blow-up]Suppose that Theorem \ref{Thm2} fails, then there exists a minimal blow-up solution $v_c:I\times\R\times\R\rightarrow\Bbb{C}$ with the maximal lifespan $ I\supset [0,\infty)$ obeying \eqref{compact-intro} and $N(0)=1$. Moreover, we have the following two scenarios of blow-up solution
\begin{enumerate}
\item The rapid frequency cascade solution:
\begin{align*}
\int_{0}^{\infty}N(t)^3dt<\infty,
\end{align*}
			
\item The quasi-soliton case:
\begin{align*}
\int_{0}^{\infty}N(t)^3dt=\infty.
\end{align*}
\end{enumerate}
	
\end{proposition}
	
In order to preclude the possibility of quasi-soliton case, following Dodson's strategy in \cite{Dodson-d=1}, we truncate the solutions to low frequency, since the low-frequency localized function is bounded in $\dot{H}^{\frac12}$. To control the error terms produced by this truncation, we need to  control  the high-frequency part carefully and delicately. To do this, we introduce the long-time Strichartz estimate, which can be understood as an  infinite vector-valued analogue of the long-time Strichartz estimate in \cite{Dodson-d=1}.
	
Before  introducing the long-time Strichartz estimate, we first give some notations and definitions. By the compactness of the minimal blow-up solutions, we can choose three small constants $0<\eta_3\ll\eta_2\ll\eta_1\ll1$ and $\eta_3<\eta_2^{100}$ satisfying
\begin{align}
&|N^\prime(t)|+|\xi^\prime(t)|\leqslant 2^{-20}\eta_1^{-\frac{1}{2}}N(t)^3,\label{xi(t)-intro}\\
&\int_{|x_1-x_1(t)|\geq \frac{2^{-20}\eta_3^{-\frac{1}{2}}}{N(t)} }\|v_c\|_{\mathcal{H}_{x_2}^1(\R)}^2dx_1+\int_{|\xi-\xi(t)|\geq 2^{-20}\eta_3^{-\frac{1}{2}}N(t) }\|\mathcal{F}_{x_1}v_c(t,\xi,x_2)\|_{\mathcal{H}_{x_2}^1(\R)}^2d\xi<\eta_2^2.\label{almost-compact-2-intro}
\end{align} 
Let $k_0$ be a positive integer and let $[a,b]$ be a compact time interval obeying
\begin{align}\label{scaling-0-intro}
\|v_c\|_{L_{t,x_1}^6L_{x_2}^2([a,b]\times\R\times\R)}=2^{k_0}.
\end{align}
By the rescaling argument, we can let
\begin{align}\label{scaling-intro}
\int_{a}^{b}N(t)^3dt=\eta_32^{k_0}.
\end{align}
	
We note that the long-time Strichartz estimate is not a priori estimate. It is proved for the almost-periodic solution. It is worth to point out  since the double endpoint Strichartz estimate fails in dimension one, we establish the long-time Strichartz estimate using the variant of atomic space $U_\Delta^p$, which was first observed in Dodson \cite{Dodson-d=1}. 

\begin{theorem}[Long-time Strichartz estimate]\label{longtimestrichartz-intro}
If $u$ is an almost periodic solution to \eqref{DCR} with the maximal lifespan $I\supset[0,\infty)$, then for any positive $k_0$, $\eta_1$, $\eta_2$, $\eta_3$ satisfying $\eqref{xi(t)-intro}$-$\eqref{scaling-intro}$ and $\eta_3<\eta_2^{100}$, we have
\begin{align}\label{longtime-intro}
\|u\|_{\widetilde{X}_{k_0}([a, b],L_{x_1}^2L_{x_2}^2)}\lesssim 1,
\end{align}
and the implicit constant in \eqref{longtime-intro} does not depend on $k_0$.  We refer to Section \ref{sec:long} for the  definition of the long-time Strichartz norm $\widetilde{X}_{k_0}$.
		
\end{theorem}

\begin{remark}[A quick comparison with the cubic analogue]
In \cite{Cheng-Guo-Guo-Liao-Shen}, the authors proved the long-time Strichartz estimate for cubic (DCR) system when $d=2$. We emphasize that the proof for our case is quite \textbf{different} from their proof. In our setting, the bilinear estimates have the better decay  and the nonlinear term is quintic, so we do not need to utilize the interaction Morawetz estimate to establish the refined bilinear estimates for the minimal blow-up solution. 
\end{remark}
	
By using the  above estimate, we can prove the following truncated interaction Morawetz estimate
\begin{align}
\bigg\|\int_{\R}\partial_{x_1}\left(\left|P_{\leqslant K}^{x_1}v_c(t,x_1,x_2)\right|^2\right)dx_2\bigg\|_{L_{t,x_1}^2([0,T]\times\R)}^2\lesssim o(K)+\sup_{t\in I}M_I(t)
\end{align}
where $\eta_3K=\int_{0}^{T}N(t)^3dt$ with $\eta_3$ sufficiently small and 
\begin{align*}
M_I(t)&=\int_{\R\times\R}\int_{\R\times\R}\frac{x_1-\tilde{x_1}}{|x_1-\tilde{x_1}|}|P_{\leqslant K}^{x_1}v_c(t,\tilde{x_1},\tilde{x_2})|^2\Im(\overline{P_{\leqslant K}^{x_1}v_c}\partial_{x_1}v_c)(t,x_1,x_2)dx_1d\tilde{x_1}dx_2d\tilde{x_2}\\
&\hspace{2ex}+\int_{\R\times\R}\int_{\R\times\R}\frac{x_1-\tilde{x_1}}{|x_1-\tilde{x_1}|}|P_{\leqslant K}^{x_1}v_c(t,\tilde{x_1},\tilde{x_2})|^2\Im(\overline{P_{\leqslant K}^{x_1}v_c}\partial_{x_1}v_c)(t,x_1,x_2)dx_1d\tilde{x_1}dx_2d\tilde{x_2}.
\end{align*}
In fact, we need to control the error term arising from the frequency cut-off
\begin{equation*}
\bigg\|\int_{\R}\partial_{x_1}(|P_{\leq K}^{x_1}{u}(t,x_1,x_2)|^2)dx_2\bigg\|_{L_{t,x_1}^2([0,T]\times\R)}^2 \lesssim \int_0^T M_I'(t) \mathrm{d}t +\mathcal{E} \lesssim \sup\limits_{t\in [0,T]}|M_I(t)| + \mathcal{E},
\end{equation*}
where $\mathcal{E}$ denotes the error term coming from the low-frequency truncation.
	
We will present the  proof of the long-time Strichartz estimate in Section \ref{sec:long} and the truncated  Morawetz estimate in Section \ref{sec:Morawetz-frequency}. Then using the Morawetz estimate and the compactness of the almost-periodic solution, we can show
\begin{align*}
\|u\|^4_{L^2_{x_1,x_2}}K   &\sim  \|u \|^4_{L^2_{x_1,x_2}}  \int_0^T N(t)^3 dt\\
&\lesssim \int_0^T N(t)^3 \bigg(\int_{|x_1-x_1(t)|\leq \frac{C\Big(\frac{\|u\|_{L_{x_1,x_2}^2 }^2}{100}\Big)}{N(t)}}\int_{\R}\left|P_{\le  10\eta_3^{-1}K}^{x_1}
u(t,x_1,x_2)\right|^2dx_1dx_2 \bigg )^2 dt \\
& \lesssim \bigg\|\int_{\R}\partial_{x_1} \left( |P_{\le 10\eta_3^{-1} K}^{x_1} u(t,x_1,x_2)|^2 \right)dx_2 \bigg\|^2_{L^2_{t,x_1}([0,T]\times \mathbb{R})} \lesssim o(K),
\end{align*}
where $\int_{0}^{T}N(t)^3dt=K$. Let $K$ be sufficiently large,  then we can get $u=0$, which  contradicts to  $u\neq0$. This precludes the possibility of the quasi-solitons.
	
In the  rapid frequency cascade case,  the long-time Strichartz estimate also plays an important role. Indeed, we can utilize it to prove  the additional regularity of the rapid cascade solution, i.e. $v_c\in L_t^\infty H_{x_1}^sL_{x_2}^2([0,\infty)\times\R\times\R)$ with $s>1$. Moreover, we can preclude the possibility of the rapid cascade solution. Therefore, we complete the proof of Theorem \ref{Thm2}.

\subsection{Organization of the paper}
The paper is organized as follows: In Section \ref{sec:pre}, we collect the  basic properties of Fourier transform and the harmonic analysis tool associated with the partial harmonic oscillator. In Section \ref{sec:local}, we present the Strichartz estimates for solutions to the linear Schr\"odinger equation under partial harmonic confinement. Using the Strichartz estimate, we prove the local well-posedness, stability theorem and small-data scattering to equation $\eqref{NLS-final}$. In Section \ref{sec:profile}, we prove the linear profile decomposition for the Schr\"odinger group $e^{it(\Delta_{\R^2}-x_2^2)}u_0$ and then apply this decomposition to show the existence of the minimal blow-up solution if Theorem \ref{Thm1} fails. Then we show that such non-trivial blow-up solution cannot exist. In Section \ref{sec:Local}, we first prove the local well-posedness, stability theorem and small-data scattering to equation $\eqref{DCR-2}$. In Section \ref{sec:minimal}, we show the existence of the minimal blow-up solution which is precompact in $L_{x_1}^2\mathcal{H}_{x_2}^1$ up to the symmetry. In Section \ref{sec:long}, we prove the long-time Strichartz estimate for the minimal blow-up solutions. In Section \ref{sec:Morawetz-frequency}, we establish the low-frequency localized Morawetz estimates, which play a crucial role in ruling out the possibility of a minimal blow-up solution. In Section \ref{sec:Thm2}, we show the impossibility of the minimal blow-up solution and give the proof of Theorem \ref{Thm2}.
\section{Preliminaries}\label{sec:pre}

In this section, we first introduce the basic notations which will be frequently used in the rest of the article. We will use the notation $A\lesssim B$, which means that there exists a constant $C>0$ such that $A\leqslant CB$. We also use the notation $A\sim B$ which means that there exist two constants $C_1,C_2>0$ such that $C_1B\leqslant A\leqslant C_2B$. For any $a(x)\in\R^d$, we denote $\langle a\rangle $ by $\langle a\rangle=\sqrt{1+|a|^2}$.
\subsection{Fourier transform and functional spaces}
In this subsection, we recall the basic definition and properties of Fourier transform and some useful functional spaces.

First, we define the Fourier transform of function $f$ as
\begin{align*}
(\mathcal{F}f)(\xi)=\frac{1}{(2\pi)^2}\int_{\R^2}f(x)e^{-ix\cdot\xi}dx
\end{align*}
and the inverse Fourier transform of $\mathcal{F}f$  as
\begin{align*}
(\mathcal{F}f)^\vee(x)=\int_{\R^2}\widehat{f}(\xi)e^{ix\cdot\xi}d\xi.
\end{align*}
Then, for $\forall s\in\R$,we can define the fractional order differential operator $|\nabla|^s$ as $(|\nabla|^sf)^\wedge(\xi)=|\xi|^s\widehat{f}(\xi)$. Similarly, the operator $\langle\nabla\rangle^s$ can be defined as $(\langle\nabla\rangle^sf)^\wedge(\xi)=\langle\xi\rangle^s\widehat{f}(\xi)=(1+|\xi|^2)^\frac{s}{2}\widehat{f}(\xi)$. 

In our paper, we also need the Fourier transform and inverse Fourier transform with respect to one variable. For example, we give the definition of Fourier transform on $x_1$-variable. For $x=(x_1,x_2)\in\R\times\R$, we define
\begin{align*}
(\mathcal{F}_{x_1}f)(\xi,x_2)=\frac{1}{2\pi}\int_{\R}f(x_1,x_2)e^{-ix_1\cdot\xi}dx_1,\quad\xi\in\R.
\end{align*}

We also need the Littlewood-Paley projectors. We take a cut-off function $\chi(t)\in C_0^\infty(\R)$ satisfying $\chi(t)=1$ if $t\leqslant1$ and $\chi(t)=0$ if $t\geqslant2$.  For $N\in2^{\Bbb{Z}}$, we let 
\begin{align*}
\chi_N(t)=\chi(N^{-1}t),\quad \varphi_N(t)=\chi_{N}(t)-\chi_{N/2}(t).
\end{align*}
Now, we can define the Littlewood-Paley projector as
\begin{align*}
P_{\leqslant N}f(x)&=\mathcal{F}^{-1}(\chi_N(|\xi|)\widehat{f}(\xi)),\quad x,\xi\in\R^2,\\
P_Nf(x)&=\mathcal{F}^{-1}(\varphi_N(|\xi|)\widehat{f}(\xi)),\quad x,\xi\in\R^2. 
\end{align*}
We also define the partial Littlewood-Paley projector as
\begin{align*}
P_{\leqslant N}^{x_1}(x_1,x_2)&=\mathcal{F}_{x_1}^{-1}(\chi_N(|\xi|)(\mathcal{F}_{x_1}f)(\xi,x_2)),\hspace{1ex}x_1,x_2,\xi\in\R\\
P_{ N}^{x_1}(x_1,x_2)&=\mathcal{F}_{x_1}^{-1}(\varphi_N(|\xi|)(\mathcal{F}_{x_1}f)(\xi,x_2)),\hspace{1ex}x_1,x_2,\xi\in\R.
\end{align*}
For simplicity, we  denote $P_j^{x_1}:=P_{2^j}^{x_1}$ for $j\in\Bbb{Z}$.

Next, we denote the standard Lebesgue spaces by $L^p(\R^2)$ and its norm is given by
\begin{align*}
\|f\|_{L^p(\R^2)}=\left(\int_{\R^2}|f(x)|^pdx\right)^\frac{1}{p}.
\end{align*}  
Let $I$ be the  time interval, we denote the mixed space-time Lebesgue space as $L_t^qL_x^r(I\times\R^2)$ with the norm
\[  \| u \|_{L_{t}^{q}L^{r}_{x}(I\times \R^{2})}=\|  \|u(t) \|_{L^{r}_{x}(\R^{2})}  \|_{L^{q}_{t}(I)}. \]
We denote $W^{s,p}$ by the inhomogeneous  Sobolev space,
\begin{align*}
W^{s,p}(\R^2):=\left\{f\in L^p(\R^2):\|f\|_{W^{s,p}(\R^2)}:=\|\langle\nabla\rangle^sf\|_{L^p(\R^2)}<+\infty\right\}.
\end{align*}

\subsection{Harmonic oscillator and associated functional space}
The harmonic oscillator $H=-\Delta+|x|^2$ is an important Schr\"odinger operator in the study of mathematics and physics. The one-dimensional harmonic oscillator can form a $L^2$-basis. We denote $E_n$ the $n$-th eigenspace associated to $H$ and $\lambda_n=2n+1$ the $n$-th eigenvalues. The eigenspace can be spanned by the Hermite function $e_n$ which is given by
\begin{align*}
e_n(x)=\frac{1}{\sqrt{n!}2^{\frac{n}{2}}\pi^\frac{1}{4}}(-1)^ne^{\frac{x^2}{2}}\frac{d^n}{dx^n}(e^{-x^2}),\quad x\in\R
\end{align*}  
and $e_n$ satisfies the following equation
\begin{align*}
(-\partial_{x}^2+x^2)e_n(x)=(2n+1)e_n(x).
\end{align*}
We denote $\Pi_n$ by the spectral projector on the $n$-th eigenspace $E_n$. For any $f\in L^2(\R^2)$, we can decompose it as $$f(x_1,x_2)=\sum\limits_{n\in\N}\Pi_nf(x_1,x_2)=\sum_{n\in\N}\langle f,e_n\rangle_{L_{x_2}^2(\R)} e_n(x_2),$$
where $$\langle f,e_n\rangle=\int_{\R}f(x_1,x_2)\overline{e_n(x_2)}dx_2.$$ 
For $s\in\R$ and $p\geqslant1$, we denote the inhomogeneous Sobolev space $\mathcal{W}^{s,p}(\R)$ by the following
\begin{align*}
\mathcal{W}^{s,p}(\R)=\left\{u\in L_x^p(\R):\|u\|_{\mathcal{W}^{s,p}(\R)}:=\left\|\langle\partial_x\rangle^su\|_{L_x^p(\R)}+\||\cdot|^su\|_{L_x^p(\R)}<\infty\right.\right\}.
\end{align*} 
If $p=2$, the above norm is equivalent to the  following, which was proved in \cite{equi}
\begin{align*}
\|f\|_{\mathcal{W}^{s,2}(\R)}^2=\sum_{n\in\N}(2n+1)^s\|\Pi_nf\|_{L^2(\R)}^2.
\end{align*}
\begin{remark}
We  note that the Sobolev spaces associated to  the Schr\"odinger operator with partial harmonic oscillator also enjoy similar equivalent properties, see   Su-Wang-Xu \cite{Su-Wang-Xu1,Su-Wang-Xu2} for more details. 
\end{remark}
In the rest of this article, we denote $\mathcal{W}^{s,2}(\R)$ by $\mathcal{H}^s(\R)$. We also define function space $L_{x_1}^p\mathcal{H}_{x_2}^s(\R\times\R)$  by
\begin{align*}L_{x_1}^{p}\mathcal{H}_{x_2}^{s}(\R\times\R)&=\bigg\{f\in L_{x_1}^{p}L_{x_2}^{2}(\R\times\R):\|f\|_{L_{x_1}^{p}\mathcal{H}_{x_2}^{s}(\R\times\R)}:=\bigg(\int_{\R}\|f(x_1,\cdot)\|_{\mathcal{H}_{x_2}^{s}(\R)}^{p}dx_1\bigg)^{1/p}\\&=\bigg(\int_{\R}\left\|\bigg(\sum_{n\in\mathbb{N}}(2n+1)^{s}|f_{n}(x_1,x_2)|^{2}\bigg)^{\frac{1}{2}}\right\|_{L_{x_2}^{2}(\R)}^{p}dx_1\bigg)^{1/p}<\infty\bigg\},\end{align*}
where $f_n=\Pi_nf$. 
We can define the Hermite-Sobolev spaces $H_{x_1}^{s_1}\mathcal{H}_{x_2}^{s_2}$ and $\dot{H}_{x_1}^{s_1}\mathcal{H}_{x_2}^{s_2}$ as 
\begin{align*}
\|f\|_{H_{x_1}^{s_1}\mathcal{H}_{x_2}^{s_2}(\R\times\R)}^2=\int_{\R}\|\langle\partial_{x_1}\rangle^{s_1}f(x_1,\cdot)\|_{\mathcal{H}_{x_2}^{s_2}(\R)}^2dx_1.
\end{align*}
and
\begin{align*}
\|f\|_{\dot{H}_{x_1}^{s_1}\mathcal{H}_{x_2}^{s_2}(\R\times\R)}^2=\int_{\R}\||\partial_{x_1}|^{s_1}f(x_1,\cdot)\|_{\mathcal{H}_{x_2}^{s_2}(\R)}^2dx_1.
\end{align*}
For any time interval $I\subset\R$ and $f:I\times\R^2\to\Bbb{C}$, we can define
\begin{align*}
\|f\|_{L_t^pL_{x_1}^qW_{x_2}^{s,r}(I\times\R\times\R)}\stackrel{\triangle}{=}\Bigg(\int_{I}\bigg(\int_{\R}\big\|\langle\partial_{x_2}\rangle^sf(t,x_1,\cdot)\big\|_{L_{x_2}^r(\R)}^qdx_1\bigg)^\frac{p}{q}dt\bigg)^\frac{1}{p}
\end{align*}
and
\begin{align*}
\|f\|_{L_t^pL_{x_1}^q\mathcal{H}_{x_2}^{s}(I\times\R\times\R)}\stackrel{\triangle}{=}\Bigg(\int_{I}\bigg(\int_{\R}\big\|f(t,x_1,\cdot)\big\|_{\mathcal{H}_{x_2}^s(\R)}^qdx_1\bigg)^\frac{p}{q}dt\bigg)^\frac{1}{p}
\end{align*}
with $1\leqslant p,q,r\leqslant\infty$. We also use the following norms for sequence $\{f_n\}_{n\in\Bbb{N}}$
\begin{align*}
\|f_n\|_{L_t^pL_{x_1}^qL_{x_2}^r\ell_n^2(I\times\R\times\R\times\Bbb{N})}\stackrel{\triangle}{=}\big\|\|f_n\|_{\ell_n^2(\N)}\big\|_{L_t^pL_{x_1}^qL_{x_2}^r(I\times\R\times\R)}.
\end{align*}

We  need the following lemma  of Dirac delta function.
\begin{lemma}\label{Dirac}
The Delta function $\delta(x)$ belongs to the space $\mathcal{H}^{-s}(\R)$ for $s\geq0$.
\end{lemma} 
\begin{proof}
By definition, we have
\begin{align}\label{eq1.6v31}
\left\|\delta_0(x)\right\|_{\mathcal{H}^{-s}_x(\R)}^2 = \sum_{n=0}^\infty {(2n+1)^{-s }} |c_n|^2,
\end{align}
where $c_n=\left\langle\delta_0(x),  e_n(x) \right\rangle = e_n(0)$. Since $e^{-x^2} = \sum\limits_{m=0}^\infty \frac{(-x^2)^m}{m!} = \sum\limits_{n=0}^\infty \frac{d^n}{dx^n}\Big|_{x=0} e^{-x^2} \cdot \frac{x^n}{n!}$,
we have
\begin{align*}
\frac{d^n}{dx^n}\bigg|_{x=0} e^{-x^2} =
\begin{cases}
0,   & \text{ $n$ is odd},\\
\frac{(-1)^\frac{n}2}{(\frac{n}2)!} n!, & \text{ $n$ is even}.
\end{cases}
\end{align*}
Thus
\begin{align*}
e_n(0)& = \begin{cases}
0,  &  \text{ $n$ is odd},\\
\frac{(-1)^n}{\sqrt{n!} 2^\frac{n}2 \pi^\frac14}  \frac{(-1)^\frac{n}2}{(\frac{n}2)!} n!,  & \text{ $n$ is even}.
\end{cases}
\end{align*}
Together with \eqref{eq1.6v31}, this implies
\begin{align*}
\left\|\delta_0(x) \right\|_{\mathcal{H}_x^{-s}(\R)}^2 \le  \pi^{-\frac14} \sum_{\substack{ n=0, \\ \text{ n even} } }^\infty \frac{n!} {2^n  \left( \left( \frac{n}2 \right)! \right)^2 (2n+1)^s } \lesssim \sum\limits_{m = 0}^\infty \frac1{2^m(4m + 1)^s} \lesssim 1.
\end{align*}

\end{proof}

Next, we establish the following Bernstein's inequalities, which will be used later.
\begin{proposition}\label{Bernstein}For any $s\geq0$, $2\leq p<q<\infty$, we have
\begin{gather}
\big\|  \| P^{x_1}_{\le N} u \|_{ L_{x_2}^2 (\mathbb{R}) }  \big\| _{L^{q} _{ x_1 } (\mathbb{R}) } \lesssim N^{\frac{1}{p} - \frac{1}{q}} \big\|   \| P_{\le N}^{x_1} u \|_{L_{x_2}^2 (\mathbb{R}) }  \big\|_{ L^p_{x_1} (\mathbb{R}) }, \label{Bern-1} \\
\big\|  \|| \partial_{x_1}|^s P_{N}^{x_1} u \|_{ L_{x_2}^2 (\mathbb{R}) }  \big\| _{L^{2} _{ x_1 } (\R) } \sim N^s \big\|   \| P_{ N}^{x_1} u \|_{L_{x_2}^2 (\mathbb{R}) }  \big\|_{ L^2_{x_1} (\mathbb{R}) }. \label{Bern-3}
\end{gather}

\end{proposition}

\begin{proof}

We denote $c_n := \langle P^{x_1}_{\le N}u ,  e_n \rangle _{L^2_{x_2}} $.  
First, we prove \eqref{Bern-1}. 
\begin{align*}
\big\|  \| P^{x_1}_{\le N} u \|_{ L_{x_2}^2 (\mathbb{R}) }  \big\|  _{L^{q} _{ x_1 } (\mathbb{R}) }^2  =  &    \big\|\sum_{n\in\N} |c_n| ^2 \big\|_{L_{x_1}^{\frac{q}{2}}(\R)} \\
=  &  \big\|P_{\leq CN}^{x_1}\big(\sum_{ n\in\N }|c_n|^2\big)\big\|_{L_{x_1}^\frac{q}2(\R)} \\ 
\lesssim & N^{\frac{2}{p} - \frac{2}{q}} \big\|P_{\leq CN}^{x_1}\big(\sum_{ n\in\N }|c_n|^2\big)\big\|_{L_{x_1}^\frac{p}2(\R)} \\
\lesssim  & N^{\frac{2}{p} - \frac{2}{q}}  \big\| \| P_{\le N}^{x_1} u \|_{L_{x_2}^2 (\mathbb{R}) }  \big\|_{ L^p_{x_1} (\mathbb{R}) }^2,
\end{align*}
which complete the proof of \eqref{Bern-1}.

Denote $ d_n := \langle |\partial_{x_1}|^s P^{x_1}_{ N} u ,  e_n \rangle _{L^2_{x_2}} $.  One can find
\begin{equation*}
\| d_n \|_{L^2_{x_1}(\R)} \sim  N^s \| \langle  P_N^{x_1}u , e_n \rangle \|_{L^2_{x_1}(\R)}. 
\end{equation*}
Consequently, we have
\begin{align*}
\big\|  \| |\partial_{x_1}|^s  P_{ N} ^{x_1}  u \|_{ L_{x_2}^2 (\mathbb{R}) }  \big\| _{L^{2} _{ x_1 } (\R) } ^ 2 =  & \int_{\mathbb{R}}    \sum_{n\in\N}  |d_n| ^2   dx_1 \\
\sim & N^{2s} \int_{\mathbb{R}} \sum_{n\in\N}  \| \langle P_N^{x_1}u , e_n \rangle \|^2 dx_1 \\
\sim & N^{2s}\big\|   \| P_{ N}^{x_1} u \|_{L_{x_2}^2 (\mathbb{R}) }  \big\|_{ L^2_{x_1} (\mathbb{R}) } ^2.
\end{align*}
Thus we have proved $\eqref{Bern-3}$.
\end{proof}

With similar strategy, we can also have the following Gagliardo-Nirenberg inequality.
\begin{lemma}[Gagliardo-Nirenberg inequality]\label{GNIN}
Let  $0< s_1 \leq  s_2$, then we have
\begin{equation*}
\big\| \| |\partial_{x_1}|^{s_1} u \|_{L_{x_2}^2(\R)} \big\|_{L_{x_1}^{2}(\R)} \lesssim \big\| \|  u \|_{L_{x_2}^2(\R)} \big\|^{\theta}_{L_{x_1}^2(\R)} \big\| \| |\partial_{x_1}|^{s_2} u \big\|_{L_{x_2}^2(\R)} \big\|^{1-\theta}_{L_{x_1}^{2}(\R)},
\end{equation*}
where
\begin{equation*}
\theta = \frac{s_2 - s_1}{s_2}.
\end{equation*}
\end{lemma}

Next, we give the Littlewood-Paley theorem in a vector-valued version.  
The proof of this proposition is rather standard, we refer to \cite{Grafakos} for details.
\begin{proposition}[Littlewood-Paley Theorem,\cite{Grafakos}]\label{Littlewood}Let $1<p<\infty$, then we have
\begin{align*}
\|f\|_{L_{x_1}^pL_{x_2}^2(\R\times\R)}&=\Big\|\big(\sum_{n\in\N}\|\Pi_nf\|_{L_{x_2}^2(\R)}^2\big)^\frac{1}{2}\Big\|_{L_{x_1}^p(\R)}\sim\bigg\|\big(\sum_{n\in\N}\sum_{k\in\Bbb{Z}}\|P_k^{x_1}\Pi_nf\|_{L_{x_2}^2(\R)}^2\big)^\frac12\bigg\|_{L_{x_1}^p(\R)},
\end{align*}
where the spectral projector $\Pi_n$ acts on the $x_2$-variable.
\end{proposition}
We end this section by giving a Moser type estimate, which is crucial in establishing the nonlinear estimate.
\begin{lemma}[Moser type estimate, \cite{Jao-CPDE,Killip-Visan-Zhang}]\label{Moser-estimate}
For $\gamma\in(0,1]$ and $p_i,q_i,r\in(1,\infty)$ satisfying $\frac{1}{r}=\frac{1}{p_i}+\frac{1}{q_i}$, $i=1,2$. Then, we have
\begin{align*}
\big\|H^\gamma(fg)\big\|_{L^r(\R)}\lesssim \|H^\gamma f\|_{L^{p_1}(\R)}\|g\|_{L^{q_1}(\R)}+\|H^\gamma f\|_{L^{p_2}(\R)}\|g\|_{L^{q_2}(\R)},
\end{align*}
where operator $H^\gamma$ is defined by the functional calculus associated to the harmonic oscillator $H=-\partial_{x_2}^2+x_2^2$.
\end{lemma}



\section{Local well-posedness and stability result}\label{sec:local}
In this section, we will prove the local well-posedness and stability result for equation $\eqref{NLS-final}$. Then we also collect the related result such as small data scattering and persistence of regularity. Compared to the classical nonlinear Schr\"odinger equation without the potential, we use the weaker norm as the scattering size. It is worth to mention that the uniform boundedness of this weaker norm is enough to obtain the scattering result. 

Before presenting the local theory, we first recall the Strichartz estimates, which is very crucial in establishing both local and global theory for dispersive equations. We shall use the following notation. 
\begin{definition}[Schr\"odinger admissible pairs] We call $(p,q)\in\R^2$ the admissible pairs if $(p,q)$ satisfies the following 
\begin{align*}
\frac{2}{p}=\frac{1}{2}-\frac{1}{q},\quad 4\leqslant p\leqslant\infty.
\end{align*}
\end{definition}
Now we can state the Strichartz estimates for linear Schr\"odinger equations with partially harmonic oscillator. This estimates were proved in Proposition 3.1 in \cite{Antonelli}.

\begin{proposition}[Strichartz estimates, \cite{Antonelli}]\label{Strichartz-Antonelli}
Let $I$ be a interval,  $(p,q)$ and $(\tilde{p},\tilde{q})$  be any Schr\"odinger admissible pairs,  we have the following estimate
\begin{align}\label{Strichartz-L^2}
\big\|e^{it(\Delta_{\R^2}-x_2^2)}f\big\|_{L_t^pL_{x_1}^qL_{x_2}^2(I \times\R\times\R)}&\lesssim\|f\|_{L_{x_1,x_2}^2(\R\times\R)}\\
\big\|\int_{0}^{t}e^{i(t-s)(\Delta_{\R^2}-x_2^2)}F(s,x_1,x_2)\big\|_{L_t^pL_{x_1}^qL_{x_2}^2(I \times\R\times\R)}&\lesssim\|F\|_{L_t^{\tilde{p}^\prime}L_{x_1}^{\tilde{q}^\prime}L_{x_2}^2(\R\times\R)}\label{Strichartz-inhomo}\\\label{Stricharz-Sobolev}
\big\|e^{it(\Delta_{\R^2}-x_2^2)}f\big\|_{L_t^pL_{x_1}^q\mathcal{H}_{x_2}^\sigma(I\times\R\times\R)}&\lesssim\|f\|_{L_{x_1}^2\mathcal{H}_{x_2}^\sigma(\R\times\R)},\quad\sigma\geq0.
\end{align}
	
\end{proposition}

\begin{remark}
The Strichartz estimates for Schr\"odinger equation with partially harmonic oscillator is very closed to that in the waveguide manifold setting. To be precise, the Strichartz estimates are also global-in-time as in the waveguide case (see \cite{Antonelli}). We refer to Barron \cite{Barron} for more recent progress on the global-in-time Strichartz estimates on waveguide manifolds via decoupling method.
\end{remark}

We also have the following nonlinear estimate.
\begin{lemma}[Nonlinear estimate]For $\varepsilon_0\in[0,\frac{1}{2})$, then we have the following nonlinear estimate
\begin{align*}
\big\|\prod_{j=1}^{5}u_j\big\|_{L_t^{\frac{6}{5}}L_{x_1}^{\frac{6}{5}}\mathcal{H}_{x_2}^{1-\varepsilon_0}(I\times\R\times\R)}\lesssim\prod_{j=1}^{5}\|u_j\|_{L_t^6L_{x_1}^6\mathcal{H}_{x_2}^{1-\varepsilon_0}(I\times\R\times\R)},
\end{align*}
where $I\times\R\times\R$ is the space-time slab.
\end{lemma}

Then combining the nonlinear estimate as well as the Strichartz estimates, one can obtain the local well-posedness in $L_{x_1}^2\mathcal{H}_{x_2}^1(\R^2)$ and $\Sigma_{x_1,x_2}(\R^2)$ respectively. 

Before stating the local well-posedness result, we  give some notations. We denote $X_1(t)$ and $X_2(t)$   by the following
\begin{align*}
X_1(t)=x_2\sin(t)-i\cos(t)\partial_{x_2},\quad X_2(t)=x_2\cos(t)+i\sin(t)\partial_{x_2}.
\end{align*}
For $f\in\mathcal{S}(\R^2)$, we have the following identity:
\begin{align}\label{transform}
|X_1(t)f(x_1,x_2)|^2+|X_2(t)f(x_1,x_2)|^2=|x_2f(x_1,x_2)|^2+|\partial_{x_2}f(x_1,x_2)|^2.
\end{align}
As a direct consequence, we have the following identity
\begin{align*}
\|u\|_{\Sigma_{x_1,x_2}(\R^2)}^2=\|u\|_{L_{x_1}^2\mathcal{H}_{x_2}^1(\R\times\R)}^2+\|X_1(t)u\|_{L_{x_1,x_2}^2(\R\times\R)}^2+\|X_2(t)u\|_{L_{x_1,x_2}^2(\R\times\R)}^2.
\end{align*}
Now we give the local well-posedness for \eqref{NLS-final}.
\begin{lemma}[Local well-posedness]\label{lwp}
For any $\eta>0$ and the initial data $u_0$ satisfying $\|u_0\|_{L_{x_1}^2\mathcal{H}_{x_2}^1(\R\times\R)}\leqslant\eta$, there exists a constant $\Lambda(\eta)$ such that
\begin{align*}
&\big\|e^{it(\Delta_{\R^2}-x_2^2)}u_0\big\|_{L_{t,x_1}^6L_{x_2}^2(I\times\R\times\R)}+\big\|X_1(t)e^{it(\Delta_{\R^2}-x_2^2)}u_0\big\|_{L_{t,x_1}^6L_{x_2}^2(I\times\R\times\R)}\\
&+\big\|X_2(t)e^{it(\Delta_{\R^2}-x_2^2)}u_0\big\|_{L_{t,x_1}^6L_{x_2}^2(I\times\R\times\R)}\leqslant\Lambda(\eta),
\end{align*}
where $I$ is the local time interval, then there exists a unique solution  $u\in C(I,L_{x_1}^2\mathcal{H}_{x_2}^1(\R\times\R))$ to \eqref{NLS-final} obeying
\begin{align*}
\|u\|_{L_{t,x_1}^6\mathcal{H}_{x_2}^{1 }(I\times\R\times\R)}\leqslant 2\|e^{it(\Delta_{\R^2}-x_2^2)}u_0\|_{L_{t,x_1}^6\mathcal{H}_{x_2}^{1}(I\times\R\times\R)}
\end{align*}
and
\begin{align*}
\|u\|_{L_t^\infty L_{x_1}^2\mathcal{H}_{x_2}^1(I\times\R\times\R)}\lesssim \|u_0\|_{L_{x_1}^2\mathcal{H}_{x_2}^1(\R\times\R)}.
\end{align*}
Furthermore, if $u_0\in\Sigma_{x_1,x_2}(\R\times\R)$, then the solution is global in $C(\R,\Sigma_{x_1,x_2}(\R\times\R))$.
\end{lemma} 

\begin{proof}
We define the  map
\begin{align*}
\Phi(u)=e^{it(\Delta_{\R^2}-x_2^2)}u_0-i\int_{0}^{t}e^{i(t-s)(\Delta_{\R^2}-x_2^2)}(|u|^4u)(s)ds.
\end{align*}
Denote the metric space $X$ by
\begin{align*}
X:=\big\{u\in C(I,L_{x_1}^2\mathcal{H}_{x_2}^1(\R\times\R)):\|u\|_{L_t^\infty L_{x_1}^2\mathcal{H}_{x_2}^1(I\times\R\times\R)}\leqslant2\eta,\|u\|_{L_{t,x_1}^6\mathcal{H}_{x_2}^1(I\times\R\times\R)}\leqslant2C\Lambda(\eta)\big\}.
\end{align*}
By  the operator $X_1(t), X_2(t)$, we have the equivalent representation of $X$ as the following
\begin{align*}
X&=\big\{u\in C(I,L_{x_1}^2\mathcal{H}_{x_2}^1(\R\times\R)):\|u\|_{L_t^\infty L_{x_1}^2\mathcal{H}_{x_2}^1(I\times\R\times\R)}\leqslant2\eta,\|X_j(t)u\|_{L_t^\infty L_{x_1,x_2}^2(I\times\R\times\R)}\leqslant2\eta,\\
&\hspace{4ex}	\|u\|_{L_{t,x_1}^6L_{x_2}^2(I\times\R\times\R)}\leqslant2C\Lambda(\eta),\|X_j(t)u\|_{L_{t,x_1}^6L_{x_2}^2(I\times\R\times\R)}\leqslant2C\Lambda(\eta)\big\}.
\end{align*}
By the direct calculus, for any $u\in X$, we have 
\begin{align*}
& \ \left\|X_1(t) \Phi(u) \right\|_{L_t^\infty L_{x_1,x_2}^2(I\times\R\times\R)}+ \left \|X_2(t) \Phi(u) \right\|_{L_t^\infty L_{x_1,x_2}^2(I\times\R\times\R)}\\
\lesssim& \left\|\partial_{x_2} u_0 \right\|_{L_{x_1,x_2}^2(\R \times\R)} + \left\|x_2 u_0 \right\|_{L_{x_1,x_2}^2(\R\times\R)}\\
&+ \|u\|_{L_{t,x_1}^6 H_{x_2}^1(I\times\R\times\R)}^4 \left(  \left\|X_1(t) u\right\|_{L_{t,x_1}^6 L_{x_2}^2(I\times\R\times\R)}
+ \left\|X_2(t) u \right\|_{L_{t,x_1}^6 L_{x_1}^2(I\times\R\times\R)} \right) .
\end{align*}
Thus
\begin{align}\nonumber
&\left\|\Phi(u) \right\|_{L_t^\infty L_{x_1,x_2}^2(I\times\R\times\R)} + \left\|X_1(t) \Phi(u) \right\|_{L_t^\infty L_{x_1,x_2}^2(I\times\R\times\R)} + \left\|X_2(t) \Phi(u) \right\|_{L_t^\infty L_{x_1,x_2}^2(I\times\R\times\R)}\\\label{conclusion}
 \le& \eta + \left(2C\Lambda(\eta) \right)^5 \leq 2\eta.
\end{align}
Similarly, we can obtain
\begin{align}\nonumber
&\left\|\Phi(u) \right\|_{L_{t,x_1}^6L_{x_2}^2(I\times\R\times\R)} + \left \|X_1(t) \Phi(u) \right\|_{L_{t,x_1}^6 L_{x_2}^2(I\times\R\times\R)} + \left\|X_2(t) \Phi(u) \right\|_{L_{t,x_1}^6 L_{x_2}^2(I\times\R\times\R)}\\\label{conclusion2}
\le& 2\Lambda(\eta)+  \left(2C\Lambda(\eta) \right)^5 \le 2C \Lambda(\eta).
\end{align}
Similarly, the map is contraction in the metric space $X$. For any $u, v\in X$, invoking the Strichartz estimate, H\"older's inequality, and Sobolev embedding again, we have
\begin{align}\label{difference}
\left\|\Phi(u)- \Phi(v)\right\|_{L_{t,x_1}^6 L_{x_2}^2}& \lesssim \big\| |u|^4 u - |v|^4 v\big\|_{L_{t,x_1}^\frac65 L_{x_2}^2}\notag\\
& \lesssim \|u-v\|_{L_{t,x_1}^6 L_{x_2}^2} \left(\|u\|_{L_{t,x_1}^6 H_{x_2}^1}^4 + \|v\|_{L_{t,x_1}^6 H_{x_2}^1}^4\right)\\
& \lesssim  \left(2C\Lambda(\eta) \right)^4 \|u-v\|_{L_{t,x_1}^6 L_{x_2}^2}. \notag
\end{align}
	
Combining \eqref{conclusion}, \eqref{conclusion2}, and \eqref{difference}, we have for $\Lambda(\eta)$ small enough, $\Phi: X \to X $ is a contractive map. Therefore, the theorem follows from the fixed point theorem. Similarly, one can obtain the local well-posedness for $u_0\in\Sigma_{x_1,x_2}(\R^2)$ in a similar way.
\end{proof}

As a direct consequence of the above lemma, we have the scattering result for small data in $\Sigma_{x_1,x_2}(\R\times\R)$.

\begin{lemma}[Small data scattering in $\Sigma_{x_1,x_2}(\R\times\R)$]\label{small-data}
Suppose that $u_0\in \Sigma_{x_1,x_2}(\R\times\R)$ and there exists a small constant $\delta_0>0$ such that  $\|u_0\|_{L_{x_1}^2\mathcal{H}_{x_2}^{1}(\R\times\R)}\leq\delta_0$, then $\eqref{NLS-final}$ admits an unique global solution  $u(t,x_1,x_2)\in C_t^0\Sigma_{x_1,x_2}(\R\times\R\times\R)\cap L_{t,x_1}^6\mathcal{H}_{x_2}^1(\R\times\R\times\R)\cap L_t^6W_{x_1}^{1,6}L_{x_2}^2(\R\times\R\times\R)$. Moreover, $u$ scatters in $\Sigma_{x_1,x_2}(\R\times\R)$, i.e. there exist $u_\pm\in\Sigma_{x_1,x_2}(\R\times\R)$ such that
\begin{align*}
\big\|	u(t,x_1,x_2)-e^{it(\Delta_{\R^2}-x_2^2)}u_{\pm}(x_1,x_2)\big\|_{\Sigma_{x_1,x_2}(\R\times\R)}\to0 \mbox{ as } t\to\pm\infty.
\end{align*} 

\end{lemma}

Next, we will give the stability theorem. 

\begin{theorem}[Stability theorem]\label{Stablity}
Let   $I$ be a compact interval and suppose that $\tilde u$ is the approximate solution to \eqref{NLS-final} satisfying $i\partial_t\tilde{u}+(\partial_{x_1}^2+\partial_{x_2}^2-x_2^2)\tilde{u}=|\tilde{u}|^4\tilde{u}+e$ for some function $e$. Assume that
\begin{align*}
\|\tilde{u}\|_{L_t^\infty L_{x_1}^2\mathcal{H}_{x_2}^{1}(I\times\R\times\R)}\leq M,\quad\|\tilde{u}\|_{L_{t,x_1}^6\mathcal{H}_{x_2}^{1}(I\times\R\times\R)}\leq L,
\end{align*}
where $L,M$ are two positive numbers. Let $t_1\in I$ and let $u(t_1)$ obey
\begin{equation}\label{stability-con1}
\left\|u(t_1)- \tilde{u}(t_1)\right\|_{ L_{x_1}^2 \mathcal{H}_{x_2}^{1 } (\R\times\R)} \le M^\prime.
\end{equation}
Furthermore, we assume the following small condition:
\begin{align}
\left\|e^{i(t-t_1)\left(\Delta_{\R^2}-x_2^2\right) } \left(u(t_1)-\tilde{u}(t_0) \right)\right\|_{L_{t,x_1}^6   \mathcal{H}_{x_2}^{1 }(I\times\R\times\R)   }  + \|e\|_{L_{t,x_1}^\frac65   \mathcal{ H}_{x_2}^{1}(I\times\R\times\R)  }  \le \epsilon, \label{stablity-con2}
\end{align}
for some $0<\varepsilon<\varepsilon_1$, where $\varepsilon_1=\varepsilon_1(M,M^\prime,L)$ is a small positive constant.  
Then, there exists a solution $u$ to \eqref{NLS} on $I\times \R  \times \R$ with an initial data $u(t_1)$ at time $t=t_1$ satisfying
\begin{align*}
&\  \left\|u-\tilde{u}\right\|_{L_{t,x_1}^6  \mathcal{H}_{x_2}^{1}}   \le C(M,M^\prime,L)\epsilon,\quad
\left\|u-\tilde{u}\right\|_{L_t^\infty L_{x_1}^2 \mathcal{H}_{x_2}^{ 1 }  }    \le C(M,M^\prime,L)M^\prime,   \\
& \text{ and }  \|u\|_{L_t^\infty L_{x_1}^2 \mathcal{H}_{x_2}^{ 1 }  \cap L_{t,x_1}^6   \mathcal{H}_{x_2}^{1}
}   \le C(M,M^\prime,L).
\end{align*}
\end{theorem}

The proof of the above stability result is standard and we refer the readers to Killip-Visan \cite{Killip-Visan-note} and Appendix A in \cite{Cheng-Guo-Guo-Liao-Shen} for the details. 

The following proposition gives the scattering criterion for solution to \eqref{NLS-final} under the boundedness of the weaker scattering norm. 
\begin{proposition}[Scattering Criterion]Suppose that $u(t,x_1,x_2)\in C_t^0\Sigma_{x_1,x_2}(\R\times\R\times\R)$ is a global solution to \eqref{NLS-final} obeying 
\begin{align*}
\|u\|_{L_{t,x_1}^6\mathcal{H}_{x_2}^{1-\varepsilon_0}(\R\times\R\times\R)}\leq L
\end{align*}
and the initial data $u_0$ satisfy $\|u_0(x_1,x_2)\|_{\Sigma_{x_1,x_2}(\R\times\R)}\leq M$ for some positive constants $M,L>0$ and $\varepsilon_0\in(0,\frac12)$. Then $u$ scatters in both time directions, i.e. there exist $u_\pm\in \Sigma_{x_1,x_2}(\R\times\R)$ such that
\begin{align*}
\Big\|u(t)-e^{it(\partial_{x_1  }^2+\partial_{x_2}^2-x_2^2)}u_\pm\Big \|_{\Sigma_{x_1,x_2}(\R\times\R)}\to0,\mbox{ as }t\to\pm\infty.
\end{align*}
\end{proposition}

\begin{proof}
By the classical scattering theory as in \cite{Killip-Visan-note,Tao-book}, it remains to prove
\begin{align}\label{remain}
\|u\|_{L_{t,x_1}^6\mathcal{H}_{x_2}^{1}\cap L_t^6W_{x_1}^{1,6}L_{x_2}^2(\R\times\R\times\R)}\leqslant C(M,L).
\end{align}
To prove \eqref{remain}, we  take $N\sim(1+\frac{L}{\delta})^6$ large enough and then  divide $\R$ into $N$-pieces of sub-intervals $I_j=[t_j,t_{j+1}]$ such that
\begin{align}\label{scatteringnorm}
\|u\|_{L_{t,x_1}^6\mathcal{H}_{x_2}^{1-\varepsilon_0}(I_j\times\R\times\R)}\leq\delta,
\end{align}
where $\delta$ will be chosen latter.
	
For every interval $I_j$, using \eqref{transform}, Proposition \ref{Strichartz-Antonelli}, Sobolev embedding and \eqref{scatteringnorm}, we obtain
\begin{align*}
&\hspace{6ex}\|u\|_{L_{t,x_1}^6\mathcal{H}_{x_2}^{1}\cap L_t^6W_{x_1}^{1,6}L_{x_2}^2(I_j\times\R\times\R)}\\&\lesssim \|u(t_j)\|_{\Sigma_{x_1 , x_2}(\R\times\R)}+\big\||u|^4u\big\|_{L_{t,x_1}^\frac{6}{5}L_{x_2}^2(I_j\times\R\times\R)}+\big\|X_1(t)\big(|u|^4u\big)\big\|_{L_{t,x_1}^\frac65L_{x_2}^2(I_j\times\R\times\R)}\\
&\hspace{3ex}+\big\|X_2(t)\big(|u|^4u\big)\big\|_{L_{t,x_1}^\frac65L_{x_2}^2(I_j\times\R\times\R)}\\
&\lesssim\|u(t_j)\|_{\Sigma_{x_1 , x_2}(\R\times\R)}+\|u\|_{L_{t,x_1}^6H_{x_2}^{1-\varepsilon_0}(I_j\times\R\times\R)}^4\Big(\|u\|_{L_{t,x_1}^6L_{x_2}^2(I_j\times\R\times\R)}+\big\|X_1(t)u\big\|_{L_{t,x_1}^6L_{x_2}^2(I_j\times\R\times\R)}\\
		&\hspace{3ex}+\big\|X_2(t)u\big\|_{L_{t,x_1}^6L_{x_2}^2(I_j\times\R\times\R)}\Big)+\|u\|_{L_{t,x_1}^4H_{x_2}^{1-\varepsilon_0}(I_j\times\R\times\R)}^4\|u\|_{L_t^6W_{x_1}^{1,6}L_{x_2}^2(I_j\times\R\times\R)}\\
		&\lesssim\|u(t_j)\|_{\Sigma_{x_1 , x_2}(\R\times\R)}+\delta^4\|u\|_{L_{t,x_1}^6H_{x}^1\cap L_t^6W_{x_1}^{1,6}L_{x_2}^2(I_j\times\R\times\R)}.
	\end{align*}
	Then choosing $\delta\ll1$, we have by the bootstrap argument
	\begin{align*}
		\|u\|_{L_{t,x_1}^6\mathcal{H}_{x_2}^1\cap L_t^6W_{x_1}^{1,6}(I_j\times\R\times\R)}\lesssim \|u(t_j)\|_{\Sigma_{x_1,x_2}(\R\times\R)}.
	\end{align*}
	Then summing over all sub-interval $I_j$, we have the desired bound
	\begin{align*}
		\|u\|_{L_{t,x_1}^6\mathcal{H}_{x_2}^{1}\cap L_t^6W_{x_1}^{1,6}L_{x_2}^2(\R\times\R\times\R)}\leqslant C(M,L).
	\end{align*}
	Hence, we complete the proof.
\end{proof}
\begin{remark}\label{rem-global}
	We remark that by the combination of mass and conservation laws, one can obtain the global well-posedness for \eqref{NLS-final} for initial data $u_0\in\Sigma_{x_1 , x_2}(\R\times\R)$.
\end{remark}
\section{ Profile decomposition and  proof of Theorem \ref{Thm1}}\label{sec:profile}
In this section, we will present a linear profile decomposition associated to the linear Schr\"odinger flow $e^{it(\Delta_{\R^2}-x_{2}^2)}$ in  $L_{x_1}^2\mathcal{H}_{x_2}^1(\R\times\R)$ and $\Sigma_{x_1,x_2}(\R\times\R)$ respectively. The linear profile decomposition in $L^2(\R^2)$  plays an important role in treating the mass-critical  problem. The $L^2$ profile decomposition was first handled by Merle-Vega \cite{Merle-Vega} in dimension two and B\'egout-Vargas \cite{Begout-Vargas} for $d\geq1$. In Subsection \ref{linear}, we give the proof of linear profile decomposition for bounded sequences $\{f_n\}\subset L_{x_1}^{2}\mathcal{H}_{x_2}^1(\R\times\R)$. Then we can use this to give the analysis of nonlinear profiles. Using the concentration-compactness argument and assuming that Theorem \ref{Thm2} holds, we can prove the existence of the minimal blow-up solution to \eqref{NLS-final} if Theorem \ref{Thm1} fails. By the interaction Morawetz estimate, we preclude the possibility of such special solutions and thus complete the proof of Theorem \ref{Thm1}.
\subsection{Linear profile decomposition}\label{linear}

In this subsection, we will establish the  linear profile decomposition.

\begin{definition}[Symmetry group $G$] For position $x_0\in\R$, frequency $\xi_0\in\R$, and scaling parameter $\lambda>0$, we define the unitary transform $g_{x_0,\xi_0,\lambda}:L_{x_1}^2\mathcal{H}_{x_2}^1(\R\times\R)\to L_{x_1}^2\mathcal{H}_{x_2}^1(\R\times\R)$ as the following
\begin{align*}
g_{x_0,\xi_0,\lambda}f(x_1,x_2)=\frac{1}{\lambda^{\frac{1}{2}}}e^{ix_1\cdot\xi_0}f\left(\frac{x_1-x_0}{\lambda},x_2\right).
\end{align*}
Let $G$ be the collection of such transformations.
\end{definition}

We establish the following linear profile decomposition for the Schr\"odinger operator $e^{it\partial_{x_1}^2}:L_{x_1}^2\mathcal{H}_{x_2}^1(\R\times\R)\to L_{t,x_1}^6\mathcal{H}_{x_2}^{1-\varepsilon_0}(\R\times\R\times\R)$, which will be used in the study of (DCR) system.
\begin{theorem}[Linear Profile Decomposition I: $L_{x_1}^2\mathcal{H}_{x_2}^1(\R\times\R)$ ]\label{LinearProfile} Let $\{u_k\}_{k\geqslant1}$ be a bounded sequence in $L_{x_1}^2\mathcal{H}_{x_2}^1(\R\times\R)$. There exists $J^*\in\{0,1,\cdots \} \cup\{\infty\}$ such that for $1\leqslant J\leqslant J^*$, the bubble function $\phi^j\in L_{x_1}^2\mathcal{H}_{x_2}^1(\R\times\R)$, $1\leqslant j\leqslant J$ and $r_k^J\in L_{x_1}^2\mathcal{H}_{x_2}^1(\R\times\R)$, group elements $\{g_{k}^j\}_{j\geq1}\subset G$ and time parameters $\{t_k^j\}_{j\geq1}\subset\R$ and we have the following decomposition
\begin{align*}
u_k(x_1,x_2)&=\sum_{j=1}^{J}g_k^je^{it_k^j\partial_{x_1}^2}\phi^j+r_k^J(x_1,x_2)\\
&:=\sum_{j=1}^J\frac{1}{(\lambda_k^j)^{\frac{1}{2}}}e^{ix_1\cdot\xi_k^j}\left(e^{it_k^j\partial_{x_1}^2}\phi^j\right)\left(\frac{x_1-x_{1,k}^j}{\lambda_k^j},x_2\right)+r_k^J(x_1,x_2).
\end{align*}
We have the following orthogonal properties
\begin{gather}
\lim_{k\to\infty}\left(\|u_k\|_{L_{x_1}^2\mathcal{H}_{x_2}^1(\R\times\R)}^2-\sum_{j=1}^J\|\phi^j\|_{L_{x_1}^2\mathcal{H}_{x_2}^1(\R\times\R)}^2-\|r_k^J\|_{L_{x_1}^2\mathcal{H}_{x_2}^1(\R\times\R)}^2\right)=0,\\
\limsup_{k\to\infty}\left\|e^{it(\partial_{x_1}^2+\partial_{x_2}^2-x_2^2)}r_k^J\right\|_{L_{t,x_1,x_2}^6(\R\times\R\times\R)}\to0,\quad J\to J^*,\label{eq5.3}\\
(\lambda_k^j)^{\frac{d-1}{2}}e^{-it_{k}^j\partial_{x_1}^2}\left(e^{-i(\lambda_k^jx_1+x_{1,k}^j)\xi_k^j}r_k^J(\lambda_k^jx_1+x_{1,k}^j,x_2)\right)\rightharpoonup0 \mbox{ in }L_{x_1}^2\mathcal{H}_{x_2}^1, k\to\infty,j\leqslant J.
\end{gather}
The  frames $\{(\lambda_k^j,t_k^j,x_{1,k}^j,\xi_k^j)\}_{j\geq1}$ is asymptotically orthogonal in the sense of
\begin{align}\label{frame-orthogonal}
\frac{\lambda^j_k}{\lambda_k^{j'}} + \frac{\lambda_k^{j'}}{\lambda_k^j} + \lambda_k^j \lambda^{j'}_k \left|\xi_k^j - \xi_k^{j'}\right|^2 + \frac{ \left|x_{1,k}^j - x_{1,k}^{j'}-2t_{k}^j(\lambda_{k}^j)^2(\xi_k^j-\xi_{k}^{j^\prime})^2 \right|^2}{\lambda_k^j \lambda_k^{j'}}
+ \frac{ \left| \left(\lambda_k^j \right)^2 t_k^j - \left(\lambda_k^{j'} \right)^2 t_k^{j'} \right|}{\lambda_k^j \lambda_k^{j'}} \to \infty
\end{align}
where $j\neq j^\prime$ and $k\to\infty$.
\end{theorem}

We also need the linear profile decomposition for bounded sequences in weighted Sobolev spaces, which is crucial in the proof of Theorem \ref{Thm1}.
\begin{theorem}[Linear Profile Decomposition II: $\Sigma_{x_1,x_2}(\R^2)$]\label{linearprofile2}Let ${u_k}$ be a bounded sequence in $\Sigma_{x_1,x_2}(\R^2)$ and $\{\lambda_k^j,t_k^j,\xi_k^j,x_{1,k}^j\}$ be the same as in Theorem \ref{LinearProfile}. In addition, we take $\lambda_k^j\to1$ or $\infty$ as $k\to\infty$ and $|\xi_k^j|\leqslant C_j$ for every $1\leq j\leq J$. Now we have the following linear profile decomposition
\begin{align*}
u_k(x_1,x_2)=\sum_{j=1}^J\phi_k^j(x_1,x_2)+r_k^J(x_1,x_2),
\end{align*}
where 
\begin{align*}
\phi_k^j(x_1,x_2)=\frac{1}{(\lambda_k^j)^\frac{1}{2}}e^{ix_1\cdot\xi_k^j}\Big(e^{it_k^j\partial_{x_1}^2}P_k^j\phi^j\Big)\left(\frac{x_1-x_{1,k}^j}{\lambda_k^j},x_2\right)
\end{align*}
and
\begin{align*}
P_k^j\phi^j(x)=\begin{cases}
\phi^j(x_1,x_2),&\mbox{ if }\lim\limits_{k\to\infty}\lambda_k^j=1,\\
P_{\leqslant(\lambda_k^j)^\theta}^{x_1}\phi^j(x_1,x_2),&\mbox{ if }\lim\limits_{k\to\infty}\lambda_k^j=\infty,
\end{cases}
\end{align*}
where $\theta\ll1$ is a fixed small number. Finally, we have the decoupling of energy and mass
\begin{gather*}
\lim_{k\to\infty}\left(E(u_k)-\sum_{j=1}^JE(\phi_k^j)-E(r_k^J)\right)=0,\\
\lim_{k\to\infty}\left(M(u_k)-\sum_{j=1}^JM(\phi_k^j)-M(r_k^J)\right)=0.\\
\end{gather*} 
\end{theorem}
Before giving the two type of linear profile decomposition, we need to establish the refined Strichartz estimates. By using the bilinear restriction estimate \cite{Tao-bilinear} and Whitney-type decomposition, we have the following refined Strichartz estimates.

\begin{proposition}[Refined Strichartz inequality]\label{refined-Strichartz}We have 
\begin{align*}
\big\|e^{it\partial_{x_1}^2}f\big\|_{L_{t,x_1,x_2}^{6}(\R\times\R\times\R)}\lesssim\|f\|_{L_{x_1}^2\mathcal{H}_{x_2}^1}^\frac{2}{3}\left(\sup_{Q\in\mathcal{D}}|Q|^{-\frac{1}{5}}	\big\|e^{it\partial_{x_1}^2}f_Q\big\|_{L_{t,x_1,x_2}^{10}}\right)^\frac{1}{3},
\end{align*}
where $f_Q(x)=\mathcal{F}^{-1}\big\{\chi_Q\widehat{f}(\xi)\big\}(x)$, and  $\mathcal{D}=\bigcup\limits_{j\in\Z}\mathcal{D}_j$ with $\mathcal{D}_j$ being the set of all dyadic intervals with the side-length of $2^j$ in $\R$,
\begin{align*}
\mathcal{D}_j=\left\{\prod_{m=1}^{d-1}[2^jk_m,2^j(k_m+1))\in\R: k=(k_1,\cdots,k_{d-1})\in\Z^{d-1}\right\}.
\end{align*}

\end{proposition}

Next, we recall the refined Fatou lemma and local smoothing estimate, which will be used in the proof of inverse Strichartz inequality.
\begin{lemma}[Refined Fatou]\label{Refine-Fatou}Let $\{f_k\}$ be a bounded sequence in $L_{t,x_1}^6\mathcal{H}_{x_2}^{1}(\R\times\R\times \R)$ satisfying
\begin{align*}
\limsup_{k\to\infty}\|f_k\|_{L_{t,x_1,x_2}^6(\R\times\R\times\R)}<\infty.
\end{align*} 
If $f_k\to f$ almost everywhere, then we have
\begin{align*}
\|f_k\|_{L_{t,x_1,x_2}^6(\R\times\R\times\R)}^{6}-\|f_k-f\|_{L_{t,x_1,x_2(\R\times\R\times\R)}^6}^6\to\|f\|_{L_{t,x_1,x_2(\R\times\R\times\R)}^6}^6
\end{align*}
when $k\to\infty$.
\end{lemma}
\begin{lemma}[Local smoothing estimate, \cite{Killip-Visan-note}]\label{Local smoothing}Fix $\varepsilon>0$,  for $\forall f\in L^2(\R\times\R)$ and $R>0$, we have
\begin{align}\label{Local-Smoothing-estimate}
\int_{\R}\int_{\R\times\R}\left|(|\partial_{x_1}|^{\frac{1}{2}}e^{it\partial_{x_1}^2}f)(x_1,x_2)\right|^2\langle x_1\rangle^{1-\varepsilon}dtdx_1dx_2\lesssim\|f\|_{L_{x_1,x_2}^2(\R\times\R)}^2.
\end{align}
\end{lemma}
We also need the following refined Sobolev embedding for Hermite-Sobolev space $\mathcal{H}^1$.
\begin{lemma}[Refined Sobolev Embedding,\cite{Cheng-Guo-Guo-Liao-Shen}]\label{refine-Sobolev} For $f\in \mathcal{H}^1(\R)$ and $R>0$,  we have
\begin{align*}
\|f\|_{L^\infty(|x|\geqslant R)}\lesssim R^{-\frac{1}{2}}\big(\|f\|_{L^2(\R)}+\|xf(x)\|_{L^2(\R)}^\frac{1}{2}\|f^\prime(x)\|_{L^2(\R)}^\frac{1}{2}\big).
\end{align*}
\end{lemma}

By the refined Strichartz estimates, refined Fatou Lemma and local smoothing estimate, we can prove the following inverse Strichartz estimate.
\begin{proposition}[Inverse Strichartz estimate] \label{inverse-strichartz}
Let  $\left\{f_k \right\}_{k\geqslant1}$ be a bounded sequence in $ L_{x_1}^2 \mathcal{H}_{x_2}^{1}(\R\times \R ) $ satisfying
\begin{align}\label{inverse-condition}
\lim_{k \to \infty} \left\|f_k \right\|_{L_{x_1}^2 \mathcal{H}_{x_2}^{1}(\R\times\R) } = A \quad\text{ and }\quad \lim_{k \to \infty} \left\|e^{it\partial_{x_1}^2} f_k \right\|_{L_{t,x_1,x_2}^{6} (\R\times\R\times\R) } = \epsilon.
\end{align}
Then, there exist $\phi \in L_{x_1}^2 \mathcal{H}_{x_2}^{1} $ and $ \left(\lambda_k,t_k,  \xi_k,  x_{1,k} \right) \in \R_+ \times \R  \times \R \times \R$, so that passing to a further subsequence  if necessary, we have
\begin{align}
& \lambda_k^{\frac{d-1}{2}}  e^{-i\xi_k \cdot (\lambda_k   x_1 + x_{1,k} ) } \left(e^{it_k \partial_{x_1}^2} f_k \right)\left(\lambda_k  x_1 + x_{1,k},x_2\right) \rightharpoonup  \phi(x_1,x_2)\  \text{ in } L_{x_1}^2 \mathcal{H}_{x_2}^{1}(\R\times\R), \text{ as } k \to \infty, \notag\\
& \lim_{k \to \infty} \left(\|f_k  \|_{L_{x_1}^2 \mathcal{H}_{x_2}^{1}(\R\times\R)}^2 - \|f_k -\phi_k \|_{L_{x_1}^2 \mathcal{H}_{x_2}^{1}(\R\times\R)}^2 \right) = \|\phi\|_{L_{x_1}^2 \mathcal{H}_{x_2}^{1}}^2 \gtrsim A^2 \left(\frac\epsilon A\right)^{12 },\label{orthogonal-inverse} \\
&  \limsup_{k \to \infty} \left\|e^{it\partial_{x_1}^2} (f_k -\phi_k ) \right\|_{L_{t,x_1,x_2}^6(\R\times\R\times\R)  }^6 \le \epsilon^{6 }   \left( 1- c \big(\frac{\varepsilon}{A}\big)^\beta \right), \label{orthogonal-inverse2}
\end{align}
where $c$ and $\beta$ are positive constants, and
\begin{align*}
\phi_k (x_1,x_2) = \frac1{\lambda_k^\frac{1}{2} } e^{iy\cdot \xi_k } \left(e^{-i\frac{t_k }{\lambda_k^2} \partial_{x_1}^2} \phi\right)\left(\frac{x_1-x_{1,k} }{\lambda_k },x_2\right).
\end{align*}
Moreover, if $\{f_k \}_{k \ge 1} $ is bounded in $\Sigma_{x_1,x_2} (\R\times\R)$, and also
\begin{align}\label{eq4.8v115}
\lim_{k \to \infty} \left\|f_k \right\|_{\Sigma_{x_1,x_2}(\R\times\R) } = A \quad\text{ and }\quad \lim_{k \to \infty} \left\|e^{it\partial_{x_1}^2} f_k \right\|_{L_{t,x_1,x_2}^6 (\R\times\R\times\R) } = \epsilon,
\end{align}
then we choose $\lambda_k\geqslant1$, $|\xi_k | \lesssim 1$ and
$\phi \in L_{x_1}^2 \mathcal{H}_{x_2}^{1 }(\R \times \R) $ such that
\begin{align}
\lambda_k  e^{-i\xi_k\cdot (\lambda_k  x_1 + x_{1,k} ) } \left(e^{it_k \partial_{x_1}^2} f_k  \right)\left(\lambda_k  x_1+ x_{1,k},x_2\right) \rightharpoonup  \phi(x_1,x_2)\  \text{ in } L_{x_1}^2 \mathcal{H}_{x_2}^{1 }(\R\times\R) , \text{ as } k \to \infty, \label{eq4.9v115}
\end{align}
and 
\begin{align}
\lim_{k \to \infty} \left( \left\|f_k \right\|_{\Sigma_{x_1,x_2}(\R\times\R)}^2 - \left\|f_k -\phi_k \right\|_{\Sigma_{x_1,x_2}(\R\times\R)}^2 \right) = \lim\limits_{k \to \infty} \|\phi_k \|_{\Sigma_{x_1,x_2}(\R\times\R) }^2\gtrsim A^2 \left(\frac\epsilon A\right)^{12 }.\label{eq3.1348}
\end{align}
\end{proposition}

\begin{proof}
First, we assume that $ \left\{f_k \right\}_{k \ge 1} $ is bounded in $L_{x_1}^2 \mathcal{H}_{x_2}^1$. By the refined Strichartz estimate(cf. Proposition \ref{refined-Strichartz}), there exists the cube $Q_k$ such that
\begin{align}\label{key-1}
\varepsilon^{3}A^{-2}\lesssim\liminf_{k \to \infty}|Q_k|^{-\frac{1}{5}}\|e^{it\partial_{x_1}^2}(f_k)_{Q_k}\big\|_{L_{t,x_1,x_2}^{10}(\R\times\R\times\R)}.
\end{align}
Then choosing $\lambda_{k}^{-1}$ to be the side-length of $Q_k$ which implies that $|Q_k|=\lambda_k^{-1}$. We also let $\xi_k $ be the center of the cube $Q_k $. Hence, by H\"older's inequality, we have
\begin{align*}
&\liminf_{ n \rightarrow \infty }|Q_k|^{-\frac{1}{5}}\big\|e^{it\partial_{x_1}^2}(f_k)_{Q_k}\big\|_{L_{t,x_1,x_2}^{10}(\R\times\R\times\R)}\\
\lesssim&\liminf_{ n \rightarrow \infty }|Q_k|^{-\frac{1}{5}}\big\|e^{it\partial_{x_1}^2}(f_k)_{Q_k}\big\|_{L_{t,x_1,x_2}^{10}(\R\times\R\times\R)}^\frac{3}{5}\big\|e^{it\partial_{x_1}^2}(f_k)_{Q_k}\big\|_{L_{t,x_1,x_2}^\infty(\R\times\R\times\R)}^{\frac{2}{5}}\\
\lesssim&\liminf_{ n \rightarrow \infty }\lambda_k^{\frac15}\varepsilon^{\frac35}\big\|e^{it\partial_{x_1}^2}(f_k)_{Q_k}\big\|_{L_{t,x_1,x_2}^\infty(\R\times\R\times\R)}^{\frac25}.
\end{align*}
Then by \eqref{key-1}, there exists a frame $(t_k,x_{1,k},(x_2)_k)$ such that
\begin{align}\label{eq4.5336}
\liminf_{k \to \infty} \lambda_k^\frac{1}{2} \left|\left(e^{it_k  \Delta_{\R^2}} \left(f_k \right)_{Q_k } \right)(x_{1,k} ,(x_2)_k )\right|\gtrsim \epsilon^6 A^{-5 }  .
\end{align}
Since $ \left|(x_2)_k \right| \leqslant R$, we may assume, up to a subsequence, $(x_2)_k \to x^*$, as $k \to \infty$, with $|x^*|\lesssim 1$.
	
By the weak compactness of $L_{x_1}^2 \mathcal{H}_{x_2}^{1}(\R\times\R)$, we have
\begin{align*}
\lambda_k^\frac{1}{2}  e^{-i\xi_k (\lambda_k  x_1+x_{1,k} )} \left(e^{it_k  \partial_{x_1}^2} f_k \right)\left(\lambda_k  x_1 + x_{1,k} , x_2 \right) \rightharpoonup  \phi(x_1,x_2) \text{ in } L_{x_1}^2 \mathcal{H}_{x_2}^{1}(\R\times\R) \text{ as } k \to \infty.
\end{align*}
In fact, we have the following convergence  in Hilbert space $H$,
\begin{align*}
g_k \rightharpoonup g \text{ in } H \Rightarrow \|g_k \|_H^2 - \|g_k - g\|_H^2 \to \|g\|_H^2.
\end{align*}
Thus, we have
\begin{align*}
\lim_{k \to \infty} \left(\|f_k  \|_{L_{x_1}^2 \mathcal{H}_{x_2}^{1}(\R\times\R)}^2 - \|f_k -\phi_k \|_{L_{x_1}^2 \mathcal{H}_{x_2}^{1}(\R\times\R)}^2 \right) = \|\phi\|_{L_{x_1}^2 \mathcal{H}_{x_2}^{1}(\R\times\R)}^2.
\end{align*}
	
Now, it remains to prove $\eqref{orthogonal-inverse}$ and $\eqref{orthogonal-inverse2}$. We take a function $h$  which the Fourier transform is   $ \chi_{\left[-\frac12,\frac12 \right]}(\xi)$.  Recall that $\delta_0(x)\in \mathcal{H}^{-1}(\R)$, thus we have $h(x_1) \delta_0(x_2)\in L_{x_1}^2 \mathcal{H}_{x_2}^{-1}(\R \times \R)$.
From \eqref{eq4.5336}, we obtain
\begin{align}\label{lower-bound}
&\left|\left\langle h(x_1) \delta_0(x_2), \phi \left(x_1,x_2+ x^* \right)\right\rangle_{x_1,x_2} \right|\notag\\
=&\lim_{k \to \infty} \left|\bigg\langle \delta_0(x_2), \int_{\R}\bar{h}(x_1) \lambda_k^\frac{1}{2}  e^{-i\xi_k\cdot (\lambda_k  x_1 + x_{1,k} )} \left(e^{it_k  \partial_{x_1}^2} f_k  \right) \left(\lambda_k  x_1 +x_{1,k} , x_2+ (x_2)_k \right) \,dx_1 \bigg\rangle_{x_2} \right| \nonumber\\
=&  \lim_{k \to \infty} \lambda_k^\frac{1}{2}  \left|\left(e^{it_k \partial_{x_1}^2} \left(f_k \right)_{Q_k }\right)(x_{1,k} ,(x_2)_k )\right| \gtrsim   \epsilon^6A^{-5 }. 
\end{align}
Consequently, we have
\begin{align*}
\left\|\phi \left(x_1,x_2+ x^* \right) \right\|_{L_{x_1}^2 \mathcal{H}_{x_2}^1(\R\times\R)} \gtrsim  \epsilon^6A^{-5 }  .
\end{align*}
On the other hand, since
\begin{align*}
\|\phi(x_1,x_2)\|_{L_{x_1}^2 \mathcal{H}_{x_2}^1(\R\times\R)}
\geqslant  \left\|\phi \left(x_1,x_2+x^* \right) \right\|_{L_{x_1}^2 \mathcal{H}_{x_2}^1(\R\times\R)} -  \left|x^* \right| \| \phi\|_{L_{x_1,x_2}^2(\R\times\R)},
\end{align*}
we get
\begin{align*}
\left\|\phi \left(x_1,x_2+x^* \right)  \right\|_{L_{x_1}^2 \mathcal{H}_{x_2}^1(\R\times\R)}\leqslant  \|\phi\|_{L_{x_1}^2 \mathcal{H}_{x_2}^1(\R\times\R)} +  \left|x^* \right| \|\phi\|_{L_{x_1,x_2}^2(\R\times\R)} \lesssim \|\phi\|_{L_{x_1}^2 \mathcal{H}_{x_2}^1(\R\times\R)}.
\end{align*}
Therefore $\|\phi\|_{L_y^2 \mathcal{H}_x^1(\R\times\R)} \gtrsim    \epsilon^{6 } A^{-5 }  $ and we complete the proof of $\eqref{orthogonal-inverse}$.
	
We turn to prove $\eqref{orthogonal-inverse2}$. By local smoothing estimate \eqref{Local-Smoothing-estimate} and the  Rellich-Kondrashov theorem, we have
\begin{align*}
e^{it\partial_{x_1}^2} \left( \lambda_k^\frac{1}{2}  e^{-i \xi_k\cdot (\lambda_k  (x_1) +x_{1,k} ) } (e^{it_k \partial_{x_1}^2} f_k )(\lambda_k  x_1 + x_{1,k} , x_2 + (x_2)_k )\right) \to
e^{it\partial_{x_1}^2} \phi(x_1,x_2), \text{ as $k \to \infty$},
\end{align*}
for $(t,x_1,x_2) \in \R \times \R\times \R$.
Then invoking the refined Fatou lemma \ref{Refine-Fatou}, we have 
\begin{align*}
\left\|e^{it \partial_{x_1}^2} f_k \right\|_{L_{t,x_1,x_2}^6(\R\times\R\times\R)  }^6 - \left\|e^{it\partial_{x_1}^2} (f_k -\phi_k )\right\|_{L_{t,x_1,x_2}^6(\R\times\R\times\R) }^6 - \left\|e^{it\partial_{x_1}^2} \phi_k  \right\|_{L_{t,x_1,x_2}^6 (\R\times\R\times\R) }^6\to 0, \text{ as } k \to \infty.
\end{align*}
Therefore, using the invariance of Galilean transform, we have
\begin{align}\label{Galilean}
\limsup\limits_{k \to \infty} \left\|e^{it\partial_{x_1}^2} (f_k -\phi_k )\right\|_{L_{t,x_1,x_2}^6(\R\times\R\times\R)  }^6
= &  \limsup\limits_{k \to \infty} \left(  \left\|e^{it \partial_{x_1}^2} f_k \right\|_{L_{t,x_1,x_2}^6(\R\times\R\times\R) }^6  - \left\|e^{it\partial_{x_1}^2} \phi_k
\right\|_{L_{t,x_1,x_2}^6(\R\times\R\times\R) }^6 \right)\nonumber\\
= & \varepsilon^6 -   \left\|e^{it\partial_{x_1}^2} \phi   \right\|_{L_{t,x_1,x_2}^6(\R\times\R\times\R) }^6. 
\end{align}
We now take $q(t) \in C^\infty$, which is  supported  on  the  interval $[-1,1]$ so that
\begin{align*}
\|q(t) e^{it \partial_{x_1}^2} h \|_{L_{t,x_1}^\frac{6}{5}} = 1.
\end{align*}
Then by \eqref{lower-bound}, we have
\begin{align*}
& \left|\int_{\R} \left\langle q(t) h(x_1) \delta_0(x_2), \phi \left(x_1,x_2+ x^* \right)\right\rangle_{x_1,x_2} \,dt  \right|\gtrsim  \epsilon^\frac{12}{\epsilon_0} A^{1- \frac{12}{\epsilon_0}}  .
\end{align*}
On the other hand, by H\"older's inequality, Sobolev's inequality,, we have
\begin{align*}
\left| \int_{\R} \left\langle q(t) h(x_1) \delta_0(x_2), \phi \left(x_1, x_2+ x^* \right) \right\rangle_{x_1,x_2} \,dt \right|
&=  \left|\int_{\R}  \left\langle e^{it \partial_{x_1}^2} \left( q(t) h(x_1) \delta_0(x_2) \right) , e^{it \partial_{x_1}^2} \phi \left(x_1, x_2+x^* \right) \right\rangle_{x_1,x_2} \,dt \right|\\
&\lesssim   \left\|e^{it \partial_{x_1}^2} \left( c(t) h(x_1)  \right) \right\|_{L_{t,x_1,x_2}^\frac{6}{5}(\R\times\R\times\R)   }  \left\|e^{it \partial_{x_1}^2} \phi \left(x_1, x_2 \right) \right\|_{L_{t,x_1,x_2}^6 (\R\times\R\times\R)} \\&\lesssim  \left\|e^{it \partial_{x_1}^2} \phi(x_1,x_2) \right\|_{L_{t,x_1,x_2}^6 (\R\times\R\times\R) }.
\end{align*}
Hence, combining the above two estimates and \eqref{Galilean}, we get \eqref{orthogonal-inverse2}. 
	
Furthermore, we assume that  $ \left\{f_k \right\}_{k \ge 1} $ is bounded in $\Sigma_{x_1,x_2} (\R\times\R)$. In this case, Bernstein's inequality implies that
\begin{align*}
\limsup\limits_{k \to \infty} \left\|P_{\ge R}^{x_1} f_k \right\|_{L_{x_1}^2 \mathcal{H}_{x_2}^{1- {\epsilon_0} }(\R\times\R)}
& \lesssim R^{-\varepsilon_0}  \limsup\limits_{k \to \infty} \left\|f_k \right\|_{\Sigma_{x_1,x_2}(\R\times\R)  } \to 0 \text{ as  }  R\to \infty.
\end{align*}
Chooing $R\in 2^{\mathbb{N}}$ large enough relying only on  $\epsilon$, we invoke the dyadic decomposition to $f_k$, then we have
\begin{align*}
\lim\limits_{k \to \infty} \left\|e^{it\partial_{x_1}^2}  f_k \right\|_{L_{t,x_1,x_2}^6(\R\times\R\times\R)  } & \leqslant \lim\limits_{k \to \infty}
\left\|e^{it\partial_{x_1}^2}P_{\leq R}^{x_1} f_k \right\|_{L_{t,x_1,x_2}^6 (\R\times\R\times\R) } +\lim\limits_{k \to \infty} \left\|e^{it\partial_{x_1}^2} P_{\ge R}^{x_1}  f_k \right\|_{L_{t,x_1,x_2}^6(\R\times\R\times\R) }.
\end{align*}
Then by the Strichartz estimate and Sobolev embedding, we have
\begin{align*}
\lim\limits_{k \to \infty} \left\|e^{it\partial_{x_1}^2} P_{\le R}^{x_1} f_k \right\|_{L_{t,x_1,x_2}^6(\R\times\R\times\R)  } & 
\ge  \lim\limits_{k \to \infty}\left\|e^{it\partial_{x_1}^2} f_k \right\|_{L_{t,x_1,x_2}^6(\R\times\R\times\R)} - C   \lim\limits_{k \to \infty} \left\| P_{\ge R}^{x_1} f_k \right\|_{L_{x_1}^2 \mathcal{H}_{x_2}^{1- {\epsilon_0} }(\R\times\R)}\\ &\ge \varepsilon.
\end{align*}
In the sequel, we can replace $f_k $ by $P_{\le R}^{x_1} f_k $ in the above case, and for $R\gg1$ large enough, we may take $ \left\{Q_k \right\}_{k  \ge 1}  \subseteq \mathcal{D}$ and $ \left|Q_k\right |\lesssim R^2$ such that $\lambda_k  \gtrsim R^{-1}$, and $|\xi_k |\lesssim R$.   Notice that if $\phi\in\Sigma_{x_1,x_2}(\R\times\R)$, we cannot have that $\lambda_k$ is always bounded. Thus we need to treat the case when $\limsup_{k \to \infty}\lambda_k=\infty$ additionally. The case of $\lambda_k<\infty$ can be treated similar to the case that $f_k\in L_{x_1}^2\mathcal{H}_{x_2}^1$. 
The remaining part of Proposition \ref{inverse-strichartz} is \eqref{eq3.1348} for the case of $\lambda_k\to\infty$.
\begin{align*}
{ \lim\limits_{k \to \infty}\left \|\phi_k \right\|_{\Sigma_{x_1,x_2}(\R\times\R)}^2 \ge \lim\limits_{k \to \infty} \left\|P^{x_1}_{\le \lambda_k^\theta} \phi \right\|_{L_{x_1}^2 \mathcal{H}_{x_2}^1(\R\times\R\times\R)}^2 }\gtrsim A^2 \left(\frac\epsilon A\right)^{12}.
\end{align*}
Then the decoupling of the $\Sigma-$norm comes from { $P_{\leq \lambda_k^\theta} \to Id$} in $L_{x_1}^2 \mathcal{H}_{x_2}^1$ and  \eqref{eq4.9v115}.
\end{proof}

Next,	utilizing the inverse Strichartz estimates and following the strategy of Killip-Visan \cite{Killip-Visan-note}, we can finish the proof of Theorem \ref{LinearProfile} and Theorem \ref{linearprofile2}. Thus we omit the proof here. 

\subsection{ Nonlinear profile decomposition}

In this part, we will prove that for sufficiently large $\lambda>0$, 
the nonlinear profile $u_\lambda$ given in \eqref{nonlinaer-profile}
\begin{align*}
\begin{cases}
i\partial_t u_\lambda +(\partial_{x_1  }^2+\partial_{x_2  }^2) u_\lambda - x_2^2 u_\lambda = |u_\lambda|^4 u_\lambda,\\
u_\lambda(0,x_1,x_2) = \frac1{\lambda^\frac{1}{2}} \phi(\frac{x_1}\lambda, x_2),
\end{cases}
\end{align*}
can be approximated by $\tilde u_\lambda$ given in  \eqref{limit-equation}
\begin{align}\label{fml-DCR-2}
\tilde u_\lambda(t,x_1,x_2) = e^{it\left(\partial_{x_2  }^2 -x_2^2\right)} \sum_{n\in \mathbb{N} }  \frac1{\lambda^\frac{1}{2}} v_n\left(\frac{t}{\lambda^2}, \frac{x_1}\lambda, x_2\right) , \ (t,x_1,x_2)\in \mathbb{R}\times \mathbb{R} \times \mathbb{R},
\end{align}
where $v_n$ solves the (DCR) system, which is given by
\begin{align}\label{DCR-5}
\begin{cases}
\left(i\partial_t + \partial_{x_1}^2 \right) v_n(t,x_1,x_2) =
\sum\limits_{\substack{n_1,n_2,n_3,n_4,n_5,n \in \mathbb{N} , \\n_1-n_2+ n_3-n_4+n_5= n} } \Pi_n\left(v_{n_1} \bar{v}_{n_2} v_{n_3}\bar{v}_{n_4}v_{n_5} \right)(t,x_1,x_2), \\
v_n(0,x_1,x_2) = \phi_n(x_1,x_2)= \Pi_n\phi(x_1,x_2).
\end{cases}
\end{align}

Now, we give the following proposition, which describes the persistence of the higher regularity. The proof is standard and thus we omit the proof. We refer to \cite{Killip-Visan-note} for more details.
\begin{proposition}[Persistence of the regularity]\label{cor-preserve-regularity}
Suppose that $\phi \in L_{x_1}^2 \mathcal{H}_{x_2}^1 (\mathbb{R} \times \mathbb{R})$ and $v$ is the global solution of \eqref{DCR-5} given as in Theorem \ref{Thm2}. For any $s_1\ge 0 $ and $ s_2 \ge 1 $, if we assume further $v|_{t=0}\in H_{x_1}^{s_1 } \mathcal{H}_{x_2}^{s_2} (\mathbb{R}\times \mathbb{R})$, then the solution $v\in C_t^0 H_{x_1}^{s_1 } \mathcal{H}_{x_2}^{s_2} (\mathbb{R}\times \mathbb{R} \times \mathbb{R})$ and satisfies
\begin{align*}
\|v\|_{L_t^\infty H_{x_1}^{s_1 } \mathcal{H}_{x_2}^{s_2}  \cap L_{t}^6 W_{x_1}^{s_1 ,6} \mathcal{H}_{x_2}^{s_2}  (\mathbb{R}\times \mathbb{R} \times \mathbb{R})} \le C\left(\|\phi \|_{H_{x_1}^{s_1 } \mathcal{H}_{x_2}^{s_2} (\mathbb{R}\times \mathbb{R})} \right).
\end{align*}
\end{proposition}

Based on Proposition \ref{cor-preserve-regularity}, we prove the  following theorem on approximating the non-linear profiles at large-scale. We will prove it via   stability theorem (cf. Theorem \ref{Stablity}).
\begin{theorem}\label{large-scale}
For any $\phi \in L_{x_1}^2 \mathcal{H}_{x_2}^1 (\R \times \R) $, $0 < \theta \ll 1$, $( \lambda_k, t_k, x_{1,k}, \xi_k) \in \mathbb{R}_+ \times \mathbb{R} \times \mathbb{R} \times \mathbb{R}$, $ | \xi_k  | \lesssim 1$ and $\lambda_k  \to \infty$ when $k  \to \infty$, then there exists a global solution $u_k \in C_t^0 L_{x_1}^2 \mathcal{H}_{x_2}^1 (\R \times \R \times \R) $ of
\begin{align*}
\begin{cases}
i \partial_t u_k +\partial_{x_1}^2 u_k + \partial_{x_2}^2 u_k - x_2^2 u_k  = |u_k|^4 u_k , \\
u_k (0,x_1,x_2) = \lambda_k^{-\frac{1}{2}} e^{ix_1\cdot \xi_k } \left( e^{it_k  \partial_{x_1}^2} { P_{\le \lambda_k^\theta}^{x_1} \phi} \right) \left( \frac{ x_1 - x_{1,k}}{\lambda_k }, x_2 \right),
\end{cases}
\end{align*}
for $k $ sufficiently large and obeying 
\begin{align*}
\|u_k  \|_{L_t^\infty L_{x_1}^2 \mathcal{H}_{x_2}^1 \cap L_{t,x_1}^6 \mathcal{H}_{x_2}^1( \mathbb{R}\times \mathbb{R} \times \mathbb{R})} \lesssim_{\|\phi\|_{L_{x_1}^2 \mathcal{H}_{x_2}^1 (\R \times \R) } }1.
\end{align*}
Moreover, suppose that $\epsilon_4 = \epsilon_4 \left( \|\phi\|_{L_{x_1}^2 \mathcal{H}_{x_2}^1 (\R \times \R \times \R) } \right) $ is a sufficiently small positive constant and $\psi  \in H_{x_1}^{10 } \mathcal{H}_{x_2}^{10} (\R \times \R) $ such that
\begin{align*}
\| \phi - \psi  \|_{L_{x_1}^2 \mathcal{H}_{x_2}^1 (\R \times \R) } \le \epsilon_4.
\end{align*}
Then there exists a solution $v \in C_t^0 H_{x_1}^2 \mathcal{H}_{x_2}^1( \mathbb{R} \times \mathbb{R} \times \mathbb{R})$ of (DCR), with
\begin{align*}
v(0, x_1 ,x_2 )  = \psi(x_1,x_2) ,   & \text{ if } t_k = 0, \\
\lim\limits_{t\to \pm \infty } \| v(t,y,x) - e^{it \partial_{x_1}^2} \psi \|_{L_{x_1}^2 \mathcal{H}_{x_2}^1 (\R \times \R) }  =  0,  & \text{ if  } t_k \to \pm \infty,
\end{align*}
such that for $k $ large enough, we have
$ \left\|u_k  \right\|_{L_t^\infty L_{x_1}^2 \mathcal{H}_{x_2}^1 \cap L_{t,x_1}^6 \mathcal{H}_{x_2}^1( \mathbb{R} \times \mathbb{R} \times \mathbb{R})}  \lesssim 1,$
with
\begin{align*}
\left\|u_k (t) - w_{\lambda_k } (t) \right\|_{L_t^\infty L_{x_1}^2 \mathcal{H}_{x_2}^1 \cap L_{t,x_1}^6 \mathcal{H}_{x_2}^1( \mathbb{R} \times \mathbb{R} \times \mathbb{R})}    \to 0, \text{ as } k \to \infty,
\end{align*}
where
\begin{align*}
w_{\lambda_k}  (t,x_1,x_2) = e^{- i(t - t_k ) |\xi_k |^2} e^{ix_1\cdot \xi_k } \lambda_k^{-\frac{1}{2}} e^{it \left (  \partial_{x_2}^2 - x_2^2 \right) } v\left( \frac{t}{\lambda_k^2} + t_k , \frac{ x_1 - x_{1,k} - 2 \xi_k ( t-t_k)}{ \lambda_k } , x_2 \right).
\end{align*}
\end{theorem}

\begin{proof}
Using the translation symmetry, we can take $x_{1,k}=0$. By the partial Galilean transformation 
\begin{align*}
u(t,x_1,x_2)\to e^{-it|\xi|^2}e^{ix_1\cdot\xi}u(t,x_1-2t\xi,x_2)
\end{align*}
for $\xi\in\R$  and since $|\xi_k|$ is bounded, we can take $\xi_k=0$. Then we have
\begin{equation*}
w_{\lambda_k}(t,x_1,x_2) = \lambda_k^{-\frac{1}{2}} e^{it (\partial_{x_2}^2-x_2^2)} v\left( \frac{t}{\lambda_k^2} + t_k , \frac{ x_1 }{ \lambda_k } , x_2 \right).
\end{equation*}
First, we will show that $w_{\lambda_k}$ is an approximation solution to $\eqref{NLS-final}$ when $t_k=0$. By the stability theorem, after a simple calculus we deduce this theorem to show that
\begin{equation}\label{fml-error-term-target}
\bigg\|  \int_{0}^{t} e^{i (t-\tau) (\partial_{x_1}^2 + \partial_{x_2}^2 - x_2^2)} e_{\lambda_k} (\tau)  d\tau   \bigg\|_{L^6_{t,x_1}  \mathcal{H}^1_{x_2} (\R \times \R \times \R) }\longrightarrow  0 , \text{as} \hspace{2ex} k\rightarrow \infty,
\end{equation}
where
\begin{align}\label{fml-error-term}
e_{\lambda_k}:=&(i\partial_t+\partial_{x_1}^2+\partial_{x_2}^2-x_2^2)w_{\lambda_k}-|w_{\lambda_k}|^4 w_{\lambda_k}\nonumber\\
=&-\lambda_k^{-\frac{5}{2}}\sum_{n\in\mathbb{N}}e^{-it(2n+1)}\sum\limits_{\substack{n_1,n_2,n_3,n_4,n_5 \in \mathbb{N}  , \\ n_1-n_2+ n_3-n_4+n_5 \neq n}}  e^{-2it (n_1-n_2+n_3-n_4+n_5-n)}\nonumber \\
&\hspace{32ex}\times   \left(   \Pi_n( v_{n_1} \bar{v}_{n_2} v_{n_3} \bar{v}_{n_4} v_{n_5} )\right)\left(\frac{t}{\lambda_k^2} , \frac{x_1}{\lambda_k} , x_2 \right).
\end{align}
We then divide the error term $\eqref{fml-error-term}$ into three terms
\begin{align*}
e_{\lambda_k}=&-\lambda_k^{-\frac{5}{2}}\sum_{n\in\mathbb{N}}e^{-it(2n+1)}\sum\limits_{\substack{n_1,n_2,n_3,n_4,n_5 \in \mathbb{N}  }}  e^{-2it (n_1-n_2+n_3-n_4+n_5-n) }\\
&\hspace{32ex}\times  P^{x_1}_{\ge 2^{-10}}  \left(\Pi_n( v_{n_1}  \overline{v}_{n_2} v_{n_3}  \overline{v}_{n_4} v_{n_5} ) \left( \frac{t}{\lambda_k^2} , \frac{x_1}{\lambda_k} , x_2 \right)  \right)\\
&+\lambda_k^{-\frac{5}{2}}\sum_{n\in\mathbb{N}}e^{-it(2n+1)}\sum\limits_{\substack{n_1,n_2,n_3,n_4,n_5 \in \mathbb{N}, \\ n_1-n_2+n_3-n_4+n_5=n  }}  e^{-2it (n_1-n_2+n_3-n_4+n_5-n)}\\
&\hspace{32ex}\times P^{x_1}_{\ge 2^{-10}}  \left(\Pi_n( v_{n_1}  \overline{v}_{n_2} v_{n_3}  \overline{v}_{n_4} v_{n_5} ) \left( \frac{t}{\lambda_k^2} , \frac{x_1}{\lambda_k} , x_2 \right)  \right)\\
&-\lambda_k^{-\frac{5}{2}}\sum_{n\in\mathbb{N}}e^{-it(2n+1)}\sum\limits_{\substack{n_1,n_2,n_3,n_4,n_5 \in \mathbb{N}, \\ n_1-n_2+n_3-n_4+n_5 \ne n  }}  e^{-2it (n_1-n_2+n_3-n_4+n_5-n)}\\
&\hspace{32ex}\times  P^{x_1}_{\le 2^{-10}}  \left(\Pi_n( v_{n_1}  \overline{v}_{n_2} v_{n_3}  \overline{v}_{n_4} v_{n_5} ) \left( \frac{t}{\lambda_k^2} , \frac{x_1}{\lambda_k} , x_2 \right)  \right)\\
:=& e^1_{\lambda_k}  +  e^2_{\lambda_k}  +  e^3_{\lambda_k}.
\end{align*}
By the Strichartz estimate, the terms $e^1_{\lambda_k}$ and $e^2_{\lambda_k}$ can be deduced to prove
\begin{equation}\label{two-term}
\|e^j_{\lambda_k}\|_{L^{\frac{6}{5}}_{t,x_1}\mathcal{H}^1_{x_2} (\R \times \R \times \R) }\longrightarrow  0\mbox{ as } k\rightarrow 0.
\end{equation}  
For the last term $e^3_{\lambda_k}$, we will show 
\begin{equation}\label{last-term}
\bigg\|  \int_{0}^{t} e^{i (t-\tau) (\partial_{x_1}^2 + \partial_{x_2}^2 - x_2^2)} e^3_{\lambda_k} (\tau)  d\tau   \bigg\|_{L^6_{t,x_1}  \mathcal{H}^1_{x_2} (\R \times \R \times \R) }\longrightarrow  0 \mbox{ as }  k\rightarrow \infty.
\end{equation}
	
First,	we treat $e^1_{\lambda_k}$. Notice that after the partial Fourier transform with respect to $x_1$, there are no low frequency. Hence, by Bernstein's inequality, we  find that
\begin{align}\label{fml-error-t-1-sub}
&\|e^1_{\lambda_k}\|_{L^{\frac{6}{5}}_{t,x_1}\mathcal{H}^1_{x_2} (\R \times \R \times \R) }   \\
\lesssim& \lambda^{-\frac{7}{2}}_k   \bigg\|  \sum_{n\in\mathbb{N}}  e^{-i  t(2n+1)}  \sum   \limits_{\substack{n_i \in \mathbb{N}\\i=1,2,3,4,5  }}  e^{-2it  (n_1-n_2+n_3-n_4+n_5-n)}\partial_{x_1} P^{x_1}_{\le 2^{-10}} \Pi_{n}( v_{n_1} (\overline{v}_{n_2} v_{n_3} \overline{v}_{n_4}  v_{n_5}) )  (\frac{t}{\lambda_k^2} , \frac{x_1}{\lambda_k} , x_2) \bigg\|_{L^{\frac{6}{5}}_{t,x_1}\mathcal{H}^1_{x_2} }.\notag
\end{align}
By change of variables, the Leibniz rule, Plancherel theorem and H\"older's inequality, we obtain the following
\begin{align}\label{fml-error-term-1}
&\mbox{(RHS) of } \eqref{fml-error-t-1-sub}   \nonumber  \\
 \lesssim&\lambda^{-1}_k   \bigg\|  \sum_{n\in\mathbb{N}}  e^{-i  \lambda^2_k  t(2n+1)}  \sum   \limits_{\substack{n_i \in \mathbb{N}\\i=1,2,3,4,5  }}  e^{-2i   \lambda^2_k  t  (n_1-n_2+n_3-n_4+n_5-n)}    \Pi_{n}( \partial_{x_1}\overline{v}_{n_2} (v_{n_1} v_{n_3} \overline{v}_{n_4}  v_{n_5}) )  (t,x_1,x_2) \bigg\|_{L^{\frac{6}{5}}_{t,x_1}\mathcal{H}^1_{x_2} (\R \times \R \times \R) }  \nonumber  \\
&+ \lambda^{-1}_k   \bigg\|  \sum_{n\in\mathbb{N}}  e^{-i  \lambda^2_k  t(2n+1)}  \sum   \limits_{\substack{ni\in \mathbb{N}\\i=1,2,3,4,5  }}  e^{-2i   \lambda^2_k  t  (n_1-n_2+n_3-n_4+n_5-n)}    \Pi_{n}( \partial_{x_1}v_{n_3} (v_{n_1} \overline{v}_{n_2} \overline{v}_{n_4}  v_{n_5}) )  (t,x_1,x_2) \bigg\|_{L^{\frac{6}{5}}_{t,x_1}\mathcal{H}^1_{x_2} (\R \times \R \times \R) }  \nonumber  \\
&+ \lambda^{-1}_k   \bigg\|  \sum_{n\in\mathbb{N}}  e^{-i  \lambda^2_k  t(2n+1)}  \sum   \limits_{\substack{n_i \in \mathbb{N}\\i=1,2,3,4,5  }}  e^{-2i   \lambda^2_k  t  (n_1-n_2+n_3-n_4+n_5-n)}  \Pi_{n}( \partial_{x_1}\overline{v}_{n_4} (v_{n_1} \overline{v}_{n_2} v_{n_3}  v_{n_5}) )  (t,x_1,x_2) \bigg\|_{L^{\frac{6}{5}}_{t,x_1}\mathcal{H}^1_{x_2} (\R \times \R \times \R) }  \nonumber  \\
&+ \lambda^{-1}_k   \bigg\|  \sum_{n\in\mathbb{N}}  e^{-i  \lambda^2_k  t(2n+1)}  \sum   \limits_{\substack{n_i \in \mathbb{N}\\i=1,2,3,4,5  }}  e^{-2i   \lambda^2_k  t  (n_1-n_2+n_3-n_4+n_5-n)}   \Pi_{n}( \partial_{x_1}v_{n_1} (\overline{v}_{n_2} v_{n_3} \overline{v}_{n_4}  v_{n_5}) )  (t,x_1,x_2) \bigg\|_{L^{\frac{6}{5}}_{t,x_1}\mathcal{H}^1_{x_2} (\R \times \R \times \R) }  \nonumber  \\
&+ \lambda^{-1}_k   \bigg\|  \sum_{n\in\mathbb{N}}  e^{-i  \lambda^2_k  t(2n+1)}  \sum   \limits_{\substack{n_i \in \mathbb{N}\\i=1,2,3,4,5  }}  e^{-2i   \lambda^2_k  t  (n_1-n_2+n_3-n_4+n_5-n)}    \Pi_{n}( \partial_{x_1}v_{n_5} (v_{n_1} \overline{v}_{n_2} v_{n_3}  \overline{v}_{n_4}) )  (t,x_1,x_2) \bigg\|_{L^{\frac{6}{5}}_{t,x_1}\mathcal{H}^1_{x_2} (\R \times \R \times \R) }  \nonumber  \\
\lesssim& \lambda^{-1} _k\|\partial_{x_2}  v\|_{L^6_{t,x_1}  \mathcal{H}^1_{x_1} (\R \times \R \times \R) }  \|v\|^4_{L^6_{t,x_1}  \mathcal{H}^1_{x_2} (\R \times \R \times \R) }\rightarrow 0,\; \text{as} \hspace{2ex} k\rightarrow \infty.
\end{align}

Now, we turn to estimate the term $e^2_{\lambda_k}$. Similarly, By Bernstein's inequality, Leibniz rule, change of variables and symmetry among $v_{n_1},v_{n_2},v_{n_3},v_{n_4},v_{n_5}$, we get the following
\begin{align*}
&\|e^2_{\lambda_k}\|_{L^{\frac{6}{5}}_{t,x_1}\mathcal{H}^1_{x_2} (\R \times \R \times \R) }\\
\sim&  \lambda^{-1}_k  \bigg\|  \bigg(  \sum_{n}  \bigg|  \langle n \rangle^{\frac{1}{2}}   \sum\limits_{\substack{n_1,n_2,n_3,n_4,n_5, \\ n_1-n_2+n_3-n_4+n_5=n}}  \Pi_{n}  (\partial_{x_1} v_{n_1} \cdot   \overline{v}_{n_2}  v_{n_3}  \overline{v}_{n_4}  v_{n_5}) (t , x_1 , x_2)  \bigg|^2  \bigg)^\frac{1}{2}  \bigg\|_{  L^{\frac{6}{5}}_{t,x_1}  L^2_{x_2} (\R \times \R \times \R) }.
\end{align*}
From equation $n=n_1-n_2+n_3-n_4+n_5$, we can easily check that
\begin{equation*}
\langle n \rangle^2\lesssim \langle n_1 \rangle^2  \langle n_2 \rangle^2  \langle n_3 \rangle^2  \langle n_4 \rangle^2  \langle n_5 \rangle^2.
\end{equation*}
Consequently, we have
\begin{equation*}
\langle n \rangle^{\frac{1}{2}} \lesssim \langle n \rangle^{-1} \langle n \rangle^{2} \lesssim \langle n \rangle^{-1} \langle n_1 \rangle^2  \langle n_2 \rangle^2  \langle n_3 \rangle^2  \langle n_4 \rangle^2  \langle n_5 \rangle^2.
\end{equation*}
Notice that $\{\langle n \rangle^{-1}\}_{n=1}^{\infty}\in l^2_n$. Therefore, by using Minkowski's inequality, we can convert $l^1$ summation respect to variable $n$ to $l^2$ summation. This make us possible to obtain the summability over variable $n$ at the cost of some regularity on $x_2$ direction. By the boundedness of $\Pi_n$, we get 
\begin{align*}
&\bigg\|  \langle n \rangle^\frac{1}{2}  \sum\limits_{\substack{n_1,n_2,n_3,n_4,n_5, \\ n_1-n_2+n_3-n_4+n_5=n}}  \langle n_1 \rangle^2  \langle n_2 \rangle^2  \langle n_3 \rangle^2  \langle n_4 \rangle^2  \langle n_5 \rangle^2  \Pi_{n}  (\partial_{x_1} v_{n_1} \cdot   \overline{v}_{n_2}  v_{n_3}  \overline{v}_{n_4}  v_{n_5}) (t , x_1 , x_2)  \bigg\|_{  L^{\frac{6}{5}}_{t,x_1}  L^2_{x_2} l^2_n}\\
\lesssim & \bigg\|  \langle n \rangle^{-1}  \sum\limits_{\substack{n_1,n_2,n_3,n_4,n_5, \\ n_1-n_2+n_3-n_4+n_5=n}}  \langle n_1 \rangle^2  \langle n_2 \rangle^2  \langle n_3 \rangle^2  \langle n_4 \rangle^2  \langle n_5 \rangle^2  \Pi_{n}  (\partial_{x_1} v_{n_1} \cdot   \overline{v}_{n_2}  v_{n_3}  \overline{v}_{n_4}  v_{n_5}) (t , x_1 , x_2)  \bigg\|_{  L^{\frac{6}{5}}_{t,x_1}  L^2_{x_2} l^2_n}\\
\lesssim & \bigg\|  \sum\limits_{\substack{n_1,n_2,n_3,n_4,n_5}}  \langle n_1 \rangle^2  \langle n_2 \rangle^2  \langle n_3 \rangle^2  \langle n_4 \rangle^2  \langle n_5 \rangle^2    \|(\partial_{x_1} v_{n_1} \cdot   \overline{v}_{n_2}  v_{n_3}   \overline{v}_{n_4}  v_{n_5}) (t , x_1 , x_2)\|_{L^2_{x_2}}    \bigg\|_{  L^{\frac{6}{5}}_{t,x_1}}.
\end{align*}
H\"older's inequality and Sobolev embedding gives
\begin{equation*}
\|\partial_{x_1} v_{n_1} \cdot  \overline{v}_{n_2}  v_{n_3}  \overline{v}_{n_4}  v_{n_5}\|_{L^2_{x_2} (\R) }\lesssim \|\partial_{x_1} v_{n_1} \|_{L^2_{x_2} (\R) }  \|v_{n_2}\|_{L^{\infty}_{x_2} (\R) }  \|v_{n_3}\|_{L^{\infty}_{x_2} (\R) } \|v_{n_4}\|_{L^{\infty}_{x_2} (\R) }   \|v_{n_5}\|_{L^{\infty}_{x_2} (\R) }.  
\end{equation*}
Since $v_{n_1},v_{n_2},v_{n_3},v_{n_4}$ and $v_{n_5}$ hold the equivalent status, it is sufficient to show
\begin{align}\label{fml-error-term-2}
&\|e^2_{\lambda_k}\|_{L^{\frac{6}{5}}_{t,x_1}\mathcal{H}^1_{x_2}}  \nonumber  \\
\lesssim & \lambda^{-1}_k \big\|  \langle n_1 \rangle^2  \langle n_2 \rangle^2  \langle n_3 \rangle^2  \langle n_4 \rangle^2  \langle n_5 \rangle^2  \|\partial_{x_1} v_{n_1}\|_{L^2_{x_2}}   \|v_{n_2}\|_{L^{\infty}_{x_2}} \|v_{n_3}\|_{L^{\infty}_{x_2}}  \|v_{n_4}\|_{L^{\infty}_{x_2}}   \|v_{n_5}\|_{L^{\infty}_{x_2}}   \big\|_{L^{\frac{6}{5}}_{t,x_1}  l^2_{n_1} l^2_{n_2} l^2_{n_3} l^2_{n_4} l^2_{n_5} (\R \times \Z^5) }  \nonumber  \\
\lesssim & \lambda^{-1}_k  \|\partial_{x_1} v\|_{L^6_{t,x_1}  \mathcal{H}^6_{x_2}(\R^2)}     \|v\|^5_{L^6_{t,x_1}  \mathcal{H}^7_{x_2}(\R^2)}  \lesssim  \lambda^{-\frac{1}{3}}_k  C\big(  \|  \psi  \|_{H^1_{x_1}  \mathcal{H}^6_{x_2}(\R^2)}  \big)   C\big(  \|  \psi  \|_{L^2_{x_1}  \mathcal{H}^7_{x_2}(\R^2)}  \big)  \rightarrow  0, \text{as} \hspace{2ex} k\rightarrow \infty.
\end{align}
	
We now turn to estimate the term $e^3_{\lambda_k}$. Unfortunately, we can not get any  extra negative power of $\lambda_k$ by directly applying $\partial_{x_1}$ to $e^3_{\lambda}$, since $e^3_{\lambda_k}$ has very low frequency on $x_1$ direction. To overcome this difficulty, we use the normal form transform on $e^3_{\lambda_k}$ as follows:
\begin{align}\label{fml-error-term-3}
&\int_{0}^{t} e^{i (t-\tau) (\partial_{x_1}^2 + \partial_{x_2}^2 - x_2^2)} e^3_{\lambda_k} (\tau)  d\tau\\
=& -\lambda_k^{-\frac{5}{2}}  \sum\limits_{\substack{n_1,n_2,n_3,n_4,n_5,n\in\mathbb{N}, \\ n_1-n_2+n_3-n_4+n_5\ne n}}  \int_{0}^t  e^{it (\partial_{x_2}^2 -2n -1)}  e^{-it \widetilde{\partial}_{x_1}^2}  P^{x_1}_{\le 2^{-10}} \bigg(  \Pi_n (v_{n_1}  \overline{v_{n_2}}v_{n_3}\overline{v_{n_4}}   v_{n_5})  \bigg(  \frac{\tau}{\lambda_k^2} , \frac{x_1}{\lambda_k} , x_2  \bigg)  \bigg)  d\tau,\notag
\end{align}
where 
\begin{equation*}
\widetilde{\partial}_{x_1}^2 := 2(n_1-n_2+n_3-n_4+n_5-n) + \partial_{x_1}^2.
\end{equation*}
Using the restriction $n_1-n_2+n_3-n_4+n_5   \ne   n$ and $|\xi|\le 2^{-10}$, one can find $\widetilde{\partial}_{x_1}^2$ is revertible and its inverse can be defined as follows:
\begin{equation*}
\mathcal{F}_{x_1} \bigg(  (-i\widetilde{\partial}_{x_1}^2)^{-1}  f  \bigg)(\xi , x_2 ) = \frac{i (\mathcal{F}_{x_1} f) (\xi , x_2) }{2(n_1-n_2+n_3-n_4+n_5-n) - \xi^2 }.
\end{equation*}
	
Therefore, by integrating by parts,  $\eqref{fml-error-term-3}$ can be divided to three sub-terms
\begin{align*}
\mbox{(RHS) of }\eqref{fml-error-term-3} = & \lambda_k^{-\frac{5}{2}}  \sum\limits_{\substack{n_1,n_2,n_3,n_4,n_5,n \in \mathbb{N} \\ n_1-n_2+n_3-n_4+n_5\ne n}}  e^{it (\partial_{x_1}^2 + \partial_{x_2}^2 - x_2^2)}  (-i\widetilde{\partial}_{x_1}^2)^{-1} \\
&\hspace{24ex}  \times  P^{x_1}_{\le 2^{-10}} \bigg(  \Pi_n (v_{n_1}  \overline{v}_{n_2}  v_{n_3}  \overline{v}_{n_4}  v_{n_5})  \bigg(  \frac{0}{\lambda_k^2} , \frac{x_1}{\lambda_k} , x_2  \bigg)  \bigg)  \\
& - \lambda_k^{-\frac{5}{2}}  \sum\limits_{\substack{n_1,n_2,n_3,n_4,n_5,n \in \mathbb{N} \\ n_1-n_2+n_3-n_4+n_5\ne n}}  e^{-2it (n_1-n_2+n_3-n_4+n_5) - it}  (-i\widetilde{\partial}_{x_1}^2)^{-1} \\ 
&\hspace{24ex}  \times  P^{x_1}_{\le 2^{-10}} \bigg(  \Pi_n (v_{n_1}  \overline{v}_{n_2}  v_{n_3}  \overline{v}_{n_4}  v_{n_5})  \bigg(  \frac{t}{\lambda_k^2} , \frac{x_1}{\lambda_k} , x_2  \bigg)  \bigg) \\	
& + \lambda_k^{-\frac{5}{2}}  \sum\limits_{\substack{n_1,n_2,n_3,n_4,n_5,n \in \mathbb{N} \\ n_1-n_2+n_3-n_4+n_5\ne n}}  e^{it (\partial_{x_2}^2 -2n -1)}  \int_{0}^t    e^{-it \widetilde{\partial}_{x_1}^2}  (-i\widetilde{\partial}_{x_1}^2)^{-1}  \\
&\hspace{24ex}  \times  \partial_{\tau}  P^{x_1}_{\le 2^{-10}} \bigg(  \Pi_n (v_{n_1}  \overline{v}_{n_2}  v_{n_3}  \overline{v}_{n_4}  v_{n_5})  \bigg(  \frac{\tau}{\lambda_k^2} , \frac{x_1}{\lambda_k} , x_2  \bigg)  \bigg)  d\tau.
\end{align*}
	 
\begin{align*}
e_{\lambda_k}^{j_1} := & \lambda_k^{-\frac{5}{2}} \bigg\| \sum\limits_{\substack{n_1,n_2,n_3,n_4,n_5,n \in \mathbb{N} \\ n_1-n_2+n_3-n_4+n_5\ne n}}  e^{it (\partial_{x_1}^2 + \partial_{x_2}^2 - x_2^2)}  (-i\widetilde{\partial}_{x_1}^2)^{-1} \\
&\hspace{24ex}  \times  P^{x_1}_{\le 2^{-10}} \bigg(  \Pi_n (v_{n_1}  \overline{v}_{n_2}  v_{n_3}  \overline{v}_{n_4}  v_{n_5})  \bigg(  \frac{0}{\lambda_k^2} , \frac{x_1}{\lambda_k} , x_2  \bigg)  \bigg) \bigg\|_{L^6_{t,x_1}\mathcal{H}_{x_2}^1(\R \times \R \times \R)},\\
e_{\lambda_k}^{j_2} := & \lambda_k^{-\frac{5}{2}} \bigg\| \sum\limits_{\substack{n_1,n_2,n_3,n_4,n_5,n \in \mathbb{N} \\ n_1-n_2+n_3-n_4+n_5\ne n}}  e^{-2it (n_1-n_2+n_3-n_4+n_5) - it}  (-i\widetilde{\partial}_{x_1}^2)^{-1} \\ 
&\hspace{24ex}  \times  P^{x_1}_{\le 2^{-10}} \bigg(  \Pi_n (v_{n_1}  \overline{v}_{n_2}  v_{n_3}  \overline{v}_{n_4}  v_{n_5})  \bigg(  \frac{t}{\lambda_k^2} , \frac{x_1}{\lambda_k} , x_2  \bigg)  \bigg) \bigg\|_{L^6_{t,x_1}\mathcal{H}_{x_2}^1(\R \times \R \times \R)},\\
e_{\lambda_k}^{j_3} := & \lambda_k^{-\frac{5}{2}} \bigg\| \sum\limits_{\substack{n_1,n_2,n_3,n_4,n_5,n \in \mathbb{N} \\ n_1-n_2+n_3-n_4+n_5\ne n}}  e^{it (\partial_{x_2}^2 -2n -1)}  \int_{0}^t    e^{-it \widetilde{\partial}_{x_1}^2}  (-i\widetilde{\partial}_{x_1}^2)^{-1}  \\
&\hspace{24ex}  \times  \partial_{\tau}  P^{x_1}_{\le 2^{-10}} \bigg(  \Pi_n (v_{n_1}  \overline{v}_{n_2}  v_{n_3}  \overline{v}_{n_4}  v_{n_5})  \bigg(  \frac{\tau}{\lambda_k^2} , \frac{x_1}{\lambda_k} , x_2  \bigg)  \bigg)  d\tau \bigg\|_{L^6_{t,x_1}\mathcal{H}_{x_2}^1(\R \times \R \times \R)}.
\end{align*}
It is easy to verify that
\begin{align*}
\bigg\|  \int_{0}^{t} e^{i (t-\tau) (\partial_{x_1}^2 + \partial_{x_2}^2 - x_2^2)} e^3_{\lambda_k} (\tau)  d\tau  \bigg\|_{L^6_{t,x_1}  \mathcal{H}^1_{x_2} (\R \times \R \times \R) } \sim e^{j_1}_{\lambda_k} + e^{j_2}_{\lambda_k} + e^{j_3}_{\lambda_k}.
\end{align*}
We treat the term $e^{j_1}_{\lambda_k}$ first. Notice that operator $P^{x_1}_{\le 2^{-10}} ( -i \widetilde{\partial}_{x_1}^2 )^{-1}$ is bounded on $L^r(\R), 1<r<\infty$. Hence, using Minkowski's inequality and Strichartz estimates, we obtain
\begin{align}
e^{j_1}_{\lambda_k} &\lesssim \lambda^{-\frac{5}{2}}_k \bigg\|   \bigg\|  \langle n \rangle^{\frac{1}{2}}  \sum\limits_{\substack{n_1,n_2,n_3,n_4,n_5,n \in \mathbb{N} \\ n_1-n_2+n_3-n_4+n_5\ne n}}  P^{x_1}_{\le 2^{-10}} (-i \widetilde{\partial}_{x_1}^2 )^{-1} \Pi_n  \big( v_{n_1}  \overline{v}_{n_2}  v_{n_3}  \overline{v}_{n_4}  v_{n_5}    \big)  \bigg(  0, \frac{x_1}{\lambda_k} , x_2  \bigg)    \bigg\|_{l^2_n (\N) }   \bigg\|_{L^2_{x_1 , x_2} (\R \times \R) }\notag\\
&\lesssim \lambda^{-\frac{5}{2}}_k  \sum\limits_{\substack{n_1,n_2,n_3,n_4,n_5 \in \mathbb{N} } }  \Big\|  \langle n \rangle^{\frac{1}{2}}  \Pi_{n}   ( v_{n_1}  \overline{v}_{n_2}  v_{n_3}  \overline{v}_{n_4}  v_{n_5}    )  (  0, \frac{x_1}{\lambda_k} , x_2  )   \Big\|_{L^2_{x_1,x_2} l^2_n (\R \times \R \times \N) } \nonumber \\
&\lesssim  \lambda^{-2}_k  C\bigg(  \|v(0 , x_1 , x_2)\|_{H^1_{x_1}  \mathcal{H}^3_{x_2} (\R \times \R) }  \bigg)^5\longrightarrow  0  \mbox{ as } \hspace{2ex} k\rightarrow \infty. \label{fml-error-term-3-1}
\end{align}
By the boundedness of $(-i \widetilde{\partial}_{x_1}^2)^{-1}P^{x_1}_{\le 2^{-10}} $, Minkowski's inequality, the fractional Leibniz rule and Sobolev embedding, the term $e^{j_2}_{\lambda_k}$ can be bounded as follows 
\begin{align}
e^{j_2}_{\lambda_k} \lesssim & \lambda^{-\frac{5}{2}}_k  \bigg\|   \langle n \rangle^{\frac{1}{2}}  \sum\limits_{\substack{n_1,n_2,n_3,n_4,n_5 \in \mathbb{N} \\ n_1-n_2+n_3-n_4+n_5\ne n}} e^{-2it  (n_1-n_2+n_3-n_4+n_5) - it}\notag  \\   
& \hspace{16ex} \times (-i \widetilde{\partial}_{x_1}^2)^{-1}P^{x_1}_{\le 2^{-10}}   \Pi_{n}   ( v_{n_1}  \overline{v}_{n_2}  v_{n_3}  \overline{v}_{n_4}  v_{n_5}    )  (  \frac{t}{\lambda_k^2} , \frac{x_1}{\lambda_k} , x_2  )      \bigg\|_{L^6_{t,x_1}  L^2_{x_2}  l^2_n (\R \times \R \times \R \times \N) }  \nonumber 
\\\lesssim & \lambda_k^{-2}  \bigg\|   \sum_{n_1,n_2,n_3,n_4,n_5\in\mathbb{N}}   \|  v_{n_1} \overline{v}_{n_2} v_{n_3}  \overline{v}_{n_4}  v_{n_5}  \|_{ H^{\frac{1}{4}}_{x_1} \mathcal{H}^1_{x_2} (\R \times \R) }   \bigg\|_{L^6_t (\R) } \nonumber \\
\lesssim &  \lambda_k^{-2}  \bigg\|    \|  v  (t , x_1 , x_2)\|^5_{ W^{\frac{5}{6} , 4}_{x_1} \mathcal{H}^5_{x_2} (\R \times \R) }    \bigg\|_{L^6_t (\R) } \nonumber \\\lesssim &  \lambda_k^{-2}  C \bigg( \|  v  (0 , x_1 , x_2)\|_{ H_{x_1}^{\frac{61}{60}} \mathcal{H}^{5}_{x_2} (\R \times \R) } \bigg) \longrightarrow 0  \mbox{ as } \hspace{2ex} k\rightarrow \infty.\label{fml-term-3-3}
\end{align}	
We now treat the most hard term $e^{j_3}_{\lambda_k}$. Using Strichartz estimate we have
\begin{align}\label{4.27}
&\lambda_k^{-\frac{5}{2}} \bigg\| \sum\limits_{\substack{n_1,n_2,n_3,n_4,n_5,n \in \mathbb{N} \\ n_1-n_2+n_3-n_4+n_5\ne n}}  e^{it (\partial_{x_2}^2 -2n -1)}  \int_{0}^t    e^{-it \widetilde{\partial}_{x_1}^2}  (-i\widetilde{\partial}_{x_1}^2)^{-1} \notag \\
&\hspace{24ex}  \times  \partial_{\tau}  P^{x_1}_{\le 2^{-10}} \bigg(  \Pi_n (v_{n_1}  \overline{v}_{n_2}  v_{n_3}  \overline{v}_{n_4}  v_{n_5})  \bigg(  \frac{\tau}{\lambda_k^2} , \frac{x_1}{\lambda_k} , x_2  \bigg)  \bigg)  d\tau \bigg\|_{L^6_{t,x_1}\mathcal{H}^1_{x_2}(\R\times\R\times\R)} \notag\\
&\lesssim \lambda_k^{-\frac{5}{2}} \bigg\| \big( i\partial_t + \partial_{x_1}^2 + \partial_{x_2}^2 - x_2^2 \big) \sum\limits_{\substack{n_1,n_2,n_3,n_4,n_5,n \in \mathbb{N} \\ n_1-n_2+n_3-n_4+n_5\ne n}}  e^{it (\partial_{x_2}^2 -2n -1)}  \int_{0}^t    e^{-it \widetilde{\partial}_{x_1}^2}  (-i\widetilde{\partial}_{x_1}^2)^{-1} \notag \\
&\hspace{24ex}  \times  \partial_{\tau}  P^{x_1}_{\le 2^{-10}} \bigg(  \Pi_n (v_{n_1}  \overline{v}_{n_2}  v_{n_3}  \overline{v}_{n_4}  v_{n_5})  \bigg(  \frac{\tau}{\lambda_k^2} , \frac{x_1}{\lambda_k} , x_2  \bigg)  \bigg)  d\tau \bigg\| _{L^1_t  L^2_{x_1}  \mathcal{H}^1_{x_2} (\R \times \R \times \R) }.
\end{align}
Hence, one can find
\begin{align*}
&\mbox{(RHS) of }\eqref{4.27}\\
 \lesssim&  \lambda_k^{-\frac{5}{2}}  \bigg\|   \bigg(   \sum_{n\in\mathbb{N}}  \langle n \rangle   \bigg(    \sum\limits_{\substack{n_1,n_2,n_3,n_4,n_5 \in \mathbb{N} \\ n_1-n_2+n_3-n_4+n_5\ne n}}  \bigg|    e^{-2it  (n_1-n_2+n_3-n_4+n_5-n)  -  it} \\   
&\hspace{3ex}\times  (-i \widetilde{\partial}_{x_1}^2 )^{-1}    P^{x_1}_{\le 2^{-10}}  \partial_t  \Pi_n  \big(  v_{n_1}    \overline{v}_{n_2}    v_{n_3}  \overline{v}_{n_4}  v_{n_5}   \big) \big(  \frac{t}{\lambda_k^{2}} , \frac{x_1}{\lambda_k} , x_2   \big)    \bigg|   \bigg)^2    \bigg)^{\frac{1}{2}}   \bigg\|_{L^1_t  L^2_{x_1}  L^2_{x_2} (\R \times \R \times \R) }\\
  \lesssim&  \lambda_k^{-\frac{5}{2}}  \bigg\|   \bigg(   \sum_{n}  \langle n \rangle   \bigg(    \sum\limits_{\substack{n_1,n_2,n_3,n_4,n_5\in \mathbb{N} \\ n_1-n_2+n_3-n_4+n_5\ne n}}  \bigg\|  \partial_t   \Pi_n  \big(  v_{n_1}    \overline{v}_{n_2}   v_{n_3}  \overline{v}_{n_4}  v_{n_5}     \big)  \big(  \frac{t}{\lambda_k^{2}} , \frac{x_1}{\lambda_k} , x_2   \big)    \bigg\|_{L^2_{x_1,x_2} (\R \times \R) }   \bigg)^2    \bigg)^{\frac{1}{2}}   \bigg\|_{L^1_t (\R) }\\
 \lesssim&  \lambda_k^{-2}  \sum_{n_1,n_2,n_3,n_4,n_5}  \|   \partial_t  v_{n_1} \cdot \overline{v}_{n_2}  v_{n_3}  \overline{v}_{n_4}  v_{n_5}   \|_{L^1_t L^2_{x_1}  \mathcal{H}^1_{x_2} (\R \times \R \times \R) } +  \lambda_k^{-2}  \sum_{n_1,n_2,n_3,n_4,n_5}  \|     v_{n_1}  \partial_t \overline{v}_{n_2} \cdot  v_{n_3}  \overline{v}_{n_4}  v_{n_5}   \|_{L^1_t L^2_{x_1}  \mathcal{H}^1_{x_2} (\R \times \R \times \R) } \\
&  + \lambda_k^{-2}  \sum_{n_1,n_2,n_3,n_4,n_5}  \|    v_{n_1} \overline{v}_{n_2}   \partial_t  v_{n_3}  \cdot  \overline{v}_{n_4}  v_{n_5}   \|_{L^1_t L^2_{x_1}  \mathcal{H}^1_{x_2} (\R \times \R \times \R) } +  \lambda_k^{-2}  \sum_{n_1,n_2,n_3,n_4,n_5}  \|     v_{n_1} \overline{v}_{n_2}  v_{n_3} \partial_t \overline{v}_{n_4}  \cdot  v_{n_5}   \|_{L^1_t L^2_{x_1}  \mathcal{H}^1_{x_2} (\R \times \R \times \R) } \\
& + \lambda_k^{-2}  \sum_{n_1,n_2,n_3,n_4,n_5}  \|    v_{n_1} \overline{v}_{n_2}  v_{n_3}  \overline{v}_{n_4}   \partial_t  v_{n_5}   \|_{L^1_t L^2_{x_1}  \mathcal{H}^1_{x_2} (\R \times \R \times \R) }.
\end{align*} 
We only give detailed proof of $\sum\limits_{n_1,n_2,n_3,n_4,n_5}  \|   \partial_t  v_{n_1} \cdot \overline{v}_{n_2} v_{n_3}  \overline{v}_{n_4}  v_{n_5}   \|_{L^1_t L^2_{x_1} (\R \times \R) }$, since the other four terms can be handled similarly. By H\"older's inequality and the fact that $v$ solves $\eqref{fml-DCR-2}$, we obtain
\begin{align*}
&\sum_{n_1,n_2,n_3,n_4,n_5}  \|   \partial_t  v_{n_1} \cdot \overline{v}_{n_2}  v_{n_3}  \overline{v}_{n_4}  v_{n_5}   \|_{L^1_t L^2_{x_1}  \mathcal{H}^1_{x_2} (\R \times \R \times \R) }  \lesssim    \|v\|^4_{L^5_t  L^{10}_{x_1}  H^5_{x_2} (\R \times \R \times \R) }  \|\partial_t  v\|_{L^5_t  L^{10}_{x_1}  \mathcal{H}^5_{x_2} (\R \times \R \times \R) } \\
\lesssim  &  \|v\|^4_{L^5_t  L^{10}_{x_1}  \mathcal{H}^5_{x_2}}  \bigg(   \|  \partial_{x_1}^2  v\|_{L^5_t  L^{10}_{x_1}  \mathcal{H}^5_{x_2}}   +    \bigg\|    \sum\limits_{\substack{\tilde{n}_1,  \tilde{n}_1,  \tilde{n}_3,  \tilde{n}_4,  \tilde{n}_5 , n_1 \in \mathbb{N} \\ \tilde{n}_1-\tilde{n}_2+\tilde{n}_3-\tilde{n}_4+\tilde{n}_5\ne n_1 }}  \Pi_{n_1}  (   v_{\tilde{n}_1}  \overline{v}_{n_2}  v_{\tilde{n}_3}  \overline{v}_{\tilde{n}_4}  v_{\tilde{n}_5}   )    \bigg\|_{L^5_t  L^{10}_{x_1}  \mathcal{H}^5_{x_2}}   \bigg).
\end{align*}
We then apply H\"older's inequality and the Sobolev embedding to the resonant part, then we get
\begin{align*}
&\bigg\|    \sum\limits_{\substack{\tilde{n}_1,  \tilde{n}_1,  \tilde{n}_3,  \tilde{n}_4,  \tilde{n}_5,n_{1}\in \mathbb{N} \\ \tilde{n}_1-\tilde{n}_2+\tilde{n}_3-\tilde{n}_4+\tilde{n}_5\ne \tilde{n}}}  \Pi_{\tilde{n}}  (   v_{\tilde{n}_1}  \overline{v}_{\tilde{n}_2}  v_{\tilde{n}_3}  \overline{v}_{\tilde{n}_4}  v_{\tilde{n}_5}   )    \bigg\|_{L^5_t  L^{10}_{x_1}  \mathcal{H}^5_{x_2} (\R \times \R \times \R) }\\  
\lesssim & \bigg\|    \sum\limits_{\substack{\tilde{n}_1,  \tilde{n}_1,  \tilde{n}_3,  \tilde{n}_4,  \tilde{n}_5 \in \mathbb{N} \\ \tilde{n}_1-\tilde{n}_2+\tilde{n}_3-\tilde{n}_4+\tilde{n}_5\ne \tilde{n}}}  \langle n_{1} \rangle^{\frac{5}{2}}   \Pi_{\tilde{n}}   (   v_{\tilde{n}_1}  \overline{v}_{\tilde{n}_2}  v_{\tilde{n}_3}  \overline{v}_{\tilde{n}_4}  v_{\tilde{n}_5}   )    \bigg\|_{L^5_t  L^{10}_{x_1}  L^2_{x_2}  l^2_{\tilde{n}} (\R \times \R \times \R \times \N) } \\
\lesssim  &  \bigg\|    \sum\limits_{\substack{\tilde{n}_1,  \tilde{n}_1,  \tilde{n}_3,  \tilde{n}_4,  \tilde{n}_5 \in \mathbb{N} \\ \tilde{n}_1-\tilde{n}_2+\tilde{n}_3-\tilde{n}_4+\tilde{n}_5\ne \tilde{n}}}     \langle n_{1} \rangle^{-1}  \langle \tilde{n}_1  \rangle^{4}    \langle \tilde{n}_2  \rangle^{4}  \langle \tilde{n}_3  \rangle^{4}  \langle \tilde{n}_4  \rangle^{4}  \langle \tilde{n}_5  \rangle^{4}     \Pi_{\tilde{n}}  (   v_{\tilde{n}_1}  \overline{v}_{\tilde{n}_2}  v_{\tilde{n}_3}  \overline{v}_{\tilde{n}_4}  v_{\tilde{n}_5}   )    \bigg\|_{L^5_t  L^{10}_{x_1}  L^2_{x_2}  l^2_{\tilde{n}} (\R \times \R \times \R \times \N) } \\ 
\lesssim  &  \|v\|^5_{L^{25}_t  L^{50}_{x_1}  \mathcal{H}^{5}_{x_2} (\R \times \R \times \R) } \lesssim\|v\|^5_{L^{25}_t  W^{\frac{2}{5},\frac{50}{21}}_{x_1}  \mathcal{H}^{5}_{x_2} (\R \times \R \times \R) } \lesssim  C^5 \big(  \|\psi\|_{H_{x_1}^{1}  \mathcal{H}^{5}_{x_2} (\R \times \R) }  \big).
\end{align*}
The persistence of regularity implies that \begin{equation*}
\|  \partial_{x_1}^2 v  \|_{L^5_t  L^{10}_{x_1}  \mathcal{H}^5_{x_2} (\R \times \R \times \R) }  \lesssim  C  \big(   \|  \partial_{x_1}^2 \psi  \|_{  H^{\frac{1}{5}}_{x_1}  \mathcal{H}^5_{x_2} (\R \times \R) }   \big).
\end{equation*}
As a conclusion, we obtain
\begin{gather}\label{lambda-3}
e^{j_3}_{\lambda_k}  \lesssim  \lambda_k^{-2}  C^4 \Big(  \|\psi\|_{  H^{\frac{1}{5}}_{x_1}  \mathcal{H}^{5}_{x_2} }  \Big)  \cdot  \bigg(   C  \big(   \|  \partial_{x_1}^2  \psi  \|_{ H^{\frac{1}{5}}_{x_1}  \mathcal{H}^5_{x_2} }   \big)   +   C  \big(   \|\psi\|_{H_{x_1}^{\frac{16}{25}}  \mathcal{H}^{25}_{x_2}}    \big)     \bigg)\longrightarrow  0  \mbox{ as}\hspace{1ex}  k\rightarrow  \infty.
\end{gather}
Thus by \eqref{fml-error-term-1}, \eqref{fml-error-term-3} and \eqref{lambda-3}, we have that $e_{\lambda_k}^3\to0$ which implies that the error term $e_{\lambda_{k}}\to\infty$.
	
On the other hand, one can verify that
\begin{gather*}
\lim\limits_{k \rightarrow 0}  \|  w_{\lambda_k} (0, x_1 , x_2)  -  u_{\lambda_k}  (0, x_1 , x_2)   \|_{L^2_{x_1}  \mathcal{H}_{x_2}^1 (\R \times \R) }  =  \|  \psi - \phi  \|_{L^2_{x_1}  \mathcal{H}_{x_2} (\R \times \R) }  \le  \varepsilon_4 ,\\
\|  w_{\lambda_k}  \|_{L^\infty_t  L^2_{x_1}  \mathcal{H}^1_{x_2} (\R \times \R \times \R) }   =   \bigg\|   \lambda_k^{-\frac{1}{2}}  v( \frac{t}{\lambda_k} , \frac{x_1}{\lambda_k}) , x_2   \bigg\|_{  L^\infty_t  L^2_{x_1}  \mathcal{H}^1_{x_2} (\R \times \R \times \R) }  =  \| v \|_{  L^\infty_t  L^2_{x_1}  \mathcal{H}^1_{x_2} (\R \times \R \times \R) },
\end{gather*} 
and
\begin{equation*}
\|  w_{\lambda_k}  \|_{L^6_t  L^6_{x_1}  \mathcal{H}^1_{x_2} (\R \times \R \times \R) }   =   \bigg\|   \lambda_k^{-\frac{1}{2}}  v( \frac{t}{\lambda_k} , \frac{x_1}{\lambda_k}) , x_2   \bigg\|_{  L^6_t  L^6_{x_1}  \mathcal{H}^1_{x_2} (\R \times \R \times \R) }  =  \| v \|_{  L^6_t  L^6_{x_1}  \mathcal{H}^1_{x_2} (\R \times \R \times \R) },
\end{equation*}
hence we can use the stability theorem(cf. Theorem \ref{Stablity}) to finish the proof when $t_k=0$.
	
When $t_k \rightarrow \pm \infty$ while $k \rightarrow \infty$. Let $v$ solve $\eqref{fml-DCR-2}$, we have
\begin{equation*}
\lim\limits_{k \rightarrow 0} \| v(t , x_1 , x_2)  -  e^{it \partial_{x_1}^2 }  \psi \|_{L^2_{x_1}  \mathcal{H}^1_{x_2} (\R \times \R) }  \longrightarrow  0 .
\end{equation*}
In the linear profile decomposition, we take that $T=\infty$ such that $t_k\to T$ as $k\to\infty$, then we add the error term that $g_n^j(e^{it_k\partial_{x_1}^2}-e^{iT\partial_{x_1}^2}\phi^j)$ then we can reduce the case to $t_k=0$.
\end{proof}

\subsection{Existence of the almost-periodic solution}In this subsection, we prove that if Theorem \ref{Thm1} fails and Theorem \ref{Thm2} holds, then there exists a special solution which is pre-compact in $\Sigma_{x_1,x_2}$ modulo the symmetry.

By Remark \ref{rem-global}, the solution $u$ is global. Therefore, we can define   
\begin{align*}
\Lambda(L)=\sup\|u\|_{L_{t,x_1}^6\mathcal{H}_{x_2}^{1-\varepsilon_0}(\R\times\R\times\R)},
\end{align*}
where the $\sup$ is taken over all  solution $u\in C(\R,\Sigma_{x_1 , x_2}(\R\times\R))$ to the Cauchy problem \eqref{NLS-final} obeying
\begin{align*}
E(u(t))+M(u(t))\leq L.
\end{align*}
By the local well-posedness and small-data scattering, we know that $\Lambda(L)<\infty$ for some sufficient small $L\ll1$. We denote $L_{max}$ by
\begin{align*}
L_{max}=\sup\{L\geq0:\Lambda(L)<\infty\}.
\end{align*}

Suppose that $L_{max}<\infty$, we can show the existence of minimal blow-up solution which is almost-periodic under the assumption of Theorem \ref{Thm2}. This is achieved by showing the Palais-Smale type condition.

\begin{theorem}[Existence of minimal blow-up solution]\label{reduction}Assume that $L_{max}<\infty$, then there exists a solution $u_c\in C(\R,\Sigma_{x_1,x_2}(\R\times\R))$ to the equation $\eqref{NLS-final}$ satisfying
\begin{align}\label{con-minimal}
E(u_c)+M(u_c)=L_{max},\quad\|u_c\|_{L_{t,x_1}^6\mathcal{H}_{x_2}^{1-\varepsilon_0}((0,\infty)\times\R\times\R)}=\|u_c\|_{L_{t,x_1}^6\mathcal{H}_{x_2}^{1-\varepsilon_0}((-\infty,0)\times\R\times\R)}=\infty.
\end{align}
Moreover, there exists a function $\tilde{x}_1(t): \R \to \R$ and a large number $C(\eta)>0$ such that
\begin{align}\label{compact-minimal}
\int_{|x_1-\tilde{x}_1(t)|>C(\eta)}\|u_c(t,x_1)\|_{\mathcal{H}_{x_2}^{1}(\R)}^2dx_1<\eta,\quad\forall t\in\R.
\end{align}
	
\end{theorem}  
As a standard argument, we can reduce the proof of Theorem \ref{reduction} to prove the following Palais-Smale condition.
\begin{proposition}[Palais-Smale condition]\label{PS-condition}
Suppose that $L_{max}<\infty$ and let $u_n\in C_t^0\Sigma_{x_1,x_2}(\R\times\R\times\R)$ be  sequences of solutions to \eqref{NLS-final} satisfying
\begin{align*}
&E(u_n)+M(u_n)\to L_{max}\\
&\|u_n\|_{L_{t,x_1}^6\mathcal{H}_{x_2}^{1-\varepsilon_0}((-\infty,t_n)\times\R\times\R)}=\|u_n\|_{L_{t,x_1}^6\mathcal{H}_{x_2}^{1-\varepsilon_0}((-\infty,t_n)\times\R\times\R)}\to\infty \mbox{ as }n\to\infty.
\end{align*} 
After passing the subsequence, there exists a spatial sequence $x_{1,n}\in\R$ and $w\in\Sigma_{x_1,x_2}(\R\times\R)$ such that 
\begin{align*}
u_n(t_n,x_1+x_{1,n},x_2)\to w(x_1,x_2)\quad\mbox{in }\Sigma_{x_1,x_2}(\R\times\R)
\end{align*} 
as $n\to\infty$.
\end{proposition}
\begin{proof}
It is sufficient to prove that there exists $w(x_1 , x_2) \in L^2_{x_1}\mathcal{H}^1_{x_2}$ such that
\begin{equation*}
u_n(x_1 + x_{1,n} , x_2 , t + t_n) \rightarrow w(x_1 , x_2),
\end{equation*}
in $L^2_{x_1}\mathcal{H}^1_{x_2}(\R\times\R)$ up to a subsequence.
	
We consider the case when $t_n = 0$, otherwise replace $u_n(t)$ with $u_n(t+t_n)$ the left cases can be handled similarly. Moreover, we suppose $\theta = 0$ since the Galilean transform is harmless in the following proof. 
	
Applying linear profile decomposition (cf. Theorem \ref{LinearProfile}) to $\{u_n(0)\}_{n \ge 1}$, we have
\begin{align*}
u_n(0,x_1,x_2) &  = \sum\limits_{k=1}^K  \frac1{ (\lambda_n^k)^{\frac{1}{2}} }   e^{ix_1 \cdot \xi_n^k}   \left(e^{it_n^k \partial_{x_1}^2}  P_n^k  \phi^k\right)
\left(\frac{x_1-x_{1,n}^k}{\lambda_n^k}, x_2 \right)   + w_n^K(x_1,x_2),
\end{align*}
up to a subsequence. The remainder satisfies
\begin{equation}\label{eq6.1}
\limsup\limits_{n\to \infty} \left\|e^{it\partial_{x_1}^2} w_n^K \right\|_{ L_t^6 L_{x_1}^6 \mathcal{H}_{x_2}^{1-\epsilon_0}(\mathbb{R} \times \mathbb{R}\times\R )} \to 0, \text{ as } K \to \infty,
\end{equation}
and we obtain asymptotic decoupling of the mass and energy:
\begin{align}
& \lim\limits_{n\to \infty} \left(  M(u_n(0)) - \sum\limits_{k=1}^K M \bigg( \frac{1}{ (\lambda_n^k)^{\frac{1}{2}} }   e^{ix_1 \cdot \xi_n^k}   \left(e^{it_n^k \partial_{x_1}^2}  P_n^k  \phi^k\right)\left(\frac{x_1-x_{1,n}^k}{\lambda_n^k}, x_2 \right)  \bigg) - M(w_n^K)  \right) = 0, \label{eq6.2} \\
& \lim\limits_{n\to \infty} \left(  E(u_n(0)) - \sum\limits_{k=1}^K E  \bigg(  \frac{1}{ (\lambda_n^k)^{\frac{1}{2}} }   e^{ix_1 \cdot \xi_n^k}   \left(e^{it_n^k\partial_{x_1}^2}   P_n^k  \phi^k\right)\left(\frac{x_1-x_{1,n}^k}{\lambda_n^k}, x_2 \right)  \bigg) - E(w_n^K)  \right) = 0,\  \forall\, K, \label{eq5.5v30}
\end{align}
For convenience, we denote that $S(u)=M(u)+E(u)$. Therefore, there are two possibilities:

{\bf Case 1.}  $\sup\limits_{k} \limsup\limits_{n\to \infty}S\bigg(  \frac{1}{ (\lambda_n^k)^{\frac{1}{2}} }   e^{ix_1 \cdot \xi_n^k}   \left(e^{it_n^k \partial_{x_1}^2}   P_n^k  \phi^k\right)\left(\frac{x_1-x_{1,n}^k}{\lambda_n^k}, x_2 \right)  \bigg) = L_{max}$.
Combining \eqref{eq6.2}, \eqref{eq5.5v30} with the fact that $\phi^k$ are nontrivial in $L^2_{x_1}\mathcal{H}^1_{x_2}$, we obtain
\begin{align*}
u_n(0,x_1,x_2) =  \frac1{ (\lambda_n)^{\frac{1}{2}} }   e^{ix_1 \cdot \xi_n}   \left(e^{it_n \partial_{x_1}^2}   P_n  \phi \right)
\left(\frac{x_1-x_{1,n}}{\lambda_n}, x_2 \right) + w_n(x_1,x_2), 
\end{align*}
with $\lim\limits_{n\to \infty} \|w_n\|_{\Sigma_{x_1 , x_2}(\R\times\R)} = 0$. We assert that $\lambda_n \equiv 1$. Otherwise if $\lambda_n \to  \infty$, by Theorem \ref{large-scale}, for sufficiently large $n$ there exists a unique global solution $u_n$ such that
\begin{align*}
u_n(0,x_1,x_2) = \frac1{ (\lambda_n)^{\frac{1}{2}} }   e^{ix_1 \cdot \xi_n}   \left(e^{it_n \partial_{x_1}^2}   P_n  \phi \right)
\left(\frac{x_1-x_{1,n}}{\lambda_n}, x_2 \right)
\end{align*}
and
\begin{align*}
\limsup\limits_{n\to \infty} \|u_n\|_{ L_t^6 L_{x_1}^6 \mathcal{H}_{x_2}^{1-\epsilon_0} (\mathbb{R} \times \mathbb{R}\times\R )} \le C(L_{max}).
\end{align*}
Let $n\rightarrow \infty$, which is a contradiction with \eqref{Palais-Smale}.
	
Therefore $u_n(0,x_1,x_2) =   e^{ix_1 \cdot \xi_n}   \left(e^{it_n \partial_{x_1}^2}   P_n  \phi \right)
\left( x_1-x_{1,n}, x_2 \right)  + w_n(x_1,x_2)$.
If $t_n \equiv 0$, boundedness of $\xi_n $ implies convergence immediately. For the rest case $t_n \to -\infty$, we use the Galilean transform 
\begin{align*}
e^{it \partial_{x_1}^2} e^{ix_1 \xi_n} \tilde{\phi} (x) = e^{-it |\xi_n|^2} e^{ix_1 \xi_n} (e^{it \partial_{x_1}^2} \tilde{ \phi} )(x-2t \xi_n), \quad t\in \mathbb{R},
\end{align*}
and obtain that
\begin{align*}
&\left\|e^{it(\partial_{x_1}^2+\partial_{x_2}^2-x_2^2) }   \left( e^{ix_1 \xi_n}   ( e^{it_n\partial_{x_1}^2} P_n \phi)(x_1-x_{1,n} , x_2) \right)  \right\|_{ L_{t,x_1}^6 \mathcal{H}_{x_2}^{1-\epsilon_0} ((-\infty,0) \times \mathbb{R}\times\R )}\\
=  & \ \|e^{it\partial_{x_1   }^2} P_n \phi\|_{ L_{t,x_1}^6 \mathcal{H}_{x_2}^{1-\epsilon_0}((-\infty,t_n)\times \mathbb{R}\times\R  )}\to 0, \text{ as } n\to \infty.
\end{align*}
From Lemma \ref{lwp}, we have
\begin{align*}
\|u_n\|_{ L_{t,x_1}^6 H_{x_2}^{1-\epsilon_0}((-\infty, 0) \times \mathbb{R}\times\R  )} \le 2 \delta_0 < \infty,
\end{align*}
for $n$ sufficiently large, which contradicts to \eqref{Palais-Smale}. The case $t_n \to \infty$ is similar.
	
{ \bf Case 2.}  $\sup\limits_{k}  \limsup\limits_{n\to \infty}  S\left( \frac1{ (\lambda_n^k)^{\frac{1}{2}} }   e^{ix_1 \cdot \xi_n^k}   \left(e^{it_n^k \partial_{x_1}^2}   P_n^k  \phi \right)\left(\frac{x_1-x_{1,n}^k}{\lambda_n^k}, x_2 \right) \right) \le L_{max} - 2\delta$ for some $\delta > 0$.
	
In this case, for each fixed finite $K\le K^*$, then for all $1 \le k \le K$ and $n$ sufficiently large, we have
\begin{align*}
S\left( \frac1{ (\lambda_n^k)^{\frac{1}{2}} }   e^{ix_1 \cdot \xi_n^k}   \left(e^{it_n^k \partial_{x_1}^2}   P_n^k  \phi \right)\left(\frac{x_1-x_{1,n}^k}{\lambda_n^k}, x_2 \right) \right)  \le L_{max}- \delta.
\end{align*}
By the definition of $L_{max}$, there exists global solution $v_n^k$ to
\begin{equation*}
\begin{cases}
i\partial_t v_n^k +( \partial_{x_1}^2 + \partial_{x_2}^2 - x^2 _ 2)  v_n^k = |v_n^k|^4 v_n^k,\\
v_n^k(0,x_1,x_2) =  \frac1{ (\lambda_n^k)^{\frac{1}{2}} }   e^{ix_1 \cdot \xi_n^k}   \left(e^{it_n^k \partial_{x_1}^2}   P_n^k  \phi \right)
\left(\frac{x_1-x_{1,n}^k}{\lambda_n^k}, x_2 \right),
\end{cases}
\end{equation*}
such that
$ \|v_n^k \|_{L_{t,x_1}^6 \mathcal{H}_{x_2}^{1-\epsilon_0}} \lesssim \Lambda(L_{max} - \delta) < \infty$.
We can use
\begin{align*}
\| v_n^k\|_{L_{t,x_1}^6 \mathcal{H}_{x_2}^{1-\epsilon_0}(\R\times\R\times\R) }^2 &  \lesssim_{L_{max},\delta} M (v_n^k(0)), \text{ for $S(v_n^k(0)) \le \eta_0$, }
\end{align*}
where $\eta_0$ denotes the small data threshold in the small data scattering theorem, together with our bounds on the space-time norms of $v_n^k$ and the finity of $L_{max}$ to deduce
\begin{align}\label{eq6.3}
\| v_n^k\|_{L_{t,x_1}^6 \mathcal{H}_{x_2}^{1-\epsilon_0} }^2 &  \lesssim_{L_{max},\delta}S\left(  \frac1{ \lambda_n^k }   e^{ix_1 \cdot \xi_n^k}   \left(e^{it_n^k \Delta_{x_1  }}   P_n^k  \phi \right)\left(\frac{x_1-x_{1,n}^k}{\lambda_n^k}, x_2 \right)  \right) \lesssim_{L_{max}, \delta} 1.
\end{align}
Let
\begin{align*}
u_n^K = \sum\limits_{j=1}^K v_n^k + e^{it (\partial_{x_1}^2 + \partial_{x_2}^2 - x ^ 2 _ 2)} w_n^K.
\end{align*}
we have $u_n^K(0) = u_n(0)$. We will prove for sufficiently large $K$ and $n$, $u_n^K$ is an approximate solution to $u_n$ in the sense of the Theorem \ref{Stability-DCR}. Consequently, we have $L_{t,x_1}^6 \mathcal{H}_{x_2}^{1-\epsilon_0} $ norm of $u_n$ is finite, which is a contradiction to \eqref{Palais-Smale}.
	
By the stability theorem, we need to verify the following properties:
	
$(i)$ $\limsup\limits_{n\to \infty} \left\|u_n^K \right\|_{ L_t^6 L_{x_1}^6 \mathcal{H}_{x_2}^{1-\epsilon_0} } \lesssim_{L_{max},\delta} 1$, uniformly in $K$;
	
$(ii)$ $  \limsup\limits_{n\to \infty} \|e_n^K \|_{ L_{t,x_1}^\frac65 \mathcal{H}_{x_2}^{1-\epsilon_0} } \to 0$, as $ K \to K^* $, where $e_n^K = (i\partial_t + \partial_{x_1}^2 + \partial_{x_2}^2 - x_2 ^ 2)u_n^K - |u_n^K|^4  u_n^K$.
	
The verification of $(i)$ relies on the asymptotic decoupling of the nonlinear profiles $v_n^j$, which we record in the following lemma.
	
\begin{lemma}[Decoupling of nonlinear profiles]\label{le6.3}
Let $v_n^j$ be the nonlinear solutions defined above, then for $j\ne k$,
\begin{align}
&  \left\|\langle \mathcal{H}_{x_2}  \rangle^{1-\epsilon_0}  v_n^j \cdot  \langle \mathcal{H}_{x_2}
\rangle^{1-\epsilon_0} v_n^k   \right\|_{L_{t,x_1}^3 L_{x_2}^1(\mathbb{R} \times \mathbb{R}\times\R )} \to 0, \label{eq7.29} \\
& \left \|v_n^k \cdot \langle \mathcal{H}_{x_2}  \rangle^{1-\epsilon_0} v_n^j\right\|_{L_{t,x_1}^3 L_{x_2}^2(\mathbb{R} \times \mathbb{R}\times\R )} \to 0 \label{eq5.8v49}\mbox{ as } n\to \infty.
\end{align}
\end{lemma}
	
\begin{proof}
We only give detailed proof of \eqref{eq7.29}, since \eqref{eq5.8v49} can be proved similarly. By Theorem \ref{large-scale}, we only need to show
\begin{align*}
& \left\|\langle \mathcal{H}_{x_2} \rangle^{1- \epsilon_0} w_n^k \cdot \langle \mathcal{H}_{x_2} \rangle^{1- \epsilon_0} w_n^{k'} \right\|_{L_{t,x_1}^3 L_{x_2}^1(\mathbb{R} \times \mathbb{R}\times\R )} \to 0 \mbox{ as } n\to \infty,
\end{align*}
where $\mathcal{H}_{x_2}$ denotes the differential operator $(-\partial_{x_2}^2+x_2^2)^\frac12$ and $w_n^k$, $w_n^{k'}$ are the approximate solution of $v_n^k$ and $v_n^{k'}$ in Theorem \ref{large-scale}, respectively. 
We denote $e_j(x_2)$ the $j-th$ eigenfunction associated with $-\partial_{x_2}^2 + x_2^2$ . In the rest part of our proof, we use notation $\Pi_j$ present the following
\begin{equation*}
\Pi_j v := \langle e_j(x_2) , v \rangle _ {L^2_{x_2}(\R)},
\end{equation*}  
then we rewrite $w_n^k$ as 
\begin{align*}
&w_n^k (t , x_1 , x_2) \\
= & e^{-i(t-t_n^k) |\xi_n^k|^2} e^{ix \xi_n^k} \frac{1}{ (\lambda_n^k)^{\frac{1}{2}} } \sum_{j \in \mathbb{Z}} e^{ i t (2j+1) } e_j(x_2) \Pi_{j} v  \left( \frac{t}{ (\lambda_n^k)^2 } + t_n^k, \frac{x_1- x_{1,n}^k - 2 \xi_n^k(t-t_n^k)}{\lambda^k_n}\right).
\end{align*}
With this in hand, we obtain
\begin{align*}
& \left\|\langle \mathcal{H}_{x_2} \rangle^{1- \epsilon_0} w_n^k \cdot \langle \mathcal{H}_{x_2} \rangle^{1- \epsilon_0} w_n^{k'} \right\|_{L_{t,x_1}^3 L_{x_2}^1(\R\times\R\times\R)} \\
= & \bigg\| \langle \mathcal{H}_{x_2} \rangle^{1- \epsilon_0} \frac{1}{ (\lambda_n^k)^{\frac{1}{2}} } \sum_{j \in \mathbb{N}} e^{ i t (2j+1) } e_j(x_2) \Pi_{j} v  \left( \frac{t}{ (\lambda_n^k)^2 } + t_n^k, \frac{x_1- x_{1,n}^k - 2 \xi_n^k(t-t_n^k)}{\lambda^k_n}\right) \\
& \quad \times \langle \mathcal{H}_{x_2} \rangle^{1- \epsilon_0} \frac{1}{ (\lambda_n^{k'})^{\frac{1}{2}} } \sum_{j \in \mathbb{N}} e^{ i t (2j+1) } e_j(x_2) \Pi_{j} v  \left( \frac{t}{ (\lambda_n^{k'})^2 } + t_n^{k'}, \frac{x_1- x_{1,n}^{k'} - 2 \xi_n^{k'}(t-t_n^{k'})}{\lambda^{k'}_n}\right) \bigg\|_{L_{t,x_1}^3 L_{x_2}^1(\R\times\R\times\R)}\\
= & \frac1{ (\lambda_n^k \lambda_n^{k'})^{\frac{1}{2}} } \bigg\| \langle \mathcal{H}_{x_2} \rangle^{1- \epsilon_0} \sum_{j \in \mathbb{N}} e^{ i t (2j+1) } e_j(x_2) \Pi_{j} v  \left( \frac{t}{ (\lambda_n^k)^2 } + t_n^k, \frac{x_1- x_{1,n}^k - 2 \xi_n^k(t-t_n^k)}{\lambda^k_n}\right)  \\
& \quad \cdot \langle \mathcal{H}_{x_2} \rangle^{1- \epsilon_0} \sum_{j \in \mathbb{N}} e^{ i t (2j+1) } e_j(x_2) \Pi_{j} v  \left( \frac{t}{ (\lambda_n^{k'})^2 } + t_n^{k'}, \frac{x_1- x_{1,n}^{k'} - 2 \xi_n^{k'}(t-t_n^{k'})}{\lambda^{k'}_n}\right) \bigg\|_{L_{t,x_1}^3 L_{x_2}^1(\R\times\R\times\R)}.
\end{align*}
Since $\Pi_{j}v(x_2) \in L^3_{t,x_1}(\R\times\R)$, it is harmless to assume $v_j \in C_0^\infty(\mathbb{R} \times \mathbb{R})$, then
\begin{align*}
supp \Pi_{j} v \left( \frac{t}{(\lambda_n^k)^2 } + t_n^k, \frac{x_1 - x_{1,n}^k - 2\xi_n^k( t-t_n^k)}{\lambda_n^k}\right) & \subseteq \left\{ (t,x_1,x_2): \left|\frac{t}{(\lambda_n^k)^2} + t_n^k \right| \le T, \left|\frac{x_1 - x_{1,n}^k - 2\xi_n^k(t -t_n^k)}{\lambda_n^k} \right| \le R\right\} \\
& = \left\{ |t+ (\lambda_n^k)^2 t_n^k | \le (\lambda_n^k)^2 T, | x_1 - x_{1,n}^k - 2\xi_n^k(t- t_n^k)| \le \lambda_n^k R\right\} \\
& := S(n , k),
\end{align*}
where ${\rm supp}\; \Pi_{j} v \subset (- T, T) \times (- R, R)$. Thus we arrive at
\begin{align*}
& \frac1{ (\lambda_n^k \lambda_n^{k'})^{\frac{1}{2}} } \bigg\| \langle \mathcal{H}_{x_2}  \rangle^{1- \epsilon_0} \sum_{j \in \mathbb{N}} e^{ i t (2j+1) } e_j(x_2) \Pi_{j} v  \left( \frac{t}{ (\lambda_n^k)^2 } + t_n^k, \frac{x_1- x_{1,n}^k - 2 \xi_n^k(t-t_n^k)}{\lambda^k_n}\right) \\
& \quad \times \langle \mathcal{H}_{x_2} \rangle^{1- \epsilon_0} \sum_{j \in \mathbb{N}} e^{ i t (2j+1) } e_j(x_2) \Pi_{j} v  \left( \frac{t}{ (\lambda_n^{k'})^2 } + t_n^{k'}, \frac{x_1- x_{1,n}^{k'} - 2 \xi_n^{k'}(t-t_n^{k'})}{\lambda^{k'}_n}\right) \bigg\|_{L_{t,x_1}^3 L_{x_2}^1\left(  S( n , k ) \cap  S( n , k^{'} )  \right)}  \\
& \lesssim \frac1{  (\lambda_n^k \lambda_n^{k'})^{\frac{1}{2}} } \bigg\| \langle \mathcal{H}_{x_2}  \rangle^{1- \epsilon_0} \sum_{j \in \mathbb{N}} e^{ i t (2j+1) } e_j(x_2) \Pi_{j} v  \left( \frac{t}{ (\lambda_n^k)^2 } + t_n^k, \frac{x_1- x_{1,n}^k - 2 \xi_n^k(t-t_n^k)}{\lambda^k_n}\right) \bigg\|_{L_{t,x_1}^6 L_{x_2}^2(  S( n , k ) \cap  S( n , k^{'} )} \\
& \quad \times \left\| \langle \mathcal{H}_{x_2} \rangle^{1- \epsilon_0} \sum_{j \in \mathbb{N}} e^{ i t (2j+1) } e_j(x_2) \Pi_{j} v  \left( \frac{t}{ (\lambda_n^{k'})^2 } + t_n^{k'}, \frac{x_1- x_{1,n}^{k'} - 2 \xi_n^{k'}(t-t_n^{k'})}{\lambda^{k'}_n}\right)  \right\|_{L_{t,x_1}^6 L_{x_2}^2\left(  S( n , k ) \cap  S( n , k^{'} )  \right)}.
\end{align*}
Denote
\begin{align*}
\Lambda_n^k = \left\{(t,x_1)\in\R\times\R : \left|( \lambda_n^k)^{-2} t + t_n^k  \right| + \left| \frac{x_1 - x_{1,n}^k - 2\xi_n^k(t -t_n^k)}{\lambda_n^k} \right| \le R \right\},\\
\Lambda_n^{k'} = \left\{(t,x_1)\in\R\times\R : \left|( \lambda_n^{k'})^{-2} t + t_n^{k'} \right| + \left| \frac{x_1 - x_{1,n}^{k'}  - 2\xi_n^{k'}(t -t_n^{k'})}{\lambda_n^{k'}} \right| \le R \right\},
\end{align*}
then for
\begin{align*}
\tilde{w}_n^k (t,x,y) = \frac1{ (\lambda_n^k)^{\frac{1}{2}} } \sum\limits_{j \in \mathbb{N}} e^{-it (2j+1) } e_j(x_2)  \Pi_{j} v\left( \frac{t}{(\lambda_n^k)^2} + t_n^k, \frac{x_1 - x_{1,n}^k - 2\xi_n^k(t -t_n^k)}{\lambda_n^k} \right),
\end{align*}
we have $\operatorname{supp} \tilde{w}_n^k \subset \Lambda_n^k, \ \operatorname{supp} \tilde{w}_n^{k'} \subset \Lambda_n^{k'}$. Thus,  
\begin{align*}
& \left\| \langle \mathcal{H}_{x_2} \rangle^{1- \epsilon_0} \tilde{w}_n^k \cdot \langle \mathcal{H}_{x_2} \rangle^{1- \epsilon_0} \tilde{w}_n^{k'} \right\|_{L_{t,x_1}^3 L_{x_2}^1(\mathbb{R} \times \mathbb{R} \times \mathbb{R}  )}\\
\lesssim & \left\| \langle \mathcal{H}_{x_2} \rangle^{1- \epsilon_0} \tilde{w}_n^k \cdot \langle \mathcal{H}_{x_2} \rangle^{1- \epsilon_0} \tilde{w}_n^{k'} \right\|_{L_{t,x_1}^3 L_{x_2}^1(   \{ (\mathbb{R} \times \mathbb{R}) \setminus \Lambda_n^k \} \times \mathbb{R})}
+ \left\| \langle \mathcal{H}_{x_2} \rangle^{1- \epsilon_0} \tilde{w}_n^k \cdot \langle \mathcal{H}_{x_2} \rangle^{1- \epsilon_0} \tilde{w}_n^{k'} \right\|_{L_{t,x_1}^3 L_{x_2}^1(  \{ (\mathbb{R} \times \mathbb{R}) \setminus \Lambda_n^{k'} \} \times \mathbb{R})}\\
& +  \left\| \langle \mathcal{H}_{x_2} \rangle^{1- \epsilon_0} \tilde{w}_n^k \cdot \langle \mathcal{H}_{x_2} \rangle^{1- \epsilon_0} \tilde{w}_n^{k'} \right\|_{L_{t,x_1}^3 L_{x_2}^1( \{ \Lambda_n^k \cap \Lambda_n^{k'}  \}  \times \mathbb{R})} \\
\lesssim & \left\| \langle \mathcal{H}_{x_2} \rangle^{1- \epsilon_0} \tilde{w}_n^k \right\|_{L_{t,x_1}^6 L_{x_2}^2( \{ (\mathbb{R}\times\R) \setminus \Lambda_n^k \} \times \mathbb{R})}
\left\| \langle \mathcal{H}_{x_2} \rangle^{1- \epsilon_0} \tilde{w}_n^{k'} \right\|_{L_{t,x_1}^6 L_{x_2}^2( \{ (\mathbb{R}\times\R) \setminus \Lambda_n^k \} \times \mathbb{R})} \\
& + \left\| \langle \mathcal{H}_{x_2} \rangle^{1- \epsilon_0} \tilde{w}_n^k  \right\|_{L_{t,x_1}^6 L_{x_2}^2(\{ (\mathbb{R}\times\R) \setminus \Lambda_n^{k'} \} \times \mathbb{R})} \left\|\langle \mathcal{H}_{x_2} \rangle^{1- \epsilon_0} \tilde{w}_n^{k'} \right\|_{L_{t,x_1}^6 L_{x_2}^2 (\{ (\mathbb{R}\times\R) \setminus \Lambda_n^{k'} \} \times \mathbb{R})} 
+ \frac1{(\lambda_n^k \lambda_n^{k'})^\frac12}  area( \Lambda_n^{k} \cap \Lambda_n^{k'})^\frac13.
\end{align*}
To obtain the first two terms vanish when $k,k' \rightarrow \infty$, we need to show
\begin{align*}
&\hspace{2ex}\left\|\langle \mathcal{H}_{x_2} \rangle^{1- \epsilon_0} \tilde{w}_n^k \right\|_{L_{t,x_1}^6 L_{x_2}^2(\R\times\R\times\R)} \\
\sim & \left\| \langle j \rangle^{1- \epsilon_0} \frac1{(\lambda_n^k)^\frac12} e^{-it (2j+1) }  \Pi_{j}  v  \left( \frac{t}{(\lambda_n^k)^2 } + t_n^k, \frac{x_1 - x_{1,n}^k - 2 \xi_n^k( t- t_n^k)}{\lambda_n^k}\right) \right\|_{L_{t,x_1}^6 l_j^2(\R\times\R\times\N)}\\
= & \frac1{( \lambda_n^k)^\frac12} \left\| \langle j \rangle^{1- \epsilon_0}  \Pi_{j}  v \left( \frac{t}{(\lambda_n^k)^2} + t_n^k, \frac{x_1 - x_{1,n}^k - 2 \xi_n^k (t - t_n^k)}{\lambda_n^k}\right) \right\|_{L_{t,x_1}^6 l_j^2(\R\times\R\times\N)}\\
= & \frac1{(\lambda_n^k)^\frac12} ( \lambda_n^k )^\frac26 ( \lambda_n^k)^\frac16 \| v \|_{L_{t,x}^6 \mathcal{H}_{x_2}^{1- \epsilon_0}(\R\times\R\times\R)} = \|v \|_{L_{t,x}^6 \mathcal{H}_{x_2}^{1 - \epsilon_0}(\R\times\R\times\R)},
\end{align*}
and
\begin{align*}
\left\| \langle \mathcal{H}_{x_2} \rangle^{1- \epsilon_0} \tilde{w}_n^k \right\|_{L_{t,x_1}^6 L_{x_2}^2( \{ \mathbb{R}^2 \setminus \Lambda_n^k \} \times \mathbb{R})} =  \| v \|_{L_{t,x_1}^6 \mathcal{H}_{x_2}^{1- \epsilon_0}\left( \{ |t| + |x| \ge R\} \times \mathbb{R}\right)} \to 0, \text{ as } R \to \infty.
\end{align*}
It remains  to show
\begin{align*}
\frac1{(\lambda_n^k \lambda_n^{k'})^\frac12} area( \Lambda_n^k \cap \Lambda_n^{k'})^\frac13 \to 0, \text{ as } n \to \infty.
\end{align*}
Noteice that
\begin{align*}
& \quad area ( \Lambda_n^k \cap \Lambda_n^{k'})\\
& = area \left( \left\{ \left|\frac{t}{(\lambda_n^k)^2} + t_n^k \right| + \left| \frac{x - x_n^k - 2 \xi_n^k( t- t_n^k) }{ \lambda_n^k} \right| < R\right\} \cap \left\{ |\frac{t}{(\lambda_n^{k'})^2} + t_n^{k'} | + | \frac{x - x_n^{k'}  - 2 \xi_n^{k'}( t- t_n^{k'} ) }{ \lambda_n^{k'}} | < R\right\}\right)\\
& \lesssim \min\left( (\lambda_n^k)^3, ( \lambda_n^{k'})^3\right),
\end{align*}
hence if  $\lim\limits_{n\to \infty} \frac{\lambda_n^k}{\lambda_n^{k'}} = \infty $ or $ \lim\limits_{n\to \infty} \frac{\lambda_n^{k'}}{\lambda_n^{k}} = \infty $, we have
\begin{align*}
\frac1{(\lambda_n^k \lambda_n^{k'})^\frac12} area( \Lambda_n^{k} \cap \Lambda_n^{k'})^\frac13   \to 0 \text{ as } n \to \infty.
\end{align*}
The same thing happens if
\begin{align*}
\lim\limits_{n\to \infty} \frac{ | ( \lambda_n^k)^2 t_n^k - ( \lambda_n^{k'})^2 t_n^{k'}|}{\lambda_n^k \lambda_n^{k'}}  = \infty.
\end{align*}
For the following case
\begin{align*}
\lambda_n^k \lambda_n^{k'} | \xi_n^k - \xi_n^{k'}|^2 + \frac{ | x_n^{k'} - x_n^k - 2t_n^{k'} (\lambda_n^{k'})^2 ( \xi_n^{k'} - \xi_n^k)|^2}{\lambda_n^k \lambda_n^{k'}} \to \infty \text{ as } n \to \infty,
\end{align*}
it is easy to check that
\begin{align*}
area( \Lambda_n^k \cap \Lambda_n^{k'}) 
\lesssim  area \left( \left\{  \left| \frac{x_1 - x_{1,n}^k - 2\xi_n^k( t- t_n^k)}{\lambda_n^k} \right| < R\right\} \bigcap \left\{  \left| \frac{ x_1 - x_{1,n}^{k'} - 2\xi_n^{k'} (t - t_n^{k'})}{\lambda_n^{k'}} \right| < R\right\}\right).
\end{align*}
By change of variables
\begin{align*}
v = x - x_n^k - 2\xi_n^k( t- t_n^k), \quad w = x - x_n^{k'} - 2\xi_n^{k'}(t - t_n^{k'}),
\end{align*}
we have
\begin{align*}
area( \Lambda_n^k \cap \Lambda_n^{k'})
= \int_{(t,x) \in \Lambda_n^k \cap \Lambda_n^{k'}} \,\mathrm{d}t \mathrm{d}x
\lesssim \int_{| v| \le \lambda_n^k R, |w| \le \lambda_n^{k'} R} \,dtdx.
\end{align*}
Since
\begin{align*}
\frac{\partial(v,w)}{\partial(t,x)} =
\begin{bmatrix}
- 2\xi_n^k \ 1\\
- 2\xi_n^{k'} \ 1
\end{bmatrix},
\end{align*}
we further  have
\begin{align*}
area( \Lambda_n^k \cap \Lambda_n^{k'})
\lesssim \int_{|v| \le \lambda_n^k R, |w| \le \lambda_n^{k'} R} \left| \frac{\partial(t,x)}{\partial(v,w)} \right| \,\mathrm{d}v \mathrm{d}w
\lesssim \int_{|v| \le \lambda_n^k R, |w| \le \lambda_n^{k'} R} \frac1{|2(\xi_n^k - \xi_n^{k'})|} \,\mathrm{d}v \mathrm{d}w
\lesssim \frac{\lambda_n^k \lambda_n^{k'} R^2}{| \xi_n^k - \xi_n^{k'}|}.
\end{align*}
Using the fact $\lambda_n^k \sim \lambda_n^{k'}$ and the assumption $\lambda_n^k \lambda_n^{k'} |\xi_n^k - \xi_n^{k'}|^2 \to \infty$, we obtain that
\begin{align*}
\frac1{(\lambda_n^k \lambda_n^{k'})^\frac12} area( \Lambda_n^k \cap \Lambda_n^{k'})^\frac13
\lesssim \frac1{\lambda_n^k} \left( \frac{(\lambda_n^k)^2 R^2}{|\xi_n^k - \xi_n^{k'}|}\right)^\frac13 \sim \frac{R^\frac23}{ (\lambda_n^k)^\frac13 |\xi_n^k - \xi_n^{k'}|^\frac13} \to 0, \text{ as } n\to \infty.
\end{align*}
For the final case
\begin{align*}
\frac{ | x_{1,n}^{k'} - x_{1,n}^k - 2 t_n^{k'} ( \lambda_n^{k'})^2 ( \xi_n^{k'} - \xi_n^k)|^2}{\lambda_n^k \lambda_n^{k'}} \to \infty, \text{ as } n \to \infty.
\end{align*}
For $(t,x) \in \Lambda_n^k \cap \Lambda_n^{k'}$,
\begin{align*}
\left|\frac{t}{(\lambda_n^k)^2} + t_n^k \right| + \left|\frac{ x_1 - x_{1,n}^k - 2\xi_n^k t}{\lambda_n^k}\right| \le R,\quad
\left|\frac{t}{(\lambda_n^{k'})^2} + t_n^{k'} \right| + \left|\frac{ x_1 - x_{1,n}^{k'} - 2\xi_n^{k'} t}{ \lambda_n^{k'} }  \right| \le R,
\end{align*}
which implies
\begin{align*}
R \ge \frac{ |x_1 - x_{1,n}^k - 2\xi_n^k t|}{\lambda_n^k}
& = \frac{ |x_{1,n}^{k'} - x_{1,n}^k + x_1 - x_{1,n}^{k'} - 2\xi_n^k t|}{\lambda_n^k}\\
& \ge \frac{| x_{1,n}^{k'} - x_{1,n}^k - 2t_n^{k'}(\lambda_n^{k'})^2( \xi_n^{k'} - \xi_n^k) |}{ \lambda_n^k} \\
& \quad - \frac{ | 2 t_n^{k'}(\lambda_n^{k'})^2 ( \xi_n^{k'} - \xi_n^k ) + x_1 - x_{1,n}^{k'} - 2\xi_n^k t|}{\lambda_n^k}.
\end{align*}
We note
\begin{align*}
\frac{ | x_1 - x_{1,n}^{k'} - 2\xi_n^k t + 2 ( \xi_n^{k'} -  \xi_n^k) t_n^{k'} (\lambda_n^{k'})^2 |}{\lambda_n^k}
& =  \frac{ | x_1 - x_{1,n}^{k'} - 2\xi_n^{k'} t + 2 ( \xi_n^{k'} - \xi_n^k)  t + 2 ( \xi_n^{k'} - \xi_n^k) t_n^{k'} (\lambda_n^{k'})^2|}{ \lambda_n^k}\\
& \le \frac{ \lambda_n^{k'}}{\lambda_n^k} R + \frac{ 2 |( \xi_n^{k'} - \xi_n^k) ( \lambda_n^{k'})^2 |}{ \lambda_n^k} R < \infty.
\end{align*}
Consequently, $\Lambda_n^k \cap \Lambda_n^{k'}  = \varnothing$, when $n$ large enough.
		
\end{proof}

Then we verify claim $(i)$ above. By \eqref{eq6.3} and \eqref{eq7.29}, we have
\begin{align*}
\left\|\sum\limits_{k=1}^K v_n^k \right\|_{  L_{t,x_1}^6 \mathcal{H}_{x_2}^{1-\epsilon_0}(\R\times\R\times\R)}^6\lesssim
& \left(\sum\limits_{k =1 }^K \left\|    v_n^k     \right\|_{L_{t,x_1}^6 \mathcal{H}_{x_2}^{1-\epsilon_0}(\R\times\R\times\R)}^2 +
\sum\limits_{j\ne k} \left\|    \langle \mathcal{H}_{x_2}  \rangle^{1-\epsilon_0} v_n^j \cdot  \langle \partial_y
\rangle^{1-\epsilon_0} v_n^k     \right\|_{L_{t,x_1}^3 L_{x_2}^1(\R\times\R\times\R)}\right)^3\\
\lesssim &  \left(\sum\limits_{k=1}^K
S\left( \frac1{ (\lambda_n^k)^{\frac{1}{2}} }   e^{ix_1 \cdot \xi_n^k}   \left(e^{it_n^k\partial_{x_1}^2}   P_n^k  \phi \right)
\left(\frac{x_1-x_{1,n}^k}{\lambda_n^k}, x_2 \right) \right)  + o_K(1)\right)^3,
\end{align*}
for $K$ large enough. The mass and energy decoupling implies
\begin{align*}
\sum\limits_{k=1}^K
(E + M)u\left( \frac1{ (\lambda_n^k)^{\frac{1}{2}} }   e^{ix_1 \cdot \xi_n^k}   \left(e^{it_n^k \partial_{x_1}^2}   P_n^k  \phi \right)
\left(\frac{x_1-x_{1,n}^k}{\lambda_n^k}, x_2 \right) \right)\le L_{max}.
\end{align*}
Together with \eqref{eq5.3}, we obtain
\begin{align}\label{eq6.5}
\lim\limits_{K \to K^*} \limsup\limits_{n\to \infty} \|  u_n^K \|_{L_t^6 L_{x_1}^{6} \mathcal{H}_{x_2}^{1-\epsilon_0}} \lesssim_{L_{max}, \delta} 1.
\end{align}
Now we turn to prove the property $(ii)$. By the definition of $u_n^K$, we rewrite $e_n^K$ as
\begin{align*}
e_n^K & = (i\partial_t + \partial_{x_1}^2 + \partial_{x_2}^2 - x_2^2 )u_n^K - |u_n^K|^4 u_n^K \\
& = \sum\limits_{k=1}^K  |v_n^k|^4 v_n^k - \left|\sum\limits_{k =1}^K v_n^k\right|^4 \sum\limits_{k=1}^K  v_n^k
+ |u_n^K- e^{it(\partial^2_{x_1}+\partial^2_{x_2} - x_2^2)} w_n^K |^4 (u_n^K - e^{it(\partial^2_{x_1}+\partial^2_{x_2} - x_2^2)} w_n^K) - |u_n^K |^4 u_n^K.
\end{align*}
	
By the fractional chain rule, Minkowski's inequality, H\"older's inequality, Sobolev embedding, \eqref{eq5.8v49} and \eqref{eq6.3}, we get
\begin{align}\label{eq4.832}
&\quad  \left\| \sum\limits_{k=1}^K |v_n^k|^4 v_n^k - \left|\sum\limits_{k=1}^K  v_n^k \right|^4 \sum\limits_{k =1}^K v_n^k \right\|_{L_{t,x_1}^\frac65 \mathcal{H}_{x_2}^{1-\epsilon_0}}\\
& \lesssim \sum_{k \ne k'} \big \|\langle \mathcal{H}_{x_2} \rangle^{1-\epsilon_0} (v_n^k (v_n^{k'})^4) \big\|_{L_{t,x_1}^\frac65 L_{x_2}^2}           \notag \\
& \lesssim \sum_{k \ne k'} \big( \|v_n^{k'} \langle \mathcal{H}_{x_2} \rangle^{1-\epsilon_0} v_n^k\|_{L_{t,x_1}^3 L_{x_2}^2} \|v_n^{k'}\|_{L_{t,x_1}^6 L_{x_2}^\infty}^3+ \|v_n^k \langle \mathcal{H}_{x_2} \rangle^{1-\epsilon_0} v_n^{k'} \|_{L_{t,x_1}^3 L_{x_2}^2} \|v_n^{k'}\|_{L_{t,x_1}^6 L_{x_2}^\infty}^3\big)       \notag \\
& \sim \sum_{k \ne k'} \|v_{k'} \langle \mathcal{H}_{x_2} \rangle v_n^k \|_{L_{t,x_1}^3 L_{x_2}^2} \big(\|v_n^{k'} \|_{L_{t,x_1}^6 L_{x_2}^\infty}^3 + \|v_n^k \|_{L_{t,x_1}^6 L_{x_2}^\infty}^3\big)\notag \\
& \lesssim  \sum_{k \ne k'} \left\|v_n^j \langle \mathcal{H}_{x_2} \rangle^{1-\epsilon_0} v_n^k \right\|_{L_{t,x_1}^3 L_{x_2}^2} \big(\|v_n^{k'} \|_{L_{t,x_1}^6 H_y^{1-\epsilon_0}}^3+ \|v_n^k \|_{L_{t,x_1}^6 H_y^{1-\epsilon_0}}^3\big)\lesssim o_K(1), \text{ as } n\to \infty.\notag
\end{align}
It remains to bound $\left|u_n^K   - e^{it (\partial_{x_1}^2 + \partial_{x_2}^2 - x^2_2 )} w_n^K \right|^4 (u_n^K   - e^{it(\partial_{x_1}^2 + \partial_{x_2}^2 - x^2_2 )} w_n^K ) - |u_n^K |^4 u_n^K $. By the fractional chain rule, H\"older, Sobolev, we have
\begin{align*}
& \left\|\left|u_n^K - e^{it (\partial_{x_1}^2 + \partial_{x_2}^2  -x^2_2)  } w_n^K \right|^4 \left(u_n^J  - e^{it (\partial_{x_1}^2 + \partial_{x_2}^2  -x^2_2)  } w_n^K \right) - \left|u_n^K \right|^4 u_n^K \right\|_{  L_{t,x_1}^\frac65    \mathcal{H}_{x_2}^{1-\epsilon_0} }\\
\lesssim & \left\|u_n^K \right\|_{L_{t,x_1}^6 \mathcal{H}_{x_2}^{1-\epsilon_0}}^4 \left\|e^{it (\partial_{x_1}^2 + \partial_{x_2}^2  -x^2_2)  } w_n^K \right\|_{L_{t,x_1}^6 \mathcal{H}_{x_2}^{1-\epsilon_0}}+ \left\|u_n^K \right\|_{L_{t,x_1}^6 \mathcal{H}_{x_2}^{1-\epsilon_0}}^3 \left\|e^{it (\partial_{x_1}^2 + \partial_{x_2}^2  -x^2_2)  } w_n^K \right\|_{L_{t,x_1}^6 \mathcal{H}_{x_2}^{1-\epsilon_0}}^2 \\
& + \left\|u_n^K \right\|_{L_{t,x_1}^6 \mathcal{H}_{x_2}^{1-\epsilon_0}}^2 \left\|e^{it (\partial_{x_1}^2 + \partial_{x_2}^2  -x^2_2)  } w_n^K \right\|_{L_{t,x_1}^6 \mathcal{H}_{x_2}^{1-\epsilon_0}}^3+\left\|u_n^K \right\|_{L_{t,x_1}^6 \mathcal{H}_{x_2}^{1-\epsilon_0}}  \left\|e^{it (\partial_{x_1}^2 + \partial_{x_2}^2  -x^2_2)  } w_n^K  \right\|_{L_{t,x_1}^6 \mathcal{H}_{x_2}^{1-\epsilon_0}}^4 \\
& +  \left\|e^{it (\partial_{x_1}^2 + \partial_{x_2}^2  -x^2_2)  } w_n^K \right\|_{L_{t,x_1}^6 \mathcal{H}_{x_2}^{1-\epsilon_0}}^5\\
\lesssim & \left( \left\|u_n^K \right\|_{L_{t,x_1}^6 \mathcal{H}_{x_2}^{1-\epsilon_0}}^4 + \left\|e^{it (\partial_{x_1}^2 + \partial_{x_2}^2  -x^2_2)  } w_n^K  \right\|_{L_{t,x_1}^6 \mathcal{H}_{x_2}^{1-\epsilon_0}}^4 \right)
\left\|e^{it (\partial_{x_1}^2 + \partial_{x_2}^2  -x^2_2)  } w_n^K \right\|_{L_{t,x_1}^6 \mathcal{H}_{x_2}^{1-\epsilon_0}  }.
\end{align*}
Using \eqref{eq6.5}, and the decay property \eqref{eq6.1}, we get
\begin{align*}
\limsup\limits_{n\to \infty} \left\|\left|u_n^K - e^{it (\partial_{x_1}^2 + \partial_{x_2}^2  -x^2_2)  } w_n^K \right|^4 \left( u_n^K - e^{it (\partial_{x_1}^2 + \partial_{x_2}^2  -x^2_2)  } w_n^K  \right) - \left|u_n^K \right|^4 u_n^K  \right\|_{  L_{t,x_1}^6 \mathcal{H}_{x_2}^{1-\epsilon_0}} \to 0,
\end{align*}
as $K \rightarrow K^*$.
\end{proof}

\subsection{The proof of Theorem \ref{Thm1}}
In this part, we will preclude the possibility of the almost-periodic solution by the interaction Morawetz inequality. 
\begin{proposition}[Non-existence of the almost-periodic solution]\label{pr7.3}
The almost-periodic solution $u_c$ as in Theorem \ref{reduction} does not exist.
\end{proposition}

\begin{proof}
For each $r_0>0$, we define the interaction Morawetz action
\begin{align*}
M_{r_0}(t) = \int_{\R \times\R} \int_{\R\times \R} \operatorname{Im}\left(\overline{ {u}_c(t,x_1,x_2) }  \partial_{x_1} u_c (t,x_1,x_2) \right) \cdot \partial_{x_1} \psi_{r_0} \left(|x_2-\tilde{x_2}| \right) \left|u_c \left(t,\tilde{x_1},\tilde{x_2} \right)\right|^2 \,dx_1dx_2 d\tilde{x_1} d\tilde{x_2},
\end{align*}
where $\psi_{r_0}\colon \R \to \R$ is a radial function defined as in \cite{Planchon-Vega} with
\begin{align*}
\Delta \psi_{r_0}(r) = \int_r^\infty s \log\left(\frac{s}r\right) w_{r_0}(s) \,ds,
\end{align*}
where
\begin{align*}
w_{r_0}(s) =
\begin{cases}
\frac1{s^3}, \ \text{if } s\ge r_0,\\
0, \  \text{ if } s < r_0.
\end{cases}
\end{align*}
It is straightforward to verify that $\psi_{r_0}$ is convex and $ \left| \partial \psi_{r_0} \right| $ is uniformly bounded (independent of $r_0$), with
\begin{align*}
-\Delta^2 \psi_{r_0} (r) = \frac{2\pi}{r_0} \delta_0(r) - w_{r_0}(r).
\end{align*}
Thus, we have  for all $T_0 > 0$,
\begin{align}\label{eq5.032}
\int_{-T_0}^{T_0} \int_{\R} \left| \partial_{x_1}  \left( \left\|u_c(t,x_1,x_2) \right\|_{L_{x_2}^2(\R) }^2\right) \right|^2 \,dx_1dt
\lesssim \sup_t \|u_c\|_{L_{x_1,x_2}^2(\R\times\R)}^3\|\partial_{x_1}u_c\|_{L_{x_1,x_2}^2(\R\times\R)}^2\lesssim1.
\end{align}
By the compactness of $u_c$ and the mass conservation, we have
\begin{align}\label{eq5.232}
\frac{m_c^2  }2 \le \int_{|x_1 - \tilde{x}_1(t)|\le C(\eta)}  \left\|  u_c(t,x_1,x_2) \right\|_{L_{x_2}^2(\R)}^2 \,dx_1.
\end{align}
where $m_c  := \|u_c\|_{L_{x_1,x_2}^2(\R\times\R)} > 0$ by \eqref{con-minimal}.
	
Therefore, for each $  T_0 > 0$, integrating with respect to time variable, Sobolev's inequality, and \eqref{eq5.032}, we deduce
\begin{align*}
\frac{ m_c^2 T_0} 2 &  \le \int_{-T_0}^{T_0} \left(\int_{|x_1 - \tilde{x}_1(t)|\le C\left(\eta\right)} \left\|u_c(t,x_1,x_2) \right\|_{L_{x_2}^2(\R)}^2 \,dx_1 \right)^2 \,dt\\
& \lesssim C\left(\eta\right)  \int_{-T_0}^{T_0}  \left( \left(\int_{\R} \left(\left\|u_c(t,x_1,x_2)\right\|_{L_{x_2}^2(\R)}^2\right)^4 \,dx_1 \right)^\frac14 \right)^2 \,dt\\
& \lesssim   C\left(\eta\right)^\frac{3}{2} \int_{-T_0}^{T_0} \int_{\R} \left| \partial_{x_1}  \left(\left\|u_c(t,x_1,x_2)\right\|_{L_{x_2}^2(\R)}^2\right) \right|^2 \,dx_1 dt\lesssim C\left(\eta\right).
\end{align*}
This implies that $T_0<\infty$ which is a contradiction when $T_0\to\infty$, thus we can finish the proof of Proposition \ref{pr7.3} and Theorem \ref{Thm1}.
\end{proof}




\section{Local well-posedness  of (DCR) system}\label{sec:Local}
In this section, we prove  the local well-posedness and small-data scattering for the (DCR) system which is given by
\begin{align}\label{DCR-2}
\begin{cases}
i\partial_tu+\partial_{x_1}^2u=F(u)\stackrel{\triangle}{=}\sum
\limits_{\substack{n_1,n_2,n_3,n_4,n_5,n\in\Bbb{N}\\n_1-n_2+n_3-n_4+n_5=n}}\Pi_{n}(\Pi_{n_1}u\overline{\Pi_{n_2}u}\Pi_{n_3}u\overline{\Pi_{n_4}u}\Pi_{n_4}u)\\
u(0,x_1,x_2)=\phi(x_1,x_2),
\end{cases}
\end{align}
where $\Pi_n$ denotes the spectral projector on the $n$-th eigenspace $E_n$.

The dispersive continuous resonant (DCR) system  enjoys structure with the following mass and energy conservation laws
\begin{align}\label{equ:Msudef}
&M_{S}(u)=\int_{\R\times\R}\sum_{n\in\N}(2n+1)|\Pi_nu|^2dx=M(u_0),\\\label{equ:Esudef}
&E_S(u)=\frac{1}{2}\sum_{n\in\Bbb{N}}\int_{\R\times\R}|\partial_{x_1}u_n|^2dx_1dx_2+\frac{1}{6}\sum_{\substack{n_1,n_2,n_3,n_4,n\in\Bbb{N}\\n_1-n_2+n_3-n_4+n_5=n}}\int_{\R\times\R}u_{n_1}\overline{u}_{n_2}u_{n_3}\overline{u}_{n_4}u_{n_5}\overline{u}_{n_6}dx_1dx_2.
\end{align}

First, we will give the local well-posedness and stability result of (DCR) system. These results can be proved by using the Strichartz estimates and the nonlinear estimates.
\begin{proposition}[Strichartz estimates associated to harmonic oscillator, \cite{Carles1,Jao-CPDE}]\label{Strichartz-harmonic} For $2\leqslant q,r\leqslant\infty$ with $\frac{2}{q}=\frac{1}{2}-\frac{1}{r}$, then the following estimate holds for any $T>0$,
\begin{align*}
\big\|e^{it(-\partial^2_{x}+x^2)}f\big\|_{L_t^q\mathcal{W}_{x}^{s,r}([-T,T]\times\R)}\lesssim\|f\|_{\mathcal{H}_x^s(\R)}.
\end{align*}
\end{proposition}

By using the Strichartz estimate and H\"older's inequality, one can obtain the following nonlinear estimate.
\begin{lemma}\label{nonlinear-1}
Let $F_j\in L_{x_1}^6L^2_{x_2}$ with $\beta\in\{0,1\}$, then we have
\begin{align}\label{I}
&\bigg\|\sum_{\substack{n_1,n_2,n_3,n_4,n_5,n\in\N\\n_1-n_2+n_3-n_4+n_5=n}}\Pi_{n}(\Pi_{n_1}F_1\overline{\Pi_{n_2}F_2}\Pi_{n_3}F_3\overline{\Pi_{n_4}F_4}\Pi_{n_5}F_5)\bigg\|_{L_{x_1}^\frac{6}{5}\mathcal{H}_{x_2}^\beta(\R\times\R)}\\\nonumber
\lesssim&\sum_{\tau\in\sigma_5}\|F_{\tau(1)}\|_{L_{x_1}^6\mathcal{H}_{x_2}^\beta(\R\times\R)}\prod_{j=2}^{5}\big\|F_{\tau(j)}\|_{L_{x_1}^6L_{x_2}^2(\R\times\R)},
\end{align}
where $\sigma_5$ is a permutation of the set $\{1,2,3,4,5\}$.
\end{lemma}

\begin{proof} Without loss of generality, we can assume that $\beta=0$.
For $\varphi_0\in L_{x_1}^6L_{x_2}^2$, using Strichartz estimates  \eqref{Strichartz-harmonic} and H\"older's inequality, we have
\begin{align*}
&\hspace{3ex}\left(\sum_{\substack{n_1,n_2,n_3,n_4,n_5,n\in\N\\n_1-n_2+n_3-n_4+n_5=n}}\Pi_{n}(\Pi_{n_1}F_1\overline{\Pi_{n_2}F_2}\Pi_{n_3}F_3\overline{\Pi_{n_4}F_4}\Pi_{n_5}F_5),\varphi_0\right)\\
&=\frac{1}{\pi}\sum_{n_1,n_2,n_3,n_4,n_5,n\in\N}\int_{0}^{\pi}e^{2it(n_1-n_2+n_3-n_4+n_5-n)}\int_{\R\times\R}\Pi_{n_1}F_1\overline{\Pi_{n_2}F_2}\Pi_{n_3}F_3\overline{\Pi_{n_4}F_4}\Pi_{n_5}F_5\overline{\Pi_{n}\varphi_0}dx_1dx_2dt\\
&=\frac{1}{\pi}\int_{0}^{\pi}\int_{\R}\int_{\R}e^{it(-\partial_{x_2}^2+x_2^2)}\varphi_0(x_1,x_2)\overline{e^{it(-\partial_{x_2}^2+x_2^2)}F_1(x_1,x_2)}e^{it(-\partial_{x_2}^2+x_2^2)}F_3(x_1,x_2)\overline{e^{it(-\partial_{x_2}^2+x_2^2)}F_4(x_1,x_2)}\\
&\hspace{2ex}\times e^{it(-\partial_{x_2}^2+x_2^2)}F_5(x_1,x_2)\overline{e^{it(-\partial_{x_2}^2+x_2^2)}F_6(x_1,x_2)}dx_1dx_2dt\\
&\lesssim\int_{\R}\left\|e^{it(-\partial_{x_2}^2+x_2^2)}\varphi_0(x_1,x_2)\right\|_{L_t^\infty L_{x_2}^2([-\pi,\pi]\times\R)}\prod_{j=1}^{5}\left\|e^{it(-\partial_{x_2}^2+x_2^2)}F_j(x_1,x_2)\right\|_{L_t^5 L_{x_2}^{10}([-\pi,\pi]\times\R)}dx_1\\
&\lesssim\int_{\R}\|\varphi_0(x_1,x_2)\|_{L_{x_2}^2}\prod_{j=1}^{5}\|F_j\|_{L_{x_2}^2(\R)}dx_1.
\end{align*}
In the above inequalities, we here used the fact that
\begin{align*}
e^{it(-\partial_{x_2}^2+x_2^2)}f(x_1,x_2)=\sum_{n\in \N}e^{it(2n+1)}\Pi_{n}f(x_1,x_2).
\end{align*}
Thus, using H\"older's inequality again, we have
\begin{align*}
\bigg\|\sum_{\substack{n_1,n_2,n_3,n_4,n_5,n\in\N\\n_1-n_2+n_3-n_4+n_5=n}}\Pi_{n}(\Pi_{n_1}F_1\overline{\Pi_{n_2}F_2}\Pi_{n_3}F_3\overline{\Pi_{n_4}F_4}\Pi_{n_5}F_5)\bigg\|_{L_{x_1}^\frac{6}{5}L_{x_2}^2(\R\times\R)}\lesssim\prod_{j=1}^{5}\big\|F_j\|_{L_{x_1}^6L_{x_2}^2(\R\times\R)}.
\end{align*}
For the case of $\beta=1$, since
\begin{align*}
&\hspace{5ex}	\bigg\|\sum_{\substack{n_1,n_2,n_3,n_4,n_5,n\in\N\\n_1-n_2+n_3-n_4+n_5=n}}\Pi_{n}(\Pi_{n_1}F_1\overline{\Pi_{n_2}F_2}\Pi_{n_3}F_3\overline{\Pi_{n_4}F_4}\Pi_{n_5}F_5)\bigg\|_{L_{x_1}^\frac{6}{5}\mathcal{H}_{x_2}^\beta(\R\times\R)}\\&=	\bigg\|\sum_{\substack{n_1,n_2,n_3,n_4,n_5,n\in\N\\n_1-n_2+n_3-n_4+n_5=n}}H^\frac{1}{2}\Pi_{n}(\Pi_{n_1}F_1\overline{\Pi_{n_2}F_2}\Pi_{n_3}F_3\overline{\Pi_{n_4}F_4}\Pi_{n_5}F_5)\bigg\|_{L_{x_1}^\frac{6}{5}L_{x_2}^2(\R\times\R)},
\end{align*}
where $H=-\partial_{x_2}^2+x_2^2$ and the fact that $H^\frac{1}{2}$ can commute with $\Pi_n$, we can obtain the desired estimate by using the similar strategy and Moser type estimate(cf. Lemma \ref{Moser-estimate}).
\end{proof}

\begin{remark}\label{nonlinear}
Indeed, we also have the following nonlinear estimates via the argument above:
\begin{align}\label{Estimate}
\bigg\|\sum_{\substack{n_1,n_2,n_3,n_4,n_5\in\N\\n_1-n_2+n_3-n_4+n_5=n}}\Pi_{n}(\Pi_{n_1}F_1\overline{\Pi_{n_2}F_2}\Pi_{n_3}F_3\overline{\Pi_{n_4}F_4}\Pi_{n_5}F_5)\bigg\|_{\mathcal{H}_{x_2}^1(\R)}\lesssim\sum_{\tau\in\sigma_5}\|F_{\tau(1)}\|_{\mathcal{H}_{x_2}^1(\R)}\prod_{j=2}^{5}\big\|F_{\tau(j)}\|_{L_{x_2}^2(\R)},
\end{align}
and
\begin{align*}
\bigg\|\sum_{\substack{n_1,n_2,n_3,n_4,n_5\in\N\\n_1-n_2+n_3-n_4+n_5=n}}\Pi_{n}(\Pi_{n_1}F_1\overline{\Pi_{n_2}F_2}\Pi_{n_3}F_3\overline{\Pi_{n_4}F_4}\Pi_{n_5}F_5)\bigg\|_{L_{x_2}^2(\R)}\lesssim\sum_{\tau\in\sigma_5}\prod_{j=1}^{5}\|F_{\tau(j)}\|_{L_{x_2}^2(\R)}.
\end{align*} 
\end{remark}

\begin{remark}\label{Remark}
For the case of $\R\times\T$, the corresponding nonlinear estimate can not hold when $\beta=1$. More precisely,  the left-hand side of \eqref{Estimate} cannot be bounded by one $h^1(\Z)$ norm and four other $\ell^2(\Z)$ norms. However, it was shown in \cite{Cheng-Guo-Zhao} that the left-hand side of \eqref{Estimate} can be bounded by one $h^1$ norm and four $h^\beta(\Z )$ norms with $\beta>\frac38$. Compared to \cite{Cheng-Guo-Zhao}, Strichartz estimates of linear Schr\"odinger equation with harmonic oscillator will help us obtain the better estimate \eqref{Estimate}.
\end{remark}

Consequently, we have the following space-time nonlinear estimate
\begin{lemma}\label{nonlinear-2}Let $u$ be a solution to \eqref{DCR-2} on interval $I$ and nonlinear term $F(u)$, then we have 
\begin{align*}
\big\|F(u)\big\|_{L_{t,x_1}^{\frac{6}{5}}\mathcal{H}_{x_2}^\beta(I\times\R\times\R)}\lesssim\|u\|_{L_{t,x_1}^6L_{x_2}^2(I\times\R\times\R)}^5
\end{align*}
for $\beta\geq0$. Moreover, we have the following estimate
\begin{align*}
\|u\|_{L_{t}^\infty L_{x_1}^2\mathcal{H}_{x_2}^1(I\times\R\times\R)}\lesssim\|\phi\|_{L_{x_1}^2\mathcal{H}_{x_2}^1(I\times\R\times\R)}+\|u\|_{L_{t,x_1}^6\mathcal{H}_{x_2}^1(I\times\R\times\R)}^5.
\end{align*}
\end{lemma}

Thus, combining Lemma \ref{nonlinear-1} and Lemma \ref{nonlinear-2}, we obtain the following  local well-posedness and small-data scattering  of \eqref{DCR-2}. 
\begin{theorem}
\begin{enumerate}
\item (Local well-posedness)
Suppose that $ \left\|\phi \right\|_{L_{x_1}^2 \mathcal{H}_{x_2}^1(\R\times\R)} <\infty$. Then, there exists an unique solution to \eqref{DCR-2} satisfying 
\begin{align*}
u \in C((-T,T),L_{x_1}^2 \mathcal{H}_{x_2}^1(\R\times\R))\cap L_{t,x_1}^6 \mathcal{H}_{x_2}^1((-T,T)\times\R\times\R) 
\end{align*}
for some $T >0$.

\item (Small data scattering)
There exists a sufficient small constant $\eta >0$ such that when $\|\phi \|_{L_{x_1}^2 \mathcal{H}_{x_2}^1(\R\times\R)} < \eta$, there eixsts  an unique global solution $u$ with $u|_{t=0} = \phi$. Moreover,  $u$ scatters in $L_{x_1}^2 \mathcal{H}_{x_2}^1(\R\times\R)$, i.e, there exists $\phi^\pm \in L_{x_1}^2 \mathcal{H}_{x_2}^1(\R \times \R)$ so that
\begin{align*}
\left\|u(t) - e^{it \partial^2_{x_1}} \phi^\pm \right\|_{L_{x_1}^2 \mathcal{H}_{x_2}^1(\R\times\R) } \to 0, \text{ when } t \to \pm \infty.
\end{align*}
\end{enumerate}
\end{theorem}

By the nonlinear estimates, we can also prove the following stability lemma. The proof is similar to that in \cite{Cheng-Guo-Guo-Liao-Shen} so we omit it here.
\begin{lemma}[Stability theory]\label{Stability-DCR}Let $I\subset\R$ be a compact interval and function $\tilde{u}\in C(I,L_{x_1}^2\mathcal{H}_{x_2}^1(\R\times\R))\cap L_{t,x_1}^6\mathcal{H}_{x_2}^1(I\times\R\times\R)$  be an approximate solution satisfying the following equation
\begin{align*}
i\partial_t\tilde u+\partial_{x_1}^2\tilde u=F(\tilde u)+e.
\end{align*}
Then for any $\varepsilon>0$,  there exists $\delta>0$ such that if
\begin{align*}
\|e\|_{L_{t,x_1}^\frac{6}{5}\mathcal{H}_{x_2}^{1}(I\times\R\times\R)}+\|\tilde{u}(t_0)-u(t_0)\|_{L_{x_1}^2\mathcal{H}_{x_2}^1(I\times\R\times\R)}<\delta,
\end{align*}
then we have
\begin{align*}
\|u-\tilde{u}\|_{L_{t,x_1}^6\mathcal{H}_{x_2}^{1}\cap L_t^\infty L_{x_1}^2\mathcal{H}_{x_2}^1(I\times\R\times\R)}<\varepsilon.
\end{align*}
\end{lemma}

By the nonlinear estimates and the bootstrap argument as in Section \ref{sec:local}, we can obtain the following scattering criterion.
\begin{lemma}
Let $u$ be the global solution to equation \eqref{DCR-2}  satisfying 
\begin{align*}
\|u\|_{L_{t,x_1}^6L_{x_2}^2(I\times\R\times\R)}<\infty,
\end{align*}	
then we have
\begin{align*}
\|u\|_{L_{t,x_1}^6\mathcal{H}_{x_2}^1(\R\times\R\times\R)}<\infty.
\end{align*}
\end{lemma}

\begin{remark}
As the same reason in Remark \ref{Remark}, for the case of $\R\times\T$, Cheng-Guo-Zhao\cite{Cheng-Guo-Zhao} only reduce the scattering norm to $L_{t,x}^6h^\beta(\R\times\R\times\Z )$ with $\beta>\frac38$. However,  the better nonlinear estimate \eqref{I} can help us reduce the scattering norm to  $L^2$ level with respect  to $x_2$ variable.
\end{remark}




\section{Reduction to almost-periodic solutions}\label{sec:minimal}
In this section, we give the classification of the minimal blow-up solution to the equation \eqref{DCR-5} via Dodson's strategy and reduce the proof of Theorem \ref{Thm2} to preclude the possibility of such special solutions.

We define a group action $G$ generated by phase rotations, Galilean transform, translations and dilations. More precisely, it is given by the map $g \rightarrow T_g$, where
\begin{equation*}
(T_{ g_{ ( \xi_n , x_{1,n} , \lambda_n^k) }  } u ) (t , x) := \left( \frac{1}{ \lambda^k_n } \right)^{\frac{1}{2}}   e^{ i x_1 \cdot \xi_n } e^{ -it |\xi_n|^2 } u \left(   \frac{t}{(\lambda_n^k) ^2} , \frac{x_1 - x_{1,n} - 2 \xi_n t}{\lambda_n^k}   \right).
\end{equation*}
Notice $T_g$ maps $L^2_{ x_1} \mathcal{H}^{1}_{x_2}$ to itself, we let $G / L_{x_1}^2\mathcal{H}_{x_2}^{1}$ be the modulo space of $G$-orbit endowed with the usual quotient topology.

Then, we define the function
\begin{equation*}
\Lambda(L) = \sup  \,  \{ \|  u\|_{L_t^6 L_{x_1}^{6} L^2_{x_2}  (\mathbb{R} \times \mathbb{R}\times\R)} : \| u(0) \|_{  L_{x_1}^2\mathcal{H}_{x_2}^1(\R\times\R) } \le L \},
\end{equation*}
and
\begin{equation*}
L_{max} = \sup \,  \{   L : \Lambda (L^{\prime}) < \infty , \forall L^{\prime} < L     \}.
\end{equation*}
By the local well-posedness theory, $\Lambda(L)< \infty$ for $L$ sufficiently small. Our goal is to prove $L_{max}  = \infty$.

Suppose that Theorem \ref{Thm2} fails, we can reduce to prove the existence of the almost-periodic solution. 

\begin{theorem}\label{thm-reduction-almost-perodic}
Assume $L_{max} < \infty $. Then there exist a solution $u \in C^0_t L^2_{x_1}\mathcal{H}^1_{x_2}(I\times\R\times\R) \cap L^6_{t , x_1} L^2_{x_2}(I\times\R\times\R)$ to $\eqref{DCR-2}$, also called critical element, with $I$ the maximal lifespan interval such that
\begin{enumerate}
\item $ M_S (u) = L_{\max}$ , \\
\item $u$ blow up at both directions, i.e. for some $t_0 , \tilde{t}_0 \in \mathbb{R}$, there are
\begin{equation*}
\|u\|_{L^6_{t , x_1} L_{x_2}^2 ( I \cap (-\infty , t_0) )} = \|u\|_{L^6_{t , x_1} L_{x_2}^2 ( I \cap (\tilde{t}_0 , \infty) )} = \infty.
\end{equation*}
\item $u$ is an almost periodic solution modulo transform group $G$.
\end{enumerate} 
\end{theorem}

In order to prove the above theorem, we will show a Palais-Smale type condition.

\begin{proposition}\label{prop-reduction-almost}
Assume that $L_{max} < \infty$. Let $\{t_n\}_{n\ge1}$ be arbitrary sequence of real numbers and $\{u_n\}_{n\ge 1}$ be a sequence of solutions in $C_t^0 L^2_{x_1}\mathcal{H}^1_{x_2} (\R\times\mathbb{R} \times \mathbb{R} )$ to \eqref{DCR-2} satisfying
\begin{equation*}
\limsup\limits_{ n \rightarrow \infty } M_S(u) = L_{max},
\end{equation*}
and
\begin{align}
\|u_n\|_{L_{t,x_1}^6 L_{x_2}^2 (I_n \cap (-\infty,t_n) \times \mathbb{R}^2) } \to \infty,
\|u_n\|_{ L_{t,x_1}^6 L_{x_2}^2(I_n \cap (t_n,\infty) \times \mathbb{R}^2)  } \to \infty, \text{ as } n \to \infty.\label{Palais-Smale}
\end{align}
Then $G u_n (t_n)$ converges in $G\setminus L^2_{x_1}\mathcal{H}^1_{x_2} $ up to a subseqence.
\end{proposition}

\begin{proof}[Proposition \ref{prop-reduction-almost} $\Longrightarrow$ Theorem $\ref{thm-reduction-almost-perodic}$ ]
By the definition of $L_{max}$, we can find a sequence $\{  u_n  \}_{n \in \mathbb{N}}$, where $u_n ( t , x_1 , x_2 )$ is  defined on its maximal lifespan  $I_n$ such that
\begin{gather*}
\lim\limits_{n \rightarrow \infty}  \| u_n \|_{L^6_{ t , x_1 } L_{x_2}^2}(\R\times\R\times\R)  =  \infty ,\\
\lim\limits_{ n \rightarrow \infty } M_S (u_n) =  L_{max} .
\end{gather*}
Then, there exists $t_n \in I_n$ such that
\begin{equation*}
\lim\limits_{ n \rightarrow \infty } \|u_n\|_{L_{t,x_1}^6 L_{x_2}^2 (I_n \cap (-\infty,t_n) \times \mathbb{R}^2) } = \lim\limits_{ n \rightarrow \infty } \|u_n\|_{ L_{t,x_1}^6 L_{x_2}^2(I_n \cap (t_n,\infty) \times \mathbb{R}^2)  } = \infty . 
\end{equation*}
We may further assume $t_n = 0$ by using the invariance of  time translation. Then by  Proposition \ref{prop-reduction-almost},  $G u_n(0)$ converges in $G \setminus L^2_{x_1}\mathcal{H}^1_{x_2} $ which is equivalent to $\exists u_0 \in L^2_{x_1}\mathcal{H}^1_{x_2} $ and $g_n \in G$ such that
\begin{equation*}
\lim\limits_{ n \rightarrow \infty } g_n u_n(0) = u_0.
\end{equation*}
Since $g$ is  invariant in $ L^2_{x_1}\mathcal{H}^1_{x_2} $, we can find that $M_S(u_0) \leqslant L_{max}$ .
	
Let $u(t , x_1 , x_2)$ be the solution to $\eqref{DCR-2}$  with the maximal lifespan $I$ and the  initial data $u_0$, we claim that $u(t , x_1 , x_2)$ blow up at both  time directions. Without loss of generality, we assume that  $u(t , x_1 , x_2)$ does not blow-up forward in time. Then we have $[0 , \infty) \subset I$ and
\begin{equation*}
\| u \|_{L^6_{t , x_1}L_{x_2}^2( (0 , \infty)\times\R\times\R) } < \infty. 
\end{equation*}
Using the stability theorem $\ref{Stability-DCR}$, for sufficiently large $n$, we have $[0 , \infty) \subset I_n$ and
\begin{equation*}
\| u_n \|_{L^6_{t , x_1}L_{x_2}^2( (0 , \infty)\times\R\times\R) } < \infty,
\end{equation*}
which is a contradiction. Similarly, we can show  that $u$ blow-up backward in time. Thus, we have $ M_S(u_0) \ge L_{max} $ and hence we arrive at $ M_S(u_0) = L_{\max} $.
	
It remains to verify that the orbit $\{ Gu : t \in I \}$ is pre-compact in $G \setminus L^2_{x_1}\mathcal{H}^1_{x_2} $. For arbitrary sequence $\{ G( t^{\prime}_n ) \}\subset\{Gu(t),t\in I\}$, since $u$ blows up at both time directions and is locally in $L^6_{t , x_1}L_{x_2}^2(I\times\R\times\R)$, we have
\begin{equation*}
\|u\|_{L^6_{t , x_1} L_{x_2}^2 ( I \cap ( -\infty , t^{\prime}_n ) ) } = \|u\|_{L^6_{t , x_1} L_{x_2}^2 ( I \cap (  t^{\prime}_n , \infty ) ) } = \infty .
\end{equation*}
Utilizing Proposition $\ref{prop-reduction-almost}$, $G u( t^{\prime}_n )$ converges in $G / L^2_{x_1}\mathcal{H}^1_{x_2} $. Thus, we have proved that the orbit $\{Gu(t),t\in I\}$ is pre-compact in $G\setminus L_{x_1}^2\mathcal{H}_{x_2}^1$.
\end{proof}

\begin{proof}[Proof of Proposition $\ref{prop-reduction-almost}$]
This proof is rather standard and we refer the readers to Tao-Visan-Zhang \cite{Tao-Visan-Zhang} and Yang-Zhao \cite{YZ}.
\end{proof}

We will demostrate three equivalent expressions for almost periodic solutions.  Before doing this, we will give the description of the compactness of Hilbert spacec $L^2_{x_1}\mathcal{H}^1_{x_2}$.

\begin{proposition}\label{prop-compactness-Sigma}
Let $\mathcal{X} \subset L^2_{x_1} \mathcal{H}^1_{x_2}(\R\times\R) $ be a bounded set, then $\mathcal{X}$ is precompact if and only if for arbitary $\varepsilon > 0$, there exist $K = K(\varepsilon) > 0, R = R(\varepsilon) > 0$, for all $u \in \mathcal{X}$
\begin{equation*}
\int_{\mathbb{R}} \sum_{ j > K } (2j + 1)|c_j (x_1) |^2 dx_1 < \varepsilon,
\end{equation*}   
besides,
\begin{gather*}
\int_{ |x_1| \ge R } \sum_{j=0}^{K}(2j+1)|c_j (x_1)|^2dx_1<\varepsilon,\\
\int_{ |\xi| \ge R } \sum_{j=0}^{K}(2j+1)|\widehat{c}_j(\xi)|^2d\xi<\varepsilon,
\end{gather*}
where $c_j(x_1) := \langle  u (x_1 , x_2 ) , e_j( x_2 )  \rangle_{L^2_{x_2}(\R)}$ and $\widehat{c_j}(\xi)$ is the Fourier transform of $c_j(x_1)$.
\end{proposition}

\begin{proof}
First, we prove the necessity. We denote $P_K\mathcal{X}$ by the following
\begin{equation*}
P_K \mathcal{X} := \{  \{  (2j + 1)^{\frac{1}{2}} c_j  \}_{j = 0} ^{K}   :  u \in \mathcal{X}  \}.
\end{equation*}
We claim that $P_K \mathcal{X}$ is precompact in $ \oplus _{j = 1} ^ {K} L^2_{x_1 , x_2} $. Since $\mathcal{X}$ is pre-compact in $L_{x_1}^2\mathcal{H}_{x_2}^1(\R\times\R)$, there exists a  finite set\\\noindent $\{  \{ u^{(1)}_j \}_{j = 0} ^{ K } , \cdots , \{ u^{ n_{\varepsilon} } _j \}_{ j = 0 } ^{ K }  \}$ whose $\varepsilon$-nets cover $\mathcal{X}$. Since $n_{\varepsilon}$ is finite, there exist $K := K(\varepsilon)$ such that
\begin{equation*}
\bigg\|  \sum_{ j \ge K(\varepsilon) } (2j + 1) | c_j^{(n_i)} (x_1) |^2 \bigg\|_{L^2_{x_1}(\R)}  < \varepsilon .
\end{equation*} 
For any $u \in \mathcal{X}$, we can find $m \in \{1 , 2 , \cdots , n_{\varepsilon} \}$ such that  
\begin{equation*}
\bigg\|  \sum_{ j \ge K(\varepsilon) } (2j + 1) | c_j^{(m)} - c_j |^2 \bigg\|_{L^2_{x_1}(\R)} < \varepsilon,
\end{equation*}
consequently, we arrive at
\begin{align*}
\bigg\|  \sum_{ j \ge K(\varepsilon) } (2j + 1) | c_j |^2 \bigg\|_{L^2_{x_1}(\R)} & \le \bigg\|  \sum_{ j \ge K(\varepsilon) } (2j + 1) | c_j^{(n_i)} (x_1) |^2 \bigg\|_{L^2_{x_1}(\R)} + \bigg\|  \sum_{ j \ge K(\varepsilon) } (2j + 1) | c_j^{(n_i)} - c_j |^2 \bigg\|_{L^2_{x_1}(\R)}\\
& < 2\varepsilon,
\end{align*}
which implies $P_{K} \mathcal{X}$ is precompact in $\bigoplus _{j = 0} ^ K L^2_{x_1 , x_2} $. Hence, we can find $R := R(\varepsilon) >0$ such that
\begin{gather*}
\int_{ |x_1| \ge R(\varepsilon) } \sum_{ j = 1 } ^ {K(\varepsilon)} (2j + 1) |c_j (x_1) |^2 dx_1 < \varepsilon , \\
\int_{ |\xi| \ge R(\varepsilon) } \sum_{ j = 1 } ^ {K(\varepsilon)} (2j + 1) | \widehat{c}_j (\xi) |^2 d\xi < \varepsilon
\end{gather*}
by the pre-compactness of $\bigoplus_{j = 1}^{K(\varepsilon)}L_{x_1,x_2}^2$. Hence, we complete the proof of the necessity.

Now,	we turn to prove the sufficiency. For arbitrary $\varepsilon>0$,  there exist $K(\varepsilon)>0$ and $R(\varepsilon)>0$ such that 
\begin{align*}
&\sum_{j=1}^{K(\epsilon)}(2j+1)\int_{|x|\geq R(\epsilon)}|c_j(x_1)|^2dx_1<\varepsilon,\\
&\sum_{j=1}^{K(\epsilon)}(2j+1)\int_{|\xi|\geq R(\epsilon)}|\widehat{c_j}(\xi)|^2d\xi<\varepsilon.
\end{align*}  For any $n > 0$  and for all $\{u_n\}\subset \mathcal{X}$,  we get
\begin{align*}
& \int_{|x_1| \ge R} \sum_{j = 0}^{ n } (2j + 1) |c_j (x_1)| ^2 dx_1 + \int_{|\xi| \ge R} \sum_{j = 0}^{ n } (2j + 1) |\widehat{c}_j (\xi)| ^2 d\xi \\
\lesssim &  \int_{|x_1| \ge R} \sum_{j = 0}^{ K } (2j + 1) |c_j (x_1)| ^2 dx_1 + \int_{|\xi| \ge R} \sum_{j = 0}^{ K } (2j + 1) |\widehat{c}_j (\xi)| ^2 d\xi \\
& + \int_{|x_1| \ge R} \sum_{j = K + 1}^{ n } (2j + 1) |c_j (x_1)| ^2 dx_1 + \int_{|\xi| \ge R} \sum_{j = K + 1}^{ n } (2j + 1) |\widehat{c}_j (\xi)| ^2 d\xi \\
<  & 4\varepsilon.
\end{align*}
According to the description of the pre-compactness of $\bigoplus\limits_{j=1}^KL_{x_1}^2$, we have that $\{(c_{1},c_2,\cdots,c_{n})\in\oplus_{j = 1}^nL_{x_1}^2,\{c_j\}_{j\in\N}\in \mathcal{X}\}$ is pre-compact in $\bigoplus_{j = 1}^nL_{x_1}^2$. Therefore, we have that $\{(c_{1},c_2,\cdots,c_{K(\varepsilon)})\in\bigoplus\limits_{j = 1}^{K(\varepsilon)}L_{x_1}^2\}$ is pre-compact in $\bigoplus\limits_{j = 1}^{K(\varepsilon)}L_{x_1}^2$.
	
Thus, we can find finite $n_1 , n_2 , \cdots , n_l$ such that the $\frac{C\varepsilon}{(1 + K(\varepsilon)^2)^2}$-nets of
\begin{equation*}
\big\{  \{  c_j^{ (n_1) } \}_{j = 1} ^{K(\varepsilon)}, \cdots , \{ |c_j^{ (n_l) } \}_{j = 1} ^{K(\varepsilon)}  \big\}
\end{equation*}
cover $P_{K}\mathcal{X}$. We claim that it is remains to show  the $3\varepsilon$-nets of 
\begin{equation*}
\big\{  \{ c_j^{ (n_1) } \}_{j\in\N} , \cdots , \{ c_j^{ (n_l) } \}_{j \in\N}   \big\}
\end{equation*} 
cover $\mathcal{X}$. For all $u \in \mathcal{X}$ we can find $n_i \in \{ n_1 , n_2 , \cdots , n_l \}$ such that $ \{ c_j \}_{j = 0} ^{K} $ lies in $\frac{\varepsilon}{1 + K^2}$ neighbourhood of $ \{  c_j^{(n_i)} \}_{j = 1} ^{K} $. By Minkowski's inequality, we get
\begin{align*}
\bigg\|  \|  u - u^{(n_i)}  \|_{\mathcal{H}^1_{x_2}(\R)}  \bigg\|_{L^2_{x_1}(\R)} \le & \int_{\mathbb{R}} \sum _{j = 0} ^{K(\varepsilon)} (2j + 1) |c_j - c_j ^{(n_i)}|^2 dx_1 + \int_{\mathbb{R}} \sum _{j = K(\varepsilon) + 1} ^{\infty} (2j + 1) |c_j|^2 dx_1\\
& + \int_{\mathbb{R}} \sum _{j = K(\varepsilon) + 1} ^{\infty} (2j + 1) |c_j ^{(n_i)}|^2 dx_1\\
& < 5\varepsilon,
\end{align*}
which implies $\mathcal{X}$ is pre-compact in $L^2_{x_1}\mathcal{H}^1_{x_2}(\R\times\R) $.
\end{proof}

As a direct application, we have the quantitative equivalent expressions of the almost-periodicity modulo the symmetry group $G$.
\begin{proposition}\label{prop-periodic-equi}
The following statement are equivalent.
\begin{enumerate}
\item $u \in C^0_t L_{x_1}^2\mathcal{H}_{x_2}^1$ is almost-periodic modulo $G$.
\item $\{  G u(t) : t \in I  \}$ is precompact in $ G / L_{x_1}^2\mathcal{H}_{x_2}^1$. 
\item $\exists x_1(t) ,  N(t)$ such that $\forall \eta > 0$ , $\exists R(\eta) >0 $ satisfying
\begin{gather}
\int_{\R}\sum _{j > K(\eta) } (2j+1) |c_j(t,x_1)|^2dx_1< \eta  \label{fml-almost-p-eq-1}, \\ 
\int_{  |x_1 - x_1(t)| \ge \frac{ C(\eta) }{ N(t) }  }\sum_{ j\leq K(\eta) } |c_j(t,x_1)|^2d x_1 < \eta, \label{fml-almost-p-eq-2}\\
\int_{|\xi-\xi(t)|\geq C(\eta)N(t)}\sum_{j\leq K(\eta)}|\widehat{c_j}(t,\xi)|^2d\xi<\eta.
\end{gather}
\end{enumerate}
\end{proposition}

 Summarizing the analysis above and using almost the same argument in \cite{Tao-Visan-Zhang},  we can further derive the following theorem:
\begin{theorem}[Almost-periodicity]\label{almost-periodic}Assume that Theorem \ref{Thm2} fails, then there exists a non-trivial minimal blow-up solution $v_c\in C(I,L_{x_1}^2\mathcal{H}_{x_2}^1(\R\times\R))$ to equation \eqref{DCR-2} with $M_S(v_c)=m_0$ and
\begin{align*}
\|v_c\|_{L_{t,x_1}^6L_{x_2}^2((I\cap(t_0,\infty))\times\R\times\R)}=\|v_c\|_{L_{t,x_1}^6L_{x_2}^2((I\cap(-\infty,t_0^\prime))\times\R\times\R)}=\infty, \quad\forall t_0, t_0'\in I,
\end{align*}
where  the maximal lifespan $I \supset [0,\infty)$. Moreover, there exist $C(\eta)>0$ and three parameters $(x_1(t),\xi(t),N(t))\in\R\times\R\times(0,1]$ such that
\begin{align}\label{almost-compact-1}
\int_{|x_1-x_1(t)|\geqslant\frac{C(\eta)}{N(t)}}\|v_c(t,x_1,x_2)\|_{\mathcal{H}_{x_2}^1(\R)}^2dx_1+\int_{|\xi-\xi(t)|\geqslant C(\eta)N(t)}\|\mathcal{F}_{x_1}v_c(t,\xi,x_2)\|_{\mathcal{H}_{x_2}^1(\R)}^2d\xi<\eta.
\end{align}
Furthermore, by the invariance of Galilean transform, we can take $N(0)=1$, $\xi(0)=0$ and 
\begin{align*}
|N^\prime(t)|+|\xi^\prime(t)|\lesssim N^3(t).
\end{align*} 
\end{theorem}

Now, we can choose three small constants $0<\eta_3\ll\eta_2\ll\eta_1\ll1$ and $\eta_3<\eta_2^{100}$ satisfying
\begin{align}
&|N^\prime(t)|+|\xi^\prime(t)|\leqslant 2^{-20}\eta_1^{-\frac{1}{2}}N(t)^3,\label{xi(t)}\\
&\int_{|x_1-x_1(t)|\geq \frac{2^{-20}\eta_3^{-\frac{1}{2}}}{N(t)} }\|v_c\|_{\mathcal{H}_{x_2}^1(\R)}^2dx_1+\int_{|\xi-\xi(t)|\geq 2^{-20}\eta_3^{-\frac{1}{2}}N(t) }\|\mathcal{F}_{x_1}v_c(t,\xi,x_2)\|_{\mathcal{H}_{x_2}^1(\R)}^2d\xi<\eta_2^2.\label{almost-compact-2}
\end{align} 
Let $k_0$ be a positive integer and  $[a,b]$ be a compact time interval obeying
\begin{align}\label{scaling-0}
\|v_c\|_{L_{t,x_1}^6L_{x_2}^2([a,b]\times\R\times\R)}=2^{k_0}.
\end{align}
By the rescaling argument, we set
\begin{align}\label{scaling}
\int_{a}^{b}N(t)^3dt=\eta_32^{k_0}.
\end{align}

\begin{remark}
After the rescaling, the interval $[a,b]$ may turns to $[\frac{a}{\lambda^2},\frac{b}{\lambda^2}]$. For the convenience, we still denote  $[\frac{a}{\lambda^2},\frac{b}{\lambda^2}]$ by $[a,b]$.
\end{remark}
\begin{definition}[Local constancy interval]\label{definition-small}
Divide $[a,b]\in[0,\infty)$ into several consecutive, disjoint intervals $J_k$ with $k\in\{0,1,\cdots,2^{k_0}-1\}$ such that
\begin{align*}
\|v_c\|_{L_{t,x_1}^6L_{x_2}^2(J_k\times\R\times\R)}=1.
\end{align*}
\end{definition}

\begin{remark}\label{remark-small}
For $t_1,t_2\in J_k$, we have $N(t_1)\sim N(t_2)$ and
\begin{align*}
&N(J_k)\sim\int_{J_k}N(t)^3dt\sim\inf_{t\in J_k}N(t),\\
&\sum_{J_k\subset J}N(J_k)\sim\int_{J}N(t)^3dt,
\end{align*}
where $N(J_k)=\inf_{t\in J_k}N(t)$.
\end{remark}

\begin{definition}\label{definition}
Let $k_0$ be the same as in Definition \ref{definition-small}.	For two integers $j,k\in\N$ satisfying $0\leq j<k_0$, $0\leq k<2^{k_0-j}$, let
$$G_{k}^{j}=\cup_{\alpha=k2^j}^{(k+1)2^j-1}J^{\alpha},$$
where $J^{\alpha}$ satisfies $[a,b]=\cup_{\alpha=0}^{M-1} J^{\alpha}$  with
\begin{equation}\label{eq-y5.33}
\int_{J^{\alpha}} \big( N(t)^3 + \eta_3\|u(t)\|^4_{L_{x_1}^6L_{x_2}^2(\mathbb{R}\times \R)} \big)dt=2\eta_3.
\end{equation}
For $j\geq k_0$, we set  $G_{k }^j=[a,b].$ Now suppose that $G_k^j=[t_0,t_1]$, let $\xi(G_k^j)=\xi(t_0)$ and define $\xi(J_l)$, $\xi(J^{\alpha})$ in a similar manner.
\end{definition}

\begin{remark}
There are some differences between $J_l$ and $J^{\alpha}$.  For any $J^{\alpha}$, it can at most intersect with two small intervals $J_{k_1}$ and $J_{k_2}$ while any $J^{k}$ can intersect with many small intervals $J^{\alpha_1} ,\cdots, J^{\alpha_l}$(See Figure \ref{Figure1} and Figure \ref{Figure2}). But if  $N(t)$ is a constant, then there  exist a constant $C$ such that  at most $C$ $J^{\alpha}$ intervals intersect any one $J_k$.
\end{remark}
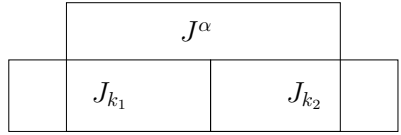
\begin{figure}[!ht]
\centering
\resizebox{0.3\textwidth}{!}{
\begin{circuitikz}
\tikzstyle{every node}=[font=\LARGE]
\draw  (1.25,10) rectangle  node {\LARGE $J_{k_1}$} (4.75,8.75);
\draw  (4.75,10) rectangle  node {\LARGE $J_{k_2}$} (8,8.75);
\draw  (2.25,11) rectangle (7,8.75);
\node [font=\LARGE] at (4.5,10.5) {$J^{\alpha}$};
\end{circuitikz}}
\caption{$J^{\alpha}$ at most intersects with two small intervals $J_{k_1}$ and $J_{k_2}$ }\label{Figure1}
\end{figure}
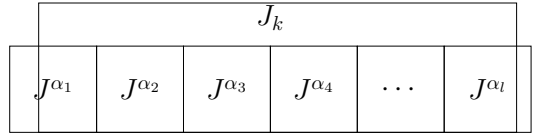
\begin{figure}[!ht]
\centering
\resizebox{0.4\textwidth}{!}{
\begin{circuitikz}
\tikzstyle{every node}=[font=\LARGE]
\draw  (-2,9.5) rectangle  node {\LARGE $J^{\alpha_1}$} (-0.5,8);
\draw  (-0.5,9.5) rectangle  node {\LARGE $J^{\alpha_2}$} (1,8);
\draw  (1,9.5) rectangle  node {\LARGE $J^{\alpha_3}$} (2.5,8);
\draw  (2.5,9.5) rectangle  node {\LARGE $J^{\alpha_4}$} (4,8);
\draw  (4,9.5) rectangle  node {\LARGE $\cdots$} (5.5,8);
\draw  (5.5,9.5) rectangle  node {\LARGE $J^{\alpha_l}$} (7,8);
\draw  (-1.5,10.25) rectangle (6.75,8);
\node [font=\LARGE] at (2.5,10) {$J_k$};
\end{circuitikz}}%
	
\caption{$J_k$ can intersect with many small intervals $J^{\alpha_1},\cdots,J^{\alpha_l}$ }\label{Figure2}
\end{figure}

\begin{remark}
By \eqref{xi(t)} and Definition \ref{definition},  for all $t \in G_k^{j}$,
\begin{align}\label{key}
|\xi(t)-\xi(G_k^{j})|\leq \int_{G_k^{j}}2^{-20}\eta_{1}^{-1/2}N(t)^3dt \leq 2^{j-19}\eta_3 \eta_{1}^{-1/2}.
\end{align}
Therefore, for all $t \in G_k^{j}$ and $i\geq j$,
$$\{\xi:2^{i-1}\leq |\xi-\xi(t)|\leq 2^{i+1} \} \subset \{\xi:2^{i-2}\leq |\xi-\xi(G_k^{j})|\leq 2^{i+2} \} \subset \{\xi:2^{i-3}\leq |\xi-\xi(t)|\leq 2^{i+3} \},$$
and
$$\{\xi:|\xi-\xi(t)|\leq 2^{i+1} \} \subset \{\xi: |\xi-\xi(G_k^{j})|\leq 2^{i+2} \} \subset \{\xi: |\xi-\xi(t)|\leq 2^{i+3} \}.$$
\end{remark}



\section{Long-time Stricharz estimates}\label{sec:long}
In this section, we devote to  prove the  new-type estimate for the minimal blow-up solution $v$ to the equation \eqref{DCR-2}, which is crucial in precluding the two special scenarios of blow-up solutions. This type estimate was first exploited by Dodson \cite{Dodson-d=1} which is called the long-time Strichartz estimate. In our case, we prove the long-time Strichartz estimate for the minimal blow-up solution to equation \eqref{DCR-5}. The proof is quite technical and complicated, so we summarize the main steps of the proof in the following diagram for convenience.

\begin{tikzcd}
\stackrel{\mbox{First multilinear estimate}}{(\mbox{cf.\hspace{1ex}Lemma \ref{mlinear1}})}\arrow[Rightarrow,dr]&&\\\
&\stackrel{\mbox{Intermediate Theorem} }{(\mbox{cf.~Theorem \ref{intermediate}})}\arrow[Rightarrow,r]&\stackrel{\mbox{Long-time Strichartz estimate}}{(\mbox{cf. Theorem \ref{longtime}})}\\
\stackrel{\mbox{Second multilinear estimate}}{(\mbox{cf.~Lemma \ref{mlinear2}})}\arrow[Rightarrow,ur]&&
\end{tikzcd}

\subsection{The preparation for the long-time Strichartz estimate}
Similar to  \cite{Dodson-d=1}, since  the endpoint Strichartz estimates fails in $d=1$, we need to use the functional space $U_\Delta^p$ and $V_\Delta^p$ which are introduced in the work of Koch-Tataru \cite{Koch-Tataru} and Hadac-Herr-Koch \cite{Hadec-Herr-Koch}. In our context, we will introduce the variant version of the standard $U_\Delta^p$ and $V_\Delta^p$ space, that is $U_\Delta^p(L_{x_1}^2L_{x_2}^2)$ and $V_\Delta^p(L_{x_1}^2L_{x_2}^2)$. Indeed, we will give the definiton of the abstract atomic space $U_\Delta^p(H;\R)$ and $V_\Delta^p(H;\R)$.   
\begin{definition}[$U_\Delta^p(H;\R)$ space] Let $1\leqslant p<\infty$ and $H$ be a complex Hilbert space. We define $U_\Delta^p(H;\R)$ as an atomic space with the atom $v^\lambda(t,x_1,x_2)$ given by  
\begin{align*}
v^\lambda(t,x_1,x_2)=\sum_{k=0}^{N}\chi_{[t_k,t_{k+1}]}(t)e^{it\partial_{x_1}^2}v_k^\lambda(x_1,x_2),\quad\sum_{k=0}^N\big\|v_k^\lambda\|_{H}^p=1.
\end{align*}
In the summation above, $N$ can be finite or infinite. If $N$ is finite, we further assume $t_0=-\infty$ and $t_{N+1}=+\infty$. Then the norm of the $U_\Delta^p (H;\R)$ can be defined in the following way
\begin{align*}\|v\|_{U_\Delta^p(H;\R)}
=\inf\big\{\sum_{\lambda}|c_\lambda|:v=\sum_{\lambda}c_\lambda v^\lambda, \hspace{1ex}v^\lambda\mbox{ is an atom}\big\},
\end{align*}
The infimum is taken over all the atom decomposition.
	
For any  interval $I\subseteq\R$, we can define the local version as
\begin{align*}
\|v\|_{U_\Delta^p(H;I)}=\|v\chi_I(t)\|_{U_\Delta^p(H;\R)}.
\end{align*} 
\end{definition}
In this paper, we set $H=L_{x_1}^2L_{x_2}^2(\R\times\R)$. 
Next, we denote the functional space $DU_\Delta^p(L_{x_1}^2L_{x_2}^2)$ which is tightly linked to the Strichartz estimate
\begin{definition}[$DU_\Delta^p(L_{x_1}^2L_{x_2}^2)$ space] Let $1\leqslant p<\infty$. We denote 
\begin{align*}
DU_\Delta^p(L_{x_1}^2L_{x_2}^2)\stackrel{\triangle}{=}\{(i\partial_t+\partial_{x_1}^2)v:v\in U_\Delta^p(L_{x_1}^2L_{x_2}^2)\},
\end{align*}
which is endowed with the norm
\begin{align*}
\|(i\partial_t+\partial_{x_1}^2)v(t,x_1,x_2)\|_{DU_\Delta^p(L_{x_1}^2L_{x_2}^2)}=\left\|\int_{0}^{t}e^{i(t-s)\partial_{x_1}^2}(i\partial_s+\partial_{x_2}^2)v(s,x_1,x_2)ds\right\|_{U_\Delta^p(L_{x_1}^2L_{x_2}^2)}.
\end{align*}
\end{definition}

Now, we define the space $V_\Delta^p(L_{x_1}^2L_{x_2}^2)$ which is the dual space of $U_\Delta^p(L_{x_1}^2L_{x_2}^2)$.
\begin{definition}[$V_\Delta^p(L_{x_1}^2L_{x_2}^2)$ space]
Let $1\leqslant p<\infty$, we define $V_\Delta^p(L_{x_1}^2L_{x_2}^2)$ as the space of the right-continous functions $v\in L_t^\infty (L_{x_1}^2L_{x_2}^2)$ endowed with the following norm
\begin{align*}
\|v\|_{U_\Delta^p(L_{x_1}^2L_{x_2}^2)}^p=\|v\|_{L_t^\infty (L_{x_1}^2L_{x_2}^2)}^p+\sum_{t_k}\sum_{k}\big\|e^{-it_{k+1}\partial_{x_1}^2}v(t_{k+1})-e^{it_k\partial_{x_1}^2}v(t_k)\big\|_{(L_{x_1}^2L_{x_2}^2)}^p<\infty.
\end{align*}
\end{definition}

\begin{remark}
For any  interval $I\subseteq\R$, we can define the local version $V_\Delta^p$ space as
\begin{align*}
\|v\|_{V_\Delta^p(H;I)}=\|v\chi_I(t)\|_{V_\Delta^p(H;\R)}.
\end{align*} 
We also have the following duality: $$(DU_\Delta^p(L_{x_1}^2L_{x_2}^2))^*=V_\Delta^{p^\prime}(L_{x_1}^2L_{x_2}^2).$$
\end{remark}

By the definition of $U_\Delta^p$ space, we give a useful lemma, which will be used in Section \ref{sec:long}.
\begin{lemma}\label{Up-dual}
For any $p>2$ and let $I$ be an interval with $t_0\in I$, then we have
\begin{align}
\left\|\int_{0}^te^{i(t-s)\partial_{x_1}^2}F(\tau)d\tau\right\|_{U_\Delta^2(I,L_{x_1}^2L_{x_2}^2)}\lesssim_p\sup_{\|G\|_{U_\Delta^p(I,L_{x_1}^2L_{x_2}^2)}=1}\int_{I}\langle G,F\rangle d\tau.	
\end{align} 
\end{lemma}

Next, we will give some useful properties of  $U_\Delta^p(L_{x_1}^2L_{x_2}^2)$ and $V_\Delta^p(L_{x_1}^2L_{x_2}^2)$ spaces.
\begin{lemma}\label{property}Suppose that $1<p<q<\infty$ and $t_0\leqslant t_1\leqslant t_2$. For $I\subset \R$,  then we have 
\begin{align}
& U_\Delta^p(I,L_{x_1}^2L_{x_2}^2)\subseteq V_\Delta^p(I,L_{x_1}^2L_{x_2}^2)\subseteq U_\Delta^q(I,L_{x_1}^2L_{x_2}^2),\label{embeddingvp}\\
&\|v\|_{U_\Delta^p(L_{x_1}^2L_{x_2}^2,[t_0,t_1])}\leqslant\|v\|_{U_\Delta^p(L_{x_1}^2L_{x_2}^2,[t_0,t_2])},\\
&\|v\|_{U_\Delta^2(L_{x_1}^2L_{x_2}^2,[t_0,t_2])}^2\leqslant\|v\|_{U_\Delta^2(L_{x_1}^2L_{x_2}^2,[t_0,t_1])}^2+\|v\|_{U_\Delta^2(L_{x_1}^2L_{x_2}^2,[t_0,t_1])}^2,\label{triangle}\\
&\|v\|_{DU_\Delta^p(L_{x_1}^2L_{x_2}^2)}\lesssim \|v_0\|_{L_{x_1}^2L_{x_2}^2} +\|(i\partial_t+\partial_{x_1}^2)v\|_{DU_\Delta^p(L_{x_1}^2L_{x_2}^2)}.
\end{align}
Furthermore, we have the following embedding 
\begin{align*}
U_\Delta^p(I,L_{x_1}^2L_{x_2}^2)\subseteq L_t^p(I,L_{x_1}^qL_{x_2}^2),\quad  L_t^{p^\prime}(I,L_{x_1}^{q^\prime}L_{x_2}^2)\subseteq DU_\Delta^2(I,L_{x_1}^2L_{x_2}^2).
\end{align*}
\end{lemma}
The proof of Lemma \ref{Up-dual} and Lemma  \ref{property} are similar to that in \cite{Dodson-d=1}, so we omit the proof.

Combining the above lemma  and the inhomogeneous Strichartz estimates, we have the following inhomogeneous Strichartz estimates in $U_\Delta^p(L_{x_1}^2L_{x_2}^2)$. 
\begin{lemma}[Inhomogeneous Strichartz estimate]
Let $I$ be the time interval which is partitioned into the several small intervals, $I=\bigcup\limits_{j=1}^{m}I_j$ with $I_j=[a_j,a_{j+1}]$. Then for $t_0\in I$, we have
\begin{align}
\left\|\int_{0}^te^{i(t-s)\partial_{x_1}^2}f(s,x_1,x_2)ds\right\|_{U_\Delta^2(L_{x_1}^2L_{x_2}^2)}\lesssim \sum_{j=1}^m\left\|\int_{I_j}e^{-is\partial_{x_1}^2}f(s)ds\right\|_{L_{x_1}^2L_{x_2}^2}+\left(\sum_{j=1}^m\|f\|_{DU_\Delta^2(I_j,L_{x_1}^2L_{x_2}^2)}^2\right)^\frac{1}{2}.\label{Inhomo}
\end{align}
\end{lemma}

Now, we present the bilinear Strichartz estimates, which is useful in the proof of long-time Strichartz estimate. 
\begin{lemma}\label{prop-bilinear}
Let $2\leq p<\infty$ and $\frac{1}{p} + \frac{1}{q} = 1$, $u_0 , v_0 \in L^2_{x_1,x_2}(\R\times\R)$. Suppose that $\mathcal{F}_{x_1} u_0 (\xi , x_2) $ is supported on $|\xi| \sim N$ and $\mathcal{F}_{x_1} v_0 (\xi , x_2)$ is supported on $|\xi| \sim M$. If $M \ll N$, we have 
\begin{align*}
\bigg\|  \big\| e^{it \partial_{x_1}^2} u_0 \big\|_{ L_{x_2}^2(\R)}  \big\| e^{it \partial_{x_1}^2} v_0 \big\|_{ L_{x_2}^2(\R) } & \bigg\| _{L^p_{t} L_{x_1}^q(I\times\R)}  \lesssim \bigg( \frac{1}{N} \bigg) ^ {\frac{1}{p}} \| u_0 \|_{L^2_{ x_1,x_2}(\R\times\R)} \| v_0 \|_{L^2_{ x_1,x_2}(\R\times\R) }.
\end{align*}
\end{lemma}

As a direct consequence, we obtain a bilinear estimate on $U^p_{\Delta}(I,L_{x_1}^2L_{x_2}^2)$ space.
\begin{lemma}\label{bilinear-Up}
Let $u , v $ be space-time functions and $I\subset\R$.  Suppose that $\mathcal{F}_{x_1} u (t,\xi , x_2) $ is supported on $|\xi| \sim N$ and $\mathcal{F}_{x_1} v (t,\xi , x_2)$ is supported on $|\xi| \sim M$. with $M \ll N$, then we have the following estimate
\begin{align*}
\bigg\|     \big\|  u \big\|_{ L_{x_2}^2(\R) }  \big\|  v \big\|_{ L_{x_2}^2(\R) }   \bigg\|_{L^p_t L^q_{x_1}(I\times\R)} \lesssim   \big(\frac{1}{N}\big) ^ {\frac{1}{p}} \|u\|_{U^p_{\Delta}(I,L_{x_1}^2L_{x_2}^2)}  \|v\|_{U^p_{\Delta}(I,L_{x_1}^2L_{x_2}^2)}
\end{align*}
with $\frac{1}{p} + \frac{1}{q} = 1$ and $p\geq2$.
\end{lemma}

Using Lemma \ref{prop-bilinear} and the Strichartz estimates, we have the following lemma.
\begin{lemma}\label{lemma-bilinear}
Let $u_0,v_0$, $N$ and $M$  be the same as in Lemma $\ref{prop-bilinear}$, the following estimates hold
\begin{align}
&\big\|   \|e^{it \partial_{x_1}^2}  u_0 \| _{L_{x_2}^2(\R)}  \|e^{it \partial_{x_1}^2}  v_0 \| _{L_{x_2}^2(\R)}  \big\|_{ L^3_{t , x_1}(I\times\R)  } \lesssim  \big( \frac{M}{N} \big) ^{\frac{1}{4}}  \|u_0\|_{L^2_{x_1} L_{x_2}^2(\R\times\R)}  \|v_0\|_{L^2_{x_1} L_{x_2}^2(\R\times\R)}, \label{fml-bilinear-L3} \\
&\big\|   \|e^{it \partial_{x_1}^2}  u_0 \| _{L_{x_2}^2(\R)}  \|e^{it \partial_{x_1}^2}  v_0 \| _{L_{x_2}^2(\R)}  \big\|_{ L^{\frac{12}{5}}_{t , x_1}(I\times\R)  }   \lesssim  \big( \frac{1}{N} \big) ^{\frac{1}{4}}  \|u_0\|_{L^2_{x_1} L_{x_2}^2(\R\times\R)}  \|v_0\|_{L^2_{x_1} L_{x_2}^2(\R\times\R)}. \label{fml-bilinear-L12}
\end{align}
\end{lemma}

\begin{proof}  
By using 	H\"older's inequality, Bernstein's inequality and Lemma $\ref{prop-bilinear}$,
\begin{align*}
&\big\|   \|e^{it \partial_{x_1}^2}  u_0 \| _{L_{x_2}^2(\R)}  \|e^{it \partial_{x_1}^2}  v_0 \| _{L_{x_2}^2(\R)}  \big\|_{ L^3_{t , x_1}(I\times\R)  } \\
\lesssim &  \big\|   \|e^{it \partial_{x_1}^2}  u_0 \| _{L_{x_2}^2(\R)}  \|e^{it \partial_{x_1}^2}  v_0 \| _{L_{x_2}^2(\R)} \big\|^{\frac{1}{2}} _{ L^2_{t , x_1}(I\times\R)  } \big\|   \|e^{it \partial_{x_1}^2}  u_0 \| _{L_{x_2}^2(\R)}  \|e^{it \partial_{x_1}^2}  v_0 \| _{L_{x_2}^2(\R)} \big\|^{\frac{1}{2}} _{ L^6_{t , x_1}(I\times\R)  } \\
\lesssim &\big\|   \|e^{it \partial_{x_1}^2}  u_0 \| _{L_{x_2}^2(\R)}  \|e^{it \partial_{x_1}^2}  v_0 \| _{L_{x_2}^2(\R)} \big\|^{\frac{1}{2}} _{ L^2_{t , x_1}(I\times\R)  } \big\|   \|e^{it \partial_{x_1}^2}  u_0 \| _{L_{x_2}^2(\R)}  \big\| ^{\frac{1}{2}} _{ L^6_{ t } L^{6} _{x_1}(I\times\R)  }    \big\|   \|e^{it \partial_{x_1}^2}  v_0 \| _{L_{x_2}^2(\R)}  \big\|^{\frac{1}{2}} _{ L^{\infty}_{ t } L^{\infty} _{x_1}(I\times\R)  }\\
\lesssim  &  \big( \frac{M}{N} \big) ^{\frac{1}{4}}  \|u_0\|_{L^2_{x_1} L_{x_2}^2(\R\times\R)}  \|v_0\|_{L^2_{x_1} L_{x_2}^2(\R\times\R)}.
\end{align*}
Since $(6,6)$ is an admissible pair,  we have
\begin{align*}
&\big\|   \|e^{it \partial_{x_1}^2}  u_0 \| _{L_{x_2}^2(\R)}  \|e^{it \partial_{x_1}^2}  v_0 \| _{L_{x_2}^2(\R)}  \big\|_{ L^{\frac{12}{5}}_{t , x_1}(I\times\R)  } \\
\lesssim &  \big\|   \|e^{it \partial_{x_1}^2}  u_0 \| _{L_{x_2}^2(\R)}  \|e^{it \partial_{x_1}^2}  v_0 \| _{L_{x_2}^2(\R)}  \big\|^{\frac{1}{2}} _{ L^2_{t , x_1}(I\times\R)  }   \big\|   \|e^{it \partial_{x_1}^2}  u_0 \| _{L_{x_2}^2(\R)}  \big\| ^\frac{1}{2} _{ L^6_{ t , x_1 }(I\times\R)  }    \big\|   \|e^{it \partial_{x_1}^2}  v_0 \| _{L_{x_2}^2(\R)}  \big\| ^\frac{1}{2} _{  L^{6} _{t , x_1}(I\times\R)  }\\
\lesssim  &  \big( \frac{1}{N} \big) ^{\frac{1}{4}}  \|u_0\|_{L^2_{x_1} L_{x_2}^2(\R\times\R)}  \|v_0\|_{L^2_{x_1} L_{x_2}^2(\R\times\R)}.
\end{align*}
\end{proof}

By Lemma \ref{prop-bilinear} and Lemma \ref{lemma-bilinear}, we have the following bilinear estimates in $U_\Delta^p(I,L_{x_1}^2L_{x_2}^2)$ spaces.
\begin{corollary}
Let $u_0,v_0$, $N,M$  be the same as in Lemma $\ref{bilinear-Up}$, the following estimates hold
\begin{align}
&\big\|   \| u \| _{L_{x_2}^2(\R)}  \|  v \| _{L_{x_2}^2(\R)}  \big\|_{ L^2_{t , x_1}(I\times\R)  } \lesssim  \big( \frac{1}{N} \big) ^{\frac{1}{2}}  \|u\|_{U^p_{\Delta}(I,L_{x_1}^2L_{x_2}^2)}  \|v\|_{U^p_{\Delta}(I,L_{x_1}^2L_{x_2}^2)}, \label{bilinearl2} \\
&\big\|   \| u \| _{L_{x_2}^2(\R)}  \| v \| _{L_{x_2}^2(\R)}  \big\|_{ L^{\frac{12}{5}}_{t , x_1}(I\times\R)  }   \lesssim  \big( \frac{1}{N} \big) ^{\frac{1}{4}}  \|u_0\|_{U^p_{\Delta}(I,L_{x_1}^2L_{x_2}^2)}  \|v_0\|_{U^p_{\Delta}(I,L_{x_1}^2L_{x_2}^2)}, \label{bilinearl12}\\
&\big\|   \| u \| _{L_{x_2}^2(\R)}  \|  v \| _{L_{x_2}^2(\R)}  \big\|_{ L^3_{t , x_1}(I\times\R)  } \lesssim  \big( \frac{M}{N} \big) ^{\frac{1}{4}}  \|u\|_{U^p_{\Delta}(I,L_{x_1}^2L_{x_2}^2)}  \|v\|_{U^p_{\Delta}(I,L_{x_1}^2L_{x_2}^2)}. \label{bilinearl3} 
\end{align}
\end{corollary}

Before presenting the long-time Strichartz estimate, we need to introduce some notations and definitions.
\begin{definition}[Galilean Littlewood-Paley projection] Let  $j>0$ be an integer and $P_j^{x_1}$ be the Littlewood-Paley projection defined as in Section \ref{sec:pre}. Then we denote the Galilean Littlewood-Paley projection by
\begin{align*}
P_{\xi_0,\leq j}^{x_1}f=e^{ix_1\cdot\xi_0}P_{\leq j}^{x_1}(e^{-ix_1\cdot\xi_0}f).
\end{align*}
We can also define $P_{\xi_0,j}^{x_1}$ $P_{\xi_0,>j}^{x_1}$ in a similar way.
	
For $1\leq p\leq\infty$ and $1\leq q\leq\infty$, define the norm
\begin{align*}
\big\|\|P_{\xi(t),j}^{x_1}f\|_{L_{x_2}^2(\R)}\big\|_{L_t^pL_{x_1}^q(\R\times\R)}=\big\|\|P_{\xi(t),j}^{x_1}\|_{L_{x_1}^qL_{x_2}^2(\R\times\R)}\big\|_{L_t^p(\R)}.
\end{align*}
\end{definition}

For the almost-periodic solution $u$ satisfying the conditions $\eqref{almost-compact-1}$-$\eqref{almost-compact-2}$, we now construct the Galilean norm:
\begin{definition}[Galilean norm]
For $0\leq i\leq j\leq k_0$, if $G_{k}^j\subset[a,b]$ is part of the partition of $[a,b]$ described in Definition \ref{definition}, we write
\begin{align*}
\big\|P_{\xi(t),i}^{x_1}u\big\|_{U_\Delta^2(G_k^j,L_{x_1}^2L_{x_2}^2)}\stackrel{\triangle}{=}\sum\limits_{G_{\alpha}^i\subset G_k^j}\big\|P_{\xi(G_{\alpha}^i),i-2\leq\cdot\leq i+2}^{x_1}u\big\|_{U_\Delta^2(G_{\alpha}^i,L_{x_1}^2L_{x_2}^2)}^2.
\end{align*}
For $i\geq j$, we define
\begin{align*}
\big\|P_{\xi(t),i}^{x_1}u\big\|_{U_\Delta^2(G_k^j,L_{x_1}^2L_{x_2}^2)}^2=\big\|P_{\xi(G_k^j),i}^{x_1}u\big\|_{U_\Delta^2(G_k^j,L_{x_1}^2L_{x_2}^2)}^2.
\end{align*}
\end{definition}

It is obvious that if $(p,q)$ is an admissible pair, then we have the following Strichartz estimate in $U_\Delta^p$ space:
\begin{align}\label{UpStrichartz1}
\|u\|_{L_t^pL_{x_1}^qL_{x_2}^2(I\times\R\times\R)}\lesssim \|u\|_{U_\Delta^p(I,L_{x_1}^2L_{x_2}^2)}\lesssim\|u\|_{U_\Delta^2(I,L_{x_1}^2L_{x_2}^2)},
\end{align}
where $I\subset \R$ is a time interval.

Now, we can give the definition of the long-time Strichartz norm.
\begin{definition}[Long-time Strichartz norms]Let $G_{k}^j\subset[a,b]$, we define
\begin{align*}
\|u\|_{X(G_k^j,L_{x_1}^2L_{x_2}^2)}^2\stackrel{\triangle}{=}\sum_{i\leq j}2^{i-j}\|P_{\xi(t),i}u\|_{U_\Delta^2(G_k^j,L_{x_1}^2L_{x_2}^2)}^2+\sum_{i>j}\|P_{\xi(t),i}u\|_{U_\Delta^2(G_k^j,L_{x_1}^2L_{x_2}^2)}^2.
\end{align*}
For any $0\leq l\leq k_0$, we write
\begin{align*}
\|u\|_{\widetilde{X}_l([a,b],L_{x_1}^2L_{x_2}^2)}^2=\sup_{0\leq j\leq l}\sup_{G_k^j\subset[a,b]}\|u\|_{X(G_k^j,L_{x_1}^2L_{x_2}^2)}^2.
\end{align*}
\end{definition}

\begin{figure}[htbp]
\centering
\includegraphics[height=10cm,width=10cm]{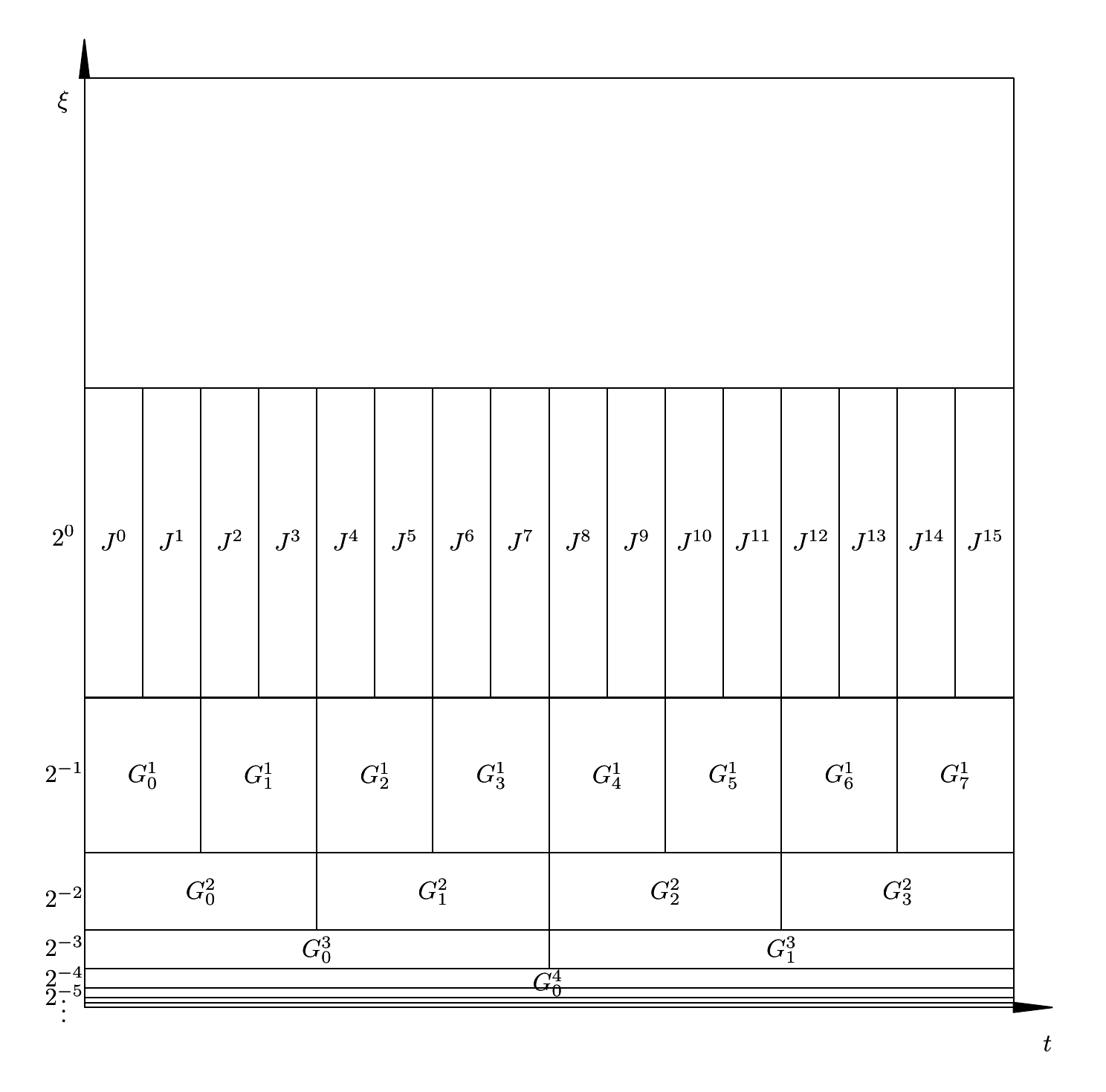}
\vspace{-1cm}
\caption{$\tilde{X}_{k_0}$ norm(cited from Yu \cite{Yu})}\label{figure}
\end{figure}

By using the Littlewood-Paley theorem in Section \ref{sec:pre}, we have the following lemma.
\begin{lemma}\label{highfreq}
Let  $j\geq 0$ be an integer and $(p,q)$ be an admissible pair,  then the following estimate holds
\begin{align*}
\big\|\|P_{\xi(t),\geq j}u\|_{L_{x_2}^2(\R)}\big\|_{L_{t}^pL_{x_1}^q(G_{k}^j\times\R)}\lesssim \|u\|_{X(G_{k}^j,L_{x_1}^2L_{x_2}^2)}.
\end{align*}
\end{lemma}

\begin{proof}
Using Proposition \ref{Littlewood} and Minkowski's inequality, we have
\begin{align*}
\big\|\|P_{\xi(t),\geq j}u\|_{L_{x_2}^2(\R)}\big\|_{L_{t}^pL_{x_1}^q(G_{k}^j\times\R)}&\sim\Big\|\big\|\big(\sum_{l\geq j}\big|P_{\xi(t),l}u\big|^2\big)^\frac12\big\|_{L_{x_2}^2(\R)}\Big\|_{L_{t}^pL_{x_1}^q(G_k^j\times\R)}\\
&\lesssim\Big(\sum_{l\geq j}\Big\|P_{\xi(t),l}u\Big\|^2_{L_t^pL_{x_1}^qL_{x_2}^2(G_k^j\times\R\times\R)}\Big)^\frac12\lesssim \|u\|_{X(G_k^j,L_{x_1}^2L_{x_2}^2)}.
\end{align*} 
Hence, we complete the proof.
\end{proof}
Similarly, for  $i<j$, we have
\begin{align}\label{frequency-property}
\big\|\|P_{\xi(t),i}\|_{L_{x_2}^2(\R)}\big\|_{L_t^pL_{x_1}^q(G_{k}^j\times\R)}\lesssim 2^{\frac{j-i}{p}}\|u\|_{X(G_{k}^j,L_{x_1}^2L_{x_2}^2)}.
\end{align}

Now we state our main theorem of this section, which is called the long-time Strichartz estimate.

\begin{theorem}[Long-time Strichartz estimate]\label{longtimestrichartz}
If $u$ is an almost periodic solution to \eqref{DCR} with the maximal lifespan $I\subset[0,\infty)$, then for any positive $k_0$, $\eta_1$, $\eta_2$, $\eta_3$ satisfying $\eqref{xi(t)}$-$\eqref{scaling}$ and $\eta_3<\eta_2^{100}$, we have
\begin{align}\label{longtime}
\|u\|_{\widetilde{X}_{k_0}([a, b],L_{x_1}^2L_{x_2}^2)}\lesssim 1,
\end{align}
and the implicit constant in \eqref{longtime} does not depend on $k_0$ or $\eta_1,\eta_2,\eta_3$.
\end{theorem}

\subsection{Two multilinear estimates}
We start with the following two multilinear estimates,  which are  keys to the proof of Theorem \ref{longtimestrichartz}.   

\begin{lemma}[First multilinear estimate]\label{mlinear1}
For any $i\geq 10$, $l\geq i-5$, $m\geq i-2$ $G_{\alpha}^i\subset G_k^j$ and $a^i_{\alpha} \in G^i_{\alpha}$, if  $N(G_{\alpha}^i)\leq \eta_3^{\frac{1}{2}}2^{i-5}$,  then we have the following multilinear estimate:
\begin{align}\label{key1}
&\bigg\|\sum\limits_{j\in \mathbb{N}}\sum_{\substack{j_1-j_2+j_3-j_4+j_5=j\\j_1,j_2,j_3,j_4,j_5\in\Bbb{N}}}\int_{a_{\alpha}^i}^te^{i(t-\tau)\partial_{x_1}^2}P^{x_1}_{\xi(G^i_{\alpha}),m}\Pi_j\big(P_{\xi(G^i_{\alpha}),l}\Pi_{j_1}uP^{x_1}_{\xi(\tau),\geq i-10}\overline{\Pi_{j_2}u}\Pi_{j_3}u\overline{\Pi_{j_4}u}\Pi_{j_5}u\big)d \tau\bigg\|_{U_\Delta^2(G_{\alpha}^i,L_{x_1}^2L_{x_2}^2)}\notag\\
\lesssim& \eta_2^{\frac{4}{3}}2^{\frac{m-l}{4}}\|P^{x_1}_{\xi(G^i_{\alpha}),l}u\|_{{U_\Delta^2(G_{\alpha}^i,L_{x_1}^2L_{x_2}^2)}}\|u\|_{X(G_{\alpha}^i,L_{x_1}^2L_{x_2}^2)}^{\frac{8}{3}}\notag\\
&+2^{\frac{i-l}{2}}2^{\frac{i-m}{4}}\|P^{x_1}_{\xi(G^i_{\alpha}),l}u\|_{U_\Delta^2(G_{\alpha}^i,L_{x_1}^2L_{x_2}^2)}\big(\eta_2^{\frac{10}{3}}\|u\|_{X(G_{\alpha}^i,L_{x_1}^2L_{x_2}^2)}^{\frac{2}{3}}+\eta_2^{\frac{11}{6}}\|u\|_{X(G_{\alpha}^i,L_{x_1}^2L_{x_2}^2)}^{\frac{13}{6}}\big).
\end{align}
Moreover, if the operator $P_{\xi(G_{\alpha}^i),l}^{x_1}$ and $P_{\xi(\tau),\geq i-10}^{x_1}$ fall on other $\Pi_{j_k}u$ and $\Pi_{j_l}u$ for $k,l\in\{1,2,3,4,5\}$, then  similar estimate also holds.
\end{lemma}

\begin{proof}
For simplicity, we just prove $\eqref{key1}$. 	By Lemma \ref{Up-dual}, we only need  to prove that for any $\|f\|_{\upppi}=1$ and  $(\mathcal{F}_{x_1}f)(t,\xi,x_2)$ is supported on $|\xi-\xi(G^i_{\alpha})|\sim 2^m$, we have 
\begin{align}
&\int_{G_{\alpha}^i}\langle f, \sum\limits_{j\in \mathbb{N}}\Pi_j\sum_{\substack{j_1-j_2+j_3-j_4+j_5=j\\j_1,j_2,j_3,j_4,j_5\in\Bbb{N}}} P^{x_1}_{\xi(G^i_{\alpha}),l}\Pi_{j_1}uP^{x_1}_{\xi(\tau),\geq i-10}\overline{\Pi_{j_2}u}\Pi_{j_3}u\Pi_{j_4}u\Pi_{j_5}u\rangle_{L_{x_1}^2L_{x_2}^2}d \tau\notag\\
\lesssim& \eta_2^{\frac{4}{3}}2^{\frac{m-l}{4}}\|P^{x_1}_{\xi(G^i_{\alpha}),l}u\|_{{U_\Delta^2(G_{\alpha}^i,L_{x_1}^2L_{x_2}^2)}}\|u\|_{X(G_{\alpha}^i,L_{x_1}^2L_{x_2}^2)}^{\frac{8}{3}}\notag\\
&+2^{\frac{i-l}{2}}2^{\frac{i-m}{4}}\|P^{x_1}_{\xi(G^i_{\alpha}),l}u\|_{U_\Delta^2(G_{\alpha}^i,L_{x_1}^2L_{x_2}^2)}\big(\eta_2^{\frac{10}{3}}\|u\|_{X(G_{\alpha}^i,L_{x_1}^2L_{x_2}^2)}^{\frac{2}{3}}+\eta_2^{\frac{11}{6}}\|u\|_{X(G_{\alpha}^i,L_{x_1}^2L_{x_2}^2)}^{\frac{13}{6}}\big).
\end{align}
we next decompose $\Pi_{j_k}u=P^{x_1}_{\xi(\tau),\leq i-10}\Pi_{j_k}u+P^{x_1}_{\xi(\tau),\geq i-10}\Pi_{j_k}u$, $k\in \{2,3,4,5\}$ and $\Pi_ju=P^{x_1}_{\xi(\tau), \leq i-10}\Pi_{j_k}u+P^{x_1}_{\xi(\tau), \geq i-10}\Pi_{j}u$, then by H\"older's inequality
\begin{align}
&\int_{G_{\alpha}^i}\big\langle f, \sum\limits_{j\in \mathbb{N}}\Pi_j\sum_{\substack{j_1-j_2+j_3-j_4+j_5=j\\j_1,j_2,j_3,j_4,j_5\in\Bbb{N}}} P^{x_1}_{\xi(G^i_{\alpha}),l}\Pi_{j_1}uP^{x_1}_{\xi(\tau),\geq i-10}\overline{\Pi_{j_2}u}\Pi_{j_3}u\Pi_{j_4}u\Pi_{j_5}u\big\rangle_{L_{x_1}^2L_{x_2}^2(\R\times\R)}d \tau\notag\\
\lesssim&\big(\int_{G^i_{\alpha}}\int_{\mathbb{R}}\|f(t,x_1)\|^2_{L_{x_2}^2(\R)}\|P^{x_1}_{\xi(G^i_{\alpha}),l}u(t,x_1)\|^2_{L_{x_2}^2(\R)}\|P_{\xi(\tau),\geq i-10}u\|^2_{L_{x_2}^2(\R)}\|u(t,x_1)\|^6_{L_{x_2}^2(\R)} dx_1 d\tau\big)^{\frac{1}{2}}\notag\\
\lesssim& \big\|P^{x_1}_{\xi(t),\geq i-10}u\big\|^4_{L_{t,x_1}^{6}L_{x_2}^2(G_{\alpha}^i\times \mathbb{R}\times \mathbb{R})}\big\|{\|P_{\xi(G^i_{\alpha}),l}^{x_1}u}\|_{{L_{x_2}^2(\mathbb{R})}}\|f\|_{{L_{x_2}^2(\mathbb{R})}}\big\|_{L_{t,x_1}^{3}(G_{\alpha}^i\times \mathbb{R})}\label{est1}\\
&\hspace{3ex}+\big\|P^{x_1}_{\xi(t),\geq i-10}u\big\|_{L_{t,x_1}^{6}L_{x_2}^2(G_{\alpha}^i\times \mathbb{R}\times \mathbb{R})}\big\|\|P_{\xi(G^i_{\alpha}),l}^{x_1}u\|_{L_{x_2}^2(\mathbb{R})}\|P^{x_1}_{\xi(t),\leq i-10}u\|^2_{{L_{x_2}^2(\mathbb{R})}}\big\|_{L_{t,x_1}^{2}(G_{\alpha}^i\times \mathbb{R})}\label{est2}\\
&\hspace{3ex}\times\big\|\|f\|_{{L_{x_2}^2(\mathbb{R})}} \|P^{x_1}_{\xi(t),\leq i-10}u\|_{{L_{x_2}^2(\mathbb{R})}}\big\|_{L_{t,x_1}^{3}(G_{\alpha}^i\times \mathbb{R})}.\label{est3}
\end{align}
We shall estimate the terms appeared in \eqref{est1}-\eqref{est3} separately. 
	
{\emph{\textbf {Estimate of  $\|P^{x_1}_{\xi(t),\geq i-10}u\|_{L_{t,x_1}^{6}L_{x_2}^2(G_{\alpha}^i\times \mathbb{R}\times \mathbb{R})}$}} }:
Using $L^p$ interpolation,   the compactness condition \eqref{almost-compact-2}, and Lemma \ref{highfreq},  one can directly verify that
\begin{align}\label{esttt1}
&\hspace{3ex}\|P^{x_1}_{\xi(t),\geq i-10}u\|^3_{L_{t,x_1}^{6}L_{x_2}^2(G_\alpha^i\times\R\times\R)}\\
&\lesssim \|P^{x_1}_{\xi(t),\geq i-10}u\|_{L_{t}^{\infty}L^2_{x_1}L_{x_2}^2(G_\alpha^i\times\R\times\R)}\|P^{x_1}_{\xi(t),\geq i-10}u\|^2_{L_{t}^{\infty}L^2_{x_1}L_{x_2}^2(G_\alpha^i\times\R\times\R)}\notag\\
&\lesssim \eta_2\|u\|^2_{X(G_{\alpha}^i,L_{x_1}^2L_{x_2}^2)}.
\end{align}
	
{\emph{\textbf {Estimate of  $\big\|{\|P_{\xi(G^i_{\alpha}),l}^{x_1}u}\|_{{L_{x_2}^2(\mathbb{R})}}\|f\|_{{L_{x_2}^2(\mathbb{R})}}\big\|_{L_{t,x_1}^{3}(G_{\alpha}^i\times \mathbb{R})}$}} :}
For this terms, we directly use the bilinear estimate \eqref{bilinearl3} and the embedding ${U_\Delta^2(G_{\alpha}^i,L_{x_1}^2L_{x_2}^2)}\hookrightarrow {U_\Delta^3(G_{\alpha}^i,L_{x_1}^2L_{x_2}^2)}$ to deduce that
\begin{align}\label{esttt2}
\big\|{\|P_{\xi(G^i_{\alpha}),l}^{x_1}u}\|_{{L_{x_2}^2(\mathbb{R})}}\|f\|_{{L_{x_2}^2(\mathbb{R})}}\big\|_{L_{t,x_1}^{3}(G_{\alpha}^i\times \mathbb{R})}
&\lesssim 2^{-\frac{|m-l|}{4}}\|f\|_{U_\Delta^3(G_{\alpha}^i,L_{x_1}^2L_{x_2}^2)}\|P^{x_1}_{\xi(G^i_{\alpha}),l}u\|_{U_\Delta^3(G_{\alpha}^i,L_{x_1}^2L_{x_2}^2)}\notag\\
&\lesssim 2^{-\frac{|m-l|}{4}}\|f\|_{U_\Delta^3(G_{\alpha}^i,L_{x_1}^2L_{x_2}^2)}\|P^{x_1}_{\xi(G^i_{\alpha}),l}u\|_{U_\Delta^2(G_{\alpha}^i,L_{x_1}^2L_{x_2}^2)}\notag\\
&\lesssim 2^{-\frac{|m-l|}{4}}\|P^{x_1}_{\xi(G^i_{\alpha}),l}u\|_{U_\Delta^2(G_{\alpha}^i,L_{x_1}^2L_{x_2}^2)}.
\end{align}
	
{\emph{\textbf {Estimate of  $\big\|\|P_{\xi(G^i_{\alpha}),l}^{x_1}u\|_{L_{x_2}^2(\mathbb{R})}\|P^{x_1}_{\xi(t),\leq i-10}u\|^2_{L_{x_2}^2(\mathbb{R})}\big\|_{L_{t,x_1}^{2}(G_{\alpha}^i\times \mathbb{R})}$}} :}
We further split $P^{x_1}_{\xi(t),\leq i-10}u$ as
\begin{align*}
P^{x_1}_{\xi(t),\leq i-10}u=P^{x_1}_{|\xi-\xi(t)|\leq \eta_3^{-1/4}2^{-20}N(t)}u+P^{x_1}_{\eta_3^{-1/4}2^{-20}N(t)\leq|\xi-\xi(t)|\leq i-10}u,
\end{align*} 
then we have
\begin{align}
&\big\|\|P_{\xi(G^i_{\alpha}),l}^{x_1}u\|_{L_{x_2}^2(\mathbb{R})}\|P^{x_1}_{\xi(t),\leq i-10}u\|^2_{L_{x_2}^2(\mathbb{R})}\big\|_{L_{t,x_1}^{2}(G_{\alpha}^i\times \mathbb{R})}\notag\\
&\lesssim\big\|\|P_{\xi(G^i_{\alpha}),l}^{x_1}u\|_{L_{x_2}^2(\mathbb{R})}\|P^{x_1}_{|\xi-\xi(t)|\leq \eta_3^{-1/4}2^{-20}N(t)}u\|^2_{L_{x_2}^2(\mathbb{R})}\big\|_{L_{t,x_1}^{2}(G_{\alpha}^i\times \mathbb{R})}\label{estttt1}\\
&+\big\|\|P_{\xi(G^i_{\alpha}),l}^{x_1}u\|_{L_{x_2}^2(\mathbb{R})}\|P^{x_1}_{\eta_3^{-1/4}2^{-20}N(t)\leq|\xi-\xi(t)|\leq i-10}u\|^2_{L_{x_2}^2(\mathbb{R})}\big\|_{L_{t,x_1}^{2}(G_{\alpha}^i\times \mathbb{R})}\label{estttt2}.
\end{align} 
	
For  the term \eqref{estttt1}, by the definition of $G^i_{\alpha}$, $2^l\gg \eta_3^{-1/4}2^{-20}N(t)$ for all $t\in G^i_{\alpha}$, then by the bilinear estimate \eqref{bilinearl2}, Remark \ref{remark-small}, H\"older's inequality and Proposition \ref{Bernstein}  
\begin{align}\label{esttt31}
&\hspace{3ex}\big\|\|P_{\xi(G^i_{\alpha}),l}^{x_1}u\|_{L_{x_2}^2(\mathbb{R})}\|P^{x_1}_{\xi(t),\leq i-10}u_{|\xi-\xi(t)|\leq \eta_3^{-\frac{1}{4}}2^{-20}N(t)}\|^2_{L_{x_2}^2(\mathbb{R})}\big\|_{L_{t,x_1}^{2}(G_{\alpha}^i\times \mathbb{R})}\notag\\
&\lesssim \big(\sum\limits_{ J_k\cap G^i_{\alpha}} 2^{-l}\|P_{\xi(G^i_{\alpha}),l}^{x_1}u\|^2_{U_\Delta^2(G_{\alpha}^i,L_{x_1}^2L_{x_2}^2)}\|P^{x_1}_{|\xi-\xi(t)|\leq \eta_3^{-1/4}2^{-20}N(t)}u\|^2_{U_\Delta^2(J_k,L_{x_1}^2L_{x_2}^2)}\notag\\
&\hspace{7ex}\times \|P^{x_1}_{|\xi-\xi(t)|\leq \eta_3^{-1/4}2^{-20}N(t)}u\|^2_{L_{t,x_1}^{\infty}L_{x_2}^2(G_{\alpha}^i\times \mathbb{R}\times \mathbb{R})}\big)^{1/2}\notag\\
&\lesssim  \big(\sum\limits_{ J_k\cap G^i_{\alpha}}N(J_k)\eta_3^{-1/4}2^{-20}2^{-l}\big)^{\frac{1}{2}}\|P_{\xi(G^i_{\alpha}),l}^{x_1}u\|_{U_\Delta^2(G_{\alpha}^i,L_{x_1}^2L_{x_2}^2)}\notag\\
&\lesssim \eta_3^{3/8}2^{(i-l)/2}\|P_{\xi(G^i_{\alpha}),l}^{x_1}u\|_{U_\Delta^2(G_{\alpha}^i,L_{x_1}^2L_{x_2}^2)}.
\end{align}
	
For term \eqref{estttt2},  H\"older's inequality gives that
\begin{align}\label{qqq}
&\big\|\|P_{\xi(G^i_{\alpha}),l}^{x_1}u\|_{L_{x_2}^2(\mathbb{R})}\|P^{x_1}_{\eta_3^{-1/4}2^{-20}N(t)\leq|\xi-\xi(t)|\leq i-10}u\|^2_{L_{x_2}^2(\mathbb{R})}\big\|^2_{L_{t,x_1}^{2}(G_{\alpha}^i\times \mathbb{R})}\\
\lesssim&\sum\limits_{\eta_3^{-1/4}2^{-20}N(t)\leq m_3\leq m_2\leq i-10}\big\|\|P_{\xi(G^i_{\alpha}),l}^{x_1}u\|_{L_{x_2}^2(\mathbb{R})}\|P^{x_1}_{\xi(t),m_2}u\|_{L_{x_2}^2(\mathbb{R})}\big\|_{L_{t,x_1}^{2}(G_{\alpha}^i\times \mathbb{R})}\notag\\\notag
&\hspace{12ex}\times \big\|\|P_{\xi(G^i_{\alpha}),l}^{x_1}u\|_{L_{x_2}^2(\mathbb{R})}\|P^{x_1}_{\xi(t),m_3}u\|_{L_{x_2}^2(\mathbb{R})}\big\|_{L_{t,x_1}^{2}(G_{\alpha}^i\times \mathbb{R})}\big\|\|P^{x_1}_{\eta_3^{-1/4}2^{-20}N(t)\leq|\xi-\xi(t)|\leq m_3 }u\|_{L_{x_2}^2(\mathbb{R})}\big\|_{L_{t,x_1}^{\infty}}^2.
\end{align}
Then by Sobolev embedding theorem, conservation of mass, Proposition \ref{Bernstein}, the compactness condition $\eqref{almost-compact-2}$, bilinear estimate \eqref{bilinearl2} and the definition of  $X(G_{\alpha}^i,L^2_{x_1}L_{x_2}^2)$, we continue:
\begin{align}\label{esttt32}
&\lesssim\sum\limits_{\eta_3^{-1/4}2^{-20}N(t)\leq m_3\leq m_2\leq i-10}\eta_2^22^{m_3-l}\|P^{x_1}_{\xi(G^i_{\alpha}),l}u\|^2_{U^2_{\Delta}(G_{\alpha}^i,L^2_{x_1}L_{x_2}^2)}\|P^{x_1}_{\xi(t),m_2}u\|^2_{U^2_{\Delta}(G_{\alpha}^i,L^2_{x_1}L_{x_2}^2)}\notag\\
&\hspace{8ex}\times\|P^{x_1}_{\xi(t),m_3}u\|^2_{U^2_{\Delta}(G_{\alpha}^i,L^2_{x_1}L_{x_2}^2)}\notag\\
&\lesssim \eta_2^22^{i-l}\|P^{x_1}_{\xi(G^i_{\alpha}),l}u\|^2_{U^2_{\Delta}(G_{\alpha}^i,L^2_{x_1}L_{x_2}^2)}\|u\|^2_{X(G_{\alpha}^i,L^2_{x_1}L_{x_2}^2)}.
\end{align}
Combining \eqref{esttt31} and \eqref{esttt32}, we have 
\begin{align}\label{esttt3}
&\big\|\|P_{\xi(G^i_{\alpha}),l}^{x_1}u\|_{L_{x_2}^2(\mathbb{R})}\|P^{x_1}_{\xi(t),\leq i-10}u\|^2_{L_{x_2}^2(\mathbb{R})}\big\|_{L_{t,x_1}^{2}(G_{\alpha}^i\times \mathbb{R})}\\\nonumber
\lesssim& 2^{\frac{i-l}{2}}\|P^{x_1}_{G^i_{\alpha},l}u\|_{U^2_{\Delta}(G_{\alpha}^i,L^2_{x_1}L_{x_2}^2)}(\eta^{\frac{3}{8}}_3+\eta_2\|u\|_{X(G_{\alpha}^i,L^2_{x_1}L_{x_2}^2)}).
\end{align}
	
{\emph{\textbf {Estimate of  $\big\|\|f\|_{L_{x_2}^2(\mathbb{R})}\|P^{x_1}_{\xi(t),\leq i-10}u\|_{L_{x_2}^2(\mathbb{R})}\big\|_{L_{t,x_1}^{3}(G_{\alpha}^i\times \mathbb{R})}$}} :}
By H\"older's inequality and \eqref{UpStrichartz1}, 
\begin{align}\label{esttt41}
&\big\|\|f\|_{L_{x_2}^2(\mathbb{R})}\|P^{x_1}_{\xi(t),\leq i-10}u\|_{L_{x_2}^2(\mathbb{R})}\big\|_{L_{t,x_1}^{3}(G_{\alpha}^i\times \mathbb{R})}\notag\\
\lesssim& \big\|\|f\|_{L_{x_2}^2(\mathbb{R})}\|P^{x_1}_{\xi(t),\leq i-10}u\|^2_{L_{x_2}^2(\mathbb{R})}\big\|^{\frac{1}{2}}_{L_{t,x_1}^{2}(G_{\alpha}^i\times \mathbb{R})}\big\|\|f\|_{L_{x_2}^2(\mathbb{R})}\big\|^{\frac{1}{2}}_{L_{t,x_1}^{6}(G_{\alpha}^i\times \mathbb{R})}\notag\\
& \lesssim\|f\|^{\frac{1}{2}}_{U^2_{\Delta}(G_{\alpha}^i,L^2_{x_1}L_{x_2}^2)} \big\|\|f\|_{L_{x_2}^2(\mathbb{R})}\|P^{x_1}_{\xi(t),\leq i-10}u\|^2_{L_{x_2}^2(\mathbb{R})}\big\|^{\frac{1}{2}}_{L_{t,x_1}^{2}(G_{\alpha}^i\times \mathbb{R})}.
\end{align}
Since  $2^m\gg \eta_3^{-1/4}2^{-20}N(t)$ for all $t\in G^i_{\alpha}$,  using the same argument in the proof of \eqref{esttt3}, we  also have
\begin{align}\label{esttt42}
\big\|\|f\|_{L_{x_2}^2(\mathbb{R})}\|P^{x_1}_{\xi(t),\leq i-10}u\|^2_{L_{x_2}^2(\mathbb{R})}\big\|_{L_{t,x_1}^{2}(G_{\alpha}^i\times \mathbb{R})}
\lesssim 2^{\frac{i-m}{2}}\|f\|_{U^2_{\Delta}(G_{\alpha}^i,L^2_{x_1}L_{x_2}^2)}(\eta^{\frac{3}{8}}_3+\eta_2\|u\|_{X(G_{\alpha}^i,L^2_{x_1}L_{x_2}^2)}),
\end{align}
 thus by \eqref{esttt41} , \eqref{esttt42} and the embedding ${U_\Delta^2(G_{\alpha}^i,L_{x_1}^2L_{x_2}^2)}\hookrightarrow {U_\Delta^3(G_{\alpha}^i,L_{x_1}^2L_{x_2}^2)}, $
\begin{align}\label{esttt4}
&\big\|\|f\|_{L_{x_2}^2(\mathbb{R})}\|P^{x_1}_{\xi(t),\leq i-10}u\|_{L_{x_2}^2(\mathbb{R})}\big\|_{L_{t,x_1}^{3}(G_{\alpha}^i\times \mathbb{R})}\notag\\
 \lesssim&2^{\frac{i-m}{4}}\|f\|_{U^2_{\Delta}(G_{\alpha}^i,L^2_{x_1}L_{x_2}^2)} (\eta^{\frac{3}{8}}_3+\eta_2\|u\|_{X(G_{\alpha}^i,L^2_{x_1}L_{x_2}^2)})\notag\\
& \lesssim2^{\frac{i-m}{4}} (\eta^{\frac{3}{8}}_3+\eta_2\|u\|_{X(G_{\alpha}^i,L^2_{x_1}L_{x_2}^2)}).
\end{align}
	
Finally, inserting \eqref{esttt1}, \eqref{esttt2}, \eqref{esttt3}, \eqref{esttt4} into \eqref{est1}-\eqref{est3}, then 
	
\begin{align}\label{keyestimate1}
&\int_{G_{\alpha}^i}\big\langle f, \sum\limits_{j\in \mathbb{N}}\Pi_j\sum_{\substack{j_1-j_2+j_3-j_4+j_5=j\\j_1,j_2,j_3,j_4,j_5\in\Bbb{N}}} P^{x_1}_{\xi(G^i_{\alpha}),l}\Pi_{j_1}uP^{x_1}_{\xi(\tau),i\geq i-10}\overline{\Pi_{j_2}u}\Pi_{j_3}u\Pi_{j_4}\Pi_{j_4}u\big\rangle_{L_{x_1}^2L_{x_2}^2(\R\times\R)}d \tau\notag\\
\lesssim& \eta_2^{\frac{4}{3}}2^{\frac{m-l}{4}}\|P^{x_1}_{\xi(G^i_{\alpha}),l}u\|_{{U_\Delta^2(G_{\alpha}^i,L_{x_1}^2L_{x_2}^2)}}\|u\|_{X(G_{\alpha}^i,L_{x_1}^2L_{x_2}^2)}^{\frac{8}{3}}\notag\\
&+2^{\frac{i-l}{2}}2^{\frac{i-m}{4}}\|P^{x_1}_{\xi(G^i_{\alpha}),l}u\|_{U_\Delta^2(G_{\alpha}^i,L_{x_1}^2L_{x_2}^2)}\big(\eta_2^{\frac{10}{3}}\|u\|_{X(G_{\alpha}^i,L_{x_1}^2L_{x_2}^2)}^{\frac{2}{3}}+\eta_2^{\frac{11}{6}}\|u\|_{X(G_{\alpha}^i,L_{x_1}^2L_{x_2}^2)}^{\frac{13}{6}}\big),
\end{align}
here we actually used $0<\eta_3\ll \eta_2.$ Thus we finish the proof.
\end{proof}

\begin{lemma}[Second multilinear estimate]\label{mlinear2}	
For any $i\geq 10$, $l\geq i-5$ , $G_{\alpha}^i\subset G_k^j$ and $a^i_{\alpha} \in G^i_{\alpha}$, if  $N(G_{\alpha}^i)\leq \eta_3^{\frac{1}{2}}2^{i-5}$,  then we have the following multilinear estimate:
\begin{align}\label{key2}
\bigg\|&\sum\limits_{j\in \mathbb{N}}\sum_{\substack{j_1-j_2+j_3-j_4+j_5=j\\j_1,j_2,j_3,j_4,j_5\in\Bbb{N}}}\int_{a_{\alpha}^i}^te^{i(t-\tau)\partial_{x_1}^2}P^{x_1}_{\xi(G^i_{\alpha}),m}\Pi_j\big(P_{\xi(G^i_{\alpha}),l}\Pi_{j_1}uP^{x_1}_{\xi(\tau),\leq i-10}\overline{\Pi_{j_2}u}\notag\\		
& \hspace{8ex}\times P^{x_1}_{\xi(\tau),\leq i-10}\Pi_{j_3}u P^{x_1}_{\xi(\tau),\leq i-10}\overline{\Pi_{j_4}u} P^{x_1}_{\xi(\tau),\leq i-10}\Pi_{j_5}u\big)d \tau\bigg\|_{U_\Delta^2(G_{\alpha}^i,L_{x_1}^2L_{x_2}^2)}\notag\\
&\lesssim 2^{\frac{i-l}{6}}\eta_2\|P^{x_1}_{\xi(G^i_{\alpha}),l}u\|_{{U_\Delta^2(G_{\alpha}^i,L_{x_1}^2L_{x_2}^2)}}\|u\|^3_{X(G_{\alpha}^i,L_{x_1}^2L_{x_2}^2)}.		
\end{align}
Moreover, if the operator $P_{\xi(G_{\alpha}^i),l}^{x_1}$ falls on other $\Pi_{j_k}u$, $k\in\{2,3,4,5\}$, then  similar estimate also holds.
\end{lemma}

\begin{proof}
For simplicity, we ignore this symbol $\sum\limits_{j\in \mathbb{N}}\sum\limits_{\substack{j_1-j_2+j_3-j_4+j_5=j\\j_1,j_2,j_3,j_4,j_5\in\Bbb{N}}}$ in our proof. As before, we just prove \eqref{key2}.  It is sufficient to prove that for $|m-l|\leq 20$, \eqref{key1} holds. First we write $u=u_L+u_H$, where $u_L=P^{x_1}_{|\xi-\xi(t)|\leq 2^{-20}\eta^{-1/4}_3 N(t)}u$ and $u_H=P^{x_1}_{|\xi-\xi(t)|\geq 2^{-20}\eta^{-1/4}_3 N(t)}u$. Then we have
\begin{align}
&\int_{a_{\alpha}^i}^te^{i(t-\tau)\partial_{x_1}^2}P^{x_1}_{\xi(G^i_{\alpha}),m}\Pi_j\big(P_{\xi(G^i_{\alpha}),l}\Pi_{j_1}uP^{x_1}_{\xi(\tau),\leq i-10}\overline{\Pi_{j_2}u} P^{x_1}_{\xi(\tau),\leq i-10}\Pi_{j_3}u P^{x_1}_{\xi(\tau),\leq i-10}\Pi_{j_4}u P^{x_1}_{\xi(\tau),\leq i-10}\Pi_{j_5}u\big)d \tau \notag\\
&=\int_{a_{\alpha}^i}^te^{i(t-\tau)\partial_{x_1}^2}P^{x_1}_{\xi(G^i_{\alpha}),m}\Pi_j\mathcal{O}_H\big(P_{\xi(G^i_{\alpha}),l}\Pi_{j_1}uP^{x_1}_{\xi(\tau),\leq i-10}\overline{\Pi_{j_2}u} P^{x_1}_{\xi(\tau),\leq i-10}\Pi_{j_3}u\notag\\
&\hspace{46ex}\times P^{x_1}_{\xi(\tau),\leq i-10}\overline{\Pi_{j_4}u} P^{x_1}_{\xi(\tau),\leq i-10}\Pi_{j_5}u\big)d \tau \label{decomp1}\\
&+\int_{a_{\alpha}^i}^te^{i(t-\tau)\partial_{x_1}^2}P^{x_1}_{\xi(G^i_{\alpha}),m}\Pi_j   \mathcal{O}_L\big(P_{\xi(G^i_{\alpha}),l}\Pi_{j_1}uP^{x_1}_{\xi(\tau),\leq i-10}\overline{\Pi_{j_2}u} P^{x_1}_{\xi(\tau),\leq i-10}\Pi_{j_3}u\notag\\ &\hspace{46ex}\times P^{x_1}_{\xi(\tau),\leq i-10}\overline{\Pi_{j_4}u} P^{x_1}_{\xi(\tau),\leq i-10}\Pi_{j_5}u\big)d \tau,\label{decomp2}
\end{align}
where the notation
\begin{align*}
\mathcal{O}_H\big(P_{\xi(G^i_{\alpha}),l}\Pi_{j_1}uP^{x_1}_{\xi(\tau),\leq i-10}\overline{\Pi_{j_2}u} P^{x_1}_{\xi(\tau),\leq i-10}\Pi_{j_3}u P^{x_1}_{\xi(\tau),\leq i-10}\overline{\Pi_{j_4}u} P^{x_1}_{\xi(\tau),\leq i-10}\Pi_{j_5}u\big)
\end{align*}
presents that at least three $u$ are replaced by $u_H$ in
\begin{align*}
P^{x_1}_{\xi(\tau),\leq i-10}\overline{\Pi_{j_2}u} P^{x_1}_{\xi(\tau),\leq i-10}\Pi_{j_3}u P^{x_1}_{\xi(\tau),\leq i-10}\overline{\Pi_{j_4}u} P^{x_1}_{\xi(\tau),\leq i-10}\Pi_{j_5}u
\end{align*}
and the notation
\begin{align*}
\mathcal{O}_L\big(P_{\xi(G^i_{\alpha}),l}\Pi_{j_1}uP^{x_1}_{\xi(\tau),\leq i-10}\overline{\Pi_{j_2}u}P^{x_1}_{\xi(\tau),\leq i-10}\Pi_{j_3}u P^{x_1}_{\xi(\tau),\leq i-10}\overline{\Pi_{j_4}u} P^{x_1}_{\xi(\tau),\leq i-10}\Pi_{j_5}u\big)
\end{align*}
presents that at least two $u$ are replaced by $u_L$ in
\begin{align*}
P^{x_1}_{\xi(\tau),\leq i-10}\overline{\Pi_{j_2}u} P^{x_1}_{\xi(\tau),\leq i-10}\Pi_{j_3}u P^{x_1}_{\xi(\tau),\leq i-10}\overline{\Pi_{j_4}u} P^{x_1}_{\xi(\tau),\leq i-10}\Pi_{j_5}u.
\end{align*}

We first estimate \eqref{decomp2}. Without loss of generality, we assume that the $\Pi_{j_2}u, \Pi_{j_3}u$ are repaced by $\Pi_{j_2}u_L, \Pi_{j_3}u_L$ respectively. Let $\theta_{l^{\prime}}$ be the left endpoint of small interval $J_{l^{\prime}}$ satisfies $J_{l^{\prime}}\subset{G^i_{\alpha}}$. On every $J_{l^{\prime}}$, we decompose \eqref{decomp2} as
\begin{align}
&\int_{a_{\alpha}^i}^{\theta_{l^{\prime}}}e^{i(t-\tau)\partial_{x_1}^2}P^{x_1}_{\xi(G^i_{\alpha}),m}\Pi_j   \big(P_{\xi(G^i_{\alpha}),l}\Pi_{j_1}uP^{x_1}_{\xi(\tau),\leq i-10}\overline{\Pi_{j_2}u_L} P^{x_1}_{\xi(\tau),\leq i-10}\Pi_{j_3}u_L P^{x_1}_{\xi(\tau),\leq i-10}\overline{\Pi_{j_4}u} P^{x_1}_{\xi(\tau),\leq i-10}\Pi_{j_5}u\big)d \tau\notag\\
&+\int_{\theta_{l^{\prime}}}^te^{i(t-\tau)\partial_{x_1}^2}P^{x_1}_{\xi(G^i_{\alpha}),m}\Pi_j   \big(P_{\xi(G^i_{\alpha}),l}\Pi_{j_1}uP^{x_1}_{\xi(\tau),\leq i-10}\overline{\Pi_{j_2}u_L} P^{x_1}_{\xi(\tau),\leq i-10}\Pi_{j_3}u_L P^{x_1}_{\xi(\tau),\leq i-10}\overline{\Pi_{j_4}u} P^{x_1}_{\xi(\tau),\leq i-10}\Pi_{j_5}u\big)d \tau.
\end{align}
We also label the left endpoint of $G^i_{\alpha}$ as $l_{G^i_{\alpha}}$ and the right endpoint of $G^i_{\alpha}$ as $r_{G^i_{\alpha}}$. Notice that there exist at least two small intervals $J_1$ and $J_2$ such that $J_1\cap G^i_{\alpha} \neq \emptyset, J_1 \not\subseteq G^{i}_{\alpha}, l_{G^i_{\alpha}} \in J_1$ and $J_2\cap G^i_{\alpha} \neq \emptyset, J_2 \not\subseteq G^{i}_{\alpha}, l_{G^i_{\alpha}} \in J_2$. On $J_1 \cap G^i_{\alpha}$, we decompose \eqref{decomp2} as
	
\begin{align}
&\int_{a_{\alpha}^i}^{l_{G^i_{\alpha}}}e^{i(t-\tau)\partial_{x_1}^2}P^{x_1}_{\xi(G^i_{\alpha}),m}\Pi_j   \big(P_{\xi(G^i_{\alpha}),l}\Pi_{j_1}uP^{x_1}_{\xi(\tau),\leq i-10}\overline{\Pi_{j_2}u_L} P^{x_1}_{\xi(\tau),\leq i-10}\Pi_{j_3}u_L P^{x_1}_{\xi(\tau),\leq i-10}\overline{\Pi_{j_4}u} P^{x_1}_{\xi(\tau),\leq i-10}\Pi_{j_5}u\big)d \tau\label{decomp21}\\
&+\int_{l_{G^i_{\alpha}}}^te^{i(t-\tau)\partial_{x_1}^2}P^{x_1}_{\xi(G^i_{\alpha}),m}\Pi_j   O\big(P_{\xi(G^i_{\alpha}),l}\Pi_{j_1}uP^{x_1}_{\xi(\tau),\leq i-10}\overline{\Pi_{j_2}u_L} P^{x_1}_{\xi(\tau),\leq i-10}\Pi_{j_3}u_L P^{x_1}_{\xi(\tau),\leq i-10}\overline{\Pi_{j_4}u} P^{x_1}_{\xi(\tau),\leq i-10}\Pi_{j_5}u\big)d \tau \label{decomp22}
\end{align}
and on $J_2 \cap G^i_{\alpha}$ as
	
\begin{align}
&\int_{a_{\alpha}^i}^{r_{G^i_{\alpha}}}e^{i(t-\tau)\partial_{x_1}^2}P^{x_1}_{\xi(G^i_{\alpha}),m}\Pi_j   \big(P_{\xi(G^i_{\alpha}),l}\Pi_{j_1}uP^{x_1}_{\xi(\tau),\leq i-10}\overline{\Pi_{j_2}u_L} P^{x_1}_{\xi(\tau),\leq i-10}\Pi_{j_3}u_L P^{x_1}_{\xi(\tau),\leq i-10}\overline{\Pi_{j_4}u} P^{x_1}_{\xi(\tau),\leq i-10}\Pi_{j_5}u\big)d \tau\label{decomp23}\\
&+\int_{r_{G^i_{\alpha}}}^te^{i(t-\tau)\partial_{x_1}^2}P^{x_1}_{\xi(G^i_{\alpha}),m}\Pi_j   \big(P_{\xi(G^i_{\alpha}),l}\Pi_{j_1}uP^{x_1}_{\xi(\tau),\leq i-10}\overline{\Pi_{j_2}u_L} P^{x_1}_{\xi(\tau),\leq i-10}\Pi_{j_3}u_L P^{x_1}_{\xi(\tau),\leq i-10}\overline{\Pi_{j_4}u} P^{x_1}_{\xi(\tau),\leq i-10}\Pi_{j_5}u\big)d \tau.\label{decomp24}
\end{align}
	
In summary, we have decomposed the nonlinear term into $F^L_1 +F^L_2$, where $F^L_1= \eqref{decomp21}$ on $J_{l^{\prime}}$ , $F^L_1= \eqref{decomp23}$ on $J_1 \cap G^i_{\alpha}$ and  $F^L_2= \eqref{decomp22}$ on $J_{l^{\prime}}$, $F^L_2= \eqref{decomp24}$ on $J_2 \cap G^i_{\alpha}$.
	
We first treat the term $\|F^L_1\|_{U_\Delta^2(G_{\alpha}^i,L_{x_1}^2L_{x_2}^2)}$.
Thanks to the embedding $V_\Delta^1(G_{\alpha}^i,L_{x_1}^2L_{x_2}^2)\hookrightarrow    U_{\Delta}^2(G_{\alpha}^i,L_{x_1}^2L_{x_2}^2)$, it is sufficent to estimate 
$\|F^L_1\|_{V_\Delta^1(G_{\alpha}^i,L_{x_1}^2L_{x_2}^2)}$.
Notice that for any partition $\mathcal{L}$ of $G^i_{\alpha}$, 
\begin{align}
\sup\limits_{\mathcal{L}}\|F^L_1\|_{V_\Delta^1(\mathcal{L}, G_{\alpha}^i,L_{x_1}^2L_{x_2}^2)}&\lesssim \notag\\
\sum\limits_{J_{l^{\prime}}\cap G^i_{\alpha}\neq \emptyset}&\big\|\int_{J_{l^{\prime}}\cap G^i_{\alpha}}e^{-i\tau\partial_{x_1}^2}P^{x_1}_{\xi(G^i_{\alpha}),m}\Pi_j   \big(P_{\xi(G^i_{\alpha}),l}\Pi_{j_1}u_{L}P^{x_1}_{\xi(\tau),\leq i-10}\overline{\Pi_{j_2}u_L} P^{x_1}_{\xi(\tau),\leq i-10}\Pi_{j_3}u_L\notag\\
&\hspace{6ex}\times P^{x_1}_{\xi(\tau),\leq i-10}\overline{\Pi_{j_4}u} P^{x_1}_{\xi(\tau),\leq i-10}\Pi_{j_5}u\big)d \tau \big\|_{L_{x_1}^{2}L_{x_2}^2}\label{essst1}
\end{align}
and
\begin{align}
&\|e^{-it\partial_{x_1}^2}F^L_1\|_{L_t ^{\infty}L_{x_1}^2L_{x_2}^2 (G^i_{\alpha} \times \mathbb{R} \times \mathbb{R})} \notag\\
\lesssim&	\sup\limits_{\theta_{l^{\prime}},\theta_{l^{\prime} } \leq a^i_{\alpha}}\sum\limits_{J_{l^{\prime}}\cap [\theta_{l^{\prime}}, a^i_{\alpha}] \neq \emptyset}\big\|\int_{J_{l^{\prime}}\cap [\theta_{l^{\prime}}, a^i_{\alpha}]  }e^{-i\tau\partial_{x_1}^2}P^{x_1}_{\xi(G^i_{\alpha}),m}\Pi_j   \big(P_{\xi(G^i_{\alpha}),l}\Pi_{j_1}u_{L}P^{x_1}_{\xi(\tau),\leq i-10}\overline{\Pi_{j_2}u_L} P^{x_1}_{\xi(\tau),\leq i-10}\Pi_{j_3}u_L\notag\\
&\hspace{6ex}\times P^{x_1}_{\xi(\tau),\leq i-10}\overline{\Pi_{j_4}u} P^{x_1}_{\xi(\tau),\leq i-10}\Pi_{j_5}u\big)d \tau \big\|_{L_{x_1}^{2}L_{x_2}^2(\mathbb{R}\times \mathbb{R})A}\label{essst21}\\
&+\sup\limits_{\theta_{l^{\prime}},\theta_{l^{\prime} } \geq a^i_{\alpha}}\sum\limits_{J_{l^{\prime}}\cap [ a^i_{\alpha},\theta_{l^{\prime}}] \neq \emptyset}\big\|\int_{J_{l^{\prime}}\cap [ a^i_{\alpha}, \theta_{l^{\prime}}]  }e^{-i\tau\partial_{x_1}^2}P^{x_1}_{\xi(G^i_{\alpha}),m}\Pi_j   \big(P_{\xi(G^i_{\alpha}),l}\Pi_{j_1}u_{L}P^{x_1}_{\xi(\tau),\leq i-10}\overline{\Pi_{j_2}u_L} P^{x_1}_{\xi(\tau),\leq i-10}\Pi_{j_3}u_L\notag\\
&\hspace{6ex}\times P^{x_1}_{\xi(\tau),\leq i-10}\overline{\Pi_{j_4}u} P^{x_1}_{\xi(\tau),\leq i-10}\Pi_{j_5}u\big)d \tau \big\|_{L_{x_1}^{2}L_{x_2}^2(\mathbb{R}\times \mathbb{R})}\label{essst22}\\
&+\sum\limits_{J_{l^{\prime}}\cap [l_{G^i_{\alpha}}, a^i_{\alpha}] \neq \emptyset}\big\|\int_{J_{l^{\prime}}\cap [l_{G^i_{\alpha}}, a^i_{\alpha}]  }e^{-i\tau\partial_{x_1}^2}P^{x_1}_{\xi(G^i_{\alpha}),m}\Pi_j   O\big(P_{\xi(G^i_{\alpha}),l}\Pi_{j_1}u_{L}P^{x_1}_{\xi(\tau),\leq i-10}\overline{\Pi_{j_2}u_L} P^{x_1}_{\xi(\tau),\leq i-10}\Pi_{j_3}u_L\notag\\
&\hspace{6ex}\times P^{x_1}_{\xi(\tau),\leq i-10}\overline{\Pi_{j_4}u} P^{x_1}_{\xi(\tau),\leq i-10}\Pi_{j_5}u\big)d \tau \big\|_{L_{x_1}^{2}L_{x_2}^2(\mathbb{R}\times \mathbb{R})}\label{essst23}\\
&+\sum\limits_{J_{l^{\prime}}\cap [ a^i_{\alpha}, r_{G^i_{\alpha}}] \neq \emptyset}\big\|\int_{J_{l^{\prime}}\cap [ a^i_{\alpha}, r_{G^i_{\alpha}}]  }e^{-i\tau\partial_{x_1}^2}P^{x_1}_{\xi(G^i_{\alpha}),m}\Pi_j   O\big(P_{\xi(G^i_{\alpha}),l}\Pi_{j_1}u_{L}P^{x_1}_{\xi(\tau),\leq i-10}\overline{\Pi_{j_2}u_L} P^{x_1}_{\xi(\tau),\leq i-10}\Pi_{j_3}u_L\notag\\
&\hspace{6ex}\times P^{x_1}_{\xi(\tau),\leq i-10}\overline{\Pi_{j_4}u} P^{x_1}_{\xi(\tau),\leq i-10}\Pi_{j_5}u\big)d \tau \big\|_{L_{x_1}^{2}L_{x_2}^2(\mathbb{R}\times \mathbb{R})}.\label{essst24}
\end{align}
	
For \eqref{essst1}, we notice that for any $J_{l^{\prime}}$	, there exists $f_{l^{\prime}}\in L^{2}_{x_1}L_{x_2}^2(\mathbb{R}\times \mathbb{R})$, $\|f_{l^{\prime}}\|_{L^{2}_{x_1}L_{x_2}^2(\mathbb{R}\times \mathbb{R})}=1$ and has Fourier transform supported on $|\xi-\xi(G^i_{\alpha})|\sim 2^m$ such that 	
\begin{align*}
&\bigg\|\int_{J_{l^{\prime}\cap G^i_{\alpha}}}e^{-i\tau\partial_{x_1}^2}P^{x_1}_{\xi(G^i_{\alpha}),m}\Pi_j\big(P_{\xi(G^i_{\alpha}),l}\Pi_{j_1}uP^{x_1}_{\xi(\tau),\leq i-10}\overline{\Pi_{j_2}u_L} P^{x_1}_{\xi(\tau), \leq i-10}\Pi_{j_3}u_L\notag\\ 
&\hspace{6ex}\times P^{x_1}_{\xi(\tau),\leq i-10}\overline{\Pi_{j_4}u} P^{x_1}_{\xi(\tau),\leq i-10}\Pi_{j_5}u\big) d \tau\bigg\|_{L^{2}_{x_1}L_{x_2}^2(\mathbb{R}\times \mathbb{R})}\\
& =\int_{J_{l^{\prime}\cap G^i_{\alpha}}}\big\langle f_{l^{\prime}}, e^{-i\tau\partial_{x_1}^2}P^{x_1}_{\xi(G^i_{\alpha}),m}\Pi_j\big(P_{\xi(G^i_{\alpha}),l}\Pi_{j_1}uP^{x_1}_{\xi(\tau),\leq i-10}\overline{\Pi_{j_2}u} P^{x_1}_{\xi(\tau), \leq i-10}\Pi_{j_3}u_L\notag\\ 
&\hspace{6ex}\times P^{x_1}_{\xi(\tau),\leq i-10}\overline{\Pi_{j_4}u} P^{x_1}_{\xi(\tau),\leq i-10}\Pi_{j_5}u\big)\big\rangle_{L^{2}_{x_1}L_{x_2}^2(\mathbb{R}\times \mathbb{R})} d \tau,
\end{align*}
then by H\"older's inequality,
\begin{align}
& \int_{J_{l^{\prime}\cap G^i_{\alpha}}}\big\langle f_{l^{\prime}}, e^{-i\tau\partial_{x_1}^2}P^{x_1}_{\xi(G^i_{\alpha}),m}\Pi_j\big(P_{\xi(G^i_{\alpha}),l}\Pi_{j_1}uP^{x_1}_{\xi(\tau),\leq i-10}\overline{\Pi_{j_2}u} P^{x_1}_{\xi(\tau), \leq i-10}\Pi_{j_3}u_L\notag\\ 
&\hspace{6ex}\times P^{x_1}_{\xi(\tau),\leq i-10}\overline{\Pi_{j_4}u} P^{x_1}_{\xi(\tau),\leq i-10}\Pi_{j_5}u\big)\big\rangle_{L^{2}_{x_1}L_{x_2}^2(\mathbb{R}\times \mathbb{R})} d \tau \notag\\
&\lesssim \int_{J_{l^{\prime}\cap G^i_{\alpha}}}\big\langle f_{l^{\prime}}, e^{-i\tau\partial_{x_1}^2}P^{x_1}_{\xi(G^i_{\alpha}),m}\Pi_j\big(P_{\xi(G^i_{\alpha}),l}\Pi_{j_1}uP^{x_1}_{\xi(\tau),\leq  i-10}\overline{\Pi_{j_2}u_L}P^{x_1}_{\xi(\tau), \leq i-10}\Pi_{j_3}u_L\notag\\ 
&\hspace{6ex}\times P^{x_1}_{\xi(\tau),\leq i-10}\overline{\Pi_{j_4}u} P^{x_1}_{\xi(\tau),\leq i-10}\Pi_{j_5}u\big)\big\rangle_{L^{2}_{x_1}L_{x_2}^2(\mathbb{R}\times \mathbb{R})} d \tau\notag\\
& \lesssim\big\|\|e^{i\tau \partial_{x_1}^2}P^{x_1}_{\xi(G^i_{\alpha}),m}f_{l^{\prime}}\|_{L_{x_2}^2(\mathbb{R})} \|P_{\xi(G^i_{\alpha}),l}u\|_{L_{x_2}^2(\mathbb{R})}\|P^{x_1}_{\xi(\tau),\leq i-10}u_L\|_{L_{x_2}^2(\mathbb{R})}\|P^{x_1}_{\xi(\tau), \leq i-10}u_L\|_{L_{x_2}^2(\mathbb{R})}\notag\\
&\hspace{4ex}\times\|P^{x_1}_{\xi(\tau),\leq i-10}u\|_{L_{x_2}^2(\mathbb{R})} \|P^{x_1}_{\xi(\tau),\leq i-10}u\|_{L_{x_2}^2(\mathbb{R})} \big\|_{L^{1}_{t,x_1}((J_{l^{\prime}\cap G^i_{\alpha}} )\times \mathbb{R})}\notag\\
&\lesssim \big\| \|P^{x_1}_{\xi(\tau), \leq  i-10}u_L\|_{L_{x_2}^2(\mathbb{R})}\big\|_{L_{t,x_1}^{\infty}(J_{l^{\prime}\cap G^i_{\alpha}} \times \mathbb{R})} \big\|\|e^{i\tau \partial_{x_1}^2}P^{x_1}_{\xi(G^i_{\alpha}),m}f_{l^{\prime}}\|_{L_{x_2}^2(\mathbb{R})} \|P^{x_1}_{\xi(\tau), \leq  i-10}u\|_{L_{x_2}^2(\mathbb{R})}\big\|_{L^2_{t,x_1}(J_{l^{\prime}\cap G^i_{\alpha}}\times \mathbb{R})}\notag\\
&\times \big\| \|P^{x_1}_{\xi(\tau), \leq  i-10}u_L\|_{L^2_{x_2}(\mathbb{R})}\big\|_{L_{t,x_1}^{\infty}(J_{l^{\prime}\cap G^i_{\alpha}} \times \mathbb{R})}\big\|\|P^{x_1}_{\xi(G^i_{\alpha}), l}u\|_{L_{x_2}^2(\mathbb{R})} \|\|P^{x_1}_{\xi(\tau), \leq i-10}u\|_{L_{x_2}^2(\mathbb{R})}\big\|_{L^2_{t,x_1}(J_{l^{\prime}\cap G^i_{\alpha}}\times \mathbb{R})}.\notag
\end{align}
By Proposition \ref{Bernstein} , bilinear estimate \eqref{bilinearl2}, and the fact that $\|u\|_{U_\Delta^2(J_{l^{\prime}},L_{x_1}^2L_{x_2}^2)}\lesssim 1$, we continue:
\begin{align}
&\lesssim  \begin{cases}
N(J_{l^{\prime}})2^{-20}\eta^{-1/4}_3 2^{- l/2} 2^{- m/2} \|e^{i\tau \partial_{x_1}^2}P^{x_1}_{\xi(G^i_{\alpha}),m}f_{l^{\prime}}(\tau)\|_{U_\Delta^2(J_{l^{\prime}}\cap G^i_{\alpha},L_{x_1}^2L_{x_2}^2)} \|P^{x_1}_{\xi(G^i_{\alpha}), l}u_H\|_{U_\Delta^2(J_{l^{\prime}}\cap G^i_{\alpha},L_{x_1}^2L_{x_2}^2)}, &J_{l^{\prime}} \subset G^i_{\alpha}\\
N(G^i_{\alpha})2^{-20}\eta^{-1/4}_3 2^{- l/2} 2^{- m/2} \|e^{i\tau \partial_{x_1}^2}P^{x_1}_{\xi(G^i_{\alpha}),m}f_{l^{\prime}}(\tau)\|_{U_\Delta^2(J_{l^{\prime}}\cap G^i_{\alpha},L_{x_1}^2L_{x_2}^2)} \|P^{x_1}_{\xi(G^i_{\alpha}), l}u_H\|_{U_\Delta^2(J_{l^{\prime}}\cap G^i_{\alpha},L_{x_1}^2L_{x_2}^2)}, &J_{l^{\prime}} \not\subseteq G^i_{\alpha}	 
\end{cases}\notag\\
&\lesssim  \begin{cases}
N(J_{l^{\prime}})\eta^{-1/4}_3 2^{- l} \|P^{x_1}_{\xi(G^i_{\alpha}), l}u_H\|_{U_\Delta^2(G^i_{\alpha},L_{x_1}^2L_{x_2}^2)}, &J_{l^{\prime}} \subset G^i_{\alpha}\\
N(G^i_{\alpha})\eta^{-1/4}_3 2^{- l} \|P^{x_1}_{\xi(G^i_{\alpha}), l}u_H\|_{U_\Delta^2(G^i_{\alpha},L_{x_1}^2L_{x_2}^2)}, &J_{l^{\prime}} \not\subseteq G^i_{\alpha}	 
\end{cases}.\label{essss1} 
\end{align}
	
By Remark \ref{remark-small} and \eqref{essss1} , we have
\begin{align}
\eqref{essst1} \lesssim \sum\limits_{J_{l^{\prime}}\cap G^i_{\alpha}\neq \emptyset} \eqref{essss1}
&\lesssim 2 N(G^i_{\alpha})\eta^{-1/4}_3 2^{- l} \|P^{x_1}_{\xi(G^i_{\alpha}), l}u_H\|_{U_\Delta^2(G^i_{\alpha},L_{x_1}^2L_{x_2}^2)}\notag\\
&+ \sum\limits_{J_{l^{\prime}}\subset G^i_{\alpha}} N(J_{l^{\prime}})\eta^{-1/4}_3 2^{- l}\|P^{x_1}_{\xi(G^i_{\alpha}), l}u_H\|_{U_\Delta^2(G^i_{\alpha},L_{x_1}^2L_{x_2}^2)}\notag\\
&\lesssim \eta^{3/4}_3 2^{i - l} \|P^{x_1}_{\xi(G^i_{\alpha}), l}u_H\|_{U_\Delta^2(G^i_{\alpha},L_{x_1}^2L_{x_2}^2)}.
\end{align}
	
Using almost the same argument above, we also have the similar results:
\begin{align*}
\eqref{essst21}\lesssim \eta^{3/4}_3 2^{i - l} \|P^{x_1}_{\xi(G^i_{\alpha}), l}u_H\|_{U_\Delta^2(G^i_{\alpha},L_{x_1}^2L_{x_2}^2)},\\
\eqref{essst22}\lesssim \eta^{3/4}_3 2^{i - l} \|P^{x_1}_{\xi(G^i_{\alpha}), l}u_H\|_{U_\Delta^2(G^i_{\alpha},L_{x_1}^2L_{x_2}^2)},\\
\eqref{essst23}\lesssim \eta^{3/4}_3 2^{i - l} \|P^{x_1}_{\xi(G^i_{\alpha}), l}u_H\|_{U_\Delta^2(G^i_{\alpha},L_{x_1}^2L_{x_2}^2)},\\
\eqref{essst24}\lesssim \eta^{3/4}_3 2^{i - l} \|P^{x_1}_{\xi(G^i_{\alpha}), l}u_H\|_{U_\Delta^2(G^i_{\alpha},L_{x_1}^2L_{x_2}^2)}.
\end{align*}
So, by the definition of $V_\Delta^1(G^i_{\alpha},L_{x_1}^2L_{x_2}^2)$, $\|F^L_1\|_{V_\Delta^1(G^i_{\alpha},L_{x_1}^2L_{x_2}^2)}\lesssim \eta^{3/4}_3 2^{i - l} \|P^{x_1}_{\xi(G^i_{\alpha}), l}u_H\|_{U_\Delta^2(G^i_{\alpha},L_{x_1}^2L_{x_2}^2)}$, further,
\begin{align}
\|F^L_1\|_{U_\Delta^2(G^i_{\alpha},L_{x_1}^2L_{x_2}^2)}\lesssim\eta^{3/4}_3 2^{i - l} \|P^{x_1}_{\xi(G^i_{\alpha}), l}u\|_{U_\Delta^2(G^i_{\alpha},L_{x_1}^2L_{x_2}^2)}\label{essst}.
\end{align}
	
Next we treat $\|F^L_2\|_{U_\Delta^2(G_{\alpha}^i,L_{x_1}^2L_{x_2}^2)}$. By \eqref{triangle}, $\|F^L_2\|^2_{U_\Delta^2(G_{\alpha}^i,L_{x_1}^2L_{x_2}^2)}\lesssim \sum\limits_{J_{l^{\prime}}\cap G^i_{\alpha} \neq \emptyset}\|F^L_2\|^2_{U_\Delta^2((J_{l^{\prime}}\cap G_{\alpha}^i),L_{x_1}^2L_{x_2}^2)}.$
Using Lemma \ref{Up-dual}, on every interval $J_{l^{\prime}}\cap G^i_{\alpha}$,  $g_{l^{\prime}}\in U_\Delta^6((J_{l^{\prime}}\cap G^i_{\alpha}),L_{x_1}^2L_{x_2}^2)$ 
\begin{align}
&\|F^L_2\|_{U^2_{\Delta}((J_{l^{\prime}}\cap G^i_{\alpha}),L_{x_1}^2L_{x_2}^2)}\notag\\
&\lesssim \sup\limits_{ \|g_{l^{\prime}}\|_{U^6_{\Delta}((J_{l^{\prime}}\cap G^i_{\alpha}),L_{x_1}^2L_{x_2}^2)}=1}\int_{J_{l^{\prime}}\cap G^i_{\alpha}}\big\langle g_{l^{\prime}}, P^{x_1}_{\xi(G^i_{\alpha}),m}\Pi_j\big(P_{\xi(G^i_{\alpha}),l}\Pi_{j_1}uP^{x_1}_{\xi(\tau),\leq i-10}\overline{\Pi_{j_2}u} P^{x_1}_{\xi(\tau), \leq i-10}\Pi_{j_3}u_L\notag\\ 
&\hspace{6ex}\times P^{x_1}_{\xi(\tau),\leq i-10}\overline{\Pi_{j_4}u} P^{x_1}_{\xi(\tau),\leq i-10}\Pi_{j_5}u\big)\big\rangle_{L^{2}_{x_1}L_{x_2}^2(\mathbb{R}\times \mathbb{R})}.\label{qweewq}
\end{align}
By H\"older's inequality and the fact that $\|u\|_{U_\Delta^2(J_{l^{\prime}},L_{x_1}^2L_{x_2}^2)}\lesssim 1$,
\begin{align}
&\eqref{qweewq} \lesssim \big\|\|P_{m}^{x_1} g_{l^{\prime}}\|_{L^2_{x_2}(\mathbb{R})}\big\|_{L_{t,x_1}^{6}((J_{l^{\prime}}\cap G^i_{\alpha})\times\mathbb{R})}\big\|\|u_{\leq i-10}\|_{L^2_{x_2}(\mathbb{R})}\big\|^2_{L_{t,x_1}^{6}((J_{l^{\prime}}\cap G^i_{\alpha})\times\mathbb{R})}\notag\\
&\hspace{6ex}\times \big\|\|P^{x_1}_{\xi(G^i_{\alpha}), l} u\|_{L^2_{x_2}(\mathbb{R})}\|u_L\|^2_{L^2_{x_2}(\mathbb{R})}\big\|_{L_{t,x_1}^{2}((J_{l^{\prime}}\cap G^i_{\alpha})\times\mathbb{R})}\notag\\
&\hspace{6ex}\lesssim \big\|\|P^{x_1}_{\xi(G^i_{\alpha}), l} u\|_{L^2_{x_2}(\mathbb{R})}\|u_L\|^2_{L^2_{x_2}(\mathbb{R})}\big\|_{L_{t,x_1}^{2}((J_{l^{\prime}}\cap G^i_{\alpha})\times\mathbb{R})}\notag.
\label{esssss1} 
\end{align}
So, 
\begin{align}
\|F^L_2\|^2_{U_\Delta^2(G_{\alpha}^i,L_{x_1}^2L_{x_2}^2)}&\lesssim \sum\limits_{J_{l^{\prime}}\cap G^i_{\alpha} \neq \emptyset}\|F^L_2\|^2_{U_\Delta^2((J_{l^{\prime}}\cap G_{\alpha}^i),L_{x_1}^2L_{x_2}^2)}.
\end{align}
However, by \eqref{esttt32}, $$\big\|\|P^{x_1}_{\xi(G^i_{\alpha}), l} u\|_{L^2_{x_2}(\mathbb{R})}\|u_L\|^2_{L^2_{x_2}(\mathbb{R})}\big\|_{L_{t,x_1}^{2}(G^i_{\alpha}\times\mathbb{R})}^2\lesssim 2^{i-l}\|P^{x_1}_{\xi(G^i_{\alpha}),l}u\|^2_{U^2_{\Delta}(G_{\alpha}^i,L^2_{x_1}L_{x_2}^2)}\|u\|^2_{X(G_{\alpha}^i,L^2_{x_1}L_{x_2}^2)}.$$
Thus 
\begin{align}\label{estttttt111}
\|F^L_2\|_{U_\Delta^2(G_{\alpha}^i,L_{x_1}^2L_{x_2}^2)}\lesssim 2^{(i-l)/2}\|P^{x_1}_{\xi(G^i_{\alpha}),l}u\|_{U^2_{\Delta}(G_{\alpha}^i,L^2_{x_1}L_{x_2}^2)}\|u\|_{X(G_{\alpha}^i,L^2_{x_1}L_{x_2}^2)}.
\end{align}
	
Next we estimate \eqref{decomp1}. Without loss of generality, we assume that the $\Pi_{j_2}u, \Pi_{j_3}u, \Pi_{j_4}u$ are replaced by $\Pi_{j_2}u_H, \Pi_{j_3}u_H, \Pi_{j_4}u_H$ respectively.

Let $\gamma^{m_3}_{\alpha^{\prime}}$ be the left endpiont of subinterval $G^{m_3}_{\alpha^{\prime}}\subset G^i_{\alpha}$, on every $G^{m_3}_{\alpha^{\prime}}$, we further decompose 
\begin{align*}
\int_{a_{\alpha}^i}^te^{i(t-\tau)\partial_{x_1}^2}P^{x_1}_{\xi(G^i_{\alpha}),m}\Pi_j\big(P_{\xi(G^i_{\alpha}),l}\Pi_{j_1}uP^{x_1}_{\xi(\tau),m_2}\overline{\Pi_{j_2}u_{H}} P^{x_1}_{\xi(\tau),m_3}\Pi_{j_3}u_{H} P^{x_1}_{\xi(\tau),\leq m_3}\overline{\Pi_{j_4}u_{H}} P^{x_1}_{\xi(\tau),\leq m_3}\Pi_{j_5}u\big)d \tau
\end{align*}
as
\begin{align}
&\int_{a^i_{\alpha}}^{\gamma^{m_3}_{\alpha^{\prime}}}e^{i(t-\tau)\partial_{x_1}^2}P^{x_1}_{\xi(G^i_{\alpha}),m}\Pi_j\big(P_{\xi(G^i_{\alpha}),l}\Pi_{j_1}uP^{x_1}_{\xi(\tau), m_2}\overline{\Pi_{j_2}u_{H}} P^{x_1}_{\xi(\tau), m_3}\Pi_{j_3}u_{H} P^{x_1}_{\xi(\tau),\leq m_3}\overline{\Pi_{j_4}u_{H}} P^{x_1}_{\xi(\tau),\leq m_3}\Pi_{j_5}u\big)d \tau\label{esttttt1}\\
&+\int_{\gamma^{m_3}_{\alpha^{\prime}}}^{t}e^{i(t-\tau)\partial_{x_1}^2}P^{x_1}_{\xi(G^i_{\alpha}),m}\Pi_j\big(P_{\xi(G^i_{\alpha}),l}\Pi_{j_1}uP^{x_1}_{\xi(\tau), m_2}\overline{\Pi_{j_2}u_{H}} P^{x_1}_{\xi(\tau), m_3}\Pi_{j_3}u_{H} P^{x_1}_{\xi(\tau),\leq m_3}\overline{\Pi_{j_4}u_{H}} P^{x_1}_{\xi(\tau),\leq m-3}\Pi_{j_5}u\big)d \tau\label{esttttt2}	
\end{align}
and  denote 	
\begin{align*}
&\int_{a^i_{\alpha}}^t e^{i(t-\tau)\partial_{x_1}^2}P^{x_1}_{\xi(G^i_{\alpha}),m}\Pi_j\big(P_{\xi(G^i_{\alpha}),l}\Pi_{j_1}uP^{x_1}_{\xi(\tau),\leq m_2}\overline{\Pi_{j_2}u_H}P^{x_1}_{\xi(\tau), m_3}\Pi_{j_3}u_HP^{x_1}_{\xi(\tau),\leq m_3}\overline{\Pi_{j_4}u_H}P^{x_1}_{\xi(\tau),\leq m_3}\Pi_{j_5}u\big)d \tau\\
&= F^{m_3}_1 + F^{m_3}_2,
\end{align*}	 
where $F^{m_3}_1=\eqref{esttttt1}$ and $F^{m_3}_2=\eqref{esttttt2}$	on every $G^{m_3}_{\alpha^{\prime}}$.
	
First we treat $\sum\limits_{\substack{0\leq m_3\leq m_2\\ \leq j-10}}\|F^{m_3}_1\|_{U_\Delta^2(G_{\alpha}^i,L_{x_1}^2L_{x_2}^2)}$.	
By the embedding $V_\Delta^1(G_{\alpha}^i,L_{x_1}^2L_{x_2}^2)\hookrightarrow    U_{\Delta}^2(G_{\alpha}^i,L_{x_1}^2L_{x_2}^2)$ , 
\begin{align}
\|F^{m_3}_1\|_{U_\Delta^2(G_{\alpha}^i,L_{x_1}^2L_{x_2}^2)}&\lesssim
\|F^{m_3}_1\|_{V_\Delta^1(G_{\alpha}^i,L_{x_1}^2L_{x_2}^2)}\notag\\
&\lesssim \sum\limits_{G^{m_3}_{\alpha^{\prime}}\subset G^j_k} \bigg\|\int_{G^{m_3}_{\alpha^{\prime}}}e^{-i\tau\partial_{x_1}^2}P^{x_1}_{\xi(G^i_{\alpha}),m}\Pi_j\big(P_{\xi(G^i_{\alpha}),l}\Pi_{j_1}uP^{x_1}_{\xi(\tau),\leq i-10}\overline{\Pi_{j_2}u_H}P^{x_1}_{\xi(\tau), m_2}\Pi_{j_3}u_H\notag\\ 
&\times P^{x_1}_{\xi(\tau), m_3}\overline{\Pi_{j_4}u_{H}}P^{x_1}_{\xi(\tau),\leq m_3}u\Pi_{j_5}P^{x_1}_{\xi(\tau),\leq m_3}u\big)d \tau\bigg\|_{L^{2}_{x_1}L_{x_2}^2(\mathbb{R}\times \mathbb{R})}.
\end{align}
	
Notice that for any $G^{m_3}_{\alpha^{\prime}}$	, there exists $f^{m_3}_{\alpha^{\prime}}\in L^{2}_{x_1}L_{x_2}^2(\mathbb{R}\times \mathbb{R})$ , $\|f^{m_3}_{\alpha^{\prime}}\|_{L^{2}_{x_1}L_{x_2}^2(\mathbb{R}\times \mathbb{R})}=1$ and has Fourier transform supported on $|\xi-\xi(G^i_{\alpha})|\sim 2^m$ such that 
	
\begin{align*}
&\bigg\|\int_{G^{m_3}_{\alpha^{\prime}}}e^{-i\tau\partial_{x_1}^2}P^{x_1}_{\xi(G^i_{\alpha}),m}\Pi_j\big(P_{\xi(G^i_{\alpha}),l}\Pi_{j_1}uP^{x_1}_{\xi(\tau), m_2}\overline{\Pi_{j_2}u_H} P^{x_1}_{\xi(\tau), m_3}\Pi_{j_3}u_H\notag\\ 
&\hspace{6ex}\times P^{x_1}_{\xi(\tau),\leq m_3}\overline{\Pi_{j_4}u_{H}} P^{x_1}_{\xi(\tau),\leq m_3}\Pi_{j_5}u\big) d \tau\bigg\|_{L^{2}_{x_1}L_{x_2}^2(\mathbb{R}\times \mathbb{R})}\\
& =\int_{G^{m_3}_{\alpha^{\prime}}}\big\langle f_{G^{m_3}_{\alpha^{\prime}}}, e^{-i\tau\partial_{x_1}^2}P^{x_1}_{\xi(G^i_{\alpha}),m}\Pi_j\big(P_{\xi(G^i_{\alpha}),l}\Pi_{j_1}uP^{x_1}_{\xi(\tau), m_2}\overline{\Pi_{j_2}u_H} P^{x_1}_{\xi(\tau), m_3}\Pi_{j_3}u_H\notag\\ 
&\hspace{6ex}\times P^{x_1}_{\xi(\tau),\leq m_3}\overline{\Pi_{j_4}u_{H}} P^{x_1}_{\xi(\tau),\leq m_3}\Pi_{j_5}u\big)\big\rangle_{L^{2}_{x_1}L_{x_2}^2(\mathbb{R}\times \mathbb{R})} d \tau.
\end{align*}
Denote $\widetilde{F}_{m_3}(t, x_1, x_2)=\sum\limits_{G^{m_3}_{\alpha^{\prime}}\subset G^i_{\alpha}}\chi_{G^{m_3}_{\alpha^{\prime}}}(t)f^{m_3}_{\alpha^{\prime}}(x_1, x_2)$, then by H\"older's inequality
	
\begin{align}
& \bigg\|\int_{G^{i}_{\alpha}}e^{-i\tau\partial_{x_1}^2}P^{x_1}_{\xi(G^i_{\alpha}),m}\Pi_j\big(P_{\xi(G^i_{\alpha}),l}\Pi_{j_1}uP^{x_1}_{\xi(\tau), m_2}\overline{\Pi_{j_2}u_H}P^{x_1}_{\xi(\tau), m_3}\Pi_{j_3}u_H\notag\\ 
&\hspace{6ex}\times P^{x_1}_{\xi(\tau),\leq m_3}\overline{\Pi_{j_4}u_H} P^{x_1}_{\xi(\tau),\leq m_3}\Pi_{j_5}u\big)d \tau\bigg\|_{L^{2}_{x_1}L_{x_2}^2(\mathbb{R}\times \mathbb{R})}\notag\\
&\lesssim \int_{G^{i}_{\alpha}}\big\langle \widetilde{F}_{m_3}(\tau), e^{-i\tau\partial_{x_1}^2}P^{x_1}_{\xi(G^i_{\alpha}),m}\Pi_j\big(P_{\xi(G^i_{\alpha}),l}\Pi_{j_1}uP^{x_1}_{\xi(\tau), m_2}\overline{\Pi_{j_2}u_H}P^{x_1}_{\xi(\tau), m_3}\Pi_{j_3}u_H\notag\\ 
&\hspace{6ex}\times P^{x_1}_{\xi(\tau),\leq m_3}\overline{\Pi_{j_4}u_H} P^{x_1}_{\xi(\tau),\leq m_3}\Pi_{j_5}u\big)\big\rangle_{L^{2}_{x_1}L_{x_2}^2(\mathbb{R}\times \mathbb{R})} d \tau\notag\\
& \lesssim\big\|\|e^{i\tau \partial_{x_1}^2}P^{x_1}_{\xi(G^i_{\alpha}),m}\widetilde{F}_{m_3}(\tau)\|_{L_{x_2}^2(\mathbb{R})} \|P_{\xi(G^i_{\alpha}),l}u\|_{L_{x_2}^2(\mathbb{R})}\|P^{x_1}_{\xi(\tau), m_2}u_H\|_{L_{x_2}^2(\mathbb{R})}\|P^{x_1}_{\xi(\tau), m_3}u_H\|_{L_{x_2}^2(\mathbb{R})}\notag\\
&\hspace{4ex}\times\|P^{x_1}_{\xi(\tau),\leq m_3}u_H\|_{L_{x_2}^2(\mathbb{R})} \|P^{x_1}_{\xi(\tau),\leq m_3}u\|_{L_{x_2}^2(\mathbb{R})} \big\|_{L^{1}_{t,x_1}(G^i_{\alpha} \times \mathbb{R})}\notag\\
& \lesssim\big\|\|e^{i\tau \partial_{x_1}^2}P^{x_1}_{\xi(G^i_{\alpha}),m}\widetilde{F}_{m_3}(\tau)\|_{L^2_{x_2}(\mathbb{R})} \|P_{\xi(G^i_{\alpha}),l}u\|_{L_{x_2}^2(\mathbb{R})}\|P^{x_1}_{\xi(\tau), m_2}u_H\|_{L_{x_2}^2(\mathbb{R})}\|P^{x_1}_{\xi(\tau), m_3}u_H\|_{L_{x_2}^2(\mathbb{R})}\notag\\
&\hspace{4ex}\times\|P^{x_1}_{\xi(\tau),\leq m_3}u_H\|_{L^2_{x_2}(\mathbb{R})} \|P^{x_1}_{\xi(\tau),\leq m_3}u\|_{L_{x_2}^2(\mathbb{R})} \big\|_{L^{1}_{t,x_1}(G^i_{\alpha} \times \mathbb{R})}.\label{estttttt1}
\end{align}

For \eqref{estttttt1}, we use the bilinear estimate \eqref{bilinearl2}, the compactness condition \eqref{almost-compact-2}, Proposition \ref{Bernstein} and H\"older's inequality to deduce that 
\begin{align}
\eqref{estttttt1}&\lesssim \sum\limits_{G^{m_3}_{\alpha^{\prime}}\subset G^i_{\alpha}}\big\|\|e^{i\tau \partial_{x_1}^2}P^{x_1}_{\xi(G^i_{\alpha}),m}\widetilde{F}_{m_3}(\tau)\|_{L_{x_2}^2(\mathbb{R})} \|P_{\xi(G^i_{\alpha}),l}^{x_1}u\|_{L_{x_2}^2(\mathbb{R})}\|P^{x_1}_{\xi(\tau), m_2}u_H\|_{L_{x_2}^2(\mathbb{R})}\|P^{x_1}_{\xi(\tau), m_3}u_H\|_{L_{x_2}^2(\mathbb{R})}\notag\\
&\hspace{4ex}\times\|P^{x_1}_{\xi(\tau),\leq m_3}u_H\|_{L_{x_2}^2(\mathbb{R})} \|P^{x_1}_{\xi(\tau),\leq m_3}u\|_{L_{x_2}^2(\mathbb{R})} \big\|_{L^{1}_{t,x_1}(G^{m_3}_{\alpha^{\prime}} \times \mathbb{R})}\notag\\
&\lesssim \sum\limits_{G^{m_3}_{\alpha^{\prime}}\subset G^i_{\alpha}}\big\|\|e^{i\tau \partial_{x_1}^2}P^{x_1}_{\xi(G^i_{\alpha}),m}\widetilde{F}_{m_3}(\tau)\|_{L_{x_2}^2(\mathbb{R})} \|P^{x_1}_{\xi(\tau), m_3}u_H\|_{L_{x_2}^2(\mathbb{R})} \|P^{x_1}_{\xi(\tau), \leq m_3}u_H\|_{L_{x_2}^2(\mathbb{R})}\big\|_{L^2_{t,x_1}(G^{m_3}_{\alpha}\times \mathbb{R})}\notag\\
&\hspace{4ex}\times  \big\|P^{x_1}_{\xi(G^i_{\alpha}), l}u\|_{L_{x_2}^2(\mathbb{R})} \|\|P^{x_1}_{\xi(\tau), m_2}u_H\|_{L_{x_2}^2(\mathbb{R})} \|P^{x_1}_{\xi(\tau), \leq m_3}u\|_{L_{x_2}^2(\mathbb{R})}\big\|_{L^2_{t,x_1}(G^{m_3}_{\alpha^{\prime}}\times \mathbb{R})}\notag\\
&\lesssim \sum\limits_{G^{m_3}_{\alpha^{\prime}}\subset G^i_{\alpha}}\big\|\|e^{i\tau \partial_{x_1}^2}P^{x_1}_{\xi(G^i_{\alpha}),m}\widetilde{F}_{m_3}(\tau)\|_{L_{x_2}^2(\mathbb{R})} \|P^{x_1}_{\xi(\tau), m_3}u_H\|_{L_{x_2}^2(\mathbb{R})} \big\|_{L^2_{t,x_1}(G^{m_3}_{\alpha}\times \mathbb{R})}\notag\\
&\hspace{4ex}\times \big\|\|P^{x_1}_{\xi(\tau), \leq m_3}u_H\|_{L_{x_2}^2(\mathbb{R})}\big\|_{L^{\infty}_{t,x_1}(G^{m_3}_{\alpha}\times \mathbb{R})}\notag\\ &\hspace{4ex}\times\big\|\|P^{x_1}_{\xi(G^i_{\alpha}), l}u\|_{L_{x_2}^2(\mathbb{R})} \|P^{x_1}_{\xi(\tau), m_2}u_H\|_{L_{x_2}^2(\mathbb{R})} \|P^{x_1}_{\xi(\tau), \leq m_3}u\|_{L_{x_2}^2(\mathbb{R})}\big\|_{L^2_{t,x_1}(G^{m_3}_{\alpha^{\prime}}\times \mathbb{R})}\notag\\
&\lesssim \eta_2 2^{(m_3-l)/2}\sum\limits_{G^{m_3}_{\alpha^{\prime}}\subset G^i_{\alpha}}\|P^{x_1}_{\xi(\tau), m_3}u\|_{U_\Delta^2(G_{\alpha^{\prime}}^{m_3},L_{x_1}^2L_{x_2}^2)}\|e^{it\partial_{x_1}^2}f^{m_3}_{\alpha^{\prime}}\|_{U_\Delta^2(G_{\alpha^{\prime}}^{m_3},L_{x_1}^2L_{x_2}^2)}\notag\\
&\hspace{4ex}\times\big\|\|P^{x_1}_{\xi(G^i_{\alpha}), l}u\|_{L_{x_2}^2(\mathbb{R})} \|P^{x_1}_{\xi(\tau), m_2}u_H\|_{L_{x_2}^2(\mathbb{R})} \|P^{x_1}_{\xi(\tau), \leq m_3}u\|_{L_{x_2}^2(\mathbb{R})}\big\|_{L^2_{t,x_1}(G^{m_3}_{\alpha^{\prime}}\times \mathbb{R})}\notag\\
&\lesssim \eta_2 2^{(m_3-l)/2}\|P^{x_1}_{\xi(\tau), m_3}u\|_{U_\Delta^2(G_{\alpha}^i,L_{x_1}^2L_{x_2}^2)}\notag\\
&\hspace{4ex}\times\big\|\|P^{x_1}_{\xi(G^i_{\alpha}), l}u\|_{L_{x_2}^2(\mathbb{R})} \|P^{x_1}_{\xi(\tau), m_2}u_H\|_{L_{x_2}^2(\mathbb{R})} \|P^{x_1}_{\xi(\tau), \leq m_3}u\|_{L_{x_2}^2(\mathbb{R})}\big\|_{L^2_{t,x_1}(G^{i}_{\alpha}\times \mathbb{R})}\label{esst1}.
\end{align}
So, 
\begin{align}
&\sum\limits_{\substack{0\leq m_3\leq m_2\\ \leq j-10}}\eqref{estttttt1}\lesssim \eta_2 \sum\limits_{\substack{0\leq m_3\leq m_2\\ \leq j-10}}2^{(m_3-l)/2}\|P^{x_1}_{\xi(\tau), m_3}u\|_{U_\Delta^2(G_{\alpha}^i,L_{x_1}^2L_{x_2}^2)}\notag\\
&\hspace{4ex}\times\big\|\|P^{x_1}_{\xi(G^i_{\alpha}), l}u\|_{L_{x_2}^2(\mathbb{R})} \|P^{x_1}_{\xi(\tau), m_2}u_H\|_{L_{x_2}^2(\mathbb{R})} \|P^{x_1}_{\xi(\tau), \leq m_3}u\|_{L_{x_2}^2(\mathbb{R})}\big\|_{L^2_{t,x_1}(G^{i}_{\alpha}\times \mathbb{R})}\notag\\
&\lesssim \eta_2 2^{(i-l)/2}\big(\sum\limits_{m_3\leq j-10} 2^{m_3-i}\|P^{x_1}_{\xi(\tau), m_3}u\|^2_{U_\Delta^2(G_{\alpha}^i,L_{x_1}^2L_{x_2}^2)}\big)^{1/2} \notag\\
&\times\big(\sum\limits_{m_3\leq j-10}\big(\sum\limits_{m_3\leq m_2\leq j-10}\big\|\|P^{x_1}_{\xi(G^i_{\alpha}), l}u\|_{L_{x_2}^2(\mathbb{R})} \|P^{x_1}_{\xi(\tau), m_2}u_H\|_{L_{x_2}^2(\mathbb{R})} \|P^{x_1}_{\xi(\tau), \leq m_3}u\|_{L_{x_2}^2(\mathbb{R})}\big\|_{L^2_{t,x_1}(G^{i}_{\alpha}\times \mathbb{R})}\big)^2\big)^{1/2}\notag\\
&\lesssim \eta_2 \|u\|_{X(G_{\alpha}^i,L_{x_1}^2L_{x_2}^2)}\notag\\
&\times\big(\sum\limits_{m_3\leq j-10}\big(\sum\limits_{m_3\leq m_2\leq j-10}\big\|\|P^{x_1}_{\xi(G^i_{\alpha}), l}u\|_{L_{x_2}^2(\mathbb{R})} \|P^{x_1}_{\xi(\tau), m_2}u_H\|_{L_{x_2}^2(\mathbb{R})} \|P^{x_1}_{\xi(\tau), \leq m_3}u\|_{L_{x_2}^2(\mathbb{R})}\big\|_{L^2_{t,x_1}(G^{i}_{\alpha}\times \mathbb{R})}\big)^2\big)^{1/2}.\label{essttt1}
\end{align}
	
For fixed $m_2$ and $m_3$, using H\"older's inequality, bilinear estimate \eqref{bilinearl2}, the compactness condition \eqref{almost-compact-2} and Proposition \ref{Bernstein}, we have
\begin{align}
&\big\|\|P^{x_1}_{\xi(G^i_{\alpha}), l}u\|_{L_{x_2}^2(\mathbb{R})} \|P^{x_1}_{\xi(\tau), m_2}u_H\|_{L_{x_2}^2(\mathbb{R})} \|P^{x_1}_{\xi(\tau), \leq m_3}u\|_{L_{x_2}^2(\mathbb{R})}\big\|^2_{L^2_{t,x_1}(G^{i}_{\alpha}\times \mathbb{R})}\notag\\
&\lesssim \big\|\|P^{x_1}_{\xi(G^i_{\alpha}), l}u\|_{L_{x_2}^2(\mathbb{R})} \|P^{x_1}_{\xi(\tau), m_2}u_H\|_{L_{x_2}^2(\mathbb{R})}\big\|^2_{L^2_{t,x_1}(G^{i}_{\alpha}\times \mathbb{R})}\big\| \|P^{x_1}_{\xi(\tau), \leq m_3}u\|_{L_{x_2}^2(\mathbb{R})}\big\|_{L^{\infty}_{t,x_1}(G^i_{\alpha} \times \mathbb{R})}\notag\\
&\lesssim  2^{m_3-l}\|P^{x_1}_{\xi(G^i_{\alpha}), m_3}u\|^2_{U_\Delta^2(G_{\alpha}^i,L_{x_1}^2L_{x_2}^2)}\|P^{x_1}_{\xi(\tau), m_2}u\|^2_{U_\Delta^2(G_{\alpha}^i,L_{x_1}^2L_{x_2}^2)}.\label{essttt2}
\end{align}
Inserting \eqref{essttt2} into \eqref{essttt1}, then 
\begin{align}
\eqref{essttt1}&\lesssim \eta_2 \|u\|_{X(G_{\alpha}^i,L_{x_1}^2L_{x_2}^2)}\notag\\
&\times\big(\sum\limits_{m_3\leq j-10}\big(\sum\limits_{m_3\leq m_2\leq j-10} 2^{(m_3-l)/2}\|P^{x_1}_{\xi(G^i_{\alpha}), m_3}u\|_{U_\Delta^2(G_{\alpha}^i,L_{x_1}^2L_{x_2}^2)}\|P^{x_1}_{\xi(\tau), m_2}u\|_{U_\Delta^2(G_{\alpha}^i,L_{x_1}^2L_{x_2}^2)}\big)^2\big)^{1/2}\notag\\
&\lesssim \eta_2 \|P^{x_1}_{\xi(G^i_{\alpha}),l}u\|_{U^2_{\Delta}(G_{\alpha}^i,L_{x_1}^2L_{x_2}^2)} \|u\|^2_{X(G_{\alpha}^i,L_{x_1}^2L_{x_2}^2)},\notag
\end{align}
that is,
\begin{align}
\sum\limits_{\substack{0\leq m_3\leq m_2\\ \leq j-10}}\|F^{m_3}_1\|_{U_\Delta^2(G_{\alpha}^i,L_{x_1}^2L_{x_2}^2)}\lesssim \eta_2 \|P^{x_1}_{\xi(G^i_{\alpha}),l}u\|^2_{U^2_{\Delta}(G_{\alpha}^i,L_{x_1}^2L_{x_2}^2)} \|u\|_{X(G_{\alpha}^i,L_{x_1}^2L_{x_2}^2)}.\label{essttt3}
\end{align}
	
Next we treat the final term $\sum\limits_{\substack{0\leq m_3\leq m_2\\ \leq j-10}}\|F^{m_3}_2\|_{U_\Delta^2(G_{\alpha}^i,L_{x_1}^2L_{x_2}^2)}$ .	
By  \eqref{triangle}, $$\sum\limits_{\substack{0\leq m_3\leq m_2\\ \leq j-10}}\|F^{m_3}_2\|_{U_\Delta^2(G_{\alpha}^i,L_{x_1}^2L_{x_2}^2)}\lesssim \sum\limits_{\substack{0\leq m_3\leq m_2\\ \leq j-10}}\big(\sum\limits_{G^{m_3}_{\alpha^{\prime}}\subset G^i_{\alpha}}\|F^{m_3}_2\|^2_{U_\Delta^2(G_{{\alpha}^{\prime}}^{m_3},L_{x_1}^2L_{x_2}^2)}\big)^{1/2}.$$
	
Noticing  by Lemma \ref{Up-dual} and H\"older's inequality, For any $G^i_{\alpha^{\prime}}$, we have
\begin{align}
&\|F^{m_3}_2\|_{U_\Delta^2(G_{{\alpha}^{\prime}}^{m_3},L_{x_1}^2L_{x_2}^2)}=\sup\limits_{ \|g^{m_3}_{\alpha^{\prime}}\|_{U^{12/5}_{\Delta}(G_{{\alpha}^{\prime}}^{m_3},L_{x_1}^2L_{x_2}^2)}=1}\int_{G^{m_3}_{\alpha^{\prime}}}\big\langle g^{m_3}_{\alpha^{\prime}},P_{\xi(G^i_{\alpha}), l}^{x_1}\Pi_{j_1}u P_{\xi(\tau),m_2}^{x_1}\overline{\Pi_{j_2}u_H} \notag\\
&\hspace{20ex}\times P_{\xi(\tau), m_3}^{x_1}\Pi_{j_3}u_H P_{\xi(\tau),\leq m_3}^{x_1}\overline{\Pi_{j_4}u_H}P_{\xi(\tau),\leq m_3}^{x_1}\Pi_{j_5}u)\big\rangle_{L^{2}_{x_1}L_{x_2}^2(\mathbb{R}  \times \mathbb{R})} d \tau\notag\\
&\hspace{13ex}\lesssim \sup\limits_{ \|g^{m_3}_{\alpha^{\prime}}\|_{U^{12/5}_{\Delta}(G_{{\alpha}^{\prime}}^{m_3},L_{x_1}^2L_{x_2}^2(\R))}=1}\big\|\|g^{m_3}_{\alpha^{\prime}}\|_{L_{x_2}^2(\R)}\|P^{x_1}_{\xi(\tau),m_3} u_H\|_{L_{x_2}^2}\|P^{x_1}_{\xi(\tau),\leq m_3}u_H\|_{L_{x_2}^2(\R)}\big\|_{L^2_{t,x_1}(G_{{\alpha}^{\prime}}^{m_3}\times\mathbb{R})}\notag\\
&\hspace{20ex}\times \big\|\|P_{\xi(G^i_{\alpha}),l}^{x_1}u\|_{L_{x_2}^2(\R)}\|P^{x_1}_{\xi(\tau),m_2} u_H\|_{L_{x_2}^2(\R)}\|P^{x_1}_{\xi(\tau),\leq m_3}u\|_{L_{x_2}^2(\R)}\big\|_{L^2_{t,x_1}(G_{{\alpha}^{\prime}}^{m_3}\times\mathbb{R})}\label{esttttttt1}.
\end{align}
	
For $\sup\limits_{ \|g^{m_3}_{\alpha^{\prime}}\|_{U^{12/5}_{\Delta}(G_{{\alpha}^{\prime}}^{m_3},L_{x_1}^2L_{x_2}^2)}=1}\big\|\|g^{m_3}_{\alpha^{\prime}}\|_{L_{x_2}^2(\R)}\|P^{x_1}_{\xi(\tau),m_3} u_H\|_{L_{x_2}^2(\R)}\|P^{x_1}_{\xi(\tau),\leq m_3}u_H\|_{L_{x_2}^2(\R)}\big\|_{L^2_{t,x_1}(G_{{\alpha}^{\prime}}^{m_3}\times\mathbb{R})}$, We first notice that by Proposition \ref{Bernstein} and \eqref{frequency-property},
\begin{align}
\sum\limits_{0\leq i\leq m_3}\big\|\|P^{x_1}_{\xi(t), i}u\|_{L_{x_2}^2(\mathbb{R})}\big\|_{L_{t,x_1}^{12}(G^{m_3}_{{\alpha}^{\prime}}\times \mathbb{R})}&\lesssim \sum\limits_{0\leq i\leq m_3}2^{i/4}\big\|\|P^{x_1}_{\xi(t), i}u\|_{L_{x_2}^2(\mathbb{R})}\big\|_{L_{t}^{12}L^{3}_{x_1}(G^{m_3}_{{\alpha}^{\prime}}\times \mathbb{R})}\notag\\
&\lesssim \sum\limits_{0\leq i\leq m_3}2^{i/4}\cdot2^{\frac{m_3-i}{12}}\|P^{x_1}_{\xi(t), i}u\|_{U^{2}_{\Delta}(G_{{\alpha}^{\prime}}^{m_3},L_{x_1}^2L_{x_2}^2)}\notag\\
&\lesssim 2^{{m_3}/4} \|u\|_{X(G_{{\alpha}^{\prime}}^{m_3},L_{x_1}^2L_{x_2}^2)}.\label{qqq1}
\end{align}
Thus, by  \eqref{qqq1}, bilinear estimate \eqref{bilinearl2},  H\"older's inequality and the embedding $U^{2}_{\Delta}(G_{{\alpha}^{\prime}}^{m_3},L_{x_1}^2L_{x_2}^2)\hookrightarrow U^{12/5}_{\Delta}(G_{{\alpha}^{\prime}}^{m_3},L_{x_1}^2L_{x_2}^2)$,
\begin{align}
&\sup\limits_{ \|g^{m_3}_{\alpha^{\prime}}\|_{U^{12/5}_{\Delta}(G_{{\alpha}^{\prime}}^{m_3},L_{x_1}^2L_{x_2}^2)}=1}\big\|\|g^{m_3}_{\alpha^{\prime}}\|_{L_{x_2}^2(\R)}\|P^{x_1}_{\xi(\tau),m_3} u_H\|_{L_{x_2}^2(\R)}\|P^{x_1}_{\xi(\tau),\leq m_3}u\|_{L_{x_2}^2(\R)}\big\|_{L^2_{t,x_1}(G_{{\alpha}^{\prime}}^{m_3}\times\mathbb{R})}\notag\\
&\hspace{4ex}\lesssim \sup\limits_{ \|g^{m_3}_{\alpha^{\prime}}\|_{U^{12/5}_{\Delta}(G_{{\alpha}^{\prime}}^{m_3},L_{x_1}^2L_{x_2}^2)}=1}\big\|\|g^{m_3}_{\alpha^{\prime}}\|_{L_{x_2}^2(\R)}\|P^{x_1}_{\xi(\tau),m_3} u_H\|_{L_{x_2}^2(\R)}\big\|_{L^{12/5}_{t,x_1}(G_{{\alpha}^{\prime}}^{m_3}\times\mathbb{R})}\notag\\
&\hspace{10ex}\times \big\|\|P^{x_1}_{\xi(\tau),\leq m_3}u\|_{L_{x_2}^2(\R)}\big\|_{L^{12}_{t,x_1}(G_{{\alpha}^{\prime}}^{m_3}\times\mathbb{R})}\notag\\
&\hspace{4ex}\lesssim 2^{({m_3}-l)/4}\sup\limits_{ \|g^{m_3}_{\alpha^{\prime}}\|_{U^{12/5}_{\Delta}(G_{{\alpha}^{\prime}}^{m_3},L_{x_1}^2L_{x_2}^2)}=1} \|g^{m_3}_{{\alpha}^{\prime}}\|_{U^{12/5}_{\Delta}(G_{{\alpha}^{\prime}}^{m_3},L_{x_1}^2L_{x_2}^2)}
\|u\|_{U^{12/5}_{\Delta}(G_{{\alpha}^{\prime}}^{m_3},L_{x_1}^2L_{x_2}^2)}\|u\|_{X(G_{{\alpha}^{\prime}}^{m_3},L_{x_1}^2L_{x_2}^2)}\notag\\
&\hspace{4ex}\lesssim 2^{({m_3}-l)/4}
\|P_{\xi(\tau), m_3 }^{x_1}u\|_{U^{2}_{\Delta}(G_{{\alpha}^{\prime}}^{m_3},L_{x_1}^2L_{x_2}^2)}\|u\|_{X(G_{{\alpha}^{\prime}}^{m_3},L_{x_1}^2L_{x_2}^2)}\notag\\
&\hspace{4ex}\lesssim 2^{({m_3}-l)/4}
\|u\|^2_{X(G_{{\alpha}^{\prime}}^{m_3},L_{x_1}^2L_{x_2}^2)}.\label{qqq2}
\end{align}
	
On the other hand, by bilinear estimate \eqref{bilinearl2}, compactness condition \eqref{almost-compact-2}, Proposition \ref{Bernstein} and H\"older's inequality,
\begin{align}
&\big\|\|P_{\xi(G^i_{\alpha}),l}^{x_1}u\|_{L_{x_2}^2(\R)}\|P^{x_1}_{\xi(\tau),m_2} u_H\|_{L_{x_2}^2(\R)}\|P^{x_1}_{\xi(\tau),\leq m_3}u_H\|_{L_{x_2}^2(\R)}\big\|_{L^2_{t,x_1}(G^i_{\alpha}\times\mathbb{R})}\notag\\
&\lesssim \big\|\|P_{\xi(G^i_{\alpha}),l}^{x_1}u\|_{L_{x_2}^2(\R)}\|P^{x_1}_{\xi(\tau),m_2} u_H\|_{L_{x_2}^2(\R)}\big\|_{L^2_{t,x_1}(G_{{\alpha}^{\prime}}^{m_3}\times\mathbb{R})}\big\|\|P^{x_1}_{\xi(\tau),\leq m_3}u_H\|_{L_{x_2}^2(\R)}\big\|_{L_{t,x_1}^{\infty}(G^i_{\alpha}\times\mathbb{R})}\notag\\
&\lesssim \eta_2 2^{{m_3}/2} 2^{-l/2}\|P_{\xi(G^i_{\alpha}),l}^{x_1}u\|_{U^{2}_{\Delta}(G^i_{\alpha},L_{x_1}^2L_{x_2}^2)} \|P_{\xi(\tau),m_2}^{x_1}u\|_{U^{2}_{\Delta}(G^i_{\alpha},L_{x_1}^2L_{x_2}^2)}\notag\\
&\lesssim \eta_2 2^{({m_3}-m_2)/2} 2^{(i-l)/2}\|P_{\xi(G^i_{\alpha}),l}^{x_1}u\|_{U^{2}_{\Delta}(G^i_{\alpha},L_{x_1}^2L_{x_2}^2)} \|u\|_{X(G^i_{\alpha},L_{x_1}^2L_{x_2}^2)}\label{qqq3}.
\end{align}
	
By \eqref{esttttttt1}, \eqref{qqq2} and \eqref{qqq3}, we have
\begin{align*}
&\big(\sum\limits_{G^{m_3}_{\alpha^{\prime}}\subset G^i_{\alpha}}\|F^{m_3}_2\|^2_{U_\Delta^2(G_{{\alpha}^{\prime}}^{m_3},L_{x_1}^2L_{x_2}^2)}\big)^{1/2}\lesssim 	2^{({m_3}-l)/4}
\sup\limits_{G_{{\alpha}^{\prime}}^{m_3}}\|u\|^2_{X(G_{{\alpha}^{\prime}}^{m_3},L_{x_1}^2L_{x_2}^2)}\\
&\hspace{16ex}\times\big(\sum\limits_{G^{m_3}_{\alpha^{\prime}}\subset G^i_{\alpha}} \big\|\|P_{\xi(G^i_{\alpha}),l}^{x_1}u\|_{L_{x_2}^2(\R)}\|P^{x_1}_{\xi(\tau),m_2} u_H\|_{L_{x_2}^2(\R)}\|P^{x_1}_{\xi(\tau),\leq m_3}u_H\|_{L_{x_2}^2(\R)}\big\|^2_{L^2_{t,x_1}(G_{{\alpha}^{\prime}}^{m_3}\times\mathbb{R})}\big)^{1/2}\\
&\hspace{8ex}\lesssim2^{({m_3}-l)/4}
\|u\|^2_{\widetilde{X}([a,b],L_{x_1}^2L_{x_2}^2)}\big\|\|P_{\xi(G^i_{\alpha}),l}^{x_1}u\|_{L_{x_2}^2}\|P^{x_1}_{\xi(\tau),m_2} u_H\|_{L_{x_2}^2(\R)}\|P^{x_1}_{\xi(\tau),\leq m_3}u_H\|_{L_{x_2}^2(\R)}\big\|_{L^2_{t,x_1}(G^i_{\alpha}\times\mathbb{R})}\\
&\hspace{8ex}\lesssim \eta_2 2^{({m_3}-l)/4}2^{({m_3}-m_2)/2} 2^{(i-l)/2}\|P_{\xi(G^i_{\alpha}),l}^{x_1}u\|_{U^{2}_{\Delta}(G^i_{\alpha},L_{x_1}^2L_{x_2}^2)} \|u\|^3_{\widetilde{X}([a, b],L_{x_1}^2L_{x_2}^2)}.
\end{align*}
	
Summing over $0\leq m_3\leq m_2\leq i-10$, we deduce that
\begin{align}
\sum\limits_{\substack{0\leq m_3\leq m_2\\ \leq j-10}}\|F^{m_3}_2\|_{U_\Delta^2(G_{\alpha}^i,L_{x_1}^2L_{x_2}^2)}&\lesssim \sum\limits_{\substack{0\leq m_3\leq m_2\\ \leq j-10}}\big(\sum\limits_{G^{m_3}_{\alpha^{\prime}}\subset G^i_{\alpha}}\|F^{m_3}_2\|^2_{U_\Delta^2(G_{{\alpha}^{\prime}}^{m_3},L_{x_1}^2L_{x_2}^2)}\big)^{1/2}\notag\\
&\lesssim \eta_2 2^{\frac{3(i-l)}{4}} \|P_{\xi(G^i_{\alpha}),l}^{x_1}u\|_{U^{2}_{\Delta}(G^i_{\alpha},L_{x_1}^2L_{x_2}^2)} \|u\|^3_{\widetilde{X}([a, b],L_{x_1}^2L_{x_2}^2)}.\label{essttttttt}
\end{align}
Combining \eqref{essst}, \eqref{estttttt111}, \eqref{essttt3} and \eqref{essttttttt}, we finish the proof.
	
\end{proof}

\subsection{The proof of long-time Strichartz estimate}
In this subsection, we give the proof of long-time Strichartz estimate(cf. Theorem \ref{longtimestrichartz}). Before presenting the proof, we first show the intermediate theorem, which is the key step in the proof of Theorem \ref{longtimestrichartz}.

\begin{theorem}[Intermediate theorem]\label{intermediate}
For any $i\geq 10$, $G_{\alpha}^i\subset G_k^j$ and $a^i_{\alpha} \in G^i_{\alpha}$, if  $N(G_{\alpha}^i)\leq \eta_3^{\frac{1}{2}}2^{i-5}$,  then there exist two constants  $\varepsilon(\eta_2,\eta_3)>0$ and $c>0$   we have the following multilinear estimate:
\begin{align}\label{intermediate1}
\bigg\|&\sum\limits_{j\in \mathbb{N}}\sum_{\substack{j_1-j_2+j_3-j_4+j_5=j\\j_1,j_2,j_3,j_4,j_5\in\Bbb{N}}}\int_{a_{\alpha}^i}^te^{i(t-\tau)\partial_{x_1}^2}P^{x_1}_{\xi(G^i_{\alpha}),m}\Pi_j\big(P_{\xi(G^i_{\alpha}),l}^{x_1}\Pi_{j_1}u\overline{\Pi_{j_2}u}\Pi_{j_3}u\overline{\Pi_{j_4}u}\Pi_{j_5}u\big)d \tau\bigg\|_{U_\Delta^2(G_{\alpha}^i,L_{x_1}^2L_{x_2}^2)}\notag\\
& \lesssim \varepsilon(\eta_2,\eta_3) 2^{-|m-l|/2}\|P_{\xi(G^i_{\alpha}),l}^{x_1}u\|_{U_\Delta^2(G_{{\alpha}}^{i},L_{x_1}^2L_{x_2}^2)}\big(1+\big\|u\big\|_{\widetilde{X}_{j}([a, b],L_{x_1}^2L_{x_2}^2)}\big)^{c}.
\end{align}
Moreover, $\varepsilon(\eta_2,\eta_3)\to0$ when $\eta_2,\eta_3\to0$ and  the constant $c$ does not depend on the choice of $\eta_2$ and $\eta_3$.
\end{theorem}

\begin{proof}
Just notice that 
\begin{align*}
&P^{x_1}_{\xi(G^i_{\alpha}),m}\Pi_j \big (P_{\xi(G^i_{\alpha}),l}^{x_1}\Pi_{j_1}u\overline{\Pi_{j_2}u}\Pi_{j_3}u\overline{\Pi_{j_4}u}\Pi_{j_5}u\big)=\\
&P^{x_1}_{\xi(G^i_{\alpha}),m}\Pi_j \big (P_{\xi(G^i_{\alpha}),l}^{x_1}\Pi_{j_1}uP^{x_1}_{\leq i-10}\overline{\Pi_{j_2}u}P^{x_1}_{\leq i-10}\Pi_{j_3}uP^{x_1}_{\leq i-10}\overline{\Pi_{j_4}u}P^{x_1}_{\leq i-10}\Pi_{j_5}u\big)\\
&+P^{x_1}_{\xi(G^i_{\alpha}),m}\Pi_j \mathcal{O}\big (P_{\xi(G^i_{\alpha}),l}^{x_1}\Pi_{j_1}uP^{x_1}_{\geq i-10}\overline{\Pi_{j_2}u}\Pi_{j_3}u\overline{\Pi_{j_4}u}\Pi_{j_5}u\big),
\end{align*}
where the symbol $\mathcal{O}$ represents that different frequncies are located in different $\Pi_{j_k}u$ for $k\in\{1,2,3,4,5\}$. 
Thus, we obtain
\begin{align}\label{qwer}
\bigg\|&\sum\limits_{j\in \mathbb{N}}\sum_{\substack{j_1-j_2+j_3-j_4+j_5=j\\j_1,j_2,j_3,j_4,j_5\in\Bbb{N}}}\int_{a_{\alpha}^i}^te^{i(t-\tau)\partial_{x_1}^2}P^{x_1}_{\xi(G^i_{\alpha}),m}\Pi_j\big(P_{\xi(G^i_{\alpha}),l}^{x_1}\Pi_{j_1}u\overline{\Pi_{j_2}u}\Pi_{j_3}u\overline{\Pi_{j_4}u}\Pi_{j_5}u\big)d \tau\bigg\|_{U_\Delta^2(G_{\alpha}^i,L_{x_1}^2L_{x_2}^2)}\notag\\
\lesssim \bigg\|&\sum\limits_{j\in \mathbb{N}}\sum_{\substack{j_1-j_2+j_3-j_4+j_5=j\\j_1,j_2,j_3,j_4,j_5\in\Bbb{N}}}\int_{a_{\alpha}^i}^te^{i(t-\tau)\partial_{x_1}^2}P^{x_1}_{\xi(G^i_{\alpha}),m}\Pi_j\big(P^{x_1}_{\xi(G^i_{\alpha}),m}\Pi_j \big (P_{\xi(G^i_{\alpha}),l}^{x_1}\Pi_{j_1}uP^{x_1}_{\leq i-10}\overline{\Pi_{j_2}u}\notag\\
&\hspace{2ex}\times P^{x_1}_{\leq i-10}\Pi_{j_3}uP^{x_1}_{\leq i-10}\overline{\Pi_{j_4}u}P^{x_1}_{\leq i-10}\Pi_{j_5}u\big)\big)d \tau\bigg\|_{U_\Delta^2(G_{\alpha}^i,L_{x_1}^2L_{x_2}^2)}\notag\\
+\bigg\|&\sum\limits_{j\in \mathbb{N}}\sum_{\substack{j_1-j_2+j_3-j_4+j_5=j\\j_1,j_2,j_3,j_4,j_5\in\Bbb{N}}}\int_{a_{\alpha}^i}^te^{i(t-\tau)\partial_{x_1}^2}P^{x_1}_{\xi(G^i_{\alpha}),m}\Pi_j\big(P^{x_1}_{\xi(G^i_{\alpha}),m}\Pi_j \mathcal{O}\big (P_{\xi(G^i_{\alpha}),l}^{x_1}\Pi_{j_1}uP^{x_1}_{\geq i-10}\overline{\Pi_{j_2}u}\notag\\
&\hspace{2ex}\times\Pi_{j_3}u\overline{\Pi_{j_4}u}\Pi_{j_5}u\big)d \tau\bigg\|_{U_\Delta^2(G_{\alpha}^i,L_{x_1}^2L_{x_2}^2)}.
\end{align}
Inserting \eqref{key1} and \eqref{key2} into \eqref{qwer}, we have proved \eqref{intermediate1} and hence complete the proof of Theorem \ref{intermediate}.
\end{proof}

Utilizing Theorem \ref{intermediate},  we  begin to prove the proof of long-time Strichartz estimate, i.e. Theorem \ref{longtimestrichartz}. 

\begin{proof}[Proof of Theorem \ref{longtimestrichartz}]
For each $0\leq j\leq k_0$, we fixed  $G^j_{\alpha}$ and compute $\|u\|_{X(G^j_{\alpha},L_{x_1}^2L_{x_2}^2)}$
	
We choose $a^i_{\alpha}\in G^i_{\alpha}$ which satisfies
\begin{align}\label{infinf}
\big\|P_{\xi(G^i_{\alpha}),i-2\leq \cdot \leq i+2}^{x_1}u(a^i_{\alpha})\big\|_{{L^2_{x_1}}}=\inf_{t\in G^i_{\alpha}}\big\|P_{\xi(G^i_{\alpha}),i-2\leq \cdot \leq i+2}^{x_1}\big\|_{L_{x_1}^{2}}.
\end{align}
By Duhamel's principle, we can write  $P_{\xi{G^i_{\alpha}},i-2\leq \cdot \leq i+2}^{x_1}u$ as
\begin{align*}
P_{\xi(G^i_{\alpha}),i-2\leq \cdot \leq i+2}^{x_1}u(t)&=P_{\xi(G^i_{\alpha}),i-2\leq \cdot \leq i+2}^{x_1}e^{i(t-a^i_{\alpha})\partial_{x_1}^2}u(a^i_{\alpha})\notag\\
&-\sum\limits_{j\in \mathbb{N}}\sum_{\substack{j_1-j_2+j_3-j_4+j_5=j\\j_1,j_2,j_3,j_4,j_5\in\Bbb{N}}}i\int^t_{a^i_{\alpha}}e^{i(t-\tau)\partial_{x_1}^2}P_{\xi(G^i_{\alpha}),i-2\leq \cdot \leq i+2}^{x_1}\Pi_j(\Pi_{j_1}u\overline{\Pi_{j_2}u}\Pi_{j_3}u\overline{\Pi_{j_4}u}\Pi_{j_5}u)d \tau,
\end{align*}
then by Strichartz estimate \eqref{Inhomo},
\begin{align*}
&\big\|P_{\xi(G^i_{\alpha}),i-2\leq \cdot \leq i+2}^{x_1}u\big\|_{{U^{2}_{\Delta}(G^i_{\alpha},L_{x_1}^2L_{x_2}^2)}}\lesssim \big\|P_{\xi(G^i_{\alpha}),i-2\leq \cdot \leq i+2}^{x_1}u(a^i_{\alpha})\big\|_{L^2_{x_1}(\mathbb{R})}\\
&\hspace{10ex}+\big\|\sum\limits_{j\in \mathbb{N}}\sum_{\substack{j_1-j_2+j_3-j_4+j_5=j\\j_1,j_2,j_3,j_4,j_5\in\Bbb{N}}}\int^t_{a^i_{\alpha}}e^{i(t-\tau)\partial_{x_1}^2}P_{\xi(G^i_{\alpha}),i-2\leq \cdot \leq i+2}^{x_1}\Pi_j(\Pi_{j_1}u\overline{\Pi_{j_2}u}\Pi_{j_3}u\overline{\Pi_{j_4}u}\Pi_{j_5}u)d \tau\big\|_{U^{2}_{\Delta}(G^i_{\alpha},L_{x_1}^2L_{x_2}^2)}.
\end{align*}
By the definition of $\|u\|_{X(G^j_{\alpha},L_{x_1}^2L_{x_2}^2)}$, 
\begin{align}
&\|u\|^2_{X(G^j_{\alpha},L_{x_1}^2L_{x_2}^2)}\lesssim \sum\limits_{0\leq i\leq j}2^{i-j}\sum\limits_{G^i_{\alpha}\subset G^j_k}\big\|P_{\xi(G^i_{\alpha}),i-2\leq \cdot \leq i+2}^{x_1}u(a^i_{\alpha})\big\|^2_{L^2_{x_1}(\mathbb{R})}+\sum\limits_{i> j}\big\|P_{\xi(G^j_{k}),i-2\leq \cdot \leq i+2}^{x_1}u(a^i_{\alpha})\big\|^2_{L^2_{x_1}(\mathbb{R})} \label{1qqq}\\
&+\sum\limits_{0\leq i\leq j}2^{i-j}\sum\limits_{G^i_{\alpha}\subset G^j_k}\big\|\sum\limits_{j\in \mathbb{N}}\sum_{\substack{j_1-j_2+j_3-j_4+j_5=j\\j_1,j_2,j_3,j_4,j_5\in\Bbb{N}}}\int^t_{a^i_{\alpha}}e^{i(t-\tau)\partial_{x_1}^2}P_{\xi(G^i_{\alpha}),i-2\leq \cdot \leq i+2}^{x_1}\Pi_j(\Pi_{j_1}u\overline{\Pi_{j_2}u}\Pi_{j_3}u\notag\\
&\hspace{22ex}\times\overline{\Pi_{j_4}u}\Pi_{j_5}u)d \tau\big\|^2_{U^{2}_{\Delta}(G^i_{\alpha},L_{x_1}^2L_{x_2}^2)}\notag\\
&+\sum\limits_{i>j}\big\|\sum\limits_{j\in \mathbb{N}}\sum_{\substack{j_1-j_2+j_3-j_4+j_5=j\\j_1,j_2,j_3,j_4,j_5\in\Bbb{N}}}\int^t_{a^i_{\alpha}}e^{i(t-\tau)\partial_{x_1}^2}P_{\xi(G^j_{k}),i-2\leq \cdot \leq i+2}^{x_1}\Pi_j(\Pi_{j_1}u\overline{\Pi_{j_2}u}\Pi_{j_3}u\overline{\Pi_{j_4}u}\Pi_{j_5}u)d \tau\big\|^2_{U^{2}_{\Delta}(G^j_{k},L_{x_1}^2L_{x_2}^2)}.\label{2qqq}
\end{align}
	
For \eqref{1qqq}, by compactness condition \eqref{eq-y5.33} and \eqref{infinf}
\begin{align}\label{keyestimate}
\eqref{1qqq}\lesssim \sum\limits_{i\geq 0}\eta_3^{-1}2^{-j}\int_{G^j_k}\|P_{\xi(t),i}^{x_1}u(t)\|_{L^2_{x_1}(\mathbb{R})}\big({N(t)}^3+\eta_3 \|u(t)\|^{{2(d+2)}/{d}}_{L_{x_1}^{{2(d+2)}/{d}}}\big) dt\lesssim 1.
\end{align}
	
For \eqref{2qqq}, we further decompose \eqref{2qqq} as the following four terms:
\begin{align}
&\eqref{2qqq}=\notag\\
&\sum\limits_{0\leq i\leq j}2^{i-j}\sum\limits_{\substack{G^i_{\alpha}\subset G^j_k\\N(G^i_{\alpha})\geq \eta_3^{1/2}2^{i-5}}}\big\|\sum\limits_{j\in \mathbb{N}}\sum_{\substack{j_1-j_2+j_3-j_4+j_5=j\\j_1,j_2,j_3,j_4,j_5\in\Bbb{N}}}\int^t_{a^i_{\alpha}}e^{i(t-\tau)\partial_{x_1}^2}P_{\xi(G^i_{\alpha}),i-2\leq \cdot \leq i+2}^{x_1}\Pi_j(\Pi_{j_1}u\overline{\Pi_{j_2}u}\notag\\
&\hspace{50ex}\times\Pi_{j_3}u\overline{\Pi_{j_4}u}\Pi_{j_5}u)d \tau\big\|_{U^{2}_{\Delta}(G^i_{\alpha},L_{x_1}^2L_{x_2}^2)}\label{qqw1}\\
&+\sum\limits_{\substack{ 0\leq i\leq j\\i\leq 10}}2^{i-j}\sum\limits_{\substack{G^i_{\alpha}\subset G^j_k\\N(G^i_{\alpha})< \eta_3^{1/2}2^{i-5}}}\big\|\sum\limits_{j\in \mathbb{N}}\sum_{\substack{j_1-j_2+j_3-j_4+j_5=j\\j_1,j_2,j_3,j_4,j_5\in\Bbb{N}}}\int^t_{a^i_{\alpha}}e^{i(t-\tau)\partial_{x_1}^2}P_{\xi(G^i_{\alpha}),i-2\leq \cdot \leq i+2}^{x_1}\Pi_j(\Pi_{j_1}u\overline{\Pi_{j_2}u}\notag\\
&\hspace{50ex}\times\Pi_{j_3}u\overline{\Pi_{j_4}u}\Pi_{j_5}u)d \tau\big\|_{U^{2}_{\Delta}(G^i_{\alpha},L_{x_1}^2L_{x_2}^2)}\label{qqw4}\\
&+\sum\limits_{\substack{ 0\leq i\leq j\\i\geq 10}}2^{i-j}\sum\limits_{\substack{G^i_{\alpha}\subset G^j_k\\N(G^i_{\alpha})< \eta_3^{1/2}2^{i-5}}}\big\|\sum\limits_{j\in \mathbb{N}}\sum_{\substack{j_1-j_2+j_3-j_4+j_5=j\\j_1,j_2,j_3,j_4,j_5\in\Bbb{N}}}\int^t_{a^i_{\alpha}}e^{i(t-\tau)\partial_{x_1}^2}P_{\xi(G^i_{\alpha}),i-2\leq \cdot \leq i+2}^{x_1}\Pi_j(\Pi_{j_1}u\overline{\Pi_{j_2}u}\notag\\
&\hspace{50ex}\times\Pi_{j_3}u\overline{\Pi_{j_4}u}\Pi_{j_5}u)d \tau\big\|_{U^{2}_{\Delta}(G^i_{\alpha},L_{x_1}^2L_{x_2}^2)}\label{qqw2}\vspace{1ex}\\
&+\sum\limits_{i>j}\big\|\sum\limits_{j\in \mathbb{N}}\sum_{\substack{j_1-j_2+j_3-j_4+j_5=j\\j_1,j_2,j_3,j_4,j_5\in\Bbb{N}}}\int^t_{a^i_{\alpha}}e^{i(t-\tau)\partial_{x_1}^2}P_{\xi(G^j_{k}),i-2\leq \cdot \leq i+2}^{x_1}\Pi_j(\Pi_{j_1}u\overline{\Pi_{j_2}u}\Pi_{j_3}u\overline{\Pi_{j_4}u}\Pi_{j_5}u)d \tau\big\|_{U^{2}_{\Delta}(G^j_{k},L_{x_1}^2L_{x_2}^2)}\label{qqw3}	
\end{align}
	
For any interval $G^i_{\alpha}$ with $N(G^i_{\alpha})\geq \eta_3^{1/2}2^{i-5}$, \eqref{xi(t)} yields that $N(t)\geq \eta^{1/2}_32^{i-6}$ for all $t\in G^i_{\alpha}$. Moreover, by Lemma \ref{remark-small}, if $J_{l}$ is a small interval, then we have
\begin{align}\label{nonlinearboundness}
\big\|\sum\limits_{j\in \mathbb{N}}\sum_{\substack{j_1-j_2+j_3-j_4+j_5=j\\j_1,j_2,j_3,j_4,j_5\in\Bbb{N}}}P_{\xi(G^j_{k}),i-2\leq \cdot \leq i+2}^{x_1}\Pi_j(\Pi_{j_1}u\overline{\Pi_{j_2}u}\Pi_{j_3}u\overline{\Pi_{j_4}u}\Pi_{j_5}u)\big\|_{L^1_{t}L^{2}_{x_1}L^2_{x_2}(J_l\times \mathbb{R}\times \mathbb{R})}\lesssim 1.
\end{align}
By Strichartz estimate \eqref{Inhomo}, compactness condition \eqref{almost-compact-1}-\eqref{key} and Remark \ref{remark-small}, we deduce that 
	
\begin{align}
&\eqref{qqw1}\notag\\
&\lesssim  \sum\limits_{0\leq i\leq j}2^{i-j}\sum\limits_{\substack{G^i_{\alpha}\subset G^j_k\\N(G^i_{\alpha})\geq \eta_3^{1/2}2^{i-5}}} 
\sum\limits_{J_{l}\subset G^j_k}\big\|\sum\limits_{j\in \mathbb{N}}\sum_{\substack{j_1-j_2+j_3-j_4+j_5=j\\j_1,j_2,j_3,j_4,j_5\in\Bbb{N}}}P_{\xi(G^i_{\alpha}),i-2\leq \cdot \leq i+2}^{x_1}\Pi_j(\Pi_{j_1}u\overline{\Pi_{j_2}u}\notag\\
&\hspace{55ex}\times\Pi_{j_3}u\overline{\Pi_{j_4}u}\Pi_{j_5}u)d \tau\big\|^2_{L^1_tL^2_{x_1}L_{x_2}^2((J_{l}\cap G^i_{\alpha}) \times \mathbb{R}\times \mathbb{R})}\notag\\
&+\sum\limits_{0\leq i\leq j}2^{i-j}\sum\limits_{\substack{G^i_{\alpha}\subset G^j_k\\N(G^i_{\alpha})\geq \eta_3^{1/2}2^{i-5}}} 
\sum\limits_{\substack{J_{l}\not\subseteq G^j_k\\J_l\cap G^j_k \neq \emptyset}}\big\|\sum\limits_{j\in \mathbb{N}}\sum_{\substack{j_1-j_2+j_3-j_4+j_5=j\\j_1,j_2,j_3,j_4,j_5\in\Bbb{N}}}P_{\xi(G^i_{\alpha}),i-2\leq \cdot \leq i+2}^{x_1}\Pi_j(\Pi_{j_1}u\overline{\Pi_{j_2}u}\vspace{2ex}\notag\\
&\hspace{55ex}\times\Pi_{j_3}u\overline{\Pi_{j_4}u}\Pi_{j_5}u)d \tau\big\|^2_{L^1_tL^2_{x_1}L_{x_2}^2((J_{l}\cap G^i_{\alpha}) \times \mathbb{R}\times \mathbb{R})}\notag\\
&\lesssim \sum\limits_{0\leq i\leq j}2^{i-j}\sum\limits_{\substack{J_{l}\subset G^j_k\\N(J_{l})\geq \eta_3^{1/2}2^{i-6}}}\big\|\sum\limits_{j\in \mathbb{N}}\sum_{\substack{j_1-j_2+j_3-j_4+j_5=j\\j_1,j_2,j_3,j_4,j_5\in\Bbb{N}}}P_{\xi(G^i_{\alpha}),i-2\leq \cdot \leq i+2}^{x_1}\Pi_j(\Pi_{j_1}u\overline{\Pi_{j_2}u}\notag\\
&\hspace{55ex}\times\Pi_{j_3}u\overline{\Pi_{j_4}u}\Pi_{j_5}u)d \tau\big\|^2_{L^1_tL^2_{x_1}L_{x_2}^2(J_{l} \times \mathbb{R}\times \mathbb{R})}\notag\\
&+\sum\limits_{0\leq i\leq j}2^{i-j}\sum\limits_{\substack{G^i_{\alpha}\subset G^j_k\\N(G^i_{\alpha})\geq \eta_3^{1/2}2^{i-5}}} 
\sum\limits_{\substack{J_{l}\not\subseteq G^j_k\\J_l\cap G^j_k \neq \emptyset}}\big\|\sum\limits_{j\in \mathbb{N}}\sum_{\substack{j_1-j_2+j_3-j_4+j_5=j\\j_1,j_2,j_3,j_4,j_5\in\Bbb{N}}}P_{\xi(G^i_{\alpha}),i-2\leq \cdot \leq i+2}^{x_1}\Pi_j(\Pi_{j_1}u\overline{\Pi_{j_2}u}\notag\\
&\hspace{55ex}\times\Pi_{j_3}u\overline{\Pi_{j_4}u}\Pi_{j_5}u)d \tau\big\|^2_{L^1_tL^2_{x_1}L_{x_2}^2((J_{l}\cap G^i_{\alpha}) \times \mathbb{R}\times \mathbb{R})}\notag\\
&\lesssim \eta_3^{-1/2}2^{-j}\sum\limits_{J_l\subset G^j_k} N(J_l)\notag\\
&+\sum\limits_{0\leq i\leq j}2^{i-j}\sum\limits_{\substack{G^i_{\alpha}\subset G^j_k\\N(G^i_{\alpha})\geq \eta_3^{1/2}2^{i-5}}} 
\sum\limits_{\substack{J_{l}\not\subseteq G^j_k\\J_l\cap G^j_k \neq \emptyset}}\big\|\sum\limits_{j\in \mathbb{N}}\sum_{\substack{j_1-j_2+j_3-j_4+j_5=j\\j_1,j_2,j_3,j_4,j_5\in\Bbb{N}}}P_{\xi(G^i_{\alpha}),i-2\leq \cdot \leq i+2}^{x_1}\Pi_j(\Pi_{j_1}u\overline{\Pi_{j_2}u}\notag\\
&\hspace{55ex}\times\Pi_{j_3}u\overline{\Pi_{j_4}u}\Pi_{j_5}u)d \tau\big\|^2_{L^1_tL^2_{x_1}L_{x_2}^2((J_{l}\cap G^i_{\alpha}) \times \mathbb{R}\times \mathbb{R})}.\notag
\end{align}
Notice that for any $G^j_k$, there are at most two interval $J_1$ and $J_2$ that intersect $G^j_k$ but are not contained in $G^j_k$, we continue:
\begin{align}
&\lesssim \eta_3^{-1/2}2^{-j}\sum\limits_{J_l\subset G^j_k} N(J_l)\notag\\
&+\sum\limits_{0\leq i\leq j}2^{i-j}\sum\limits_{\substack{G^i_{\alpha}\subset G^j_k\\N(G^i_{\alpha})\geq \eta_3^{1/2}2^{i-5}}} 
\big\|\sum\limits_{j\in \mathbb{N}}\sum_{\substack{j_1-j_2+j_3-j_4+j_5=j\\j_1,j_2,j_3,j_4,j_5\in\Bbb{N}}}P_{\xi(G^i_{\alpha}),i-2\leq \cdot \leq i+2}^{x_1}\Pi_j(\Pi_{j_1}u\overline{\Pi_{j_2}u}\notag\\
&\hspace{50ex}\times\Pi_{j_3}u\overline{\Pi_{j_4}u}\Pi_{j_5}u)d \tau\big\|^2_{L^1_tL^2_{x_1}L_{x_2}^2((J_1\cap G^i_{\alpha}) \times \mathbb{R}\times \mathbb{R})}\notag\\
&+\sum\limits_{0\leq i\leq j}2^{i-j}\sum\limits_{\substack{G^i_{\alpha}\subset G^j_k\\N(G^i_{\alpha})\geq \eta_3^{1/2}2^{i-5}}} 
\big\|\sum\limits_{j\in \mathbb{N}}\sum_{\substack{j_1-j_2+j_3-j_4+j_5=j\\j_1,j_2,j_3,j_4,j_5\in\Bbb{N}}}P_{\xi(G^i_{\alpha}),i-2\leq \cdot \leq i+2}^{x_1}\Pi_j(\Pi_{j_1}u\overline{\Pi_{j_2}u}\notag\\
&\hspace{50ex}\times\Pi_{j_3}u\overline{\Pi_{j_4}u}\Pi_{j_5}u)d \tau\big\|^2_{L^1_tL^2_{x_1}L_{x_2}^2((J_{2}\cap G^i_{\alpha}) \times \mathbb{R}\times \mathbb{R})}\notag\\
&\lesssim  \sum\limits_{0\leq i\leq j}2^{i-j}+\eta_3^{-1/2}2^{-j}\sum\limits_{J_l\subset G^j_k} N(J_l)\notag\\
&\lesssim 1\label{key111}.
\end{align}
	
For \eqref{qqw4}, one can follow the similar strategy above to show that  
\begin{align}\label{key222}
\eqref{qqw4}\lesssim 1.
\end{align}
	
For \eqref{qqw2}, notice that for any $j_1,j_2,j_3,j_4,j_5,j \in{\mathbb{N}}$, $$P^{x_1}_{m} \Pi_j(P^{x_1}_{\leq i-5}\Pi_{j_1}uP^{x_2}_{\leq i-5}\overline{\Pi_{j_2}u}P^{x_3}_{\leq i-5}\Pi_{j_3}uP^{x_1}_{\leq i-5}\overline{\Pi_{j_4}u}P^{x_5}_{\leq i-5}\Pi_{j_5}u)=0.$$
Therefore, by Theorem \ref{intermediate} and H\"older's inequality,
\begin{align}\eqref{qqw2}
&\lesssim \varepsilon(\eta_2,\eta_3)\sum\limits_{10\leq i\leq j} 2^{i-j}\sum\limits_{G^i_{\alpha}\subset G^j_{k}}\bigg(\sum\limits_{i-2\leq m \leq i+2}\sum\limits_{l\geq i-5} 2^{c(d)(m-l)}\big\|P^{x_1}_{\xi(G^i_{\alpha}),l}u\big\|_{U^{2}_{\Delta}(G^i_{\alpha},L_{x_1}^2L_{x_2}^2)}\notag\\
&\hspace{50ex}\times\big(1+\big\|u\big\|_{\widetilde{X}_{i}([a, b],L_{x_1}^2L_{x_2}^2)}\big)^{c}\bigg)^2\notag\\
&\lesssim  \varepsilon(\eta_2,\eta_3)\big(1+\big\|u\big\|_{\widetilde{X}_{j}([a, b],L_{x_1}^2L_{x_2}^2)}\big)^{c}\sum\limits_{10\leq i\leq j} 2^{i-j}\sum\limits_{G^i_{\alpha}\subset G^j_{k}}\bigg(\sum\limits_{l\geq i-5} 2^{c(d)(i-l)}\big\|P^{x_1}_{\xi(G^i_{\alpha}),l}u\big\|_{U^{2}_{\Delta}(G^i_{\alpha},L_{x_1}^2L_{x_2}^2)}\bigg)^2\notag\\
&\lesssim \varepsilon(\eta_2,\eta_3)\big(1+\big\|u\big\|_{\widetilde{X}_{j}([a, b],L_{x_1}^2L_{x_2}^2)}\big)^{c}\sum\limits_{10\leq i\leq j} 2^{i-j}\sum\limits_{G^i_{\alpha}\subset G^j_{k}}\big(\sum\limits_{l\geq i-5} 2^{c(d)(i-l)}\big)\notag\\
&\hspace{50ex}\times\bigg(\sum\limits_{l\geq i-5} 2^{c(d)(i-l)}\big\|P^{x_1}_{\xi(G^i_{\alpha}),l}u\big\|^2_{U^{2}_{\Delta}(G^i_{\alpha},L_{x_1}^2L_{x_2}^2)}\bigg)\notag\\
&\lesssim \varepsilon(\eta_2,\eta_3)\big(1+\big\|u\big\|_{\widetilde{X}_{j}([a, b],L_{x_1}^2L_{x_2}^2)}\big)^{c}\sum\limits_{10\leq i\leq j} 2^{i-j}\sum\limits_{G^i_{\alpha}\subset G^j_{k}}\sum\limits_{i-5\leq l \leq j} 2^{c(d)(i-l)}\big\|P^{x_1}_{\xi(G^i_{\alpha}),l}u\big\|^2_{U^{2}_{\Delta}(G^i_{\alpha},L_{x_1}^2L_{x_2}^2)}\label{key221}\\
&+ \varepsilon(\eta_2,\eta_3)\big(1+\big\|u\big\|_{\widetilde{X}_{j}([a, b],L_{x_1}^2L_{x_2}^2)}\big)^{c}\sum\limits_{10\leq i\leq j} 2^{i-j}\sum\limits_{G^i_{\alpha}\subset G^j_{k}}\sum\limits_{l > j} 2^{c(d)(i-l)}\big\|P^{x_1}_{\xi(G^i_{\alpha}),l}u\big\|^2_{U^{2}_{\Delta}(G^i_{\alpha},L_{x_1}^2L_{x_2}^2)}\label{key22221}.
\end{align}
	
Notice that the interval $G^i_{\alpha}$ intersects $2^{j-i}$ intervals $G^i_{\alpha}$. So 
\begin{align}\eqref{key22221}
&\lesssim \varepsilon(\eta_2,\eta_3)\big(1+\big\|u\big\|_{\widetilde{X}_{j}([a, b],L_{x_1}^2L_{x_2}^2)}\big)^{c}\sum\limits_{10\leq i\leq j}2^{i-j}2^{j-i}\sum\limits_{l > j} 2^{c(d)(i-l)}\big\|P^{x_1}_{\xi(G^i_{\alpha}),l}u\big\|^2_{U^{2}_{\Delta}(G^j_{k},L_{x_1}^2L_{x_2}^2)}\notag\\
&\lesssim \varepsilon(\eta_2,\eta_3)\big(1+\big\|u\big\|_{\widetilde{X}_{j}([a, b],L_{x_1}^2L_{x_2}^2)}\big)^{c}\sum\limits_{10\leq i\leq j}2^{c(d)(i-j)}\sum\limits_{l > j} \big\|P^{x_1}_{\xi(t),l}u\big\|^2_{U^{2}_{\Delta}(G^j_{k},L_{x_1}^2L_{x_2}^2)}\notag\\
&\lesssim \varepsilon(\eta_2,\eta_3)\big(1+\big\|u\big\|_{\widetilde{X}_{j}([a, b],L_{x_1}^2L_{x_2}^2)}\big)^{c}\|u\|^2_{X(G^j_{k},L_{x_1}^2L_{x_2}^2)}.\label{key2222}
\end{align}
	
Also, when $i-5\leq l\leq j$, every $G^i_{\alpha}$ overlap at most $100 \cdot 2^{l-i} $ intervals $G^l_{{\alpha}^{\prime}}$, thus
\begin{align*}
\sum\limits_{G^i_{\alpha}\subset G^j_k}\big\|P^{x_1}_{\xi(G^i_{\alpha}),l}u\big\|_{U^{2}_{\Delta}(G^i_{\alpha},L_{x_1}^2L_{x_2}^2)}&\lesssim 2^{l-i}\big\|P^{x_1}_{\xi(t),l}u\big\|^2_{U^{2}_{\Delta}(G^j_{k},L_{x_1}^2L_{x_2}^2)}.
\end{align*}
Further,
\begin{align}\eqref{key221}
&\lesssim \sum\limits_{10\leq i\leq j} 2^{i-j}\sum\limits_{i-5\leq l \leq j} 2^{c(d)(i-l)}2^{l-i}\big\|P^{x_1}_{\xi(t),l}u\big\|^2_{U^{2}_{\Delta}(G^j_{k},L_{x_1}^2L_{x_2}^2)}\notag\\
&\lesssim \varepsilon(\eta_2,\eta_3)\big(1+\big\|u\big\|_{\widetilde{X}_{j}([a, b],L_{x_1}^2L_{x_2}^2)}\big)^{c}\sum\limits_{10\leq i\leq j}\sum\limits_{i-5\leq l \leq j} 2^{c(d)(i-l)}2^{l-j}\big\|P^{x_1}_{\xi(t),l}u\big\|^2_{U^{2}_{\Delta}(G^j_{k},L_{x_1}^2L_{x_2}^2)}\notag\\
&\lesssim  \varepsilon(\eta_2,\eta_3)\big(1+\big\|u\big\|_{\widetilde{X}_{j}([a, b],L_{x_1}^2L_{x_2}^2)}\big)^{c}\sum\limits_{0\leq l\leq j}2^{l-j}\big\|P^{x_1}_{\xi(t),l}u\big\|^2_{U^{2}_{\Delta}(G^j_{k},L_{x_1}^2L_{x_2}^2)}\big(\sum\limits_{0\leq i\leq l+5}2^{c(d)(i-l)}\big)\notag\\
&\lesssim \varepsilon(\eta_2,\eta_3)\big(1+\big\|u\big\|_{\widetilde{X}_{j}([a, b],L_{x_1}^2L_{x_2}^2)}\big)^{c}\|u\|^2_{X(G^j_{k},L_{x_1}^2L_{x_2}^2)}.\label{key2221}
\end{align}
Combining \eqref{key2222} and \eqref{key2221}, we have
\begin{align}\label{key33}
\eqref{qqw2}\lesssim \varepsilon(\eta_2,\eta_3)\big(1+\big\|u\big\|_{\widetilde{X}_{j}([a, b],L_{x_1}^2L_{x_2}^2)}\big)^{c}\|u\|^2_{X(G^j_{k},L_{x_1}^2L_{x_2}^2)}.
\end{align}
	
Finally, if $N(G^j_k)\leq \eta_3^{1/2}$, then one can use the same argument in dealing with the term in \eqref{qqw1} to show that
\begin{align}\label{key44}
\eqref{qqw3}\lesssim 1.
\end{align}
If $N(G^j_k)\leq \eta_3^{1/2}$, then one can use the same argument in dealing with the term in \eqref{qqw2} to show that
\begin{align}\label{key55}
\eqref{qqw3}\lesssim \|u\|^2_{X(G^j_{k},L_{x_1}^2L_{x_2}^2)}.
\end{align}
	
Collecting the estimate \eqref{key111}, \eqref{key222}, \eqref{key33}, \eqref{key44} and \eqref{key55}, then for any  $j\geq 10$ and any interval $G^j_{\alpha}$,
\begin{align*}
\|u\|_{X(G^j_{\alpha},L_{x_1}^2L_{x_2}^2)}\lesssim 1+\varepsilon(\eta_2,\eta_3)\big(1+\big\|u\big\|_{\widetilde{X}_{j}([a, b],L_{x_1}^2L_{x_2}^2)}\big)^{c}\|u\|^2_{X(G^j_{k},L_{x_1}^2L_{x_2}^2)}.
\end{align*} 
By the definition of $\widetilde{X}_{j}([a, b],L_{x_1}^2L_{x_2}^2)$, 
	
\begin{align}\label{bootstrap1}
\|u\|_{\widetilde{X}_{j}([a,b],L_{x_1}^2L_{x_2}^2)}\lesssim 1+\varepsilon(\eta_2,\eta_3)\big(1+\big\|u\big\|_{\widetilde{X}_{j}([a, b],L_{x_1}^2L_{x_2}^2)}\big)^{c}\|u\|^2_{\widetilde{X}_{j}([a,b],L_{x_1}^2L_{x_2}^2)}, \quad j\geq 10.
\end{align}
	
On the other hand, using the same argument in dealing with the term in \eqref{qqw2}, one can easily check that
\begin{align}\label{bootstrap2}
\|u\|_{\widetilde{X}_{0}(G^j_{\alpha},L_{x_1}^2L_{x_2}^2)}\lesssim 1.
\end{align}
Next by Definition \ref{definition}, $G^{j+1}_k= G^{j}_{2k} \bigcup G^{j}_{2k+1}$ with $G^{j}_{2k} \bigcap G^{j}_{2k+1}=\varnothing$, and for $0\leq i\leq j$, if $G^{i}_{\alpha} \subset G^{j+1}_k= G^{j}_{2k} \bigcup G^{j}_{2k+1}$, then either $G^{i}_{\alpha} \subset G^{j}_{2k}$ or $G^{i}_{\alpha} \subset G^{j}_{2k+1}$ would happen. Thus
\begin{equation}\label{eq-y5.8}
\begin{split}
&\sum_{0\leq i< j+1}2^{i-(j+1)}\sum_{G_{\alpha}^{i}\subset G_k^{j+1}}\|P_{\xi(G_{\alpha}^{i}), i-2\leq\cdot\leq i+2}u\|^2_{U^{2}_{\Delta}(G^i_{\alpha},L_{x_1}^2L_{x_2}^2)}\\
\leq& 2^{-1}\sum_{0\leq i< j}2^{i-j} \big[\sum_{G_{\alpha}^{i}\subset G_{2k}^{j}}\|P_{\xi(G_{\alpha}^{i}), i-2\leq\cdot\leq i+2}u\|^2_{U^2_{\Delta}(G^j_{2k};L_{x_1}^2L^2_{x_2})}+\sum_{G_{\alpha}^{i}\subset G_{2k+1}^{j}}\|P_{\xi(G_{\alpha}^{i}), i-2\leq\cdot\leq i+2}u\|^2_{U^{2}_{\Delta}(G^j_{2k+1},L_{x_1}^2L_{x_2}^2)}\big]\\
&+ 2^{-1} \big[\|P_{\xi(G_{2k}^{j}), j-2\leq\cdot\leq j+2}u\|^2_{U^{2}_{\Delta}(G^j_{2k},L_{x_1}^2L_{x_2}^2)} + \|P_{\xi(G_{2k+1}^{j}), j-2\leq\cdot\leq j+2}u\|^2_{U^{2}_{\Delta}(G^j_{2k+1},L_{x_1}^2L_{x_2}^2)}\big]\\
\leq& \frac{1}{2}\big[ \|u\|^2_{X(G_{2k}^{j})}+\|\u\|^2_{X(G_{2k+1}^{j})}\big].
\end{split}
\end{equation}
Meanwhile notice that for all $t \in G_k^{j+1}$, from \eqref{key}, we have
\begin{equation}
|\xi(t)-\xi(G_k^{j+1})|\leq 2^{-18}\eta_3 \eta_{1}^{-1/2}.
\end{equation}
Therefore, for all $t \in G_k^{j+1}$ and $i\geq j$,
$$\{\xi:2^{i-1}\leq |\xi-\xi(t)|\leq 2^{i+1} \} \subset \{\xi:2^{i-2}\leq |\xi-\xi(G_k^{j+1})|\leq 2^{i+2} \} \subset \{\xi:2^{i-3}\leq |\xi-\xi(t)|\leq 2^{i+3} \},$$
which, combined with  \eqref{triangle}, yield
\begin{equation}\label{eq-y5.10}
\begin{split}
&\sum_{i\geq j+1}\|P_{\xi(G_{k}^{j+1}),i-2\leq\cdot\leq i+2} u\|^2_{U^{2}_{\Delta}(G^{j+1}_{k},L_{x_1}^2L_{x_2}^2)}\\
\leq & \sum_{i\geq j+1} \big[\|P_{\xi(G_{k}^{j+1}),i-2\leq\cdot\leq i+2}u\|^2_{U^{2}_{\Delta}(G^j_{2k},L_{x_1}^2L_{x_2}^2)} + \|P_{\xi(G_{k}^{j+1}),i-2\leq\cdot\leq i+2} u\|^2_{U^{2}_{\Delta}(G^j_{2k+1},L_{x_1}^2L_{x_2}^2)} \big]\\
\leq & \sum_{i\geq j+1} \big[\|P_{\xi(G_{2k}^{j}),i-3\leq\cdot\leq i+3}u\|^2_{U^{2}_{\Delta}(G^j_{2k},L_{x_1}^2L_{x_2}^2)} + \|P_{\xi(G_{2k+1}^{j}),i-3\leq\cdot\leq i+3} u\|^2_{U^{2}_{\Delta}(G^j_{2k+1},L_{x_1}^2L_{x_2}^2)} \big]\\
\end{split}
\end{equation}
Thus \eqref{eq-y5.8} and \eqref{eq-y5.10} give
\begin{equation}
\begin{split}
\|u\|^2_{X(G_k^{j+1})}&=\sum_{0\leq i< j+1}2^{i-(j+1)}\sum_{G_{\alpha}^{i}\subset G_k^{j+1}}\|P_{\xi(G_{\alpha}^{i}), i-2\leq\cdot\leq i+2}u\|^2_{U^2_{\Delta}(G_{\alpha}^{i},L_{x_1}^2L_{x_2}^2)}\notag\\
&\hspace{5ex}+\sum_{i\geq j+1}\|P_{\xi(G_{k}^{j+1}),i-2\leq\cdot\leq i+2} u\|^2_{U^2_{\Delta}(G_{j+1}^{k},L_{x_1}^2L_{x_2}^2)}\\
&\leq \big[ \|u\|^2_{X(G_{2k}^{j},L_{x_1}^2L_{x_2}^2)}+\|u\|^2_{X(G_{2k+1}^{j},L_{x_1}^2L_{x_2}^2)}\big],
\end{split}
\end{equation}
which directly implies
\begin{align}\label{bootstrap3}
\|u\|^2_{\widetilde{X}_{k_{\ast}+1}([a,b],L_{x_1}^2L_{x_2}^2)}\leq 2\|u\|^2_{\widetilde{X}_{k_{\ast}}([a,b],L_{x_1}^2L_{x_2}^2)}.
\end{align}
	
Finally, thanks to    \eqref{bootstrap1},  \eqref{bootstrap2} and \eqref{bootstrap3}, if we choose sufficiently small $\varepsilon(\eta_2,\eta_3)$, then by a standard bootstrap argument, 
\begin{align}\label{longtimeestimate}
\|u\|_{\widetilde{X}_{k_{0}}([a,b],L_{x_1}^2L_{x_2}^2)}\lesssim 1.
\end{align}
\end{proof}




\section{Frequency-localized Morawetz estimate}\label{sec:Morawetz-frequency}

In this section, we prove the low-frequency localized interaction Morawetz estimates which is crucial in the exclusion of the critical element in the next section. Unlike the case of \cite{Cheng-Guo-Guo-Liao-Shen}, we need to give additional calculus to derive the following interaction Morawetz estimates.

\begin{theorem}[Low-frequency localized interaction Morawetz estimates]\label{frequency-localized}
Let $v(t,x_1,x_2)$ be the almost-periodic solution as in Theorem \ref{almost-periodic} with $\int_{0}^{T}N(t)^3dt=\eta_3K$, then we have
\begin{align}
\bigg\|\int_{\R}\partial_{x_1}\left(\left|P_{\leqslant K}^{x_1}u(t,x_1,x_2)\right|^2\right)dx_2\bigg\|_{L_{t,x_1}^2([0,T]\times\R)}^2\lesssim o(K)+\sup_{t\in[0,T]}M_I(t),
\end{align}
where the symbol $o(K)$ denotes the quantity in the sense that 
\begin{align*}
\frac{o(K)}{K}\rightarrow0,\mbox{ as }K\to\infty
\end{align*}
and  the interaction Morawetz action is defined by
\begin{align*}
M_I(t)&=\int_{\R\times\R}\int_{\R\times\R}\frac{x_1-\widetilde{x_1}}{|x_1-\widetilde{x_1}|}|P_{\leqslant K}^{x_1}{u}(t,\widetilde{x_1},\widetilde{x_2})|^2\Im(\overline{P_{\leqslant K}^{x_1}u}\partial_{x_1}u)(t,x_1,x_2)dx_1d\widetilde{x_1}dx_2d\widetilde{x_2}\\
&\hspace{2ex}+\int_{\R\times\R}\int_{\R\times\R}\frac{x_1-\widetilde{x_1}}{|x_1-\widetilde{x_1}|}|P_{\leqslant K}^{x_1}u(t,\widetilde{x_1},\widetilde{x_2})|^2\Im(\overline{P_{\leqslant K}^{x_1}{u}}\partial_{x_1}{u})(t,x_1,x_2)dx_1d\widetilde{x_1}dx_2d\widetilde{x_2}.
\end{align*}
\end{theorem}

Before giving the frequency localized interaction Morawetz estimate, we need the following two lemmas:
\begin{lemma}\label{le5.30v9}
\noindent For $F_n(u)=\sum\limits_{\substack{n_1,n_2,n_3,n_4,n_5\in\N\\n_1-n_2+n_3-n_4+n_5=n}}\Pi_n(\Pi_{n_1}u\overline{\Pi_{n_2}u}\Pi_{n_3}u\overline{\Pi_{n_4}u}\Pi_{n_5}u)$, we have:
\begin{equation*}
\sum_{n\in\N}\left\{F_n(u) ,\phi\right\}_p=-\frac{1}{3}\sum_{n\in\N} \partial\left(\bar{\phi}, F_n(u)\right),
\end{equation*}
\noindent where $\{f,g\}_p := \Re(f\partial \bar{g}-g \partial \bar{f})$ is the momentum bracket.
\end{lemma}

\begin{proof}
The proof of this lemma is direct by using the symmetry of resonant system  and the definition of the momemtum bracket.
\end{proof}

Arguing like \cite{Cheng-Guo-Zhao}, we also have the following positive estimate:
\begin{lemma}\label{le4.27v40}
For any  bounded sequence $u\in L_{x_1}^2\mathcal{H}_{x_2}^1(\R\times\R)$, we have
\begin{equation*}
\sum\limits_{\substack{n_1,n_2,n_3,n_4,n_5,n\in \mathbb{N}, \\ n+n_2+n_4=n_1+n_3+n_5} }  \overline{\Pi_nu} \Pi_{n_1}u\overline{\Pi_{n_2}u}\Pi_{n_3}u\overline{\Pi_{n_4}u}\Pi_{n_5}u\geq 0 .
\end{equation*}
\end{lemma}

\begin{remark}\label{rm4.30}
The above two basic lemmas are to ensure that the term related to the momentum bracket appear in the calculus of interaction Morawetz estimate is positive. Thus we can extract this term. 
\end{remark}

\begin{proof}[Proof of Theorem \ref{frequency-localized}]
Recall that the generalized interaction Morawetz action with a weight $a(x_1-\widetilde{x_1})$,
\begin{align*}
M_I^0(t)&=\int_{\R\times\R}\int_{\R\times\R}a(x_1-\widetilde{x_1})|u(t,\widetilde{x_1},\widetilde{x_2})|^2\Im(\overline{u}\partial_{x_1}u)(t,x_1,x_2)dx_1d\widetilde{x_1}dx_2d\widetilde{x_2}\\
&\hspace{2ex}+\int_{\R\times\R}\int_{\R\times\R}a(x_1-\widetilde{x_1})|u(t,\widetilde{x_1},\widetilde{x_2})|^2\Im(\overline{u}\partial_{x_1}u)(t,x_1,x_2)dx_1d\widetilde{x_1}dx_2d\widetilde{x_2}.
\end{align*}
 Integrating by parts, together with Lemmas \ref{le5.30v9} and \ref{le4.27v40}, then as in \cite{Dodson-d=1}, we can prove
\begin{equation}\label{eq4.20v66}
\bigg\|\int_{\R}\partial_{x_1}(|{u}(t,x_1,x_2)|^2)dx_2\bigg\|_{L_{t,x_1}^2([0,T]\times\R)}^2\lesssim \int_0^T  M'(t) \,\mathrm{d}t  \lesssim \sup\limits_{t\in [0,T]}|M(t)|.
\end{equation}
 Similar as in \cite{Dodson-d=3,dodson-d=2,Dodson-d=1}, since $\dot{H}^1_{x_1}$ can not be controlled by $L^2_{x_1}$ norm, thus we need to add the low-frequency cut-off on the interaction Morawetz action $M_I(t)$. Due to the low-frequency cut-off, the error term will appear inevitably and our goal is to control the error term. More precisely, we have
\begin{equation*}
\bigg\|\int_{\R}\partial_{x_1}(|P_{\leq K}^{x_1}{u}(t,x_1,x_2)|^2)dx_2\bigg\|_{L_{t,x_1}^2([0,T]\times\R)}^2 \lesssim \int_0^T M_I'(t) \mathrm{d}t + \mathcal{E} \lesssim \sup\limits_{t\in [0,T]}|M_I(t)| + \mathcal{E},
\end{equation*}
where
\begin{align*}
&\mathcal{E} =2\int_{0}^{T}\int_{\R}\int_{\R}\int_{\R}\int_{\R}a(x_1-\widetilde{x_1})\Im(\overline{P_{\leq K}^{x_1}u}P_{\leq K}^{x_1}F(u))(t,\widetilde{x_1},\widetilde{x_2})\Im(\overline{P_{\leq K}^{x_1}u}\partial_{x_1}P_{\leq K}^{x_1}u)(t,x_1,x_2)dx_1d\widetilde{x_1}dx_2d\widetilde{x_2}dt\\
&\hspace{2ex}+\int_{0}^{T}\int_{\R}\int_{\R}\int_{\R}\int_{\R}a(x_1-\widetilde{x_1})|P_{\leq K}^{x_1}u(t,\widetilde{x_1},\widetilde{x_2})|^2\Re\big((\overline{P_{\leq K}^{x_1}F(u)}-\overline{F(P_{\leq K}u)})\notag\\
&\hspace{20ex}\times\partial_{x_1}P_{\leq K}u\big)(t,x_1,x_2)dx_1d\widetilde{x_1}dx_2d\widetilde{x_2}dt \\
&\hspace{2ex}-\int_{0}^{T}\int_{\R}\int_{\R}\int_{\R}\int_{\R}a(x_1-\widetilde{x_1})|P_{\leq K}^{x_1}u(t,\widetilde{x_1},\widetilde{x_2})|^2\Re(\overline{P_{\leq K}^{x_1}u}\notag\\
&\hspace{20ex}\times\partial_{x_1}(F(P_{\leq K}^{x_1}u)-P_{\leq K}^{x_1}F(u)))(t,x_1,x_2)dx_1d\widetilde{x_1}dx_2d\widetilde{x_2}dt.
\end{align*}
	
It suffices to prove that $ \mathcal{E} \leq o(K)$. In order to prove this, we use the fact that $a(x)$ is odd implies that $\mathcal{E}$ is Galilean invariant with respect to the $x_1$ variable. Thus, we can write $\mathcal{E}$ to the following 
\begin{equation}\label{error-terms}
\aligned
\mathcal{E} &=2\int_{0}^{T}\int_{\R^4}a(x_1-\widetilde{x_1})\Im(\overline{P_{\leq K}^{x_1}u}P_{\leq K}^{x_1}F(u))(t,\widetilde{x_1},\widetilde{x_2})\Im(\overline{P_{\leq K}^{x_1}u}(\partial_{x_1}-i\xi(t))P_{\leq K}^{x_1}u)(t,x_1,x_2)dx_1d\widetilde{x_1}dx_2d\widetilde{x_2}dt\\
&\hspace{2ex}+\int_{0}^{T}\int_{\R^4}a(x_1-\widetilde{x_1})|P_{\leq K}^{x_1}u(t,\widetilde{x_1},\widetilde{x_2})|^2\Re\big((\overline{P_{\leq K}^{x_1}F(u)}-\overline{F(P_{\leq K}u)})\notag\\
&\hspace{20ex}\times(\partial_{x_1}-i\xi(t))P_{\leq K}u\big)(t,x_1,x_2)dx_1d\widetilde{x_1}dx_2d\widetilde{x_2}dt \\
&\hspace{2ex}-\int_{0}^{T}\int_{\R^4}a(x_1-\widetilde{x_1})|P_{\leq K}^{x_1}u(t,\widetilde{x_1},\widetilde{x_2})|^2\Re(\overline{P_{\leq K}^{x_1}u}(\partial_{x_1}-i\xi(t))(F(P_{\leq K}^{x_1}u)-P_{\leq K}^{x_1}F(u))(t,x_1,x_2)dx_1d\widetilde{x_1}dx_2d\widetilde{x_2}dt\\
&:=   I + II + III.
\endaligned
\end{equation}
	
Now it suffices to control the three  error terms in \eqref{error-terms} respectively. We introduce the useful scaling transform.
For $\lambda = \frac{2^{k_0}}{K}$, let $u_\lambda(t,x_1,x_2)= \lambda^\frac12 u(\lambda^2 t, \lambda x,y)$.  Then, we have
\begin{equation}\label{est-galilean}
\aligned
\left\|\left(\partial_{x_1}-i\xi(t)\right)P_{\le  K}^{x_1} u\right\|_{L^4_tL^{\infty}_{x_1}L_{x_2}^2\left(\left[0,T\right]\times \R\times\R\right)}
&=2^{-k_0}K \left\|\left(\partial_{x_1}-i \lambda \xi(t)\right)P_{\le \lambda  K}^{x_1}u_{\lambda}\right\|_{L^4_tL^{\infty}_{x_1}L_{x_2}^2\left(\left[0,\frac{T}{\lambda^2}\right]\times \R\times\R\right)} \\
&\lesssim 2^{-k_0}K\sum_{j=0}^{k_0+2}2^j \left\|P_{\lambda\xi(t),j}u_{\lambda}\right\|_{L^4_tL^{\infty}_{x_1}L_{x_2}^2\left(\left[0,\frac{T}{\lambda^2}\right]\times \R\times\R\right)}   \lesssim K.
\endaligned
\end{equation}
As in \cite{Cheng-Guo-Guo-Liao-Shen,Dodson-d=1,dodson-d=2}, we use a useful trick of frequency decomposition, i.e. considering $u= P^{x_1}_{\leq \frac{K}{32}}u+P^{x_1}_{\geq \frac{K}{32}}u$ and the following basic fact:

\begin{lemma}
Let $\varphi$ be a real-valued radially symmetric bump function satisfying 
\begin{align*}
\varphi(x)=\begin{cases}
1,&|x|\leq1,\\
0,&|x|\geq2,
\end{cases}
\end{align*}
then	we have for any $\xi_1, \xi_2\in \mathbb{R}$, 
\begin{equation}
\left|\varphi\left(\frac{ \xi_2+\xi_1}{K} \right)- \varphi\left(\frac{ \xi_1}{K} \right)\right| \lesssim \frac{|\xi_2|}{K}. 
\end{equation}
\end{lemma}
To simplify the notations, we denote $\Pi_{n_i}u$ by $u_{n_i}$, $i=1,2,3,4,5$;  $u_{n_i}^{l}=P_{\leq\frac{K}{32}}^{x_1}u_{n_i}$ and $u_{n_i}^h=P_{\geq\frac{K}{32}}^{x_1}u_{n_i}$. 
	
By using the H\"older inequality, Bernstein's inequality and the conservation of mass, we obtain that
\begin{align*}
\mathcal{E}&\lesssim\|N_1\|_{L_{t}^\frac{4}{3}L_{x_1}^1L_{x_2}^2}\|(\partial_{x_1}-i\xi(t))P_{\leq K}^{x_1}u\|_{L_{t}^4L_{x_1}^\infty L_{x_2}^2}\|P_{\leq K}^{x_1}u\|_{L_{t}^\infty L_{x_1,x_2}^2}^2.
\end{align*}
	
Next, we will give the estimate of $N_1=P_{\leq K}^{x_1}F(u)-F(P_{\leq K }^{x_1}u)$. By the high-low frequency truncation mentioned above, we can rewrite the term in $N_1$ to the following and obtain that
\begin{align}
\|N_1\|_{L_{t}^\frac{4}{3}L_{x_1}^{1}L_{x_2}^2\left(\left[0,T\right]\times \R\times\R\right)}&=
\bigg\|P_{\leq K}^{x_1}\bigg(\sum_{\substack{n_1,n_2,n_3,n_4,n_5,n\in\N\\n_1-n_2+n_3-n_4+n_5=n}}\Pi_n(u_{n_1}^l\overline{u_{n_2}^l}u_{n_3}^l\overline{u_{n_4}^l}u_{n_5}^l)\bigg)\label{N-est1}\\
&\hspace{8ex}-\sum_{\substack{n_1,n_2,n_3,n_4,n_5,n\in\N\\n_1-n_2+n_3-n_4+n_5=n}}\Pi_n(u_{n_1}^l\overline{u_{n_2}^l}u_{n_3}^l\overline{u_{n_4}^l}u_{n_5}^l)\bigg\|_{L_{t}^\frac{4}{3}L_{x_1}^{1}L_{x_2}^2\left(\left[0,T\right]\times \R\times\R\right)}\notag\\
&+\bigg\|P_{\leq K}^{x_1}\bigg(\sum_{\substack{n_1,n_2,n_3,n_4,n_5,n\in\N\\n_1-n_2+n_3-n_4+n_5=n}}\Pi_n\mathcal{O}(u_{n_1}^l\overline{u_{n_2}^l}u_{n_3}^l\overline{u_{n_4}^l}u_{n_5}^h)\bigg)\notag\\
&\hspace{8ex}-\sum_{\substack{n_1,n_2,n_3,n_4,n_5,n\in\N\\n_1-n_2+n_3-n_4+n_5=n}}\Pi_n\mathcal{O}(u_{n_1}^l\overline{u_{n_2}^l}u_{n_3}^l\overline{u_{n_4}^l}P_{\leq K}^{x_1}u_{n_5}^h)\bigg\|_{L_{t}^\frac{4}{3}L_{x_1}^{1}L_{x_2}^2\left(\left[0,T\right]\times \R\times\R\right)}\label{N-est2}\\
&+\bigg\|\sum_{\substack{n_1,n_2,n_3,n_4,n_5,n\in\N\\n_1-n_2+n_3-n_4+n_5=n}}\Pi_n\mathcal{O}\big(u_{n_1}^h\overline{u_{n_2}^h}u_{n_3}^l\overline{u_{n_4}^l}u_{n_5}^l\big)\bigg\|_{L_{t}^\frac{4}{3}L_{x_1}^{1}L_{x_2}^2\left(\left[0,T\right]\times \R\times\R\right)}\label{N-est3}\\
&+\bigg\|\sum_{\substack{n_1,n_2,n_3,n_4,n_5,n\in\N\\n_1-n_2+n_3-n_4+n_5=n}}\Pi_n\big(u_{n_1}^h\overline{u_{n_2}^h}u_{n_3}^h\overline{u_{n_4}^h}u_{n_5}^h\big)\bigg\|_{L_{t}^\frac{4}{3}L_{x_1}^{1}L_{x_2}^2\left(\left[0,T\right]\times \R\times\R\right)}\label{N-est4}\\
&:=J_1+J_2+J_3+J_4,\notag
\end{align}
where $\mathcal{O}$ include the conjugate part and the different frequency.  We claim that 
\begin{align}\label{claim1}
J_1+J_2+J_3+J_4\lesssim o(1).
\end{align}
By the claim \eqref{claim1} and the H\"older inequality, we can verify that 
\begin{align*}
II&=\int_{0}^{T}\int_{\R^4}a(x_1-\widetilde{x_1})|P_{\leq K}^{x_1}u(t,\widetilde{x_1},\widetilde{x_2})|^2\Re\big((\overline{P_{\leq K}^{x_1}F(u)}-\overline{F(P_{\leq K}u)})\\
&\hspace{4ex}\times(\partial_{x_1}-i\xi(t))P_{\leq K}u\big)(t,x_1,x_2)dx_1d\widetilde{x_1}dx_2d\widetilde{x_2}dt\\
&\lesssim o(K).
\end{align*}
	
Now, it remains to prove the claim \eqref{claim1}. By the direct computation, we can show that 
\begin{align*}
P_{\leq K}^{x_1}\bigg(\sum_{\substack{n_1,n_2,n_3,n_4,n_5,n\in\N\\n_1-n_2+n_3-n_4+n_5=n}}\Pi_n\mathcal{O}(u_{n_1}^l\overline{u_{n_2}^l}u_{n_3}^l\overline{u_{n_4}^l}u_{n_5}^l)\bigg)=\sum_{\substack{n_1,n_2,n_3,n_4,n_5,n\in\N\\n_1-n_2+n_3-n_4+n_5=n}}\Pi_n\mathcal{O}\big(u_{n_1}^l\overline{u_{n_2}^l}u_{n_3}^lu_{n_4}^l\overline{u_{n_5}^h}\big).
\end{align*}
This implies $J_1=0$.
	
We turn to give the estimate of the term $J_2$. First, recall the following multiplier estimate
\begin{align*}
\big|\varphi\big(\frac{\xi_2+\xi_1}{K}\big)-\varphi\big(\frac{\xi_1}{K}\big)\big|\lesssim\frac{|\xi_2|}{K}\sup_{|\xi|\sim K}|\partial_{x}\varphi(\xi)|.
\end{align*}
Then using the Fourier inverse formula and the fact that $P_{\leq K}^{x_1}$ can commute with $\Pi_n$, we can obtain that 
\begin{align*}
J_2&\lesssim\frac{1}{K}\bigg\|\sum_{\substack{n_1,n_2,n_3,n_4,n_5,n\in\N\\n_1-n_2+n_3-n_4+n_5=n}}\Pi_n(P_{\leq K}^{x_1}u_{n_5}^h\partial_{x_1}(u_{n_1}^l\overline{u_{n_2}^l}u_{n_3}^l\overline{u_{n_4}^l}))\bigg\|_{L_{t}^\frac{4}{3}L_{x_1}^1L_{x_2}^2([0,T]\times\R\times\R)}.
\end{align*}
Now we give the estimate the general term of the summation above, i.e. 
\begin{align*}
J_2(n_1,n_2,n_3,n_4,n_5)=P_{\leq K}^{x_1}u_{n_5}^h\partial_{x_1}(u_{n_1}^l\overline{u_{n_2}^l}u_{n_3}^l\overline{u_{n_4}^l}).
\end{align*}
Using the H\"older inequality, we have
\begin{align*}
\|J_2(n_1,n_2,n_3,n_4,n_5)\|_{L_{t}^\frac{4}{3}L_{x_1}^1([0,T]\times\R)}&\lesssim\|u_{n_5}^hu_{n_3}^l\overline{u_{n_4}^l}\|_{L_{t,x_1}^2([0,T]\times\R)}\|\partial_{x_1}(e^{-ix_1\xi(t)}u_{n_1}^l\overline{u_{n_2}^le^{-ix_1\xi(t)}})\|_{L_{t}^{4}L_{x_1}^2([0,T]\times\R)}.
\end{align*}
We remark that in the following inequality, we act the derivatives to $u_{n_1}^l\overline{u_{n_2}^l}$. The case of $\partial_{x_1}(u_{n_3}^l\overline{u_{n_4}^l})$ can be treated  similarly.
Then taking the $L_{x_2}^2$ norm, we have
\begin{align*}
&\|J_2\|_{L_{t}^\frac{4}{3}L_{x_1}^1L_{x_2}^2([0,T]\times\R\times\R)}\\
\lesssim&\frac{1}{K}\big\|\|u^h\|_{L_{x_2}^2(\R)}\|u^l\|_{L_{x_2}^2(\R)}\|\overline{u^l}\|_{L_{x_2}^2(\R)}\big\|_{L_{t,x_1}^2([0,T]\times\R)}\big\|\partial_{x_1}(\|e^{-ix_1\xi(t)}u^l\|_{L_{x_2}^2(\R)}\|\overline{e^{-ix_1\xi(t)}u^l}\|_{L_{x_2}^2(\R)})\big\|_{L_{t}^{4}L_{x_1}^2([0,T]\times\R)}.
\end{align*}
By using the intermediate estimate, i.e. Theorem \ref{intermediate}, we have 
\begin{align*}
\|J_2\|_{L_{t}^\frac{4}{3}L_{x_1}^1L_{x_2}^2([0,T]\times\R\times\R)}\lesssim\frac{1}{K}\big\|\partial_{x_1}(\|e^{-ix_1\xi(t)}u^l\|_{L_{x_2}^2(\R)}\|\overline{u^le^{-ix_1\xi(t)}}\|_{L_{x_2}^2(\R)})\big\|_{L_{t}^{4}L_{x_1}^2([0,T]\times\R)}.
\end{align*} 
	
Using the compactness of almost-periodic solution and \eqref{est-galilean}, we have
\begin{align*}
\big\|\|P_{\xi(t),\geq C(\eta)N(t)}u\|_{L_{x_2}^2(\R)}(\partial_{x_1}-i\xi(t))\|u^l\|_{L_{x_2}^2(\R)}\big\|_{L_{t}^4L_{x_1}^2([0,T]\times\R)}\lesssim \eta K.
\end{align*}
We further assume that $C(\eta)\gg\eta_1^{-1}$. Then by the compactness condition \eqref{almost-compact-1}, we have that
\begin{align*}
&\sum_{10C(\eta)\leq 2^j\leq \frac{K}{C(\eta)}}2^jN(J_l)\big\|(\|P_{\xi(J_l),\geq C(\eta)N(J_l)}^{x_1}u\|_{L_{x_2}^2(\R)})(\|P_{\xi(t),2^jN(J_l)}^{x_1}u\|_{L_{x_2}^2(\R)})\big\|_{L_{t,x_1}^2(J_l\times\R)}^\frac{1}{2}\\
&\hspace{2ex}\times\big\|\|P_{\xi(J_l),2^jN(J_l)}^{x_1}u\|_{L_{x_2}^2(\R)}\big\|_{L_t^\infty L_{x_1}^2(J_l\times\R)}^\frac{1}{2}\big\|\|P_{\xi(J_l),\leq C(\eta)N(t)}^{x_1}u\|_{L_{x_2}^2(\R)}\big\|_{L_{t,x_1}^\infty(J_l\times\R)}^\frac{1}{2}\\
&\lesssim o(K^\frac{3}{4})\big(C(\eta)N(J_l)\big)^\frac{1}{4}.
\end{align*}
In the last inequality, we use the bilinear Strichartz estimates. On the other hand, we have that
\begin{align*}
10C(\eta)N(J_l)\big\|\|P_{\xi(J_l),\leq C(\eta)N(J_l)}u\|_{L_{x_2}^2(\R)}\big\|_{L_t^4L_{x_1}^\infty(J_l\times\R)}\big\|\|P_{\xi(J_l),\leq 10C(\eta)N(J_l)}u\|_{L_{x_2}^2(\R)}\big\|_{L_t^\infty L_{x_2}^2(J_l\times\R)}\lesssim C(\eta)N(J_l).
\end{align*}
Summing over all $J_l\subset[0,T]$ and using the condition
\begin{align*}
\int_{0}^{T}N(t)^3dt=\eta_3 K,
\end{align*}  we finally get the control for $J_2$ when $K\to\infty$(i.e. $\eta(K)\to0$),
\begin{align*}
\|J_2\|_{L_{t}^\frac{4}{3}L_{x_1}^1L_{x_2}^2([0,T]\times\R\times\R)}\lesssim\frac{1}{K}(\eta K+C(\eta)K^\frac{1}{4}+C(\eta)o(K))\lesssim o(1).
\end{align*}
	
Next, we give the proof of $J_3$ and $J_4$.  By a direct computation, 
\begin{align*}
\|J_4+J_3\|_{L_t^\frac{4}{3}L_{x_1}^1L_{x_2}^2([0,T]\times\R\times\R)}&\lesssim\bigg\|\sum_{\substack{n_1,n_2,n_3,n_4,n_5,n\in\N\\n_1-n_2+n_3-n_4+n_5=n}}\Pi_n\mathcal{O}\big(u_{n_1}^h\overline{u_{n_2}^h}u_{n_3}^l\overline{u_{n_4}^l}u_{n_5}^l\big)\bigg\|_{L_{t}^\frac{4}{3}L_{x_1}^{1}L_{x_2}^2\left(\left[0,T\right]\times \R\times\R\right)}\notag\\
&\hspace{2ex}+\bigg\|\sum_{\substack{n_1,n_2,n_3,n_4,n_5,n\in\N\\n_1-n_2+n_3-n_4+n_5=n}}\Pi_n\big(u_{n_1}^h\overline{u_{n_2}^h}u_{n_3}^h\overline{u_{n_4}^h}u_{n_5}^h\big)\bigg\|_{L_{t}^\frac{4}{3}L_{x_1}^{1}L_{x_2}^2\left(\left[0,T\right]\times \R\times\R\right)}\\
&\lesssim\big\|\|u^h\|_{L_{x_2}^2(\R)}\|\overline{u^h}\|_{L_{x_2}^2(\R)}\|{u^l}\|_{L_{x_2}^2(\R)}\|\overline{u^l}\|_{L_{x_2}^2(\R)}\|u^l\|_{L_{x_2}^2(\R)}\big\|_{L_{t}^\frac43L_{x_1}^1([0,T]\times\R)}\\
&\hspace{2ex}+\big\|\|u^h\|_{L_{x_2}^2(\R)}\|\overline{u^h}\|_{L_{x_2}^2(\R)}\|{u^l}\|_{L_{x_2}^2(\R)}\|\overline{u^h}\|_{L_{x_2}^2(\R)}\|u^h\|_{L_{x_2}^2(\R)}\big\|_{L_{t}^\frac43L_{x_1}^1([0,T]\times\R)}\\
&\lesssim\big\|\|\|u^h\|_{L_{x_2}^2(\R)}\|u^l\|_{L_{x_2}^2(\R)}\|\overline{u^l}\|_{L_{x_2}^2(\R)}\|_{L_{t,x_1}^2}^\frac{3}{2}\big\|\|u^h\|_{L_{x_2}^2(\R)}\|_{L_t^\infty L_{x_1}^2([0,T]\times\R)}^\frac{1}{2}\\
&\hspace{2ex}+\big\|\|u^h\|_{L_{x_2}^2(\R)}\big\|_{L_t^4L_{x_1}^\infty}^3\big\|\|u^h\|_{L_{x_2}^2(\R)}\big\|_{L_t^\infty L_{x_1}^2([0,T]\times\R)}^3\\
&\lesssim o(1),
\end{align*}
where we use the long-time Strichartz estimate and the bilinear Strichartz estimates in the proof of long-time Strichartz estimate. 
Combining the estimate above, we obtain that 
\begin{align*}
II\lesssim o(K).
\end{align*}	
	
For the third term $III$,   integrating by parts, we have
\begin{align} \label{III-1}
III &  =- \int_{0}^{T}\int_{\R^4}a(x_1-\widetilde{x_1})|P_{\leq K}^{x_1}u(t,\widetilde{x_1},\widetilde{x_2})|^2\notag\\
&\hspace{18ex}\times\Re\big((\overline{P_{\leq K}^{x_1}F(u)}-\overline{F(P_{\leq K}u)})(\partial_{x_1}-i\xi(t))P_{\leq K}u\big)(t,x_1,x_2)dx_1d\widetilde{x_1}dx_2d\widetilde{x_2}dt \\
& \quad  - \int_0^T\int_{\R^4}
\partial_{x_1} a(x_1-\widetilde{x_1})|P_{\le K}^{x_1} u(t,\widetilde{x_1},\widetilde{x_2})|^2\notag\\
&\hspace{18ex}\times \Re\left( \overline{P_{\le  K}^{x_1} u(t,x_1,x_2)}
\left(F \left( P_{\le  K}^{x_1} u \right)- P_{\le  K}^{x_1} F \left(u\right)\right)(t,x_1,x_2) \right)  dx_1d\widetilde{x_1}dx_2d\widetilde{x_2}dt\label{III-2}.
\end{align}
The estimate of \eqref{III-1} is similar to the estimate of $II$, thus we have 
\begin{align*}
\eqref{III-1}\lesssim o(K).
\end{align*}
It is sufficient to prove the term \eqref{III-2}. We note that $\partial_{x_1} a(x_1-\widetilde{x_1})$ is an $L^1$ function  thus we can  use Young's inequality, \eqref{claim1} and conservation law,
\begin{equation*}
\operatorname{\eqref{III-2}}\lesssim \left\|P_{\le  K}^{x_1} u\right\|^3_{L_t^{12}L_{x_1}^{3}L_{x_2}^2([0,T]\times\R\times\R)} \left\|F(P_{\le  K}^{x_1} u  )- P_{\le  K}^{x_1} F( u) \right\|_{L_t^{\frac{4}{3}}L_x^1 L_{x_2}^2([0,T]\times\R\times\R)} \lesssim o(K).
\end{equation*}
	
Finally, we turn to the term $I$. Using the symmetric property of the  nonlinearity, we see
\begin{equation*}
\int_{\R}\Im \bigg(\sum_{\substack{n_1,n_2,n_3,n_4,n_5,n\in\N\\n_1-n_2+n_3-n_4+n_5=n}}\overline{P_{\leq K}^{x_1}u} \Pi_n\big(P_{\leq K}^{x_1}u_{n_1}\overline{P_{\leq K}^{x_1}u_{n_2}}P_{\leq K}^{x_1}u_{n_3}\overline{P_{\leq K}^{x_1}u_{n_4}}P_{\leq K}^{x_1}u_{n_5}\big)\bigg)(\widetilde{x_2})d\widetilde{x_2}=0.
\end{equation*}
Thus, we can write
\begin{align*}
&\hspace{4ex}	\int_{\R}\sum_{n\in\N}\Im\Big(\overline{\Pi_n\big(P_{\leq K}^{x_1}u\big)}P_{\leq K}^{x_1}\sum_{\substack{n_1,n_2,n_3,n_4,n_5\in\N\\n_1-n_2+n_3-n_4+n_5=n}}\Pi_n(u_{n_1}\overline{u_{n_2}}u_{n_3}\overline{u_{n_4}}u_{n_5})\Big)\\
&=\int_{\R}\sum_{n\in\N}\Im\Big(\overline{\Pi_n\big(P_{\leq K}^{x_1}u\big)}\bigg(P_{\leq K}\sum_{\substack{n_1,n_2,n_3,n_4,n_5\in\N\\n_1-n_2+n_3-n_4+n_5=n}}\Pi_n(u_{n_1}\overline{u_{n_2}}_{n_3}\overline{u_{n_4}}u_{n_5})\\
&\hspace{38ex}-\Pi_n\big(P_{\leq K}^{x_1}u_{n_1}\overline{P_{\leq K}^{x_1}u_{n_2}}P_{\leq K}^{x_1}u_{n_3}\overline{P_{\leq K}^{x_1}u_{n_4}}P_{\leq K}^{x_1}u_{n_5}\big)\Big)\bigg).
\end{align*}
By using \eqref{est-galilean}, \eqref{N-est2}, \eqref{N-est3} and \eqref{N-est4}, we have  
\begin{align}
\bigg\|\sum_{n\in N}\overline{P_{\geq \frac{K}{32}}^{x_1}\Pi_n(P_{\leq K}^{x_1}u)}\Big(P_{\leq K}&\sum_{\substack{n_1,n_2,n_3,n_4,n_5\in\N\\n_1-n_2+n_3-n_4+n_5=n}}u_{n_1}\overline{u_{n_2}}u_{n_3}\overline{u_{n_4}}u_{n_5}\notag\\
&-\big(P_{\leq K}^{x_1}u_{n_1}\overline{P_{\leq K}^{x_1}u_{n_2}}P_{\leq K}^{x_1}u_{n_3}\overline{P_{\leq K}^{x_1}u_{n_4}}P_{\leq K}^{x_1}u_{n_5}\Big)\bigg\|_{L_{t,x_1,x_2}^1([0,T]\times\R\times\R)}\lesssim 1.\label{est-term-1}
\end{align} 
	
Using the decomposition $u=u^h+u^l$, where $u^l=P_{\leq \frac{K}{32}}^{x_1}u$ and $u^h=P_{\geq \frac{K}{32}}^{x_1}u$.
Similar to \eqref{est-term-1},	we can also obtain the following estimates: 
\begin{align}\label{eq6.44v69}
& \bigg\| \sum\limits_{n\in\N}
\overline{ P_{\le \frac{K}{32}}^{x_1}  \Pi_n(P_{\leq K}^{x_1}u)}  \cdot \sum_{\substack{n_1,n_2,n_3,n_4,n_5\in\N\\n_1-n_2+n_3-n_4+n_5=n}} \mathcal{O}\left( u_{n_1}^h \overline{ u_{n_2}^h}   u_{n_3}^l \overline{u_{n_3}^l } u_{n_5}^l\right)
\bigg\|_{L_{t,x_1,x_2}^{1} ([0,T] \times \mathbb{R}\times\R)} \\
& \  +  \bigg\| \sum\limits_{n\in\N}
\overline{ P_{\le \frac{K}{32}}^{x_1}  \Pi_n(P_{\leq K}^{x_1}u)}  \cdot \sum_{\substack{n_1,n_2,n_3,n_4,n_5\in\N\\n_1-n_2+n_3-n_4+n_5=n}}\mathcal{O}\left( u_{n_1}^h \overline{ u_{n_2}^h}   u_{n_3}^h \overline{u_{n_3}^h } u_{n_5}^h\right)
\bigg\|_{L_{t,x_1,x_2}^{1}  ([0,T] \times \mathbb{R}\times\R)}
\lesssim 1. \notag
\end{align}
 Moreover, the Fourier transform of
\begin{align}
&  P_{\le K}^{x_1} \left( \sum_{\substack{n_1,n_2,n_3,n_4,n_5\in\N\\n_1-n_2+n_3-n_4+n_5=n}} \mathcal{O}\left( u_{n_1}^l \overline{u_{n_2}^l}  u_{n_3}^l  \overline{u_{n_4}^l} u_{n_5}^h\right)\right)   \notag\\
& \quad -  \sum_{\substack{n_1,n_2,n_3,n_4,n_5\in\N\\n_1-n_2+n_3-n_4+n_5=n}} \mathcal{O}\left(u_{n_1}^l \overline{u_{n_2}^l}  u_{n_3}^l  \overline{u_{n_4}^l} u_{n_5}^h \right)\notag\\
&=P_{\le K}^{x_1} \left( \sum_{\substack{n_1,n_2,n_3,n_4,n_5\in\N\\n_1-n_2+n_3-n_4+n_5=n}} \mathcal{O}\left( u_{n_1}^l \overline{u_{n_2}^l}  u_{n_3}^l  \overline{u_{n_4}^l} u_{n_5}^h\right)\right)   \notag\\
& \quad -  \sum_{\substack{n_1,n_2,n_3,n_4,n_5\in\N\\n_1-n_2+n_3-n_4+n_5=n}} \mathcal{O}\left(u_{n_1}^l \overline{u_{n_2}^l}  u_{n_3}^l  \overline{u_{n_4}^l} u_{n_5}^h \right)\label{fourier-1}
\end{align}
is supported on $|\xi|\geq \frac{K}{2}$ with respect to $x_1$. And therefore the Fourier transform of
\begin{align}
& P_{\leq \frac{K}{32}}^{x_1}\Pi_n(P_{\leq K}u)\bigg(  P_{\le K}^{x_1} \left( \sum_{\substack{n_1,n_2,n_3,n_4,n_5\in\N\\n_1-n_2+n_3-n_4+n_5=n}} \mathcal{O}\left( u_{n_1}^l \overline{u_{n_2}^l}  u_{n_3}^l  \overline{u_{n_4}^l} u_{n_5}^h\right)\right)  \notag \\
& \quad -  \sum_{\substack{n_1,n_2,n_3,n_4,n_5\in\N\\n_1-n_2+n_3-n_4+n_5=n}} \mathcal{O}\left(u_{n_1}^l \overline{u_{n_2}^l}  u_{n_3}^l  \overline{u_{n_4}^l} u_{n_5}^h \right)\bigg)\label{fourier-2}
\end{align}
is supported on $|\xi|\geq \frac{K}{4}$(resp. $x_1$).
The Sobolev embedding implies,
\begin{equation}\label{eq6.48v69}
\left\|P_{\leq K}^{x_1} u\right\|^4_{L^8_tL^{\infty}_{x_1}L_{x_2}^2([0,T]\times\R\times\R)} \lesssim K.
\end{equation}
Using the fact that $N(t)\leq 1$ and the compactness condition  \eqref{almost-compact-2}, we have
\begin{equation}\label{eq6.49v66}
\left\|\left(\partial_x-i\xi(t)\right)P_{\le  K} u\right\|_{L^2_{x_1}L_{x_2}^2(\R\times\R)} \lesssim o(K).
\end{equation}
By using integration by parts,  \eqref{eq6.48v69}, and \eqref{eq6.49v66}, we have
\begin{align}\label{eq6.47v69}
&\quad  \int_0^T \int_{\R^4}  a(x_1-\widetilde{x_1})
P_{>\frac{K}{2}}\left( P_{\geq \frac{K}{32}}^{x_1}\Pi_n(u)\cdot\sum_{\substack{n_1,n_2,n_3,n_4,n_5\in\N\\n_1-n_2+n_3-n_4+n_5=n}}\mathcal{O}\big(u_{n_1}^l\overline{u_{n_2}^l}u_{n_3}^l\overline{u_{n_4}^l}u_{n_5}^l\big)\right)(t,\widetilde{x_1},\widetilde{x_2})\\
& \qquad  \cdot \Im\left(\overline{P_{\le  K}^{x_1}  u }\left(\partial_{x_1}-i\xi(t)\right)P_{\le K}^{x_1}  u \right)(t,x_1,x_2)dx_1d\widetilde{x_1}dx_2d\widetilde{x_2}\mathrm{d}t \notag \\
&= \int_0^T\int_{\R^4}   \partial_{x_1} a(x_1-\widetilde{x_1})\frac{\partial_{\widetilde{x_1}}}{\partial_{\widetilde{x_1}}^2}
P_{>\frac{K}{2}}^{x_1}\left( P_{\geq \frac{K}{32}}^{x_1}\Pi_n(u)\cdot\sum_{\substack{n_1,n_2,n_3,n_4,n_5\in\N\\n_1-n_2+n_3-n_4+n_5=n}}\mathcal{O}\big(u_{n_1}^l\overline{u_{n_2}^l}u_{n_3}^l\overline{u_{n_4}^l}u_{n_5}^l\big)\right)(t,\widetilde{x_1},\widetilde{x_2})\notag \\
& \qquad \cdot \Im\left(\overline{P_{\le  K}^{x_1}  u }\left(\partial_{x_1}-i\xi(t)\right)P_{\le K}^{x_1}  u \right)(t,x_1,x_2)dx_1d\widetilde{x_1}dx_2d\widetilde{x_2}dt  \notag \\
&\lesssim \frac{1}{K} \left\|\left(\partial_{x_1}-i\xi(t)\right) P_{\le K}^{x_1} u \right\|_{L_t^\infty L^2_{x_1}L_{x_2}^2([0,T]\times\R\times\R) }  \left\| P_{\le  K}^{x_1} u\right\|^4_{L^8_tL^{\infty}_{x_1}L_{x_2}^2([0,T]\times\R\times\R)} \lesssim o(K). \notag
\end{align}
In the above inequality, we also use the fact that the support property of \eqref{fourier-1} and \eqref{fourier-2}.
	
By \eqref{est-term-1}, \eqref{eq6.44v69}, and \eqref{eq6.47v69}, we have
\begin{align*}
I &= 2\int_{0}^{T}\int_{\R^4}a(x_1-\widetilde{x_1})\Im(\overline{P_{\leq K}^{x_1}u}P_{\leq K}^{x_1}F(u))(t,\widetilde{x_1},\widetilde{x_2})\Im(\overline{P_{\leq K}^{x_1}u}(\partial_{x_1}-i\xi(t))P_{\leq K}^{x_1}u)(t,x_1,x_2)dx_1d\widetilde{x_1}dx_2d\widetilde{x_2}dt\\
&\lesssim o(K).
\end{align*}
The proof of Theorem \ref{frequency-localized} is now complete.
\end{proof}



\section{The proof of Theorem \ref{Thm2}}\label{sec:Thm2}
In this section, we give the proof for Theorem \ref{Thm2}.

Recall that if Theorem \ref{Thm2} fails, then Theorem \ref{almost-periodic} implies that there exist a minimal blow-up solution $u\in C^0_tL^2_{x_1}\mathcal{H}^1_{x_2}(I\times\R\times\R)$ with maximal lifespan $I\supset [0,\infty)$ and three parameters $N(t):I\to (0,1]$, $x_1(t):I\to\R$, $\xi(t):I\to\R$ such that
\begin{align*}
\left\{\frac{1}{N(t)^\frac{1}{2}}e^{-ix_1\cdot\xi(t)}u\left(t,\frac{x_1-x_1(t)}{N(t)},x_2\right):t\in I\right\}
\end{align*}
is pre-compact in $L_{x_1}^2L_{x_2}^2(\R\times\R)$. Without loss of generality, we can further assume that $N(0)=1$.  Thus, we can consider the following two scenarios respectively, that is
\begin{enumerate}
\item(Rapid frequency cascade) $$\int_{0}^\infty N(t)^3dt=K<\infty,$$
\item(Quasi-soliton) $$\int_{0}^\infty N(t)^3dt=K=\infty.$$
\end{enumerate}
Our goal of this subsection is to prove Theorem \ref{Thm2} by reducing it to  the preclusion  of these two type solutions. The main ingredients are the frequency-localized Morawetz estimates and the additional regularity of the critical elements.

\subsection{Rapid frequency cascade}

\begin{theorem}If $u$ is an almost-periodic solution to (DCR) system $\eqref{DCR-2}$ with the form of Theorem \ref{almost-periodic} and obeys 
\begin{align*}
\int_{0}^{\infty}N(t)^3dt=K<\infty,
\end{align*}
then $u=0$.
\end{theorem}

\begin{proof}
First, we will show that the almost-periodic solution $u$ enjoys the additional regularity:
\begin{align*}
\big\|u(t)\big\|_{\dot{H}_{x_1}^5L_{x_2}^2(\R\times\R)}\lesssim\eta_3^{-5}K^5.
\end{align*}
We assume that there exists $T>0$ such that 
\begin{align*}
\int_{0}^{T}\|u(t)\|_{L_{x_1}^6L_{x_2}^2(\R\times\R)}dt=2^{k_0}.
\end{align*}
Using the following scaling transform,
\begin{align*}
u_\lambda(t,x_1,x_2)=\lambda u(\lambda^2t,\lambda x_1,x_2),
\end{align*}
where $\lambda=\eta_3\frac{2^{k_0}}{K}$, thus
\begin{align*}
\int_{0}^{\frac{T}{\lambda^2}}N_\lambda(t)^3dt=\eta_32^{k_0}.
\end{align*}
The long-time Strichartz estimate \eqref{longtimestrichartz} gives 
\begin{align}\label{qaz}
\|u_\lambda\|_{\tilde{X}_{k_0}([0,\frac{T}{\lambda^2}]\times\R\times\R)}\leq C_0,
\end{align}
for some constant $C_0>0$ .
Using scaling and \eqref{remark-small}, we obtain $N(t)\leq \eta_32^{k_0}$,
then	 \eqref{xi(t)} yields that 
\begin{align*}
|\xi(t)|\leqslant 2^{-20}\eta_{1}^{-\frac12}\eta_32^{k_0}.
\end{align*}
Without loss of generality, we can also assume that $K\geq1$.
\begin{align*}
\|P_{\geq N}^{x_1}u_\lambda\|_{U_\Delta^2(I;L_{x_1}^2L_{x_2}^2)}\sim\sum_{j:2^j\geq N}\|P_j^{x_1}u_\lambda\|_{U_\Delta^2(I;L_{x_1}^2L_{x_2}^2)},
\end{align*}
where $I=[0,\frac T{\lambda^2}]$.  
By the Duhamel principle and 	 Intermediate  Theorem \ref{intermediate}, we have
\begin{align}\label{qazw}
\left\|P_{>N}^{x_1}u_{\lambda}\right\|_{U_{\Delta}^{2}(I,L_{x_1}^2L_{x_2}^2)}& \lesssim\inf_{t\in I}\left\|P_{>N}^{x_1}u_{\lambda}(t)\right\|_{L_{x_1}^{2}L_{x_2}^2(I\times\R\times\R)} \notag\\
&+\epsilon_{2}^{1/2}C_{0}^{3}\big\Vert P_{>\frac{N}{64}}^{x_1}u_{\lambda}\big\Vert_{U_{\Delta}^{2}(I,L_{x_1}^{2}L_{x_2}^2(I\times\R\times\R))}.
\end{align}

Since the equation is mass-critical,  $U_\Delta^2$ and $L^2$ norm are invariant under the above scaling symmetry. Thus, \eqref{qaz} and \eqref{qazw} yield that for $N\geq \eta_{3}^{-1}K$,
\begin{align}\label{qwertyui}
\big\|P_{> \eta_3^{-1}K}^{x_1}u\big\|_{L_{x_1}^2L_{x_2}^2(\R\times\R)}\leq C_0,
\end{align}
and
\begin{align}\label{qwerty}
\left\|P_{>N}^{x_1}u\right\|_{U_{\Delta}^{2}([0,T],L_{x_1}^2L_{x_2}^2)}& \lesssim\inf_{t\in[0,T] }\left\|P_{>N}^{x_1}u(t)\right\|_{L_{x_1}^{2}L_{x_2}^2(\R\times\R)} \notag\\
&+\epsilon_{2}^{1/2}C_{0}^{3}\big\Vert P_{>\frac{N}{64}}^{x_1}u\big\Vert_{U_{\Delta}^{2}([0,T],L_{x_1}^{2}L_{x_2}^2(\R\times\R))}.
\end{align}

Notice that $\varlimsup\limits_{t\to\infty}N(t)=0$ and $|\xi(t)|\leqslant 2^{-20}\varepsilon_{3}^{-\frac12}K$, we obtain
\begin{align}\label{qwertyu}
\lim_{T\to\infty}\inf_{t\in[0,T] }\|P_{\geq N}^{x_1}u\|_{L_{x_1}^2L_{x_2}^2(\R\times\R)}\to 0.
\end{align}
Thus by \eqref{qwerty} and \eqref{qwertyu}, we have
\begin{align}\label{qwerty2}
\left\|P_{>N}^{x_1}u\right\|_{U_{\Delta}^{2}([0,\infty),L_{x_1}^2L_{x_2}^2)}\leq \epsilon_{2}^{1/2}C_{1}\cdot C_0^3\big\Vert P_{>\frac{N}{64}}^{x_1}u\big\Vert_{U_{\Delta}^{2}([0,\infty),L_{x_1}^{2}L_{x_2}^2(\R\times\R))},
\end{align}
where $C_1$ is the implicit constant in \eqref{qwerty}.

Therefore, if we choose sufficiently small $\eta_2$ such that $\eta_2^{1/2}C_1C_0^3\leq 2^{-100}$, then \eqref{qwertyui}, \eqref{qwerty2},  Bernstein's inequality and the embedding $U^2_{\Delta}([0,\infty),L_{x_1}^{2}L_{x_2}^2(\R\times\R))\hookrightarrow L_t^{\infty}\dot{H}_{x_1}^{5}L_{x_2}^2([0,\infty)\times\mathbb{R}\times\mathbb{R})$ imply that 
\begin{align}\label{ert}
\|P_{> N}^{x_1}u\|_{L_t^{\infty}\dot{H}_{x_1}^{5}L_{x_2}^2([0,\infty)\times\mathbb{R}\times\mathbb{R})}&\lesssim 	\sum_{k\geq 0}\|P_{2^{6k} N< \cdot < 2^{6(k+1)}N}^{x_1}u\|_{L_t^{\infty}\dot{H}_{x_1}^{5}L_{x_2}^2([0,\infty)\times\mathbb{R}\times\mathbb{R})}\notag\\
&\lesssim \sum_{k\geq 0} 2^{30(k+1)}N^5 \|P_{2^{6k} N< \cdot < 2^{6(k+1)}N}^{x_1}u\|_{L_t^{\infty}L^2_{x_1}L_{x_2}^2([0,\infty)\times\mathbb{R}\times\mathbb{R})}\notag\\
&\lesssim \sum_{k\geq 0} 2^{30(k+1)}N^5 2^{-100k}\notag\\
&\lesssim N^5.
\end{align}
	
On the other hand, by Bernstein's inequality, 
\begin{align}\label{ertt}
\|P_{\leq N}^{x_1}u\|_{L_t^{\infty}\dot{H}_{x_1}^{5}L_{x_2}^2([0,\infty)\times\mathbb{R}\times\mathbb{R})}\lesssim N^5\|P_{\leq N}^{x_1}u\|_{L_t^{\infty}L^2_{x_1}L_{x_2}^2([0,\infty)\times\mathbb{R}\times\mathbb{R})}\lesssim N^5.
\end{align}
Choosing $N\sim \eta_3^{-1}K$, $N\geq \eta_3^{-1}K$, then by \eqref{ert} and \eqref{ertt},
\begin{align*}
\|u\|_{L_t^{\infty}\dot{H}_{x_1}^{5}L_{x_2}^2([0,\infty)\times\mathbb{R}\times\mathbb{R})}\lesssim\varepsilon_{3}^{-5}K^5.
\end{align*} 
	
Next, we show the impossibility of the non-trivial rapid frequency cascade solution. We choose a sequence $\{t_n\}$ satisfying $N(t_n)\to 0$ and $\lim\limits_{n\to\infty}\xi(t_n)=\xi_\infty\leq 2^{-20}\varepsilon_{3}^{-\frac12}KA$. By the Galilean transformation with respect to $x_1$-direction, we can shift $\xi_\infty$ to the origin. Nevertheless we still have
\begin{align*}
\|u\|_{L_t^{\infty}\dot{H}_{x_1}^{5}L_{x_2}^2([0,\infty)\times\mathbb{R}\times\mathbb{R})}\lesssim\varepsilon_{3}^{-5}K^5.
\end{align*} 
By using the compactness of the almost-periodic solution and Lemma \ref{GNIN}, we have
\begin{align*}
\left\|u(t) \right\|_{\dot{H}_x^1 L_{x_2}^2(\R\times\R)} & \le  \left\|P_{\xi(t), \ge C(\eta) N(t)} u(t) \right\|_{\dot{H}_x^1 L_{x_2}^2(\R\times\R)} + \left\|P_{\xi(t), \le C(\eta) N(t)} u(t) \right\|_{\dot{H}_x^1 L_{x_2}^2(\R\times\R)}\\
& \lesssim \left\|P_{\xi(t), \ge C(\eta) N(t)} u(t) \right\|_{L^2_{x_1} L_{x_2}^2}^\frac45 \left\|P_{\xi(t), \ge C(\eta) N(t)} u(t) \right\|_{\dot{H}^5_{x_1} L_{x_2}^2(\R\times\R)}^\frac15\\
&\hspace{3ex} + (C(\eta) N(t) + |\xi(t)|) \left\| u(t)\right\|_{L^2_{x_1} L_{x_2}^2(\R\times\R)}\\
& \lesssim \left\|P_{\xi(t), \ge C(\eta) N(t)} u(t) \right\|_{L^2_{x_1} L_{x_2}^2(\R\times\R)}^\frac45 \left\|u(t) \right\|_{\dot{H}^5_{x_1} L_{x_2}^2(\R\times\R)}^\frac15 \\
&\hspace{3ex}+ \left(C(\eta) N(t) + |\xi(t)|\right) \left\| u(t) \right\|_{L^2_{x_1} L_{x_2}^2(\R\times\R) }\\
& \lesssim \eta^\frac45 \epsilon_3^{-1} K +  C(\eta) N(t) + |\xi(t)|,
\end{align*}
for any $\eta >0$.
Therefore, we have
\begin{align*}
\|u(t_n) \|_{\dot{H}_x^1 L_{x_2}^2(\R\times\R)} \to 0  \text{ as $n \to \infty$}.
\end{align*}
By the Gagliardo-Nirenberg inequality and nonlinear estimates(cf. Remark \ref{nonlinear}), we deduce that
\begin{align*}
E_S(u(t_n)) \lesssim \|u(t_n)\|_{\dot{H}_x^1 L_{x_2}^2(\R\times\R)}^2 + \|u(t_n) \|_{L_x^2 L_{x_2}^2(\R\times\R)}^4 \|u(t_n)\|_{\dot{H}_x^1 L_{x_2}^2(\R\times\R)}^2 \to 0 \text{ as } n\to \infty,
\end{align*}
thus by the energy conservation law, $E_S(u(t))=0$, which implies $u \equiv 0$.
This excludes the rapid frequency cascade case.
\end{proof}

\subsection{Quasi-soliton case}	 
In this subsection, we will preclude the possibility of the quasi-soliton.

\begin{theorem}\label{soliton}
If $u$ is an almost-periodic solution to (DCR) system $\eqref{DCR-2}$ with the form of Theorem \ref{almost-periodic} and obeys 
\begin{align*}
\int_{0}^{\infty}N(t)^3dt=K=\infty,
\end{align*}
then $u=0$.
\end{theorem}

\begin{proof}
In this case, by H{\"o}lder, Gagliardo-Nirenberg, interpolation, Sobolev, and also the definition of the H\"older continuity in \cite{Dodson-d=1}, we have
\begin{align*}
& \int_{|x_1-x_1(t)|\leq \frac{C\Big(\frac{ \|u\|_{L^2_{x_1,x_2} }^2}{100}\Big)}{N(t)}}\int_{\R}|P_{\le  10\eta_3^{-1}K}^{x_1} u(t,x_1,x_2)|^2 dx_1dx_2 \\
 \lesssim& \int_{|x_1-x_1(t)|\leq \frac{C\Big(\frac{ \|u\|_{L_{x_1,x_2}^2}^2}{100}\Big)}{N(t)}}\int_{\R} |x_1- x_1(t)|^\frac12 \frac{ |P_{\le  10\eta_3^{-1}K}^{x_1} u(t,x_1,x_2) |^2 - |P_{\le  10\eta_3^{-1}K}^{x_1} u(t, x_1(t),x_2)|^2}{ | x_1 - x_1(t)|^\frac12} \,dx_1dx_2 \\
&+\int_{|x_1-x_1(t)|\leq \frac{C\Big(\frac{ \|u\|_{L_{x_1,x_2}^2}^2}{100}\Big)}{N(t)}}\int_{\R} |P_{\le  10\eta_3^{-1}K}^{x_1} u(t,x_1(t),x_2)|^2 \,dx_1dx_2\\
\lesssim& \Bigg(\frac{C \Big(\frac{\|u\|_{L_{x_1,x_2}^2}^2}{100}\Big)}{N(t)}\Bigg)^{\frac{3}{2}}\bigg\|\int_{\R}|P_{\le  10\eta_3^{-1}K}^{x_1} u(t,x_1,x_2)|^2dx_2\bigg\|_{\dot{C}_{x_1}^{\frac{1}{2}}(\mathbb{R})} +  \frac{\|u\|_{L^2 }^2}{100}\\
 \lesssim& \Bigg(\frac{C\Big(\frac{ \|u\|_{L^2 l^2}^2}{100}\Big)}{N(t)} \Bigg)^{\frac{3}{2}}\left\|\int_{\R}\partial_{x_1} \left( |P_{\le 10\eta_3^{-1}K}^{x_1} u(t,x_1,x_2)|^2 \right)dx_2 \right\|_{L^2_{x_1}(\mathbb{R})} +  \frac{\|u\|_{L^2 }^2}{100},
\end{align*}
where $\dot{C}_{x_1}^\frac12(\mathbb{R})$ is the homogeneous H\"older norm in $x_1$ direction.
	
By Theorem \ref{frequency-localized} and \eqref{eq6.49v66}, we have
\begin{equation*}
\Big\|\int_{\R}\partial_{x_1} \left( |P_{\le 10\eta_3^{-1}K}^{x_1} u(t,x_1,x_2)|^2  \right)dx_2 \Big\|^2_{L^2_{t,x_1}([0,T]\times \mathbb{R})} \lesssim o(K),
\end{equation*}
therefore, for $K \ge C\Big(\frac{ \left\|u \right\|_{L_{x_1,x_2}^2}^2}{100}\Big)$, we have
\[\frac{\left\|u\right\|_{L^2_{x_1,x_2}}^2}{2} \le
\int_{|x_1-x_1(t)|\leq \frac{C\Big(\frac{ \|u \|_{L^2_{x_1,x_2} }^2}{100}\Big)}{N(t)}} \int_{\R}\left|P_{\le  10\eta_3^{-1}K}^{x_1} u
(t,x_1,x_2)\right|^2 dx_1dx_2. \]
Thus
\begin{align}
\|u\|^4_{L^2_{x_1,x_2}}K   &\sim  \|u \|^4_{L^2_{x_1,x_2}}  \int_0^T N(t)^3 dt\notag\\
&\lesssim \int_0^T N(t)^3 \bigg(\int_{|x_1-x_1(t)|\leq \frac{C\Big(\frac{\|u\|_{L_{x_1,x_2}^2 }^2}{100}\Big)}{N(t)}}\int_{\R}\left|P_{\le  10\eta_3^{-1}K}^{x_1}u(t,x_1,x_2)\right|^2dx_1dx_2 \bigg )^2 dt \notag\\
& \lesssim \bigg\|\int_{\R}\partial_{x_1} \left( |P_{\le 10\eta_3^{-1} K}^{x_1} u(t,x_1,x_2)|^2 \right)dx_2 \bigg\|^2_{L^2_{t,x_1}([0,T]\times \mathbb{R})} \lesssim o(K).\label{qazq}
\end{align}
When  $K$ is sufficiently large, \eqref{qazq} implies $u=0$, which is a contradiction. Thus we can finish  the proof of Theorem \ref{soliton} and Theorem \ref{Thm2}. 
\end{proof}

\noindent\textbf{Acknowledgement:}  We appreciate Professor Xing Cheng for helpful discussions and  beneficial suggestions on this paper. 
This work is supported by National key R\&D program of China: 2021YFA1002500.  Z. Zhao was supported by the NSF grant of China (No. 12101046, 12271032) and the Beijing Institute of Technology Research Fund Program for Young Scholars.
J. Zheng was supported by  NSF grant of China (No. 12271051) and Beijing Natural Science Foundation 1222019.



\begin{thebibliography}{99}
	
	
\bibitem{Antonelli}  A. Antonelli, R. Carles, and J. D. Silva, \emph{Scattering for nonlinear Schr\"odinger equation under partial harmonic confinement},
Commun. Math. Phys. {\bf 334} (2015), 367-396.

\bibitem{AC} A. Ardila and R. Carles, \emph{Global dynamics below the ground states for NLS under partial harmonic confinement}, Comm. Res. Sci. {\bf 19} (2021), 993-1032.
	
\bibitem{Barron} A. Barron, \emph{On global-in-time Strichartz estimates for the semiperiodic Schrödinger equation}, Anal. \& PDE. {\bf 14} (2021), 1125-1152.
	
\bibitem{Begout-Vargas} P. B\'egout and A. Vargas, \emph{Mass concentration phenomena for the $L^2$-critical nonlinear Schr\"odinger equation,} Trans. Amer. Math. Soc. \textbf{359} (2007), 5257-5282.

\bibitem{BBJV} J. Bellazzini, N. Boussaid, L. Jeanjean and N. Visciglia, \emph{Existence and stability of standing waves for supercritical NLS with a partial confinement,} Comm. Math. Phys. {\bf 353} (2017), 229-251.
	
\bibitem{Bourgain-JAMS} J.~Bourgain, {\em Global well-posedness of defocusing critical nonlinear {S}chr\"{o}dinger equation in the radial case}, J. Amer. Math. Soc., \textbf{12} (1999), pp.~145--171.
	
\bibitem{Carles1} R. Carles, \emph{Remarks on nonlinear Schr\"odinger equations with harmonic potential}, Ann. Henri Poincar\'e, {\bf3} (2002), 757-772.
	
	
\bibitem{Carles-Keraani} R. Carles and S. Keraani,\emph{ On the role of quadratic oscillations in nonlinear Schr\"odinger equations. II. The $L^2$-critical case}, Trans. Amer. Math. Soc. \textbf{359} (2007), 33-62.
	
\bibitem{Cazenave} T. Cazenave, \emph{Semilinear Schr\"odinger equations}, Courant Lecture Notes in Mathematics, \textbf{10}. New York University, Courant Institute of Mathematical Sciences, New York; American Mathematical Society, Providence, RI, 2003.
	
\bibitem{Chapert} A. Chapert, A weakly turbulent solution to the cubic nonlinear harmonic oscillator on $\R^2$ perturbed by a real smooth potential decayingto zero at infinity, Commun. Part. Differ. Equa. 49 (2024), 185-216.
	
\bibitem{Cheng-Guo-Guo-Liao-Shen} X. Cheng, C. Guo, Z. Guo, X. Liao and J. Shen, \emph{Scattering of the three-dimensional cubic nonlinear Schr\"odinger equation with partial harmonic potentials}, arxiv:2105.02515,  to appear in Anal. \& PDE.
	
	
\bibitem{Cheng-Guo-Yang-Zhao} X. Cheng, Z. Guo, K. Yang and L. Zhao, \emph{On scattering for the  cubic defocusing nonlinear Schr\"odinger equation
on the waveguide $\R^2\times \Bbb{T}$,} Rev. Mat. Iberoam. \textbf{36} (2020), 985-1011.  
	
\bibitem{Cheng-Guo-Zhao} X. Cheng, Z. Guo and Z. Zhao, \emph{ On scattering for the defocusing quintic nonlinear Schr\"odinger equation on the
two-dimensional cylinder,} SIAM J. Math. Anal. \textbf{52} (2020), 4185-4237.
	
\bibitem{Cheng-Zhao-Zheng} X. Cheng Z. Zhao and J. Zheng, \emph{Well-posedness for energy-critical NLS on waveguide manifold}, Jour.  Math. Anal. Appl. {\bf494} (2021), no.2, 124654. 
	
\bibitem{CKSTT-Ann} J. Colliander, M. Keel, G. Staffilani, H. Takaoka, T. Tao, \emph{Global well-posedness and scattering for the energy-critical nonlinear Schr\"odinger equations in $\R^3$}, Ann. of Math. \textbf{167} (2008), 767-865.
	
\bibitem{CKSTT-Invent} J. Colliander, M. Keel, G. Staffilani, H. Takaoka and T. Tao, \emph{Transfer of energy to high frequencies in the cubic defocusing nonlinear Schr\"odinger equation,} Invent. Math. {\bf181} (2010), 39-113.
	
\bibitem{Deng-Su-Zheng} M. Deng, X. Su and J. Zheng,  Growth of Sobolev norms for 2D cubic nonlinear Schr\"odinger equation with partial harmonic potential, arXiv:2304.02995.
	
\bibitem{Deng} Y. Deng, C. Fan, K. Yang, Z. Zhao and J. Zheng, On bilinear Strichartz estimates on waveguides with applications, Journal of Functional Analysis, {\bf287} (2024) 110595.
	
\bibitem{Dodson-d=3} B. Dodson, \emph{ Global well-posedness and scattering for the defocusing {$L^{2}$}-critical nonlinear {S}chr{\"o}dinger equation when $d \geq 3$}, Journal of the American Mathematical Society, \textbf{25} (2012), 429-463.
	
\bibitem{dodson-focusing} B. Dodson, \emph{Global well-posedness and scattering for the mass critical nonlinear Schr\"odinger equation with mass below the mass of the ground state}, Adv. Math. {\bf285} (2015), 1589-1618.

\bibitem{Dodson-d=1} B. Dodson, \emph{Global well-posedness and scattering for the defocusing, $L^2$-critical, nonlinear Schr\"odinger equation when $d = 1$},  Amer. J. Math. {\bf 138} (2016), no. 2, 531-569.	
	
	
\bibitem{dodson-d=2} B. Dodson, \emph{Global well-posedness and scattering for the defocusing, $L^{2}$-critical nonlinear Schr{\"o}dinger equation when $d = 2$}, Duke Math. J. \textbf{165} (18), 3435-3516.
	
\bibitem{D5} B. Dodson, \emph{ Global well - posedness and scattering for the focusing, energy-critical nonlinear Schr{\"o}dinger problem in dimension $d = 4$ for initial data below a ground state threshold}, Ann. Scient. \'Ec. Norm. Sup.  \textbf{52} (2019),  139-180.
	
\bibitem{equi}	 J. Dziuba\'nski and P. Glowacki, \emph{Sobolev spaces related to Schr\"odinger operators with polynomial potentials,} Math. Z. \textbf{262} (2009), 881-894.
	
\bibitem{FO} R. Fukuizumi and M. Ohta, \emph{Instability of standing waves for nonlinear Schr\"odinger equations with potentials}, Differential Integral Equations {\bf 16 } (2003), no. 6, 691-706.
	
\bibitem{Grafakos} L. Grafakos, \emph{Classical Fourier analysis}, Graduate Texts in Mathematics, vol. 249 (2008), Berlin: Springer.
	
\bibitem{Hadec-Herr-Koch} M. Hadec, S. Herr and H. Koch, \emph{Well-posedness and scattering for KP-II equation in a critical space,} Ann.  Inst. H. Poincar\'e Anal. Non Lin\'eaire {\bf 26} (2009), 917-941.
	
\bibitem{Hani-Pausader} Z. Hani and B. Pausader, \emph{On scattering for the quintic defocusing nonlinear Schr\"odinger equation on $\mathbb{R}\times \mathbb{T}^2$}, Comm. Pure Appl. Math. {\bf67} (2014), no. 9, 1466-1542.
	
\bibitem{Hani-Pausader-Tzvetkov-Visciglia} Z. Hani, B. Pausader, N. Tzvetkov, and N. Visciglia, \emph{Modified scattering for the cubic Schr\"odinger equation on product spaces and applications}, Forum of Mathematics, PI. (2015), Vol. 3, 1-63.
	
\bibitem{Hani-Thomann} Z. Hani and L. Thomann, \emph{Asymptotic behavior of the nonlinear Schr\"odinger equation with harmonic trapping}, Comm. Pure Appl. Math. {\bf69} (2016), no. 9, 1727-1776.
	
\bibitem{Nakanishi} S.~Ibrahim, N.~Masmoudi and K.~ Nakanishi, {\em Scattering threshold for the focusing Klein-Gordon equation}. Analysis and PDE \textbf{4} (2011),  405-460.
	
\bibitem{Ionescu-Pausader} A. D. Ionescu and B. Pausader, \emph{Global well-posedness of the energy-critical defocusing NLS on $\mathbb{R}\times \mathbb{T}^3$}, Comm. Math. Phys. {\bf 312} (2012), no. 3, 781-831.
	
\bibitem{Jao-CPDE} C. Jao, \emph{The energy-critical quantum harmonic oscillator}, Comm. Partial Differential Equations {\bf 41} (2016), no. 1, 79-133.
	
\bibitem{Jao-DCDS} C. Jao, \emph{Energy-critical NLS with potentials of quadratic growth}, Discrete Contin. Dyn. Syst. {\bf38} (2018), no. 2, 563-587.
	
\bibitem{Jao-APDE} C. Jao, \emph{Refined mass-critical Strichartz estimates for Schr\"odinger operators}, Anal. \& PDE {\bf 13 } (2020), no. 7, 1955-1994.
	
	
\bibitem{Jao-Killip-Visan-1D} C. Jao, R. Killip, and M. Visan, \emph{Mass-critical inverse Strichartz theorems for 1d Schr\"odinger operators}, Rev. Mat. Iberoam. {\bf 35 } (2019), no. 3, 703-730.
	
\bibitem{JP} C. Josserand and Y. Pomeau, \emph{Nonlinear aspects of the theory of Bose-Einstein condensates}, Nonlinearity {\bf 14}(5), R25-R62(2001).
	
\bibitem{Kenig-Merle} C. E. Kenig and F. Merle, \emph{Global well-posedness, scattering and blow-up for the energy-critical focusing non-linear wave equation}, Acta Math. {\bf201} (2008), no. 2, 147-212.
	
	
\bibitem{KTV2009} R. Killip, T. Tao, and M. Visan, The cubic nonlinear Schr\"odinger equation in two dimensions with radial data. J. Eur. Math. Soc., 11 (2009), 1203-1258.
	
	
\bibitem{Killip-Visan-note} R. Killip, M. Visan, \emph{Nonlinear Schr\"odinger equations at critical regularity}. Proceedings for the Clay summer school ``Evolution Equations'', Eidgen\"ossische technische Hochschule, Z\"urich, 2008.
	
\bibitem{Killip-Visan-AJM} R. Killip  and  M. Visan, \emph{The focusing energy-critical nonlinear Schr\"odinger equation in dimensions five and higher,} Amer. J. Math., {\bf132} (2010) , 361-424. 
	
\bibitem{KVZ2008} R. Killip, M. Visan, and X. Zhang, The mass-critical nonlinear Schr\"odinger equation with radial data in dimensions three and
higher. Anal. \& PDE, {\bf1} (2008) 229-266. 
	
\bibitem{Killip-Visan-Zhang} R. Kllip, M. Visan, and X. Zhang, \emph{Energy-critical NLS with quadratic potentials}, Comm. Partial Differential Equations {\bf 34} (2009), no. 10-12, 1531-1565.
	
\bibitem{Koch-Tataru} H. Koch, D. Tataru,  \emph{Dispersive estimates for principally normal pseudodifferential operators,} Comm. Pure Appal. Math. {\bf 58} (2005), 217-284.
	
\bibitem{Koch-Tataru2} H. Koch and D. Tataru, \emph{$L^p$ eigenfunction bounds for the Hermite operator}, Duke Math. J. {\bf 128 } (2005), no. 2, 369-392.
	
\bibitem{Koch-Tataru-Visan-book} H. Koch, D. Tataru, and M. Visan, \emph{Dispersive equations and nonlinear waves. Generalized Korteweg-de Vries, nonlinear Schr\"odinger, wave and Schr\"odinger maps}, Oberwolfach Seminars, {\bf 45}. Birkh\"auser/Springer, Basel, 2014. xii+312 pp.

\bibitem{Luo-normalized} Y. Luo, \emph{Normalized ground states and threshold scattering for focusing NLS on $\R^d\times\T$ via semivirial-free geometry,} Preprint, arxiv: 2205. 04969.
	
\bibitem{Luo-JMPA} Y. Luo, \emph{On long time behavior of the focusing energy-critical NLS on $\R^d\times\T$ via semivirial-vanishing geometry}, J. Math. Pures Appl. {\bf177} (2023), 415-454.


\bibitem{Luo-MA} Y. Luo, \emph{Sharp scattering for focusing intercritical NLS on high-dimensional waveguide manifolds,} Math. Ann. {\bf389} (2024),  63-83.

\bibitem{Masmoudi-Nakanishi} N. Masmoudi and K. Nakanishi, \emph{	From nonlinear Klein-Gordon equation to a system of coupled nonlinear Schr\"odinger equations}, Math. Ann. {\bf 324} (2002), 359-389.
	
\bibitem{Merle-Vega} F. Merle and L. Vega, \emph{Compactness at blow-up time for $L^2$ solutions of the critical nonlinear Schr\"odinger equation in 2D}, Internat. Math. Res. Notices (1998), no. 8, 399-425.
	
\bibitem{Pitaeskii} L. Pitaevskii and S. Stringari, \emph{Bose-Einstein condensation}, International Series of Monographs on Physics, vol. {\bf 116}. The Clarendon Press Oxford University Press, Oxford (2003).
	
\bibitem{Planchon-Vega} F. Planchon and L. Vega, \emph{Bilinear virial identities and applications}, Ann. Sci. \'Ec. Norm. Sup\'er. (4) {\bf 42} (2009), no. 2, 261-290.
	
\bibitem{Ryckman-Visan} E. Ryckman  and  M. Visan , \emph{Global well-posedness and scattering for the defocusing energy-critical nonlinear Schr\"odinger equation in ${{\mathbb{R}}^{1+4}}$}, Amer. J. Math. {\bf 129} (2007) , 1-60. 
	
\bibitem{Schneider} T. Schneider, \emph{Nonlinear optics in telecommunications}, Springer, Berlin, 2004.
	
\bibitem{Stefanov} M. Stanislavova and A. Stefanov, Ground states for the nonlinear Schr\"odinger equation under a general trapping potential,  J. Evol. Equ. {\bf21} (2021), 671-697.
	
\bibitem{Su-Wang-Xu1} X. Su, Y. Wang and G. Xu, \emph{Riesz transforms and Sobolev spaces associated to the partial harmonic oscillator}, preprint, arxiv: 2207.10461.
	
\bibitem{Su-Wang-Xu2} X. Su, Y. Wang and G. Xu, \emph{ A Mikhlin-type multiplier theorem for the partial harmonic oscillator}, Forum  Math. {\bf35} (2023), no. 3, 831-841.
	
\bibitem{Tao-bilinear} T. Tao, \emph{A sharp bilinear restriction estimate for paraboloids}, Geom. Funct. Anal. \textbf{13} (2003), 1359-1384.
	
\bibitem{Tao-book} T. Tao, \emph{Nonlinear dispersive equations: local and global analysis,} CBMS Regional Conference Series in Mathematics, \textbf{106}. American Mathematical Society, Provident, R.I., 2006.

\bibitem{Tao-Visan-Zhang} T. Tao, M. Visan, and X. Zhang, \emph{Global well-posedness and scattering for the defocusing mass-critical nonlinear Schr\"odinger equation for radial data in high dimensions}, Duke Math. J. {\bf 140 }  (2007),  165-202.
	
\bibitem{Tao-Visan-Zhang2} T. Tao, M. Visan, and X. Zhang, \emph{Minimal-mass blowup solutions of the mass-critical NLS}, Forum Math. {\bf20} (2008),
	881-919.
	
	
\bibitem{TV} N. Tzvetkov and N. Visciglia, \emph{Well-posedness and scattering for NLS  on $\R^d\times\Bbb{T}$ in the energy space}, Rev.  Mat. Iberoam. {\bf 32} (2016), no. 4, 1163-1188.
	
\bibitem{Visan} M. Visan , \emph{The defocusing energy-critical nonlinear Schr\"odinger equation in higher dimensions,} Duke Math. J., {\bf 138} (2007) , 281-374. 
	
\bibitem{YZ} K. Yang and L. Zhao, \emph{Global well-posedness and scattering for mass-critical, defocusing, infinite dimensional vector-valued resonant nonlinear Schr\"odinger system}, SIAM J. Math. Anal. {\bf50} (2018), no. 2, 1593-1655.
	
\bibitem{Yu}X. Yu, \emph{ Global well-posedness and scattering for the defocusing  $\dot{H}^\frac12$-critical nonlinear Schr\"odinger equation in 
$\R^2$}, Anal. \& PDE, {\bf14} (2021), 2225-2268.
	
\bibitem{Z} J. Zhang, \emph{Sharp threshold for blowup and global existence in nonlinear Schr\"odinger equations under a harmonic potential}, Comm. Partial Differential Equations {\bf 30} (2005), no. 10-12, 1429-1443.
	
\bibitem{Z3} J. Zhang, \emph{Sharp threshold of global existence for nonlinear Schr\"odinger equation with partial confinement}, Nonlinear Anal. {\bf 196 } (2020), 111832.
	
\bibitem{Zhao-d=4} Z. Zhao, \emph{Global well-posedness and scattering for the defocusing cubic Schr\"odinger equation on waveguide $\mathbb{R}^2 \times \mathbb{T}^2$},  J. Hyperbolic Differ. Equ. {\bf16} (2019), no. 1, 73-129.
	
	
\bibitem{Zhao-JDE} Z. Zhao, \emph{On scattering for the defocusing nonlinear Schr\"odinger equation on waveguide $\R^m\times\Bbb{T}$(when $m=2,3$),} J. Diff. Equ. {\bf 275} (2021), 598-637. 
	
\bibitem{Zhao-Zheng} Z. Zhao, J. Zheng, Long time dynamics for defocusing cubic nonlinear Schr\"{o}dinger equations on three dimensional product space. SIAM J. Math. Anal. {\bf53} (2020), 3644-3660.

\end{thebibliography}

\end{document}